\documentclass[
	a4paper,
	12pt,
	twoside,
	openright,
	oldfontcommands,
]{memoir}

\setlrmarginsandblock{1.1in}{1.1in}{*}   
\setulmarginsandblock{1.5in}{1.8in}{*}   
\checkandfixthelayout

\createmark{chapter}{left}{shownumber}{}{.\space}
\createmark{section}{right}{shownumber}{}{ \space}

\makeevenhead{ruled}{\slshape\leftmark}{}{}
\makeoddhead{ruled}{}{}{\slshape\rightmark}

\setheadfoot{14.5pt}{0.75in}  
\setheaderspaces{*}{0.4in}{*}  

\clearpage
\makeatletter
\newlength{\numberheight}

\makechapterstyle{veelo_mod}{%
  \setlength{\afterchapskip}{40pt}

  \renewcommand*{\chapnamefont}{\normalfont\LARGE\flushright}

  \renewcommand*{\printchaptername}{%
    \chapnamefont\MakeUppercase{\@chapapp}}
  
  \setlength{\numberheight} {12mm}
  \setlength{\beforechapskip}{18mm}
  \makeoddfoot{plain}{}{}{\thepage}}
\makeatother
\chapterstyle{veelo_mod}

\setsecnumdepth{subsection}  
\settocdepth{subsection} 

\setsecnumformat{\csname the#1\endcsname\ \,} 

\nonzeroparskip

\captionnamefont{\footnotesize}   
\captiontitlefont{\footnotesize} 	

\def\monthword{\ifcase\month \or
  Januar\or Februar\or M\"arz\or April\or Mai\or Juni\or
  Juli\or August\or September\or Oktober\or November\or Dezember\fi}
  
\usepackage[english,ngerman,german]{babel}

\usepackage{fontenc}

\usepackage{minitoc}

\mtcsettitle{minitoc}{Outline}
\tightmtctrue
\makeindex

\usepackage[sort,numbers]{natbib}

\usepackage[resetlabels]{multibib}
\newcites{own}{List of Publications}

\usepackage{xparse}

\usepackage{etoolbox}     

\usepackage{url}

\usepackage{relsize}

\usepackage[all,cmtip]{xy}
\usepackage{a4wide}

\usepackage{amsmath, amsthm, amssymb}
\usepackage{fontenc}

\usepackage[utf8]{inputenc}
\usepackage[T1]{fontenc}

\usepackage{graphicx}
\usepackage{pgfplots}

\usepackage{enumerate}

\usepackage{prettyref}

\usepackage{mathtools}

\usepackage{mathabx}  
\usepackage{accents}  

\usepackage{bm} 

\usepackage{stmaryrd}

\usepackage[draft]{todonotes}  
\usepackage[final]{showkeys}   

\usepackage{color}

\usepackage{xy}

\usepackage{extarrows}

\usepackage{tikz}
\usetikzlibrary{matrix,arrows}

\usepackage{arydshln}
\usepackage{tabularx}
\usepackage{array}

\usepackage{tabu}

\usepackage{mytxfonts} 

\usepackage{cancel}

\newtheorem*{theorem*}{Theorem}
\newtheorem*{example*}{Example}

\newtheorem{theorem}{Theorem}[section]
\newtheorem{definition}[theorem]{Definition}
\newtheorem{proposition}[theorem]{Proposition}
\newtheorem{lemma}[theorem]{Lemma}
\newtheorem{remark}[theorem]{Remark}
\newtheorem{algorithm}[theorem]{Algorithm}
\newtheorem{example}[theorem]{Example}
\newtheorem{corollary}[theorem]{Corollary}
\newtheorem{detail}{Detail}

\newcommand{\orig}{{\rm orig}}

\newcommand{\st}{\text{ {\rm s.t.} }}

\newcommand{\epi}{\mathop{\rm epi}}

\newcommand\lev{ {\rm lev}}

\DeclareMathOperator{\MAX}{MAX}

\newcommand{\rarr}{\rightarrow}
\newcommand{\larr}{\leftarrow}
\newcommand{\Rarr}{\Rightarrow}
\newcommand{\Larr}{\Leftarrow}
\newcommand{\iseq}{\; \Longleftrightarrow \;}

\newcommand{\minfty}{{- \infty}}
\newcommand{\pinfty}{{+ \infty}}

\newcommand{\zerovec}{\boldsymbol{0}}
\newcommand{\vectornorm}[1]{\lVert#1\rVert} 

\let \oldtodo \todo
\renewcommand{\todo}[1]{\oldtodo{TODO:\\#1}}
\newcommand{\checken}[1]{\oldtodo[color=red!20]{CHECKEN: \\ #1}}
\newcommand{\rem}[1]{\oldtodo[color=green]{BEM.:\\ #1}}
\newcommand{\iftime}[1]{\oldtodo[color=blue!20]{IF TIME:\\#1}}
\let\diss\iftime 
\newcommand{\may}[1]{\oldtodo[color=blue!20]{M\!\;\!a\!\;\!y\,2015\!\;\!:\\#1}}   
\renewcommand{\todo}[1]{}
\renewcommand{\checken}[1]{}
\renewcommand{\rem}[1]{}
\renewcommand{\iftime}[1]{}

\DeclareMathOperator{\oxlim}{\mathcal{O}_X-lim}
\DeclareMathOperator{\tlim}{\mathcal{T}-lim}

\DeclareMathOperator{\spanop}{span}
\NewDocumentCommand \linspan { m } {\spanop(#1) }
\NewDocumentCommand \Linspan { m } {\spanop\left(#1\right) }
\DeclareMathOperator{\sdirsum}{\uplus}
\NewDocumentCommand \openball { O{} O{} O{} }
{
  \ifstrequal {#3}{}
    { 
      \ifstrequal{#2}{}
	{ \mathbb{B}_{#1} }
	{ \mathbb{B}_{#1}\!\!\;(#2) } 
    }
    { 
      \ifstrequal{#2}{}
	{ \mathbb{B}^{(#3)}_{#1} } 
	{ \mathbb{B}^{(#3)}_{#1}\!\!\;(#2) }  
    }
}
\NewDocumentCommand \closedball { O{} O{} O{} }
{
  \ifstrequal {#3}{}
    { 
      \ifstrequal{#2}{}
	{ \overline{\mathbb{B}}_{#1} }
	{ \overline{\mathbb{B}}_{#1}\!\!\;(#2) } 
    }
    { 
      \ifstrequal{#2}{}
	{ \overline{\mathbb{B}}^{(#3)}_{#1} } 
	{ \overline{\mathbb{B}}^{(#3)}_{#1}\!\!\;(#2) }  
    }
}
\NewDocumentCommand \sphere { O{} O{} O{} }
{
  \ifstrequal {#3}{}
    { 
      \ifstrequal{#2}{}
	{ \mathbb{S}_{#1} }
	{ \mathbb{S}_{#1}\!\!\;(#2) }
    }
    { 
      \ifstrequal{#2}{}
	{ \mathbb{S}^{(#3)}_{#1} } 
	{ \mathbb{S}^{(#3)}_{#1}\!\!\;(#2) }  
    }
}

\newcommand{\hill}[1]{\wideparen{#1}}

\newcommand{\id}{{\rm id}}

\newcommand{\R}{\mathbb{R}}
\newcommand{\N}{\mathbb{N}}
\newcommand{\Z}{\mathbb{Z}}
\newcommand{\C}{\mathbb{C}}

\def\tT{{\mbox{\tiny{T}}}}
\def\argmin{\mathop{\rm argmin}}
\def\dom{\mathop{\rm dom}}
\def\sup{\mathop{\rm sup}}

\catcode`@=11

\def\nvphantom{\v@true\h@false\nph@nt}
\def\nhphantom{\v@false\h@true\nph@nt}
\def\nphantom{\v@true\h@true\nph@nt}
\def\nph@nt{\ifmmode\def\next{\mathpalette\nmathph@nt}%
  \else\let\next\nmakeph@nt\fi\next}
\def\nmakeph@nt#1{\setbox\z@\hbox{#1}\nfinph@nt}
\def\nmathph@nt#1#2{\setbox\z@\hbox{$\m@th#1{#2}$}\nfinph@nt}
\def\nfinph@nt{\setbox\tw@\null
  \ifv@ \ht\tw@\ht\z@ \dp\tw@\dp\z@\fi
  \ifh@ \wd\tw@-\wd\z@\fi \box\tw@}

\newcommand{\adapt}[2]{\hphantom{#1}\nhphantom{#2}#2}

\newrefformat{chap}{Chapter \ref{#1}}
\newrefformat{sec}{Section \ref{#1}}
\newrefformat{subsec}{Subsection \ref{#1}}
\newrefformat{subsubsec}{Subsubsection \ref{#1}}
\newrefformat{cor}{Corollary \ref{#1}}
\newrefformat{exa}{Example \ref{#1}}
\newrefformat{rem}{Remark \ref{#1}}
\newrefformat{prop}{Proposition \ref{#1}}
\newrefformat{det}{Detail \ref{#1}}
\newrefformat{def}{Definition \ref{#1}}

\newcommand{\defeq}{\coloneq}
\newcommand{\eqdef}{\eqcolon}

\newcommand{\defiseq}{\vcentcolon \iseq}

\newcommand{\topology}[1]{          \mathcal{O}_{#1}   }
\newcommand{\topospace}[1]{     (#1 ,\mathcal{O}_{#1} )}

\newcommand{\ordertotopspace}[2]
			{     (#1^{#2}, \mathcal{T}_{\leq#2}        )}

\newcommand{\orderspace}[2]{     (#1^{#2}, \leq#2        )}

\newcommand{\decorspace}[2]
			{     (#1#2, \mathcal{O}#2        )}

\newcommand{\compdecorspace}[2]
			{     (#1#2, \mathcal{O}#2        )_{\infty#2}}

    \newcounter{saveenumerate}
    \makeatletter
    \newcommand{\enumeratext}[1]{%
    \setcounter{saveenumerate}{\value{enum\romannumeral\the\@enumdepth}}
    \end{enumerate}
    #1
    \begin{enumerate}
    \setcounter{enum\romannumeral\the\@enumdepth}{\value{saveenumerate}}%
    }
    \makeatother

\renewcommand{\otimes}{\varotimes \!\!\;}

\let\oldstrictfi\strictfi
\renewcommand{\strictfi}{\!\! \oldstrictfi \!\!}

\let\oldstrictif\strictif
\renewcommand{\strictif}{\!\! \oldstrictif \!\!}

	\newcommand{\upref}[2][)]{\textup{\ref{#2}#1}}

\newcommand{\DissVersionForMeOrHareBrainedOfficialVersion}[2]{#2} 

\newcommand{\MayChangesPartiallyPerformedVersionOrHareBrainedOfficialVersion}[2]{#2} 

  \usepackage{collect}
  \definecollection{introscollection}
  \newcommand{\addtointroscoll}[1]{\begin{collect*}{introscollection}{}{}{}{} #1 --------------  \end{collect*} }
  \renewcommand{\addtointroscoll}[1]{#1}

\begin{document}
\clearpage   
\selectlanguage{\ngerman}
%
%
\thispagestyle{empty}

\title{\vspace{-4cm}
    \normalsize Vom Fachbereich Mathematik
    \normalsize der Technischen Universit\"at Kaiserslautern\\
    \normalsize zur Verleihung des akademischen Grades\\
    \normalsize Doktor der Naturwissenschaften
    \normalsize (Doctor rerum naturalium, Dr. rer. nat.)\\
    \normalsize genehmigte \\[1.5ex]
    \large \textbf{Dissertation}*\\
\vspace{2cm}
				 \huge \textbf{Coercive functions from a topological viewpoint and properties
	of minimizing sets of convex functions appearing in image restoration }
}
\author{\vspace*{0.7cm}\\
{\LARGE René Ciak}
 	 	\vspace{6cm}\\ 
		\hspace{-1.6cm}\hspace{-4.46cm}
		Gutachter: \\
		\hspace{-1.6cm}\hspace{-1.53cm} Prof. Dr. Gabriele Steidl \\ 
		\hspace{-1.6cm}Prof. Dr. Gerlind Plonka--Hoch\\ 
}

\date{ \vspace{0.3cm} 
       Tag der Disputation: 9. Oktober 2014 
       \vspace{2cm}\\  
       \normalsize D 386
       \vspace{-2cm}}

\maketitle
\thispagestyle{empty}



\clearpage   
\newpage

\thispagestyle{empty}
\ \\
\large*\normalsize up to minor differences, see last page
\vspace*{\fill}

\begin{tabular}{l}
\textbf{Promotionskommission}
\end{tabular}\\[1ex]
\begin{tabular}{l l}
			Vorsitzender: & Prof. Dr. Claus Fieker, TU Kaiserslautern \\
			Erstgutachterin:  & Prof. Dr. Gabriele Steidl, TU Kaiserslautern\\
			Zweitgutachterin: &  Prof. Dr. Gerlind Plonka--Hoch, Universit\"at G\"ottingen \\ 
			Weiterer Pr\"ufer: & Prof. Dr. J\"urgen Franke, TU Kaiserslautern
\end{tabular} 
\bigskip\\
\begin{tabular}{lr}
\end{tabular}


\frontmatter

\newpage

\selectlanguage{\english}
\chapter{Table of notation}


\subsection*{Sets, ordered sets and level sets }
\begin{tabular}{p{3.75cm}p{\textwidth-2\tabcolsep-3.75cm}}
  $A \subseteq B$
    & $A$ is subset of $B$
\\
  $A \subset B$
    & $A$ is strict subset of $B$
\\
  $\N$
    & Set $\{1,2,3,\dots \}$ of natural numbers 
\\
  $\N_0$
    & Set $\{0,1,2,3, \dots \} = \{0\} \cup \N$
\\
  $\R$
    & Set of real numbers
\\
  $\R_0^+$
    & The real interval $[0, \pinfty)$
\\
  $\C$
    & Set of complex numbers
\\[0.8ex]
  $\MAX_\leq(Z)$, $\MAX(Z)$
    & (Possibly empty) set of maximal elements of an ordered set $(Z, \leq)$
\\
  $\max_\leq(Z)$, $\max (Z)$
    & Maximum of a totally ordered set $(Z, \leq)$ really having 
      a maximum
\\
  $\min_\leq(Z)$, $\min (Z)$
    & Minimum of a totally ordered set $(Z, \leq)$ really having 
      a minimum
\\[0.8ex]
  $\lev_{\leq \tau} \Psi, \lev_{\tau} \Psi$
    & (Lower) level set $\{x \in X : \Psi(x) \leq \tau \}$ 
      of the function $\Psi: X \rarr (Z, \leq)$
\\
  $\lev_{< \tau} \Psi$
    & Strict (lower) level set $\{x \in X : \Psi(x) < \tau \}$ 
\\
  $\lev_{= \tau} \Psi$
    & Iso-level set $\{x \in X : \Psi(x) = \tau \}$ 
\\
  $\closedball[r][a][] [d]$, $\closedball[r][a]$
    & Closed ball $\{x \in X : d(x, a) \leq r \}$ in a metric space 
      $(X,d)$
\\
  $\openball[r][a][] [d]$, $\openball[r][a]$
    & Open ball $\{x \in X : d(x,a) < r\}$
\\
  $\sphere[r][a][] [d]$, $\sphere[r][a]$
    & Sphere $\{x \in X: d(x,a) = r\}$
\\
  $\closedball[r][][] [\|\cdot\|]$,$\closedball[r]$
    & Closed ball $\{x \in X : \|x\| \leq r \}$ around $\zerovec$ in 
      a normed space $(X, \|\cdot \|)$
\\
  $\openball[r][][] [\|\cdot\|], \openball[r]$
    & Open ball $\{x \in X : \|x\| < r\}$ around $\zerovec$
\\
  $\sphere[r][][] [\|\cdot\|], \sphere[r]$
    & Sphere $\{x \in X: \|x\| = r\}$ around $\zerovec$
\\
  $\closedball[r][a][n] [\|\cdot\|], \closedball[r][a][n] $
    & Closed ball $\{x \in \R^n: \|x\| \leq r \}$ in $(\R^n, \|\cdot \|)$
\\
  $\openball[r][a][n] [\|\cdot\|], \openball[r][a][n]$
    & Open ball $\{x \in \R^n : \|x\| < r\}$ in $(\R^n, \|\cdot \|)$
\\
  $\sphere[r][a][n] [\|\cdot\|], \sphere[r][a][n]$
    & Sphere $\{x \in \R^{\boldsymbol{n+1}}: \|x\| = r\}$
	in $(\R^{\boldsymbol{n+1}}, \|\cdot \|)$
\\
  $H_{p, \alpha}^{\leq}$
    & Closed halfspace $\{x \in \R^n : \langle x, p \rangle \leq \alpha \}$
\\
  $H_{p, \alpha}^{<}$
    & Open halfspace $\{x \in \R^n : \langle x, p \rangle < \alpha \}$
\\
  $H_{p, \alpha}^{=}$
    & Hyperplane $\{x \in \R^n : \langle x, p \rangle = \alpha \}$
\\
  $\dom \Phi$
    & Effective domain $\{ x \in X : \Phi(x) < \pinfty \}$
      of the function $\Phi$
\\
  $OP(\Phi, \Psi)$
    & The set 
      $\{\tau \in \R : \dom \Phi \cap \lev_\tau \Psi \not = \emptyset\}$
    of parameters $\tau \in \R$ for which $\dom \Phi$ and
      $\lev_\tau \Psi$ overlap
\end{tabular}

\subsection*{Topological spaces and systems of sets}

\begin{tabular}{p{3cm}p{\textwidth-2\tabcolsep-3cm}}
  $(X, \mathcal{O})$
    & A topological space, i.e. a set $X$ equipped
      with some topology $\mathcal{O}$
\\
  $(X_\infty, \mathcal{O}_\infty)$
    & One point compactification of a topological 
      space $(X, \mathcal{O})$
\\[0.8ex]
  $\mathcal{U}(x)$
    & Neighborhood system of the point $x$ of a topological space 
      $(X, \mathcal{O})$
\\
  $\mathcal{B}(x)$
    & A neighborhood basis of the point $x$ of a topological space 
      $(X, \mathcal{O})$
\\
  $\mathcal{K}(X, \mathcal{O}), \mathcal{K}(X)$
    & System of all compact subsets of a topological space 
      $(X, \mathcal{O})$
\\
  $\mathcal{A}(X, \mathcal{O}), \mathcal{A}(X)$
    & System of all closed subsets of a topological space
      $(X, \mathcal{O})$
\\
  $\mathcal{KA}(X, \mathcal{O})$
    & System 
      of all compact and closed subsets of a topological space
      $(X, \mathcal{O})$
\\[0.8ex]
  $U \Cap \mathcal{O}$
    & Subspace topology $\{U \cap O: O \in \mathcal{O}\}$ 
      for  the subset $U$ of a topological space $(X, \mathcal{O})$
\\
  $\mathcal{O}_\leq$
    & Usual order topology for a totally ordered set $(X, \leq)$
\\
  $\mathcal{T}_\leq$
    & Right order topology for a totally ordered set $(X, \leq)$
\\
  $\mathcal{T}_\geq$
    & Left order topology for a totally ordered set $(X, \leq)$
\\
  $\mathcal{T}$
    & Right order topology for $[\minfty, \pinfty]$
\\[0.8ex]
  $(\R, \mathcal{O})$
    & $\R$ equipped with its natural topology
\\
  $(\R^n, \mathcal{O}^{\otimes n})$
    & $\R^n$ equipped with its natural topology
\end{tabular}

\subsection*{Hulls and topological operations}


\begin{tabular}{p{3cm}p{\textwidth-2\tabcolsep-3cm}}
  ${\rm co}(S)$
    & Convex hull of the set $S$
\\
  ${\rm aff}(S)$
    & Affine hull of the set $S$
\\[0.8ex]
  $\overline S$
    & Closure of the set $S$
\\
  ${\rm int}(S)$
    & Interior of the set $S$
\\
  ${\rm int}_A(S)$
    & Interior of the set $S$, relative to $A$
\\
  ${\rm ri}(S)$
    & Relative interior ${\rm int}_{{\rm aff}(S)}(S)$ of the set $S$
\\
  ${\rm rb}(S)$
    & Relative boundary $\overline S \setminus {\rm ri}(S)$ 
	of the set $S$
\end{tabular}

\subsection*{Linear Algebra}

\begin{tabular}{p{3cm}p{\textwidth-2\tabcolsep-3cm}}
  $S_1 \oplus \dots \oplus S_k$
    & Direct sum of the \emph{subsets} $S_1, \dots, S_k$ of 
      some vecotor space
\\
  $A^*$
    & Transpose of the matrix $A$
\\
  $v^\tT$
    & Transpose of the vector $v$
\\
  $e_1, \dots, e_n$
    & Standard basis vectors 
      $(1,0, \dots, 0)^\tT, \dots, (0,0, \dots 1)^\tT $
      of $\R^n$
\\
  $\mathcal{N}(A)$
    & Nullspace of the linear mapping $A$, resp. of the matrix $A$
\\
  $\mathcal{R}(A)$
    & Range of the linear mapping $A$, resp. of the matrix $A$
\\
  $0_{X}$
    & The trivial linear mapping $0_{X}: X \rarr \R, x \mapsto 0$ 
\end{tabular}

\newpage
\thispagestyle{plain}

\subsection*{Operators, functions and families of functions}

\begin{tabular}{p{3cm}p{\textwidth-2\tabcolsep-3cm}}
 $F_1 \sdirsum F_2$
   & Semidirect sum of the functions $F_i: X_i \rarr \R \cup \{\pinfty\}$
     defined on subspaces $X_i$ with $X_1 + X_2 = X_1 \oplus X_2$,
     given by \\&
     $(F_1 \sdirsum F_2)(x_1 + x_2) \defeq F_1(x_1) + F_2(x_2) $
\\
  $|y|$ 
    & The vector in $\R^n$ which is derived from $y = (a,b)^\tT \in \R^{n+n}$
      according to $|y|_i \defeq \sqrt{a_i^2 + b_i^2}$,
      $i = 1 \dots n$.
\\
 $\nabla$
   & Gradient operator (the continuous one or a discrete one)
\\ 
 $\partial \Phi(x)$
   & Subdifferential of the function $\Phi$ at $x$ 
\\
\end{tabular}
\begin{tabular}{p{3cm}p{\textwidth-2\tabcolsep-3cm}}
  $\Phi^*$
   & (Fenchel) conjugate function of $\Phi$ \\
  ${\rm cl} \Phi$
    & Closure of the function $\Phi$
\\
 $\iota_S$ 
   & Indicator function $\iota_S:\R\rightarrow\R\cup\{+\infty\}$ of $S$ defined by \\
   & $\iota_S(x)\;=\; \begin{cases} 0 & x\in S, \\
     														\infty & \text{otherwise}
     							\end{cases}
     $ 
\\
  ${\rm  gr} g$
    & Graph of the function $g$
\\
  $\Gamma_0(X)$
    & Set of all proper convex and lower semicontinuous \\& functions
      mapping a nonempty affine subset $X$ of $\R^n$ to $[\minfty, \pinfty]$
\end{tabular}

\cleardoublepage
\section*{Summary}
\thispagestyle{empty}
Many tasks in image processing can be tackled by 
modeling an appropriate data fidelity term 
$\Phi: \R^n \rarr \R \cup \{\pinfty\}$ and then solve 
one of the regularized minimization problems
\begin{align*} 
  &(P_{1,\tau}) \qquad \argmin_{x \in \mathbb R^n} \left\{ \Phi(x) \st \Psi(x) \le \tau \right\}
\\
  &(P_{2,\lambda}) \qquad \argmin_{x \in \mathbb R^n} \{ \Phi(x) +  \lambda \Psi(x) \}, \; \lambda > 0
\end{align*}
with some function $\Psi: \R^n \rarr \R \cup \{\pinfty\}$
and a good choice of the parameter(s).
Two tasks arise naturally here:
\begin{enumerate}
  \item 
    Study the solver sets ${\rm SOL}(P_{1,\tau})$ and 
    ${\rm SOL}(P_{2,\lambda})$ of the minimization problems.
  \item 
    Ensure that the minimization problems have solutions.
\end{enumerate}

This thesis provides contributions to both tasks:
Regarding the first task for a more special setting we prove that there 
are intervals $(0,c)$ and $(0,d)$ such that the setvalued curves 
\begin{align*}
  \tau \mapsto {}& {\rm SOL}(P_{1,\tau}),  \; \tau \in (0,c)
\\
  \lambda \mapsto {}& {\rm SOL}(P_{2,\lambda}), \; \lambda \in (0,d)
\end{align*}
are the same, besides an order reversing parameter change 
$g: (0,c) \rarr (0,d)$. Moreover we show 
that the solver sets are changing all the time while 
$\tau$ runs from $0$ to $c$ and $\lambda$ runs from $d$ to $0$.

In the presence of lower semicontinuity
the second task is done if we have additionally coercivity.
We regard lower semicontinuity and coercivity from a topological 
point of view and develop a new technique for proving 
lower semicontinuity plus coercivity.
The key point is that 
a function
$f: \R^n \rarr [\minfty, \pinfty]$
is lower semicontinuous and coercive, iff
a certain continuation of $f$ to the one point compactification of 
$\R^n$ is continuous with respect to the right order topology on 
$[\minfty, \pinfty]$.

Dropping any lower semicontinuity assumption we also prove a theorem
on the coercivity of a sum of functions. 
More precisely, this theorem gives information on which subspaces 
of $\R^n$ a sum $F + G$ 
of functions $F, G : \R^n \rarr [\minfty, \pinfty]$ 
is coercive, provided that $F$ and $G$ are of a certain form, namely
\begin{align*}  
  &F = F_1 \sdirsum F_2 \quad \text{  and  }  \quad G = G_1 \sdirsum G_2 
\end{align*}
with functions
$F_1: X_1 \rarr \R \cup \{\pinfty\}$, 
$F_2: X_2 \rarr \R \cup \{\pinfty\}$,
$G_1: Y_1 \rarr \R \cup \{\pinfty\}$, \linebreak[2]
and 
$G_2: Y_2 \rarr \R \cup \{\pinfty\}$,
where 
\begin{gather*}
  \R^n = X_1 \oplus X_2 = Y_1 \oplus Y_2.
\end{gather*}
For such functions the theorem basically states that $F + G$ 
is coercive on 
$X_1 + Y_1 = (X_2 \cap Y_2)^\perp$ if 
$X_1 \perp X_2$, $Y_1 \perp Y_2$ and 
certain boundedness conditions hold true.

\selectlanguage{\ngerman}
\thispagestyle{empty}
\newpage
\section*{Zusammenfassung}
\thispagestyle{empty}
Viele Aufgaben in der Bildverarbeitung lassen sich wie folgt
angehen:
Nach Modellierung eines Datenterms
$\Phi: \R^n \rarr \R \cup \{\pinfty\}$ löst man 
eines der folgenden regularisierten Minimierungsprobleme 
\begin{align*} 
  &(P_{1,\tau}) \qquad \argmin_{x \in \mathbb R^n} \left\{ \Phi(x) \st \Psi(x) \le \tau \right\}
\\
  &(P_{2,\lambda}) \qquad \argmin_{x \in \mathbb R^n} \{ \Phi(x) +  \lambda \Psi(x) \}, \; \lambda > 0
\end{align*}
mit einer Funktion $\Psi: \R^n \rarr \R \cup \{\pinfty\}$
und jeweils gut gewähltem Parameterwert.
Es stellen sich unter anderem folgende Aufgaben:
\begin{enumerate}
  \item 
    Untersuche die Lösungsmengen ${\rm SOL}(P_{1,\tau})$ und
    ${\rm SOL}(P_{2,\lambda})$ der Minimierungsprobleme.
  \item 
    Stelle sicher, daß die Minimierungsprobleme überhaupt Lösungen besitzen.
\end{enumerate}

Diese Arbeit enthält Beiträge zu beiden Aufgaben:
Bezüglich der ersten Aufgabe wird (in einem spezielleren Rahmen)
die Existenz von Intervallen
$(0,c)$ und $(0,d)$ bewiesen derart, daß die mengenwertigen Kurven
\begin{align*}
  \tau \mapsto {}& {\rm SOL}(P_{1,\tau}),  \; \tau \in (0,c)
\\
  \lambda \mapsto {}& {\rm SOL}(P_{2,\lambda}), \; \lambda \in (0,d)
\end{align*}

die selben sind, bis auf einen ordnungsumkehrenden 
Parameterwechsel $g: (0,c) \rarr (0,d)$. 
Desweiteren zeigen wir, daß die Lösungsmengen
${\rm SOL}(P_{1,\tau})$ bzw. ${\rm SOL}(P_{2,\lambda})$
sich die ganze Zeit ändern, während $\tau$ aufsteigend 
das Intervall $(0,c)$ durchläuft bzw. $\lambda$ absteigend
das Intervall $(0,d)$ durchläuft.

Falls Halbstetigkeit von unten gegeben ist, ist die 
zweite Aufgabe gelöst, wenn zusätzlich Koerzivität vorliegt.

Wir betrachten in dieser Arbeit sowohl Halbstetigkeit von unten 
als auch Koerzivität von einem topologischen Standpunkt.
Grundlegend ist hierbei, daß eine Funktion 
$f: \R^n \rarr [\minfty, \pinfty]$
genau dann halbstetig von unten und koerziv ist, wenn 
eine gewisse Fortsetzung von $f$ auf die Einpunktkompaktifizierung von 
$\R^n$ stetig bzgl. der von den Halbstrahlen $(a, \pinfty]$, 
$a \in [\minfty, \pinfty)$ erzeugten Topologie ist.
Hieraus wird eine neue Beweistechnik für den gemeinsamen Nachweis von
Halbstetigkeit von unten und Koerzivität entwickelt.

\newpage
\thispagestyle{empty}
Desweiteren beweisen wir einen Satz über die Koerzivität
der Summe zweier Funktionen, ohne Halbstetigkeit von unten vorauszusetzen.
Genauer gesagt liefert dieser Satz Informationen darüber 
auf welchen Unterräumen des $\R^n$ die Summe $F + G$ 
von Funktionen $F, G : \R^n \rarr [\minfty, \pinfty]$ 
koerziv ist, wenn diese Funktionen von der Bauart
\begin{align*}  
  &F = F_1 \sdirsum F_2 \quad \text{  and  }  \quad G = G_1 \sdirsum G_2
\end{align*}
sind mit Funktionen 
$F_1: X_1 \rarr \R \cup \{\pinfty\}$, 
$F_2: X_2 \rarr \R \cup \{\pinfty\}$,
$G_1: Y_1 \rarr \R \cup \{\pinfty\}$, \linebreak[2]
und 
$G_2: Y_2 \rarr \R \cup \{\pinfty\}$,
worin 
\begin{gather*}
  \R^n = X_1 \oplus X_2 = Y_1 \oplus Y_2.
\end{gather*}
Für Funktionen solchen Typs besagt der Satz im Wesentlichen, daß 
$F + G$ genau dann koerziv auf dem Unterraum
$X_1 + Y_1 = (X_2 \cap Y_2)^\perp$ ist, wenn
$X_1 \perp X_2$, $Y_1 \perp Y_2$ und 
gewisse Beschränktheitsvoraussetzungen erfüllt sind.

\thispagestyle{plain}
\cleardoublepage
\selectlanguage{\english}

\tableofcontents*

\nociteown{*}

\bibliographyown{my_publications}
\bibliographystyleown{alpha_unsrt}


\newpage
\thispagestyle{empty}
\cleardoublepage

\newpage

\selectlanguage{\english}
\mainmatter

\clearpage   


\chapter{Introduction and overview} \label{chap:intro}

\minitoc


\section{Definitions, notations and conventions} \label{sec:definitions_notations_and_conventions}

Writing $A \subseteq B$ means that $A$ is a subset of $B$,
whereas writing $A \subset B$ indicates that $A$ 
is a proper subset of $B$.
A function $f: X \rarr Y$ is {\bf genuine} or {\bf non-trivial},
iff $X$ (and therefore also $Y$) is nonempty.

  A {\bf (direct) decomposition} of a vector space $V$ {\bf into subspaces}
  $V_1, V_2 \dots V_n$ is a tupel $(V_1, V_2, \dots V_n)$ of subspaces,
  such that every $v \in V$ can be written in a unique way in the form 
  $v = v_1 + v_2 + \dots + v_n$ with $v_i \in V_i$ for $i = 1 \dots n$.
  A bit sloppily but practically we will also write
  $V = V_1 \oplus V_2 \oplus \dots \oplus V_n$
  and call this a (direct) decomposition or direct sum.
  For a given subspace $U_1$ of $V$ a subspace $U_2$ is called
  {\bf complementary} to $U_1$ iff $V = U_1 \oplus U_2$.

The set of all $n$-tuples of real numbers is denoted by 
  $\R^n$, where $n \in \N_0$.
Note that $\R^0$, containing only the empty tupel, 
is the trivial real vector space.
By $e_1, e_2, \cdots , e_n$ we name the 
vectors 
$(1,0, 0, \dots 0)^T, (0,1,0, \dots, 0)^T, \dots (0, \dots 0,1)^T$,
which form the {\bf standard basis} of $\R^n$.
The {\bf trivial linear mapping} $X \rarr \R, x\mapsto 0$ 
between a real vector space $X$ and the real numbers 
will be denoted by $0_{X}$.
The {\bf nullspace (kernel)} of a matrix/linear 
operator $A$ is denoted by  ${\cal N}(A)$ and
its {\bf range} by ${\cal R}(A)$.
The {\bf transpose} of a matrix $A$ is denoted by $A^*$.
For Euclidean vectors $v$ we will also write $v^T$.
For a vector $y = (a, b)^\tT \in \R^{n+n}$ let 
$|y|$ denote the vector in $\R^n$ whose components are
$\sqrt{a_i^2 + b_i^2} \eqdef |y|_i$, $i = 1 \dots n$.
Usually $y$ appears in the form 
$y = \nabla x$ with a linear mapping $\nabla : \R^n \rarr \R^n \times \R^n$
modeling a discrete gradient. 

We also remark that, in the presence of a direct decomposition of
$\R^n$ into subspaces like $\R^n = X_1 \oplus X_2 \oplus X_3$,
we will use the unique decomposition $x=x_1 + x_2 + x_3$ 
of $x\in \R^n$ in its components
$x_1 \in X_1, x_2 \in X_2, x_3 \in X_3$ without emphasizing
the underlying direct decomposition every time.
Furthermore
we will use the notation 
$S = S_1 \oplus S_2 \oplus \dots \oplus S_k$
for subsets $S, S_1, \dots, S_k$ of $\R^n$ iff
every $s \in S$ has a unique decomposition
$s = s_1 + s_2 + \dots + s_k$ into components 
$s_j \in S_j$, $j\in \{1,\dots, k\}$.
For \emph{convex} subsets $C_1, C_2$ of $\R^n$ we have 
$C_1 + C_2 = C_1 \oplus C_2$ iff
${\rm aff} (C_1) + {\rm aff} (C_2) = {\rm aff} (C_1) \oplus {\rm aff} (C_2)$, 
see \prettyref{thm:directness_of_sum_of_convex_sets_eq_carries_over_to_affine_hulls}
for more details.

The {\bf convex hull} of a set $S \subseteq \R^n$ is denoted by 
${\rm co}(S)$.
The {\bf affine hull} of a set $S \subseteq \R^n$ is named by 
${\rm aff}(S)$.
%
The (topological) {\bf closure} and the {\bf interior} of a set 
$S \subseteq \R^n$ will be denoted by $\overline{S}$ and 
${\rm int}(S)$,  
respectively. 
Note that, for any subset $A \subseteq \R^n$, the identity 
$\overline{A}^B = \overline{A}$ holds for all $B \supseteq A$ that 
are closed subsets of $\R^n$;
in particular it does not matter whether we form the closure 
of a subset $A$ of $\R^n$ with respect to $\R^n$ or 
with respect to any affine supperset of $A$, including ${\rm aff}(A)$.
The {\bf relative interior} of a convex set $C$ will be denoted by
${\rm ri}(C)$.
The {\bf relative boundary} of a convex set $C$ will be denoted by
${\rm rb}(C) \defeq \overline{C} \setminus {\rm ri}(C)$.

  For a totally ordered set $(Z, \leq)$ 
  we set 
  \begin{gather*}
    \MAX_\leq(Z) 
    \defeq \{\widehat z \in Z : z \text{ is a maximum of } Z \}
  \end{gather*}
  If it is clear from the context which total order is given to $Z$
  we will shortly also write $\MAX(Z)$.
  If $(Z, \leq)$ has a maximum $\widehat z$ then 
  $\MAX_\leq(Z) = \{\widehat z\}$.
  If $(Z, \leq)$ has no maximum then $\MAX_\leq(Z) = \emptyset$. 



Let $\R_0^+ \defeq [0,+\infty)$ and let
$
\Gamma_0(\mathbb R^n)
$ denote the set of proper, convex, closed functions mapping $\R^n$
into the extended real numbers $\R \cup \{ +\infty \}$.
For nonempty, affine subsets $X \subseteq \R^n$, we define 
$\Gamma_0(X)$ 
in an analogous way.
The {\bf closure} of a convex function 
$f: \R^n \rightarrow \R \cup \{-\infty, +\infty\}$ is denoted by 
${\rm cl}f$. The closure of a proper convex function
is its lower semicontinuous hull. See 
\prettyref{thm:closure_of_proper_convex_function}
for some of the properties of the closure operator.
For a given function 
$\Psi : X \rarr Z$ between a set $X$ and a totally ordered 
set $(Z, \leq)$ we distinguish different 
types of {\bf level sets} by the following notations:
\begin{align*}  
  {\rm lev}_\tau \Psi 
  \defeq 
  {\rm lev}_{\leq\tau} \Psi 
  \defeq 
  \{x\in X: \Psi(x) \le \tau \}
  &&\text{and}&&
  {\rm lev}_{<\tau} \Psi 
  \defeq 
  \{x\in X: \Psi(x) < \tau \}.
\end{align*}
Usually the term ``level set'' refers to the first type with ``$\leq$''.

Important lower level sets are the closed balls
$
  \overline{\mathbb{B}}_{r}(a)[\|\cdot\|]
  \defeq 
  \{x \in \R^n : \|x\| \leq r\}
$
of radius $r \in [0, +\infty)$, midpoint $a \in \R^n$ with respect to a norm 
$\|\cdot \|$.
If it is clear from the context which norm is meant
we use the abbreviation $\overline{\mathbb{B}}_{r}(a)$.
If $a=\zerovec$ we even more shortly write $\overline{\mathbb{B}}_{r}$.
For spheres 
$
  \overline{\mathbb{S}}_{r}(a)[\|\cdot\|]
  \defeq 
  \{x \in \R^n : \|x\| = r\}
$
and open balls
$
  \mathbb{B}_{r}(a)[\|\cdot\|]
  \defeq 
  \{x \in \R^n : \|x\| < r\}
$
with midpoint $a$, radius $r \in [0, +\infty)$ and $r\in(0,+\infty)$, 
respectively, we apply similar abbreviations.
If more general a metric space $(X, d)$ is given we use the notations
$
  \openball[R][a] \defeq \{x \in X: d(x,a) < R \}, 
  \closedball[R][a] \defeq \{x \in X: d(x,a) \leq R \}
$ and 
$\sphere[R][a] \defeq{} \{x\in X: d(x,a) = R \}$
for the 
{\bf open ball},  {\bf closed ball} and {\bf sphere} of 
radius $R \in \R$ around $a \in X$, respectively. If
$X = \R^n$ is endowed with the usual Euclidean metric we also
will use the notations
$
  \openball[R][a][n]  \defeq \{x \in \R^n: \|x - a \| < R \},  
  \closedball[R][a][n] \defeq \{x \in \R^n: \|x - a \| \leq R \}
$ and 
$ \sphere[R][a][n-1] \defeq{} \{x\in \R^n: \|x - a \| = R \}$.
If the dimension $n$ of the underlying Euclidean space is clear from 
the context we also 
use the abbreviations 
$ \openball[R][a],\closedball[R][a]$ and $\sphere[R][a]$.
If $a = \zerovec$ and/or $r = 1$ we 
sometimes
omit the corresponding parts of the notations and write
e.g. $\sphere[r]$, $\sphere[][a]$, $\sphere[][]$ or $\closedball$.

Further important level sets are {half-spaces} and {hyperplanes}.
We use the notations 
$H_{p,\alpha}^{\leq} \defeq \{x \in \R^n: \langle p,x \rangle \leq \alpha\}$,
$H_{p,\alpha}^{>} \defeq \{x \in \R^n: \langle p,x \rangle > \alpha\}$ and
$H_{p,\alpha}^{=} \defeq \{x \in \R^n: \langle p,x \rangle = \alpha\}$
for the {\bf closed halfspaces}, the {\bf open halfspaces} and 
{\bf hyperplanes}, respectively.
\MayChangesPartiallyPerformedVersionOrHareBrainedOfficialVersion{%
The alternatives $H_{p,\alpha}^{\geq}$ and $H_{p,\alpha}^{<}$ are defined analogously.
\may{This Sentence was newly inserted}%
}
{}%

The set of {\bf overlapping parameters} between a set $A$ and a family
$(B_\tau)_{\tau \in T}$ of sets $B_\tau$ with some index set $T$ is
$
  OP(A, (B_\tau)_{\tau \in T}) 
  \defeq \{\tau \in T: A \cap B_\tau \not = \emptyset\}
$. In this 
thesis
we will consider the case $A = \dom \Phi$ and
$B_\tau = {\rm lev}_{\tau}\Psi, \tau \in \R$ for functions
$\Phi, \Psi: \R^n \rightarrow \R \cup \{+\infty\}$ and use the notation
\begin{equation*}
  OP(\Phi, \Psi) 
  \defeq 
  OP(\dom \Phi, ({\rm lev}_{\tau}\Psi)_{\tau \in \R})
  = \{\tau \in \R: \dom \Phi \cap {\rm lev}_{\tau}\Psi \not = \emptyset\}.
\end{equation*}

Furthermore, the {\bf indicator function} $\iota_S$ of a set $S$
is defined by
$$
\iota_S(x) \defeq
\left\{
\begin{array}{cl}
0&{\rm if} \; x \in S,\\
+\infty&{\rm otherwise}.
\end{array}
\right.
$$

For $x_0 \in \R^n$ the {\bf subdifferential} $\partial \Psi(x_0)$  of $\Psi$ at $x_0$ is the set
$$
\partial \Psi(x_0) \defeq \{p \in \R^n: \Psi(x_0) + \langle p,x-x_0\rangle \le \Psi(x)  \; \text{ for all } x \in \R^n\}.
$$
If $\Psi$ is proper, convex and $x_0 \in {\rm ri}({\rm dom} \Psi)$, then
$\partial \Psi(x_0) \not = \emptyset$.

Additionally we will need the {\bf Fenchel conjugate function} of $\Psi$ defined by
$$
\Psi^*(p) \defeq \sup_{x \in \R^n} \{ \langle p,x \rangle - \Psi(x)\}.
$$

Finally the graph of a function $g$ is denoted by ${\rm gr} \, g$.

\subsection*{Topological notations and notions}

\begin{definition}
  We say that a topological space $(X,\mathcal{O})$ is {\bf nonempty}, iff
  $X$ is nonempty.
\end{definition}

\begin{definition}
  Let $U$ be a subset of a set $X$ and let $\mathcal{O}$ be a system of subsets of
  $X$. Then we denote the system
  \begin{equation*}
    \{U \cap O : O\in \mathcal{O} \}
  \end{equation*}
  abbreviated by $U \doublecap \mathcal{O}$.
\end{definition}

If $\mathcal{O}$ is a topology on $X$ then $U \Cap \mathcal{O}$ is a 
topology on $U$; cf. also 
\prettyref{subsec:subspaces}.

\begin{definition}
  An {\bf open neighborhood} of a point $x$ in a topological space 
  $(X, \mathcal{O})$ is just 
  a
  subset $O\in \mathcal{O}$ that contains $x$.
  
  A {\bf neighborhood} of a point $x$ from a topological space $(X, \mathcal{O})$
  is just a subset $U \subseteq X$ containing an open neighborhood of $x$.

  The system of all neighborhoods of $x$ will be denoted by 
  $\mathcal{U}[\mathcal{O}](x)$ or, if the underlying topological space 
  is clear from the context, 
  simply also by $\mathcal{U}(x)$.

  A system $\mathcal{B}(x)$ of 
  open 
  subsets of $X$ is called an 
  $\mathcal{O}$--{\bf neighborhood basis} of a point $x \in X$, iff 
  every neighborhood $U \in \mathcal{U}(x)$ contains some 
  $B \in \mathcal{B}(x)$.
\end{definition}
  We will feel free to adopt our notations for neighborhood systems
  according to the notations for the underlying topological space.
  For instance in the context of a topological space
  $(X',\mathcal{O}')$ we usually write $\mathcal{U}'(x')$
  instead of $\mathcal{U}(x')$.

\begin{remark}
  Having a neighborhood basis $\mathcal{B}(x)$ for every point $x$ 
  of a topological space $\decorspace{X}{}$ 
  we can first reconstruct all neighborhood systems 
  $\mathcal{U}(x)$, $x \in X$, and then also the whole topology 
  by means of the formulas
  \begin{align*}
    \mathcal{U}(x) 
    &= 
    \{U \subseteq X \,|\, \exists B \in \mathcal{B}(x) : U \supseteq B \}
  &&\text{and}
    &\mathcal{O} 
    =
    \{O \subseteq X \,|\, \forall x \in O : O \in \mathcal{U}(x) \}.
  \end{align*}
  See 
  \cite[2.9 Satz]{Qrnb2001} and its proof for more details.
\end{remark}

Regarding the following definition we note that 
``limit point'' is really meant 
as limit point and not as accumulation point.
\begin{definition}
  A sequence $(x_n)_{n\in \N}$ in a topological space 
  $(X, \mathcal{O}_X)$
  is said to have an element $x \in X$  as {\bf limit point}
  iff every neighborhood of $x$ contains almost all 
  sequence members, i.e. -- more formally expressed -- iff 
  \begin{gather*}
	  \forall U \in \mathcal{U}(x)
    \;\;  \exists N \in \N
    \;\;  \forall n \geq N :
	  x_n \in U
  \end{gather*}
  holds true.
  The set of all limit points will be denoted by 
  $\oxlim_{n \rarr \pinfty} x_n$ or simply by
  $\lim_{n \rarr \pinfty} x_n$,
  if it is clear which topology is given to $X$.
  If the sequence has at last one limit point we call the 
  sequence {\bf convergent}.
\end{definition}

\begin{definition}
  A topological space $(X, \mathcal{O})$ is called a 
  {\bf Hausdorff space} iff any two distinct points 
  have two disjoint open neighborhoods, 
  i.e. for every pair of distinct point 
  $x_1, x_2 \in X$ there are open \emph{disjoint} sets 
  $O_1, O_2 \in \mathcal{O}$ with $x_1 \in O_1$ and $x_2 \in O_2.$
\end{definition}

\begin{definition} \label{def:compactness_of_a_space}
  A topological space $(X,\mathcal{O})$ is called {\bf compact} if every
  covering of $X$ by sets from $\mathcal{O}$ has a 
  finite subcover.
\end{definition}

If the topological space appears as a subspace of another space, see
\prettyref{subsec:subspaces}, the
following equivalent definition can also be used:

\begin{definition}
  Let $(\widehat X, \widehat {\mathcal{O}})$ be a topological space.
  A subspace $(X,X \Cap \widehat {\mathcal{O}})$ is called {\bf compact} if every
  open covering of $X$ with open sets from $\widehat {\mathcal{O}}$ has a 
  finite subcover.
\end{definition}

\begin{remark}
  In some texts the word ``compact'' is only used for spaces that 
  are in addition Hausdorff spaces.
\end{remark}

\begin{definition}
  Let $(X, \mathcal{O})$ be a topological space. We say that $K \subseteq X$
  is a {\bf compact subset} of $(X,\mathcal{O})$, iff 
  $(K, K \Cap \mathcal{O})$ is a compact space.
  We denote the system 
  $\{K \subseteq X : K \text{ is a compact subset of } (X, \mathcal{O})\}$
  by $\bm{\mathcal{K}(X, \mathcal{O})} $ 
  or sometimes only by 
  $\bm{\mathcal{K}(X)}$, if it is clear which topology is given to $X$.
  
  Similarly we denote the system of closed subsets of $(X, \mathcal{O})$
  by $\bm {\mathcal{A}(X,\mathcal{O})}$ or by 
  $\bm {\mathcal{A}(X)}$ or even only by $\bm {\mathcal{A}}$.
  Finally the system of compact and closed subsets of $(X, \mathcal{O})$
  will be denoted by $\bm {\mathcal{KA}(X,\mathcal{O})}$ or by 
  $\bm {\mathcal{KA}(X)}$.
\end{definition}

Note that 
$
  \mathcal{KA}(X,\mathcal{O}) 
  = \mathcal{K}(X,\mathcal{O}) \cap \mathcal{A}(X,\mathcal{O})
  \subseteq 
  \mathcal{K}(X,\mathcal{O})
$ 
can be a strict subset of $\mathcal{K}(X,\mathcal{O})$, cf. 
\prettyref{exa:compact_not_closed_subset}.

The following definition is taken from 
\DissVersionForMeOrHareBrainedOfficialVersion
{\cite[section Locally Compact Spaces on p. 146]{Kelley1955}.}
{\cite[p. 146]{Kelley1955}.}

\begin{definition}
  A topological space is {\bf locally compact}, iff each point has
  at least one compact neighborhood.
\end{definition}

\begin{example}
  The Euclidean space $\R^n$, endowed with the natural topology, 
  is not compact, but locally compact, 
  since $\closedball[1][x]$ is a compact neighborhood 
  for an arbitrary point $x \in \R^n$.
\end{example}

Cf.
\prettyref{rem:metric_continuity_vs_topological_continuity}
for the following definition.
\begin{definition}
  A function $f: (X,\mathcal{O}) \rarr (X',\mathcal{O}')$ 
  between topological spaces $(X, \mathcal{O})$ and 
  $(X', \mathcal{O}')$ is called {\bf continuous in} \bm{$x_0$} 
  iff for all open neighborhoods $O'_{f(x_0)}\in \mathcal{O}'$  of $f(x_0)$ 
  there is an
  open neighborhood $O_{x_0} \in \mathcal{O}$ of $x_0$ with
  $f[O_{x_0}] \subseteq O'_{f(x_0)}$ 
  (which is to say $O_{x_0} \subseteq f^{-}[O'_{f(x_0)}]$).
  We call $f$ {\bf continuous} if $f$ is continuous in 
  all points $x \in X$, i.e. if 
  for all open sets $O' \in \mathcal{O}'$ the pre-image 
  $O \defeq f^{-}[O']$ is an open set from $\mathcal{O}$.
\end{definition}

For the next two definitions cf. e.g. 
\cite[p. 90]{Kelley1955} and \cite[p. 94]{Kelley1955}.

\begin{definition}
  A mapping $g:\topospace{Y} \rarr \topospace{Z}$ between 
  topological spaces is called {\bf open} iff 
  every open subset of $\topospace{Y}$ is mapped by $g$
  to an open subset of $\topospace{Z}$.
  Analogously $g$ is called {\bf closed} iff 
  every closed subset of $\topospace{Y}$ is mapped by $g$
  to a closed subset of $\topospace{Z}$. 
\end{definition}

Note that a bijective mapping is open, respectively closed, 
iff its inverse mapping is continuous.

\section{Motivation from image processing} \label{sec:introduction}

Many tasks in image processing such as deblurring, inpainting,
removal of different kinds of noise or
reconstruction of a sparse signal can be tackled by 
minimizing a (parameter containing) function, designed  
for the respective purpose. Often this function 
can be written as a weighted sum 
\begin{gather*}
  \Phi + \lambda \Psi
\end{gather*}
of two functions $\Phi, \Psi \in \Gamma_0(\R^n)$,
where $\Phi$ serves as data fidelity term and $\Psi$ as regularization 
term which influence is controlled by the parameter $\lambda$.
At this point vectors $x \in \R^n$ model gray value images, 
where $n= n_x n_y$ is the total number of pixels.

Both the family of penalized problems
\begin{gather*}
  \argmin_{x \in \R^n} (\Phi(x) + \lambda \Psi(x))
\end{gather*}
and the related families of constrained problems 
\begin{align*}
  \argmin_{x \in \R^n} ( \Phi(x) \st \Psi(x) \leq \tau )
   & \iseq  \argmin_{x \in \R^n} ( \Phi(x) + \iota_{\lev_\tau \Psi} ),
\\
    \argmin_{x \in \R^n} ( \Psi(x) \st \Phi(x) \leq \sigma )
   & \iseq  \argmin_{x \in \R^n} ( \Psi(x) + \iota_{\lev_\sigma \Phi})
\end{align*}
(for certain parameter ranges) are considered in the literature. 
Some examples are:

\begin{itemize}
  \item The family of penalized problems 
     \begin{gather*}                                		
       \argmin_{x \in \R^n} \big( \|Ax -b\|_2^2 + \lambda \|x\|_1 \big),
     \end{gather*}
    along with the families of constraint problems
    \begin{align*}
      \argmin_{x \in \R^n} \big( \|Ax -b\|_2^2 \st \|x\|_1 \leq \tau \big)
	&\iseq \argmin_{x \in \R^n} \big( \|Ax -b\|_2 \st \|x\|_1 \leq \tau \big)
      \\
	&\iseq \argmin_{x \in \R^n} \big( \|Ax -b\|_2 + \iota_{\lev_\tau(\|\cdot\|_1)}(x) \big)
    \intertext{(LASSO problem) and}
      \argmin_{x \in \R^n} \big( \|x\|_1 \st \|Ax - b\|_2 \leq \sqrt\sigma \big)
        &\iseq \argmin_{x \in \R^n} \big( \|x\|_1 + \iota_{\lev_{\sqrt\sigma}(\|A \cdot -b\|_2)}(x) \big),
    \end{align*}    
    (Basis pursuit denoising), cf. e.g. \cite{vandenBergFriedlander2008probing},
    \cite{LO11},
    \cite{vandenBergFriedlander2011sparse}, \cite{DFL08}.
  \item The family of penalized problems 
     \begin{gather*}                                		
       \argmin_{x \in \R^n} \big( \|Ax -b\|_2^2 + \lambda \big\| |\nabla x| \big\|_1 \big),
     \end{gather*}
    along with the families of constraint problems 
    \begin{align*}
      \argmin_{x \in \R^n} \big( \|Ax -b\|_2^2 \st \big\| |\nabla x| \big\|_1 \leq \tau \big)
	&\iseq \argmin_{x \in \R^n} \big( \|Ax -b\|_2 + \iota_{\lev_\tau(\| |\nabla \cdot| \|_1)}(x) \big)
    \intertext{and}
      \argmin_{x \in \R^n} \big( \big\| |\nabla x| \big\|_1 \st \|Ax - b\|_2 \leq \sqrt\sigma \big)
        &\iseq \argmin_{x \in \R^n} \big( \big\| |\nabla x| \big\|_1 + \iota_{\lev_{\sqrt\sigma}(\|A \cdot -b\|_2)}(x) \big),
    \end{align*}    
    cf. e.g. \cite{NgWeissYuan2010solving},
    \cite{WenChan2012parameter},
    \cite{WBA09}.
  \item The family of penalized problems 
     \begin{gather*}                                		
       \argmin_{x \in \R^n} \Bigg(
	  \underbrace{
	     \sum_{k=1}^n \big([Ax]_k - b_k \log([Ax]_k) \big) 
	  }_{\eqdef \Phi(x)}
	      + \lambda \| |\nabla x| \|_1 \Bigg)
	  ,
     \end{gather*}
    along with the families of constraint problems 
    \begin{align*}
      \argmin_{x \in \R^n} \big( \Phi(x) \st \| |\nabla x| \|_1 \leq \tau \big)
	&\iseq \argmin_{x \in \R^n} \big( \Phi(x) + \iota_{\lev_\tau(\| |\nabla \cdot| \|_1)}(x) \big)
    \intertext{and}
      \argmin_{x \in \R^n} \big( \| |\nabla x| \|_1 \st \Phi(x)  \leq \sigma \big)
        &\iseq \argmin_{x \in \R^n} \big( \| |\nabla x| \|_1 + \iota_{\lev_{\sigma}(\Phi(\cdot))}(x) \big),
    \end{align*}    
    cf. e.g. \cite{FigueiredoBioucasdias2009deconvolution},
    \cite{SteidlTeuber2010removing},
    \cite{CiShSt2012}.
\end{itemize}

All this minimization problems are of the form
\begin{gather} \label{eq:basic_minimization_problem}
  \argmin (F + G_\eta)
\end{gather}
with functions $F, G \in \Gamma_0 (\R^n)$ and some regularization 
parameter $\eta$; for $\eta \neq 0$ the function $G_\eta$ 
is often of the form 
\begin{gather*}
  G_\eta(\cdot) = G(\eta L \cdot)
\end{gather*}
with a matrix $L \in \R^{m, n}$
and a norm $G(\cdot) = \|\cdot\|$ on $\R^m$ in the penalized cases
and the indicator function $G = \iota_{\lev_1 G}$
in the constraint cases, respectively.

Two questions arise naturally:
How can a good regularization parameter be chosen?
How can $\argmin (F+G_\eta) \neq \emptyset$ be ensured?
Regarding the first question for penalized problems
\begin{gather*}
  \argmin_{x \in \R^n} \big( F(x) + \lambda  \|L x \|  \big)
\end{gather*}
there are for instance methods from statistics for choosing a value for
$\lambda$,
cf. 
\cite{Wahba1990spline},
\cite{BardsleyGoldes2009regularization},
\cite{HankeHansen1993regularization}.
However, in cases where we have knowledge about the original image 
$x_\orig$, say in the sense of knowing a good upper bound for 
$\|L x_\orig \|$, we can use this upper bound as value for 
the regularization  parameter in the constrained problem 
\begin{gather*}
  \argmin_{x \in \R^n} (F(x) \st \|L (x) \| \leq \tau).
\end{gather*}
If we have knowledge about the noise level, say in the sense 
of knowing approximately $F(x_\orig)$, we can similar 
choose this approximate value in the constrained problem
\begin{gather*}
  \argmin_{x \in \R^n} (\|L x \| \st F(x) \leq \sigma).
\end{gather*}
But even if we had chosen a good parameter
$\tau$, resp. $\sigma$, the questions remains
how we can find a corresponding value for $\lambda$.

Regarding the second question it is well known that 
the lower semicontinous function $F + G_\eta \eqdef H_\eta$ has 
a minimizer if it is coercive, i.e. fulfills 
$H_\eta(x) \rarr \pinfty$ as $\|x\| \rarr \infty$.
Often it is possible to prove coercivity of $H_\eta$ 
by hand. Since this can be laboriously it would be 
good to have some 
easy 
tools which ensure 
coercivity of such a sum.

This thesis provides contributions to both the question 
on how to find for given $\tau$ a corresponding value 
$\lambda$ and performs also coercivity investigations.

\section{Contributions and a useful inequality}

\subsection{A method for proving coercivity and lower semicontinuity}
As already mentioned coercivity is a usefull property 
for proving the existence of a minimizer. 
The defining condition
$H(x) \rarr \pinfty$ as $\|x\| \rarr \pinfty$ 
looks somewhat like a continuity condition.

As we will see in \prettyref{thm:characterize_lsc_and_coercive}
a lower semicontinous function $H: \R^n \rarr [\minfty, \pinfty] $ 
is indeed coercive iff a certain extension 
$\widehat H: \widehat X \rarr [\minfty, \pinfty]$ to a compact 
topological superspace of $\R^n \eqdef X$ 
is continuous with respect to a certain topology $\mathcal{T}_\leq$ on 
$[\minfty, \pinfty]$, making the latter to a compact space as well. 
This equivalence between the lower semicontinuity plus 
coercivity of the mapping $H$ and the existence of such a certain
\emph{compact continuation} $\widehat H$
leads to a -- as far as the author knows -- new technique 
of proving lower semicontinuity plus coercivity. The rough idea is 
as follows:
Assume we know that a function 
$g: \R^n \rarr [\minfty, \pinfty]$ can be written as, say, composition 
$g = g_2 \circ g_1$ of easier functions 
$g_1: \R^n \rarr Y$, $g_2: Y \rarr [\minfty, \pinfty]$,
where $Y$ is some topological space, such that 
each of them allows a compact continuation 
$\widehat g_1: \widehat X \rarr \widehat Y$ and
$\hill g_2: \hill Y \rarr [\minfty, \pinfty]$. 
Under certain conditions then also the existence of the needed
compact continuation $\widehat g$ of  $g$ can be concluded. 
The needed compact continuation $\widehat g$ is simply obtained  
if we can directly form the concatenation 
$\hill g_2 \circ \widehat g_1$, i.e. if 
$\widehat Y = \hill Y$. Also if $\id_Y$ allows a compact continuation 
$\widehat{\id_Y}: \widehat Y \rarr \hill Y$ we are done after setting 
$\widehat g = \hill{g_2} \circ \widehat{\id_Y} \circ \widehat g_1$.
More surprising and more important is the fact that 
the needed compact continuation $\widehat g$
also exists (under certain conditions)
if the mapping $\id_Y$ allows a compact continuation 
$\hill{\id_Y}: \hill Y \rarr \widehat Y$, cf. 
\prettyref{thm:concat_compactly_continuable_mappings} and 
\linebreak[1]
\prettyref{thm:concat_compactly_continuable_mappings_special_case}.
Although the developed theory is quite rudimentary it is already 
strong enough to easily prove 
for example
the following often applied result in image restoration which was
indeed the starting point of my work.
\\
\\
{\it
  Assume that the following mappings are given:
  \begin{enumerate}
    \item 
      Two matrices / linear mappings 
      $H:\R^n \rarr \R^d, K: \R^n \rarr \R^e$ with
      \begin{gather*}
        \mathcal{N}(H) \cap \mathcal{N}(K) = \{\zerovec\}.
      \end{gather*}
    \item 
      Two proper, lower semicontinuous and coercive mappings
      $
	\phi: \R^d \rarr [\minfty, \pinfty]
      $, \linebreak
      $
	\psi: \R^e \rarr [\minfty, \pinfty]
      $.
  \end{enumerate}
  Then the mapping $h: \R^n \rarr [\minfty, \pinfty]$, given by
  \begin{gather*}
    x \mapsto \phi(Hx) + \psi(Kx)
  \end{gather*}
  is lower semicontinuous and coercive. In particular the mapping $h$
  takes his infimum $\inf h \in [\minfty,\pinfty]$ at some point in $\R^n$.
}
\\\\
The corresponding proof can be found in 
\prettyref{sec:application_to_an_example}.

\subsection{Properties of lower semicontinuous mappings from a topological viewpoint}

In the previous section we have mentioned the topology 
$\mathcal{T}$ for $[\minfty, \pinfty]$. More precise this 
is the right order topology which is induced by the 
natural order on $[\minfty, \pinfty]$. This is the 
natural topology for studying lower semicontinuity,
since a function $\R^n \rarr [\minfty, \pinfty]$ is lower 
semicontinuous iff it is continuous with respect to 
the topology $\mathcal{T}$ on $[\minfty, \pinfty]$.
After investigating some properties of the 
topological space $([\minfty, \pinfty], \mathcal{T})$
we will see in
\prettyref{subsec:known_properties_of_lsc_functions_revisited}
that some well known (and easy to prove)
properties of lower semicontinous functions are just 
special cases of common theorems from topology.
For instance the general statement 
\begin{center}
  {\it ``The concatenation $g \circ f$ of continuous mappings $f,g$
is again continuous.'' }
\end{center}
becomes in this context the 
property 
\begin{center}
  {\it
  ``The concatenation $g \circ f$ of a 
continuous mapping $f$ with a \\ lower semicontinuous mapping 
$g$ is again lower semicontinous.''
}
\end{center}
In the same way we can also regard the fact that a 
lower semicontinuous function $f$ takes its infimum on every compact set:
The general statement
\begin{center}
  {\it
  ``A continuous function maps compact sets onto compact sets''
  }
\end{center}
reads in our context
\begin{center}
  {\it
  ``A lower semicontinous function maps compact sets 
  \\on sets which contain their infimum.''
  }
\end{center}

\subsection{Coercivity of a sum of functions}

%
\prettyref{thm:sum_coercive_on_certain_subspaces} can be used as an
easy to apply tool for investigating 
coercivity of a sum of functions. 
More precisely, this theorem gives information on which subspaces 
of $\R^n$ a sum $F + G$ 
of functions $F, G : \R^n \rarr [\minfty, \pinfty]$ 
is coercive, provided that $F$ and $G$ are of a certain form, namely
\begin{align*}  
  &F = F_1 \sdirsum F_2 \quad \text{  and  }  \quad G = G_1 \sdirsum G_2 
\end{align*}
with functions
$F_1: X_1 \rarr \R \cup \{\pinfty\}$, 
$F_2: X_2 \rarr \R \cup \{\pinfty\}$,
$G_1: Y_1 \rarr \R \cup \{\pinfty\}$, \linebreak[2]
and 
$G_2: Y_2 \rarr \R \cup \{\pinfty\}$,
where 
\begin{gather*}
  \R^n = X_1 \oplus X_2 = Y_1 \oplus Y_2.
\end{gather*}
For such functions the theorem basically states that $F + G$ 
is coercive on 
$X_1 + Y_1 = (X_2 \cap Y_2)^\perp$ if 
$X_1 \perp X_2$, $Y_1 \perp Y_2$ and 
certain boundedness conditions hold true.

If the conditions 
$X_1 \perp X_2$, $Y_1 \perp Y_2$ are not fulfilled 
there is no guarantee that $F + G$ is coercive on 
$X_1 + Y_1$. But at least
$F+G$ is then still coercive on all those subspaces $Z_1$ of $\R^n$
that are complementary to $Z_2 \defeq X_2 \cap Y_2$.

\subsection{Relation between the constrained and unconstained problems for a rather general setting}

In \cite{CiShSt2012} Ciak et al. considered for an underlying 
orthogonal decomposition $\R^n = X_1 \oplus X_2$ of $\R^n$
the primal minimizations problems
\begin{eqnarray*}
  (P_{1,\tau}) && \argmin_{x \in \R^n}  \left\{ \Phi(x) \st \| Lx\| \le \tau \right\} \\  
  (P_{2,\lambda}) && \argmin_{x \in \R^n} \left\{ \Phi(x) \; + \; \lambda \| Lx\| \right\} 
\end{eqnarray*}
along with the dual problems 
\begin{eqnarray*}
  (D_{1,\tau}) && \argmin_{p \in \R^m} \left\{\Phi^*(-L^* p) \; + \; \tau \|p\|_* \right\},  \\
  (D_{2,\lambda}) && \argmin_{p \in \R^m} \left\{ \Phi^*(-L^* p)  \st  
  \| p \|_* \le \lambda \right\} .
\end{eqnarray*}
The function $\Phi$ there has the special form
\begin{gather*} 
  \Phi(x) = \Phi(x_1 + x_2) =  \phi(x_1),
\end{gather*}
where 
$\phi:X_1 \rightarrow \R \cup \{+\infty\}$ is a function 
fulfilling some properties. 

In this thesis we extend this setting by allowing a third component 
in the orthogonal decomposition of $\R^n = X_1 \oplus X_2 \oplus X_3$ 
and demand
\begin{gather*}
  \Phi(x) = \Phi(x_1 + x_2 + x_3)
  =
  \begin{cases}
    \phi(x_1) & \text{ if } x_3 = \zerovec,
  \\
    \pinfty   & \text{ if } x_3 \neq \zerovec.
  \end{cases}
\end{gather*}
This extension can become interesting when dealing with data in
a high dimensional real vector space if
the data is actually contained in a lower dimensional subspace.
Moreover, this extended form has the advantage that 
a symmetry between $\Phi$ and $\Phi^*$ is recognizable
much better in this extended setting as we shall see in 
\prettyref{lem:subgradient_and_conjugate_function_in_our_setting}.

\subsection{A simple but useful equality}

Here we want to mention  
\prettyref{lem:same_blow_up_factor_gets_all_attachted_affines_strangers_out_of_ball}
from the appendix along with its preceding vivid explanation.
    The simple but helpful inequality presented in that lemma is
    \begin{gather*}
      \|h_1\| \leq C \|h_1 + h_2\|
    \end{gather*}
    for all $h_1$ and $h_2$ in subspaces $X_1, X_2$ of $\R^n$ with 
    trivial intersection.
    Originally 
    this inequality was made and proved in the context of 
    \prettyref{lem:essentially_smoothness_does_not_depend_on_periodspace}, 
    in which proof it was twice used for showing differentiability.
    However it turned out that using this inequality also simplifies 
    the boundedness proof in
    \cite[Lemma 3.1 (i)]{CiShSt2012}
    as done in the proof of part 
    \upref{enu:intersection_of_lower_level_sets_bounded} of    
    \prettyref{lem:levelsets_and_existence_of_minimizer_of_sum}.
    Moreover this inequality was helpful in showing convergence 
    of a sequence which appeared in the proof of 
    \prettyref{lem:addition_restricted_is_homeo}.

\section{Overview} \label{sec:overview}

This thesis consists of three parts, organized in Chapters
\ref{chap:coercivity_and_lscty_from_the_top_point_of_view},
\ref{chap:coercivity_of_sum_of_functions} and
\ref{chap:penalizers_and_constraints_in_convex_problems}.
In the first part we develop a theory giving rise to a 
-- as far as the author knows -- new technique of proving
lower semicontinuity plus coercivity of functions 
$h$. The main ingredients are as follows:
\begin{itemize}
  \item 
    Equivalence of lower semicontinuity plus coercivity to 
    the existence of a certain compact continuation $\widehat h$
    of $h$.
  \item 
    An analysis of compact continuations,
    giving a criteria 
    for ensuring that a concatenate function $h = g \circ f$ allows 
    a compact continuation $\widehat h$ if $g$ and $f$ have 
    a compact continuation $\widehat f$ and $\hill g$.
\end{itemize}
Having a function $h: \R^n \rarr [\minfty, \pinfty]$ we can hence 
perform the strategy to write this mapping as composition 
$h = g \circ f$ with mappings $f$ and $g$ that allow certain 
compact continuations in a first step. In a second step 
we can then try to get the needed extension of $h$.

The first part is organized as follows:
After recalling some set theoretic topology we introduce the 
right order topology for the set $[\minfty, \pinfty]$ 
and prove the mentioned equivalence.
Then the concept of compact continuations is introduced.
An application of the theory to an example concludes the first part.

The second part also deals with coercivity. However,  
lower semicontinuity no longer plays a role in this part. 
After giving definitions and developing some lemmata we address
the easy case of linear mappings before moving towards
the main theorem of this chapter, giving
information on which subspaces of $\R^n$
certain sums $(F_1 \sdirsum F_2) + (G_1 \sdirsum G_2)$ are coercive.

In the third part we are interested in the relation between the
convex constrained optimization problem
\begin{align} \label{eq:intro_constraint} 
  &(P_{1,\tau}) \qquad \argmin_{x \in \mathbb R^n} \left\{ \Phi(x) \st \Psi(x) \le \tau \right\}
\intertext {and the unconstrained optimization problem}
\label{eq:intro_nonconstraint} 
  &(P_{2,\lambda}) \qquad \argmin_{x \in \mathbb R^n} \{ \Phi(x) +  \lambda \Psi(x) \}, \;\; \lambda \ge 0.
\end{align}

The constrained problem \eqref{eq:intro_constraint} is interesting only for
$\tau \in OP(\Phi, \Psi)$ and can then
be rewritten as the following unconstrained one:
\begin{equation} \label{eq:constraint_1}
\argmin_{x \in \mathbb R^n} \{ \Phi(x) +  \iota_{{\rm lev}_\tau \Psi} (x) \}.
\end{equation}

In the inverse problems and machine learning context 
the problems \eqref{eq:intro_constraint} and 
\eqref{eq:intro_nonconstraint} are referred
to as Ivanov regularization and
Tichonov
regularization of 
optimization problems of the form
$\argmin_{x \in \mathbb R^n} \{ \Phi(x)\}$.

Let ${\rm SOL}(P_\bullet)$ denote the set of solutions of problem $(P_\bullet)$.
While it is rather clear that under mild conditions on $\Phi$ and $\Psi$ a vector
$\hat x \in {\rm SOL}(P_{2,\lambda})$, $\lambda >0$ is also a solution of $(P_{1,\tau})$
exactly for $\tau = \Psi(\hat x)$, the opposite direction has in general no simple explicit solution.
At least it is known that, under certain conditions, for $\hat x \in {\rm SOL} (P_{1,\tau})$
there exists $\lambda \ge 0$ such that $\hat x \in {\rm SOL} (P_{2,\lambda})$.
This result, beeing stated in 
\prettyref{thm:constraint_vs_nonconstraint} and 
\prettyref{cor:constraint_vs_nonconstraint},  
can be shown by using that the relation
\begin{gather*}
  \R_0^+ \, \partial \Psi(x) = \partial \iota_{{\rm lev}_{\Psi(x)} \Psi} (x)
\end{gather*}
from \cite[p. 245]{HL93} holds true under certain conditions.
This result is presented in 
Lemma \ref{level_sets} and proved by 
using an epigraphical projection or briefly inf-projection, cf. \cite[p. 18+]{RW04},
which allows reducing the intrinsic problem to one dimension. 

After developing some assisting theory we consider 
particular
problems 
where 
\begin{align*}
 &\Phi(x) \defeq  \phi(x_1)
  && \text{ and } && \Psi \defeq \| L \cdot\| \text{ with } L \in \R^{m,n};
\end{align*}
here $x_1$ is the orthogonal projection of $x\in \dom \Phi$
onto a subspace $X_1$ of $\mathbb R^n$
and 
$\phi: X_1 \rightarrow \mathbb R \cup \{ +\infty \}$ is a function
which fulfills the following conditions:
\begin{enumerate}
  \item 
    $\dom \phi$ is an open subset of $X_1$ with $\zerovec \in \overline{ \dom \phi }$,
  \item 
    $\phi$ is proper, convex and lower semicontinuous
    as well as strictly convex and essentially smooth,
    and
  \item 
    $\phi$ has a  minimizer.
\end{enumerate}
We use the dual problems 
to prove that in a certain interval there is a one-to-one
correspondence between $\tau$ and $\lambda $ in the sense that
$
{\rm SOL}(P_{1,\tau}) = {\rm SOL}(P_{2,\lambda}) 
$
exactly for the corresponding pairs.
Furthermore, given $\tau$, the value $\lambda$ is determined by $\lambda \defeq \| \hat p\|_*$,
where $\hat p$ is any solution of the dual problem of $(P_{1,\tau})$.
See \prettyref{thm:theo_lambda_tau} for more details.
\\[1ex]
The third part is organized as follows:
We first deal with two ways of interpreting each of the
minimization problems $(P_{1,\tau})$ and $(P_{2,\lambda})$ 
and show that these perspectives, though related, are 
not equivalent in general.
In \prettyref{sec:penalizers_and_constraints} we state
a known relation between $(P_{1,\tau})$ and $(P_{2,\lambda})$ 
for a rather general setting, see 
\prettyref{thm:constraint_vs_nonconstraint}.
In particular, we provide some novel proofs 
by making use of 
an epigraphical projection.
We also recall Fenchel's Duality relation. 
Finally we discuss 
the mentioned \prettyref{thm:constraint_vs_nonconstraint} 
more in detail. In particular a 
relation between one of its regularity assumptions and Slaters Constraint 
Qualification is given.
In close connection with \prettyref{sec:penalizers_and_constraints} is
\prettyref{sec:homegeneous_penalizers_and_constraints}, where we restrict ourselves to homogeneous
regularizers and to essentially smooth data terms, which are 
strictly convex on a certain subspace of $\R^n$.
We prove a relation between the  parameters $\tau$ and $\lambda$ 
such that the solution sets of the corresponding constrained and unconstrained problems coincide
and determine the $\lambda$ corresponding to $\tau$ by duality arguments.
The intermediate 
\prettyref{sec:assisting_theory} provides some theorems and lemmata 
needed in the proofs of 
\prettyref{sec:homegeneous_penalizers_and_constraints},
some of which are interesting in themselves.
In the Appendix some useful theorems are collected.
The parts there which are not own work but are taken from the literature
are clearly indicated by giving references.

Applications can be found in Section 4 of 
\cite{CiShSt2012}.
Ideas from this chapter were also used in 
\cite{TeuberSteidlCahn2013minimization}.

\chapter{Coercivity and lower semicontinuity from the topological point of view}  
\label{chap:coercivity_and_lscty_from_the_top_point_of_view}

\minitoc


For convenience we will call a topological space also just ``space''
in this chapter.

\section{On the relation between closed and compact subsets} \label{sec:relation_between_closed_and_compact_subsets}

In this section we recall a known theorem, describing the relation 
between compactness and closeness.
\begin{theorem} \label{thm:relation_closed_compact}
  ~
  \begin{enumerate}
    \item \label{enu:closeness_implies_compactness}
      Each closed subset of a compact space is compact.
    \item \label{enu:compactness_implies_closeness}
      Each compact subset of a Hausdorff space is closed.
  \end{enumerate}

\end{theorem}

The subsequent proof 
resembles the proof of Bemerkung 2 in
\cite[ch. 1.8 on p. 26]{Jnch2001}
and the proof of a Lemma in
\cite[ch.1.8 on p. 28]{Jnch2001}.
\begin{proof}

  \upref{enu:closeness_implies_compactness}
  Let $\topospace{X}$ be a compact space and $A$ a closed subset of 
  this space.
  Let $A$ be covered by open sets $O_i \in \mathcal{O}_X, i \in I$.
  Adding the open set $X \setminus A \in \mathcal{O}_X$ to the 
  $O_i, i \in I$, yields an open covering of $\topospace{X}$.
  Due the compactness of $\topospace{X}$ finitely many of the $O_i$ 
  together with $X \setminus A$ suffice to cover $X$. Due to
  $(X \setminus A) \cap A = \emptyset$ these finitely many $O_i$
  must already cover $A$. So $(A, A \Cap \mathcal{O}_X)$ is compact.
  \\
  \upref{enu:compactness_implies_closeness}
  Let $\topospace{X}$ be a Hausdorff space and $A$ some compact subset.
  For proving the closeness of $A$ it suffices to show that each
  $x\in X \setminus A$ is an interior point of $X \setminus A$, i.e. that
  there is an open neighborhood $U$ of $x$ with $U \subseteq X \setminus A$.
  To this end we fix $x \in X \setminus A$. Since $\topospace{X}$ is a
  Hausdorff space, there are disjoint open neighborhoods
  $O_a \in \mathcal{U}(a)$ and $U_a \in \mathcal{U}(x)$ for every $a \in A$.
  The open cover of the compact set $A$ by the $O_a, a \in A$ has a 
  finite subcover; i.e. there are finitely many $a_1, \dots, a_n \in A$ with
  $\bigcup_{i=1}^{n} O_{a_i} \supseteq A$. The set $\bigcap_{i=1}^{n}U_{a_i}$ 
  is an open neighborhood of $x$ with 
  \[
    \bigcap_{i=1}^{n} U_{a_i} \cap A
    \subseteq
    \bigcap_{i=1}^{n} U_{a_i} \cap \bigcup_{j=1}^{n} O_{a_j}
    =
    \bigcup_{j=1}^{n} \left( \bigcap_{i=1}^{n} U_{a_i} \cap O_{a_j} \right)
    \subseteq
    \bigcup_{j=1}^{n} \left( U_{a_j} \cap O_{a_j} \right)
    =
    \emptyset,
  \]
  i.e. $\bigcap_{i=1}^{n} U_{a_i} \subseteq X \setminus A$.
  So $x$ is indeed an interior point of $X \setminus A$.
\end{proof}

We point out that even a compact topological space 
can have compact subsets which are not closed.
An example for this behavior is obtained when equipping 
the interval $[\minfty, \pinfty]$ with the right order topology,
see \prettyref{exa:compact_not_closed_subset}.

\section{Remarks on the topology induced by a metric space} \label{sec:metric_space}

In this subsection we first recall some well known facts for 
the topology induced by a metric.
Then we recall the equivalence of
metric continuity concepts and topological continuity concepts.

\begin{definition} 
  Let $(X, d)$ be a metric space. The topology generated by the ''open`` balls
  $\openball[r][x]$, $r>0$, $x \in X$, i.e. the topology
  \[
    \mathcal{O}[d] 
    \defeq 
    \{O \subseteq X : O \text{ is union of ''open`` balls }\},
  \]
  will be called {\bf topology induced by $\bm{d}$}. If it is clear 
  from the context we will also use the short form $\mathcal{O}$ for
  $\mathcal{O}[d]$
\end{definition}

\begin{remark}
  The 
  open
  balls $\openball[r][x]$, $r>0$, $x \in X$ 
  are really open sets from $\mathcal{O}[d]$.
\end{remark}

\begin{remark} \label{rem:metric_continuity_vs_topological_continuity}
  Let $(X,d)$, $(X', d')$ be metric spaces and $(X,\mathcal{O})$, 
  $(X',\mathcal{O}')$ the induced topological spaces.
  For a mapping $f: X \rarr X'$ the metric continuity notions and the 
  topological continuity notions are the same;
  speaking in particular about the \emph{continuity} in a 
  single point $x_0$ we have the equivalence of the following statements
  \begin{enumerate}
    \item 
      $f: (X,d) \rarr (X',d')$ is continuous in $x_0$ in the metric sense, i.e.
      \\
      $
	\forall \varepsilon>0 \;\; \exists \delta>0 
	\;\; \forall x \in X:
	\\
	d(x,x_0) < \delta \implies d'(f(x), f(x_0)) < \varepsilon
      $
    \item
      $f: (X,\mathcal{O}) \rarr (X',\mathcal{O}')$ is continuous in $x_0$ 
      in the topological sense, i.e.,
      \\
      for every open neighborhood $O'\in \mathcal{O}'$  of $f(x_0)$ there is an
      open neighborhood $O \in \mathcal{O}$ of $x_0$ with $f[O] \subseteq O'$ 
      (which is to say $O \subseteq f^{-}[O']$).
  \end{enumerate}

  Similarly, speaking about continuity of the whole function, we have the 
  equivalence of the statements
  \begin{enumerate}
    \item 
      $f: (X,d) \rarr (X',d')$ is continuous in the metric sense, i.e.
      \\
      $
	\forall x_0 \in X \;\; \forall \varepsilon>0 \;\; \exists \delta>0 
	\;\; \forall x \in X:
	\\
	d(x,x_0) < \delta \implies d'(f(x), f(x_0)) < \varepsilon
      $
    \item
      $f: (X,\mathcal{O}) \rarr (X',\mathcal{O}')$ is continuous in the topological sense, i.e.
      \\
      $
	\forall O' \in \mathcal{O'}: f^{-}[O'] \in \mathcal{O}
      $.
  \end{enumerate}

\end{remark}

\section{Creating topological spaces from given ones} 
\label{sec:topology__creating_new_spaces}

%

In this section we give a short introduction in four known 
ways of generating topological spaces from given ones: 
\begin{itemize}
  \item
    In \prettyref{subsec:subspaces} we
    discuss how a subset of  a topological space can be made to a subspace 
    by giving it the ''correct`` topology. 
  \item 
    In \prettyref{subsec:product_spaces} we
    show how to equip finite products of topological spaces with a 
    meaningful topology.
  \item
    In \prettyref{subsec:identification_spaces} we deal with the vivid notion 
    of glueing a given object and how we can formalize it in the language of
    topology.
  \item 
    In \prettyref{subsec:one_point_compactification} we extend every 
    topological space to a compact one by adding one single new point.
\end{itemize}

In each of this four subsections we give motivations for the definition.
We remark that our motivation for the identification topology
seems to be new.

%

\subsection{Subspaces} \label{subsec:subspaces}
Let $(X,d)$ be a metric space and 
$
  (\widecheck X, \widecheck d)
  =
  (\widecheck X, d|_{\widecheck X \times \widecheck X})
$
some metric subspace. After choosing a point 
$\widecheck x \in \widecheck X \subseteq X$
and some ''radius`` $r > 0$ we can think of an open ball of radius $r$ 
around $\widecheck x$ in two ways -- on on the one hand with respect to 
$(\widecheck X, \widecheck d)$ and the other hand with respect to
$(X, d)$. Though they are different in general, they are linked via

\begin{align*}
  \openball[r][\widecheck x]{}[\widecheck d]
  =&
    \{x \in \widecheck X:  d(x, \widecheck x) < r \}
    \\
  =& 
    \widecheck X \cap \{x \in X : d(x, \widecheck x) < r\}  
    \\
  =& 
    \widecheck X \cap \openball[r][\widecheck x]{}[d].
\end{align*}

For any $\widecheck x_i \in \widecheck X$ and $r_i > 0, i\in I$ we
therefore have
\[
  \bigcup_{i \in I} \openball[r_i][\widecheck x_i]{}[\widecheck d]
  =
  \widecheck X
  \cap 
  \bigcup_{i\in I} \openball[r_i][\widecheck x_i]{}[d].
\]
So 
$
  \mathcal{O}[(\widecheck X, \widecheck d)] 
  = 
  \widecheck X \Cap \mathcal{O}[(X, d)].
$
This gives rise to the following definition.

\begin{definition}
  Let $(X,\mathcal{O})$ be a topological space and $\widecheck X \subseteq X$.
  We call $(\widecheck X, \widecheck {\mathcal{O}})$ a {\bf subspace} of 
  $(X, \mathcal{O})$, iff 
  $\widecheck {\mathcal{O}} = \widecheck X \Cap \mathcal{O}$. The topology
  $\widecheck X \Cap \mathcal{O}$ is called {\bf subspace topology} for 
  $\widecheck X \subseteq X$.
  To the contrary a topological space $(X, \mathcal{O})$ is called a 
  {\bf superspace} of a space $(\widecheck X, \widecheck {\mathcal{O}})$, iff
  the latter is a subspace of the first.
\end{definition}

The following remark illuminates that the above topology is the 
appropriate topology 
for subsets of a already given topological space. It states that the
continuity of a function
$f: (X, \mathcal{O}) \rarr (Y, \mathcal{P})$
does not get lost by restricting its domain and by extending its codomain:
\begin{remark}
  Let $(X, \mathcal{O})$ be a topological space with some subspace
  $
    (\widecheck X, \widecheck {\mathcal{O}}) 
    = 
    (\widecheck X, \widecheck X \Cap \mathcal{O})
  $ and let $(Y, \mathcal{P})$ be a topological space with some superspace
  $(\widehat Y, \widehat {\mathcal{P}}) $.
  Then the following holds true for all mappings $f: X \rarr Y$:
  \begin{enumerate}
    \item 
      $f: (X,\mathcal{O}) \rarr (Y,\mathcal{P})$ is continuous
      $
	\implies
	f|_{\widecheck X}: (\widecheck X,\widecheck{\mathcal{O}}) \rarr (Y, \mathcal{P})
      $
      is continuous.
 
    \item 
      $f: (X,\mathcal{O}) \rarr (Y,\mathcal{P})$ is continuous
      $
	\iseq
	f: (X,\mathcal{O}) \rarr (\widehat Y, \widehat {\mathcal{P}})
      $
      is continuous.
  \end{enumerate}

\end{remark}

\subsection{Product spaces} \label{subsec:product_spaces}

Let $(Y_1, \mathcal{O}_1), \dots , (Y_n, \mathcal{O}_n)$ be 
topological spaces. We search a topology $\mathcal{O}$ for the Cartesian
product $Y \defeq Y_1, \times \dots \times Y_n$ such that for any sequence
$y^{(k)}$ in $Y$ the equivalence
\begin{gather*}
  (\forall i \in \{1,\dots , n\}:
  y_i^{(k)} \rarr y_i^{\ast})
  \iseq
  y^{(k)} \rarr y^{\ast}
\end{gather*}
holds true.
To this end we express the left hand side as explicit statement
\begin{equation}
  \forall i \in \{1, \dots , n\} \;\; 
  \forall U_i \in \mathcal{U}_i(y_i^{\ast}) \;\;
  \exists \widecheck k_i \in \N \;\;
  \forall k \geq \widecheck k_i: \;\;
  y_i^{(k)} \in U_i
\label{eq:convergence_in_factors}
\end{equation}
and compare it with the explicit formulation
\begin{equation}
  \forall U \in \mathcal{U}(y^\ast) \;\; 
  \exists \widecheck k \in \N \;\;
  \forall k \geq \widecheck k: \;\;
  y^{(k)} \in U
\label{eq:convergence_in_product_space}
\end{equation}
for the right--hand side. 
On the one hand, to guarantee 
``\prettyref{eq:convergence_in_product_space} 
  $\Rarr$
  \prettyref{eq:convergence_in_factors}'',
we should demand that every product
$U \defeq U_1 \times \dots \times U_n$, 
where $U_i \in \mathcal{U}_i(y_i^\ast)$,
is already a neighborhood of $y^\ast$.
On the other hand, to guarantee
``\prettyref{eq:convergence_in_product_space} 
  $\Larr$
  \prettyref{eq:convergence_in_factors}'',
all those subsets $\widecheck Y \subseteq Y$, which do not contain any
product $U_1 \times \dots \times U_n$ with $U_i \in \mathcal{U}_i(y_i^\ast)$, 
should be barred from beeing a neighborhood of $y^\ast$; i.e we should
demand that every $U \in \mathcal{U}(y^\ast)$ contains some product
$U_1 \times \dots \times U_n$ of neighborhoods 
$U_i \in \mathcal{U}_i(y_i^{\ast})$.
Altogether it seems reasonable to demand
\[
  U \in \mathcal{U}(y^\ast) 
  \defiseq
  \exists U_1 \in \mathcal{U}_1(y_i^\ast), 
  \dots, 
  U_n \in \mathcal{U}_n(y_n^\ast):
  U \supseteq U_1 \times \dots \times U_n
\]
This leads to the following
\begin{definition}
  Let $(Y_1,\mathcal{O}_1), (Y_2,\mathcal{O}_2), \dots , (Y_n, \mathcal{O}_n)$ 
  be \emph{finitely} many topological spaces. A topology $\mathcal{O}$ 
  on the Cartesian product
  $Y_1 \times Y_2 \dots \times Y_n \eqdef Y$ is said to be \emph{the}
  {\bf product topology} of 
  $\mathcal{O}_1, \mathcal{O}_2, \dots, \mathcal{O}_n$, if one of the
  following equivalent conditions is fulfilled:
  \begin{enumerate}
    \item 
      The neighborhood system $\mathcal{U}(y^\ast)$ of a point $y^\ast \in Y$ 
      exactly consists
      of the sets $U = U_1 \times \dots \times U_n$, where 
      $U_i \in \mathcal{U}_i(y_i^\ast)$, $i \in \{1,\dots, n\}$, and of all
      subsets of $Y$ which are supersets of these sets $U$.
    \item
      The topology $\mathcal{O}$ consists exactly of those subsets 
      $O \subseteq Y$, which are of the form 
      $O_1 \times \dots \times O_n$ with any
      $O_i \in \mathcal{O}_i, i \in \{1, \dots , n\}$, 
      or can be written as union of sets of this form.
  \end{enumerate}
  The {\bf product space} $(Y, \mathcal{O})$ will be denoted by
  \begin{align*}
      (
	Y_1 \times Y_2 \times \dots \times Y_n
	,
	\mathcal{O}_1 
	  \varotimes 
	\mathcal{O}_2 
	  \varotimes
	  \dots 
	  \varotimes 
	\mathcal{O}_n
      )
    \shortintertext{or by}
      (Y_1, \mathcal{O}_1) 
	\varotimes 
      (Y_2, \mathcal{O}_2) 
	\varotimes 
	\dots 
	\varotimes 
      (Y_n, \mathcal{O}_n).
  \end{align*}
As a shorter notation for
  $
    \underbrace{(Y,\mathcal{O}) \varotimes \dots \varotimes (Y, \mathcal{O})}
    _{ n \text { times} }
  $
  we will also write $(Y^n, \mathcal{O}^{\otimes n})$.

\end{definition}

 
%

In most cases we deal with $Y = \R$ equipped with its natural topology
$\mathcal{O} = \mathcal{O}[d]$, where $d$ is the natural metric 
defined by $d(x,y) = |x-y|$.
The product topology 
$\mathcal{O}^{\otimes n}$ for 
$\R^n$ equals its natural topology, i.e.
the topology generated by every norm on $\R^n$.

\begin{remark}
  Let $\mathcal{O}_1, \mathcal{O}_2, \mathcal{O}_3$ be some topologies.
  Then we have
  $
    \mathcal{O}_1 \otimes \mathcal{O}_2 \otimes \mathcal{O}_3
    =
    (\mathcal{O}_1 \otimes \mathcal{O}_2) \otimes \mathcal{O}_3
    =
    \mathcal{O}_1 \otimes (\mathcal{O}_2 \otimes \mathcal{O}_3)  
  $,
  i.e. building product spaces is an associative operation.
\end{remark}

The following remark illuminates that the above defined topology is the 
appropriate topology for the Cartesian product of already given topological spaces.
It states that a ``multivalued'' function is continuous iff its 
component functions are continuous.
\begin{remark}
  A mapping
  $
  f:(X,\mathcal{O}) \rarr 
  (Y_1, \mathcal{O}_1) \varotimes (Y_2, \mathcal{O}_2) \varotimes 
  \dots \varotimes (Y_n, \mathcal{O}_n),
  x \mapsto f(x) = (f_1(x), f_2(x), \dots, f_n(x))
  $
  is continuous if and only if all its component functions
  $f_i: (X, \mathcal{O}) \rarr (Y_i, \mathcal{O}_i)$, $i \in \{1, \dots n\}$, 
  are continuous.
\end{remark}

Next we state Tichonov's Theorem for the 
simple case of building the product of only finitely many 
compact spaces. For a proof see
\cite[Theorem 5.7 on p. 167]{Munkres1975}.
\begin{theorem}[Tichonov's Theorem for finite products]\label{thm:little_tichonov}
  The product space of finitely many compact spaces is 
  compact.
\end{theorem}

\begin{remark}
  We only introduced the product space of finitely many topological spaces.
  Although it is possible to declare a product space also for
  infinitely many topological spaces, we have decided to avoid this, more
  complicated and harder to grasp, construction, since we will not need it.
\end{remark}

We conclude this subsection with a remark showing that 
the order in which the actions of building subspaces and 
product spaces are done have no influence on the finally resulting 
topological space:

\begin{remark} 
\label{rem:productspace_of_supspaces_is_a_subspace_of_productspace}
  Given two topological spaces $(\widehat X_1,\widehat {\mathcal{O}}_1 )$ and 
  $(\widehat X_2, \widehat {{\mathcal{O}}}_2)$, the Cartesian product $X_1 \times X_2$ 
  of two subsets $X_1 \subseteq \widehat X_1$ and $X_2 \subseteq \widehat X_2$ 
  has to be equipped with a topology.
  Two natural ways of equipping $X_1 \times X_2$ with a topology seem possible:
  On the one hand $X_1 \times X_2$ can be interpreted as subset of 
  $\widehat X_1 \times \widehat X_2$ and thus be equipped with the subspace topology
  \begin{gather*}
    (X_1 \times X_2)
    \Cap
    (\widehat {\mathcal{O}}_1 \otimes \widehat {\mathcal{O}}_2).
  \end{gather*}
  On the other hand $X_1 \times X_2$ can be seen as Cartesian product of the sets
  $X_1$ and $X_2$ and thus be equipped with the product topology
  \begin{gather*}
    (X_1 \Cap \widehat {\mathcal{O}}_1) 
    \otimes 
    (X_2 \Cap \widehat {\mathcal{O}}_2 ).
  \end{gather*}
  Luckily these topologies are actually identical since the sets 
  \begin{gather*}
    (\widehat{O}_1 \times \widehat {O}_2) \cap ({X_1 \times X_2})
    =
    (\widehat {O}_1 \cap X_1) \times (\widehat{O}_2 \cap X_2),
  \end{gather*}
  where $\widehat{O}_1 \in \widehat {\mathcal{O}}_1$,
  $\widehat{O}_2 \in \widehat {\mathcal{O}}_2$, form 
  a base for both topologies.
\end{remark}

\subsection{Identification or quotient spaces} \label{subsec:identification_spaces}

In the following example let $\mathcal{O}$ be 
the natural topology of $\R$ and 
$\mathbb{S} \defeq \{ x\in \R^2 : \|x\|_2 = 1 \}$.

\rem{Im Latex Code steht noch die alte Version mit der offenen Frage,
ob beide Definitionen wirklich äquivalent sind}
\DissVersionForMeOrHareBrainedOfficialVersion{\\---------------------------------------------------------------------------------------------------------------------
    \\ {\bf Die alte Version steht noch im Tex Code und sehnt sich nach Erl\"osung}
\\---------------------------------------------------------------------------------------------------------------------
}{} 

\begin{example} \label{exa:motivation_of_identification_topology}
  Consider the surjective and continuous mapping
  $
    f: 
    ([0,2\pi], [0, 2\pi] \Cap \mathcal{O} )
    \rarr
    (\mathbb{S}, \mathbb{S} \cap \mathcal{O}^{\otimes 2})
  $,
  given by
  \begin{gather*}
    x
    \mapsto
    e^{ix} 
    = 
    \begin{pmatrix}
      \cos x
      \\
      \sin x
    \end{pmatrix}.    
  \end{gather*}
  
  The impression occurs that the straight line 
  $([0, 2\pi], [0, 2\pi] \Cap \mathcal{O})$ 
  is transformed
  to the circle line $(\mathbb{S}, \mathbb{S} \Cap \mathcal{O}^{\otimes 2})$
  by gluing the endpoints $0$ and $2\pi$ to one and the same point
  $ (1,0)^\tT = f(0) = f(2\pi)$ of the circle line. At any other point 
  $x \in (0, 2\pi)$, where nothing is glued, it seems that nothing 
  essential changes:
  A small interval--like--neighborhood $U'$ of $f(x)$ seems to be just 
  the image $f[U]$ of some small interval--neighborhood $U$ of $x$. 
\DissVersionForMeOrHareBrainedOfficialVersion{In contrast, a small interval--like--neighborhood $U'$ of 
  $(1,0)^\tT = f(0) = f(2\pi) $ appears to come into beeing, from gluing a small
  neighborhood, say $[0, \varepsilon_1)$, of $0 \in [0, 2\pi]$,  with a small
neighborhood, say $(2\pi-\varepsilon_2, 2\pi]$, of $2\pi \in [0, 2\pi]$.}
{In contrast it seems that a small interval--like--neighborhood $U'$ of 
  $(1,0)^\tT = f(0) = f(2\pi) $ is obtained from gluing a small
  neighborhood, say $[0, \varepsilon_1)$, of $0 \in [0, 2\pi]$,  with a small
neighborhood, say $(2\pi-\varepsilon_2, 2\pi]$, of $2\pi \in [0, 2\pi]$.}
  So whatever point $x' \in \mathbb{S}$ we consider: It always
  seems that a neighborhood $U'$ of $x'$ is build by taking a suitable
  $U_x \in \mathcal{U}(x)$, for every $x$ with $f(x) = x'$, and then
  getting $U'$ as union of the images of the $U_x$, i.e. via
  \begin{gather} \label{eq:example_neighborhoods_glued_together}
    U'
    =
    \bigcup_{x\in [0, 2\pi]: f(x)=x'} f[U_x],
  \end{gather}  
  or, to put it more vividly, by \emph {glueing} neighborhoods 
  $U_x$, $x \in f^{-}[U']$.
\end{example}

The next remark serves as a bridge between the previous example
and the subsequent definition of an identifying mapping.
It picks up \eqref{eq:example_neighborhoods_glued_together}
and shows how this naturally lead to the definition of 
identification topology and identifying mapping.
This way of motivating the identification topology seems to be new.

\begin{remark}[Motivation for the definition of the identification topology]
  Consider a surjective mapping $f: (X, \mathcal{O}) \rarr X'$ between a 
  topological space $(X, \mathcal{O})$ and some set $X'$.
  Assume that there is a topology $\mathcal{O}'$ on $X'$
  such that every neighborhood $U'$ of an arbitrarily chosen point 
  $x'$ results from gluing neighborhoods of all preimage points 
  $x\in f^-[\{x'\}]$; i.e. assume that there is a topology $\mathcal{O}'$
  on $X'$ whose neighborhood systems fulfill
      \begin{gather} \label{eq:generall_case_neighborhoods_glued_together}
	\mathcal{U}'(x') 
	=
	\bigg \{
	  U' \subseteq X' \mathrel{\bigg|} 
	  \forall x \in f^-[\{x'\}] \;\; \exists U_x \in \mathcal{U}(x):
	  U' = \bigcup_{x\in f^-[\{x'\}]} f[U_x]
	\bigg \}
      \end{gather}
  for every $x' \in X'$.
  Is it then possible to describe $\mathcal{O}'$ in a more direct manner?
  Due to the following equivalences for a subset $O' \subseteq X'$
  we can give a positive answer to this question:
  \begin{align*}
        & O' \in \mathcal{O}'
	\\
    \iseq {}&
	\forall x' \in O' : O' \in \mathcal{U'}(x')
	\\
    \overset{\eqref{eq:generall_case_neighborhoods_glued_together}}{\iseq} {}&
	\forall x' \in O' \;\; 
	\forall x \in f^-[\{x'\}] \;\; \exists U_x \in \mathcal{U}(x):
	O' = \bigcup_{x \in f^-[\{x'\}]} f[U_x]
	\\
    \overset{(\ast)}{\iseq} {}&
	\forall x' \in O' \;\; 
	\forall x \in f^-[\{x'\}] \;\; \exists \widetilde {U_x} \in \mathcal{U}(x):
	O' \supseteq \bigcup_{x \in f^-[\{x'\}]} f[\widetilde{U_x}]
	\\
    \iseq {}&
	\forall x' \in O' \;\; 
	\forall x \in f^-[\{x'\}] \;\; 
	\exists \widetilde {O_x} \in \mathcal{U}(x) \cap \mathcal{O}:
	O' \supseteq f[\bigcup_{x \in f^-[\{x'\}]} \widetilde{O_x}]
	\\
    \iseq {}&
	\forall x' \in O' \;\; \exists O \in \mathcal{O} :
	f^-[\{x'\}] \subseteq O  \wedge  O' \supseteq  f[O] 
	\\
    \iseq {}&
	\exists \widehat O \in \mathcal{O} \;\; \forall x' \in O':
	f^-[\{x'\}] \subseteq \widehat O  \wedge O'
	\supseteq f[\widehat O]
	\\
    \iseq {}&
	\exists \widehat O \in \mathcal{O}:
	f^-[O'] \subseteq \widehat O  \wedge f^-[O'] 
	\supseteq \widehat O 
	\\
    \iseq {}&
	\exists \widehat O \in \mathcal{O}:
	f^-[O'] = \widehat O
	\\
    \iseq {}&
	f^-[O'] \in \mathcal{O}.
  \end{align*}
  Note that the harder implication ''$\Leftarrow$'' in $(\ast)$ holds true, since
  $U_x \defeq f^-[O'] \supseteq \widetilde {U_x}$ is a neighborhood for
  each $x \in f^-[\{x'\}]$ and fulfills $f[U_x] = O'$, in virtue of
  $f$'s surjectivity.

  Summarizing we can say that necessarily 
  \begin{gather*}
    \mathcal{O}'
    =
    \big \{O' \subseteq X' : f^-[O'] \in \mathcal{O} \big \}.
  \end{gather*}
  This motivates the following definition.
  Take note, though, that we did \emph{not} prove that the 
  topology $\big \{O' \subseteq X' : f^-[O'] \in \mathcal{O} \big \}$ 
  actually induces neighborhood systems 
  which fulfill \eqref{eq:generall_case_neighborhoods_glued_together}.
\end{remark}

\begin{definition}
  We say that a mapping 
  $f: (X, \mathcal{O}) \rarr (X', \mathcal{O}')$
  between two topological spaces $(X, \mathcal{O})$ and $(X', \mathcal{O}')$
  is 
  {\bf identifying}
  or that it
  {\bf glues} 
  $\bm{(X, \mathcal{O})} $
  {\bf to}
  $\bm{(X', \mathcal{O}')}$, iff it is \emph{surjective} and 
\begin{gather*}
  \mathcal{O}' = \{O' \subseteq X' : f^-[O'] \in \mathcal{O}\}.
\end{gather*}
%
  The topology $\mathcal{O}'$ is called 
  {\bf quotient topology} or {\bf identification topology}
  {\bf induced by \boldmath$f$ and $\mathcal{O}$} and
  $(X', \mathcal{O}')$ is called the
  {\bf quotient space} or {\bf identification space} 
  {\bf induced by \boldmath$f$ and $\mathcal{O}$}.
\end{definition}
\iftime{Staturierte Topologie rein: SONST: -- }

The identification topology is uniquely determined by
the surjective mapping $f$, cf. Remark 
\ref{rem:identifying_implies_continuous_AND_identification_topology_unique}.

\begin{remark} \label{rem:identifying_implies_continuous_AND_identification_topology_unique}
    If a topological space $(X, \mathcal{O})$ is glued to 
    a topological space $(X', \mathcal{O}')$ by a mapping $g$
    then the, by definition surjective, mapping $g$ is in particular
    continuous; to see this just compare
    \begin{align*}
      g \text{ is continuous } 
      &\iseq 
      \left(
      \forall S' \subseteq X':
      S' \in \mathcal{O}_{X'}
      \implies
      g^-[S'] \in \mathcal{O}_{X}
      \right)      
    \shortintertext{with}
      g \text{ is identifying } 
      &\iseq 
      \left(
      \forall S' \subseteq X':
      S' \in \mathcal{O}_{X'}
      \iseq
      g^-[S'] \in \mathcal{O}_{X}
      \right)
      \\
      &\color{white}\iseq \color{black} \wedge \
      g \text{ is surjective}.
    \end{align*}
    More precisely one can read from the above lines, that a surjective mapping 
    $g$ glues a topological space $(X, \mathcal{O}) $
    to a topological space $(X', \mathcal{O}')$, iff
    $\mathcal{O}'$ is the finest topology on $X'$ for which 
    $g: (X, \mathcal{O})  \rarr  (X', \mathcal{O}')$ is still continuous.
\end{remark}

The relation between ''homeomorphic``, ''identifying`` and ''continuous``
is shown in the following diagram.
\[ \label{diag:relation_homeo_identifying_continuous}
  \xymatrix{
    f: \topospace{X} \rarr \topospace{X'} \text{ is a homeomorphism} 
	\ar@<-1ex>@{=>}[d]
	\\
    f: \topospace{X} \rarr \topospace{X'} 
	    \text{ glues } \topospace{X} 
	    \text{ to }    \topospace{X'}
	\ar@<-1ex>@{=>}[u]_{f \text{ bij.} }   
	\ar@<-1ex>@{=>}[d]  
	\\
    f: \topospace{X} \rarr \topospace{X'} \text{ is continuous }
	\ar@<-1ex>@{=>}[u]_{f \text { surj. and (open or closed)} }
  }
\]
%
The relations between the first and second row are easy to see, by 
\prettyref{rem:identifying_implies_continuous_AND_identification_topology_unique}. 
The implication from the second to the third row is also clear by this 
Remark. 
It remains to deal with the implication from the third to the second row.
Before illustrating this condition and then moving towards its 
justification in \prettyref{thm:sufficient_conditions_for_indentifying}
we would like to warn the reader that 
restricting identifying mappings is more problematic than 
restricting continuous mappings or homeomorphisms:
The restriction of a continuous mappings resp. homeomorphism are 
again continuous mappings resp. homeomorphisms.
In contrast 
the restriction of an identifying mapping is not
necessarily again identifying, cf. Example 
\ref{exa:restricted_function_no_longer_identifying}.
Now we return to our discussion of the implication from the 
third row to the second row.
As stated in 
\prettyref{rem:identifying_implies_continuous_AND_identification_topology_unique},
every identifying mapping $(X, \mathcal{O}) \rarr (X', \mathcal{O}')$
is continuous. However, the opposite is not true. The identity mapping
$\id_{\{0,1\}}: \{0,1\} \rarr \{0,	1\}$
between 
$(X, \mathcal{O})   = (\{0,1\}, \{X, \emptyset, \{0\} \} )$ and 
$(X', \mathcal{O}') = 
 (X, \mathcal{O}')  =(\{0,1 \}, \{X, \emptyset        \} )$ is
a simple, but maybe not very natural, example.
A more natural example for a surjective continuous mapping, 
which is not identifying is given in 
\prettyref{exa:continuous_but_not_identifying_projection}.

The proof of the following lemma, can also be found 
in \cite[p. 109]{Qrnb2001}. 

\begin{lemma} 
\label{lem:continuous_mapping_from_compact_space_to_hausdorff_space_is_closed}
  A continuous mapping $g:\topospace{Y} \rarr \topospace{Z}$ 
  from a compact space $\topospace{Y}$ into a Hausdorff
  space $\topospace{Z}$ is always a closed mapping.
  In particular $g$ is a homeomorphism if $g$ is additionally 
  bijective.
\end{lemma}

\begin{proof}
  A closed subset of the compact space $\topospace{Y}$ is again 
  compact by part \upref{enu:closeness_implies_compactness}
  of \prettyref{thm:relation_closed_compact}; 
  therefore it is mapped by the continuous mapping $g$
  to a compact subset of $\topospace{Z}$,
  which is a closed subset of this Hausdorff space, by
  part \upref{enu:compactness_implies_closeness} of
  \prettyref{thm:relation_closed_compact}.
  Hence $g$ is a closed mapping.
  If $g$ is in addition bijective then the mapping $g$
  is also an open mapping, since the image $g[O]$ of 
  every open subset  $O \in \topology{Y}$ can then be written 
  in the form 
  $
    g[O] 
    = 
    g[Y \setminus (Y \setminus O)]
    =
    g[Y] \setminus g[Y \setminus O]
    =
    Z \setminus g[Y \setminus O]
  $,
  showing that $g[O]$ is the complement of the closed set 
  $g[Y \setminus O]$ and hence an open subset of $\topospace{Z}$.
  Therefore the mapping $g$ is open and continuous and hence 
  a homeomorphism. 
\end{proof}

Each identifying mapping is also continuous. 
The converse is not true in general.
Yet the next theorem gives some  sufficient criteria for ensuring 
that a continuous function is even identifying.

\begin{theorem} \label{thm:sufficient_conditions_for_indentifying}
  A surjective continuous mapping 
  $g: \topospace{Y} \rarr \topospace{Z}$ is identifying,
  if at least one of the following additional properties is fulfilled:
  \begin{enumerate}
    \item \label{enu:open_or_closed_implies_indentifying}
      $g$ is a closed or open mapping.
    \item \label{enu:mapping_from_compact_to_hausdorff_implies_indentifying}
      $\topospace{Y}$ is a compact space and $\topospace{Z}$ is a Hausdorff space.
  \end{enumerate}
  
\end{theorem}

Before proving this theorem, we give an example for a continuous, 
but not identifying mapping $g$, which is defined on a simple subset 
$Y$ of $\R^2$ and maps onto a compact interval $Z$.
By \prettyref{thm:sufficient_conditions_for_indentifying}
it is clear that $Y$ must not be a compact subset of 
$(\R^2, \mathcal{O}^{\otimes 2})$ and that $g$ must not be open and closed.
We note that our example was inspired by an example, given by Kelly
in \cite[ch. Quotient spaces, p. 95]{Kelley1955}, illustrating  that there
are continuous mappings which are neither open nor closed. 
The natural topology of $\R$ is denoted by $\mathcal{O}$.

\begin{example} \label{exa:continuous_but_not_identifying_projection}
  The interval $[-1,1]$ can be generated by putting a single point, say
  $(0,1) \in \R^2$, into the 
  gap of $[-1,1]\setminus \{0\} =[-1,0) \cup (0,1]$. This operation
  is modeled by the mapping
  $g: \topospace{Y} \rarr \topospace{Z}$, 
  $g(y) \defeq y_1$, 
  where
  \begin{align*}
    Y
    &\defeq
    \left(
    [-1,1]\setminus \{0\} \times \{0\}
    \right)
    \cup
    \{(0,1)\}\\
  \shortintertext{and}
    Z
    &\defeq
    [-1,1]
  \end{align*}
  are endowed with the subspace topologies 
  $\mathcal{O}_Y = Y \Cap \mathcal{O}^{\otimes 2}$ and
  $\mathcal{O}_Z = Z \Cap \mathcal{O}$, respectively. 
  The projection $g$ is continuous but, however, not
  identifying: Consider the point $ 0 \in [-1,1]$ and its only preimage
  point $(0,1)\in Y$. The isolated point $(0,1) \in Y$ has 
  $\{(0,1)\}\eqdef U$ as smallest open neighborhood. Yet $g[U]=\{0\}$ is no
  neighborhood of $0$. 
  We remark that the same reasoning shows that $g$ is not
  an open mapping; moreover $g$ is neither a closed mapping
  since it maps the closed subset 
  $[-1,0) \times \{0\}$ of $\topospace{Y}$ to $[-1,0)$ which is not
  a closed subset of $\topospace{Z}$.

\end{example}

\begin{proof}[Proof of Theorem \ref{thm:sufficient_conditions_for_indentifying}]

  \upref{enu:open_or_closed_implies_indentifying}
    Since $g: \topospace{Y} \rarr \topospace{Z}$ is surjective we have
      \begin{align}
	  g \text{ is continuous } 
	  &\iseq 
	  \left(
	  \forall \widecheck Z \subseteq Z:
	  \widecheck Z \in \mathcal{O}_{Z}
	  \implies
	  g^-[\widecheck Z] \in \mathcal{O}_{Y}
	  \right)
      \label{eq:g_is_contiuous_iff}
      \shortintertext{and}
	  g \text{ is identifying } 
	  &\iseq 
	  \left(
	  \forall \widecheck Z \subseteq Z:
	  \widecheck Z \in \mathcal{O}_{Z}
	  \iseq
	  g^-[\widecheck Z] \in \mathcal{O}_{Y}
	  \right).
      \label{eq:surjective_g_is_identifying_iff}
      \end{align}
      So our task of proving 
      ''$g$ is continuous $\implies$ $g$ is identifying``
      reduces to verify the statement
      \begin{equation}
        \forall \widecheck Z \subseteq Z:
	g^-[\widecheck Z] \in \mathcal{O}_Y
	\implies
	\widecheck Z \in \mathcal{O}_Z.
      \label{eq:reduced_task__g_continuous_implies_g_identifying}
      \end{equation}
      
      In the first case that $g$ is open, i.e. fulfills
      $g[O_Y] \in \mathcal{O}_Z$ for all $O_Y \in \mathcal{O}_Y$
      we are done by writing $\widecheck Z = g[g^-[\widecheck Z]]$
      and setting $O_Y \defeq g^-[\widecheck Z]$.
      In the second case that $g$ is closed, i.e. fulfills
      $g[A_Y] \in \mathcal{A}_Z$ for all $A_Y \in \mathcal{A}_Y$ 
      -- where $\mathcal{A}_Y$ and $\mathcal{A}_Z$ are the systems 
      of the closed subsets of $\topospace{Y}$ and $\topospace{Z}$, 
      respectively --
      we translate all involved statements of the previous reasoning
      from their ''open set viewpoint`` formulation
      \prettyref{eq:g_is_contiuous_iff},
      \prettyref{eq:surjective_g_is_identifying_iff} and
      \prettyref{eq:reduced_task__g_continuous_implies_g_identifying}
      to the corresponding ''closed set viewpoint`` formulation,
      by means of building complements.
      Then the reasoning goes the same way as before.

  \upref{enu:mapping_from_compact_to_hausdorff_implies_indentifying}
%
%
  By Lemma 
  \ref{lem:continuous_mapping_from_compact_space_to_hausdorff_space_is_closed}
  the function $g$ maps every closed subset of $\topospace{Y}$ to a 
  closed subset of $\topospace{Z}$ and therefore fulfills
  \upref{enu:open_or_closed_implies_indentifying}, which implies that
  $g$ is identifying.  
\end{proof}

In the next theorem we consider two functions 
$g : \topospace{Y} \rarr \topospace{Z}$ and 
$g': \topospace{Y'} \rarr \topospace{Z}$ which are 
identical except that their domains of definition
do not need to be totally identical; rather $\topospace{Y}$ 
shall only to be glued to $\topospace{Y'}$ by an identifying mapping 
$I: \topospace{Y} \rightarrow \topospace{Y'}$.
The theorem states that $g$ is continuous respectively identifying, 
iff so is $g'$.
\begin{gather*}
  \xymatrix{
    &   \topospace{Y}  \ar[dd]_{I} \ar[dr]|-{g}   \\
    & & \topospace{Z}   \\
    &   \topospace{Y'} \ar[ur]|-{g'}      
  }
\end{gather*}

\begin{theorem} \label{thm:mappings_from_ori_space_vs_mappings_from_glued_space}
  Let 
  $g : \topospace{Y} \rarr \topospace{Z}$ and
  $g': \topospace{Y'} \rarr \topospace{Z}$
  be mappings between topological spaces, which are 
  related via $g = g' \circ I$,
  with a mapping $I$ that glues $\topospace{Y}$ to $\topospace{Y'}$. 
  Then the following statements hold true:
  \begin{enumerate}
    \item \label{enu:continuous_mappings__from_ori_space_vs_from_glued_space}
      $g$ is continuous $\iseq g'$ is continuous. 
    \item \label{enu:identifying_mappings__from_ori_space_vs_from_glued_space}
      $g$ glues $\topospace{Y}$ to $\topospace{Z}$
      $\iseq$
      $g'$ glues $\topospace{Y'}$ to $\topospace{Z}$.
  \end{enumerate}
\end{theorem}

See also \cite[p. 95 -- 96]{Kelley1955}  
for the first part of the subsequent proof.
\begin{proof}
  Since $I$ is identifying we have, for every subset 
  $\widecheck Z $ of $Z$, the equivalences 
  \[
    g'^-[\widecheck Z] \in \mathcal{O}_{Y'}
    \iseq
    I^-[g'^-[\widecheck Z]] \in \mathcal{O}_{Y}
    \iseq
    g^-[\widecheck Z] \in \mathcal{O}_{Y}.
  \]
  Having this in mind we get
  \begin{align*}
    g \text{ is continuous } 
      &\iseq 
      \left(
      \forall \widecheck Z \subseteq Z:
      \widecheck Z \in \mathcal{O}_{Z}
      \implies
      g^-[\widecheck Z] \in \mathcal{O}_{Y}
      \right)
      \\
      &\iseq
      \left(
      \forall \widecheck Z \subseteq Z:
      \widecheck Z \in \mathcal{O}_{Z}
      \implies
      g'^-[\widecheck Z] \in \mathcal{O}_{Y'}
      \right)
      \\
      &\iseq
      g' \text{ is continuous}
  \shortintertext{and}
    g \text{ is identifying } 
      &\iseq 
      \left(
      \forall \widecheck Z \subseteq Z:
      \widecheck Z \in \mathcal{O}_{Z}
      \iseq
      g^-[\widecheck Z] \in \mathcal{O}_{Y}
      \right)
      \\
      &\iseq
      \left(
      \forall \widecheck Z \subseteq Z:
      \widecheck Z \in \mathcal{O}_{Z}
      \iseq
      g'^-[\widecheck Z] \in \mathcal{O}_{Y'}
      \right)
      \\
      &\iseq
      g' \text{ is identifying}.
  \end{align*}
\end{proof}

We end this subsection with a warning: in general a restriction 
of an identifying mapping is no longer identifying as the following 
example shows.
Again $\mathcal{O}$ is the natural topology of $\R$ and 
$\mathbb{S} \defeq \{ x\in \R^2 : \|x\|_2 = 1 \}$.

\begin{example} \label{exa:restricted_function_no_longer_identifying}
  We consider, once more, the both surjective and continuous mapping \linebreak
  $
    f: 
    ([0,2\pi], [0, 2\pi] \Cap \mathcal{O} )
    \rarr
    (\mathbb{S}, \mathbb{S} \cap \mathcal{O}^{\otimes 2})
  $,
  given by
  \begin{gather*}
    x
    \mapsto
    e^{ix} 
    = 
    \begin{pmatrix}
      \cos x
      \\
      \sin x
    \end{pmatrix}.    
  \end{gather*}
  This mapping is identifying by part
  \upref{enu:mapping_from_compact_to_hausdorff_implies_indentifying}
  of Theorem 
  \ref{thm:sufficient_conditions_for_indentifying}.
  Restricting this mapping to the subset 
  $\check X \defeq [0,2\pi)$ we get the continuous bijection 
  $
    f|_{\check X}: 
    ([0,2\pi), [0, 2\pi) \Cap \mathcal{O} ) 
    \rarr 
    (\mathbb{S}, \mathbb{S} \cap \mathcal{O}^{\otimes 2})
  $
  which is no longer identifying,
  since an identifying bijection would necessarily be an homeomorphism,
  cf. the Diagram on page 
  \pageref{diag:relation_homeo_identifying_continuous}.
  However the spaces $([0,2\pi), [0, 2\pi) \Cap \mathcal{O} )$ and 
  $(\mathbb{S}, \mathbb{S} \cap \mathcal{O}^{\otimes 2})$ are clearly
  not homeomorphic, since only the latter one is compact.
  
\end{example}

\subsection{One-point compactification of a topological space}
\label{subsec:one_point_compactification}

We start with a well known special case before we give the general definition.
\begin{example}[{{\bf and Definition}}] \label{exa:one_point_compactification_of_euclidean_space}
  It is often convenient to regard $\R^n$ as the subset
  $\sphere[][][n] \setminus \{(0,0,\dots, 0, 1)\} \eqdef \dot S$ of the 
  sphere 
  $
    \topospace{\sphere[][][n]}
    \defeq
    (\sphere[][][n], \sphere[][][n] \Cap \mathcal{O}^{\otimes (n+1)})
  $ by means of the homeomorphism
  \begin{align*}
    &\pi: (\dot S, \dot S \Cap  \mathcal{O}_{\sphere[][][n]})
      \rarr (\R^n, \mathcal{O}^{\otimes n}),\\
    &\pi: (s_1,s_2, \dots, s_n; s_{n+1})^\tT
      \mapsto \tfrac{1}{1-s_{n+1}}(s_1, s_2, \dots, s_n)^\tT,
  \end{align*}
  known as stereographic 
  projection\DissVersionForMeOrHareBrainedOfficialVersion{\footnote{
      An image illustrating the stereographic projection can be found in
      \cite{wiki:stereographical_projection}, 
      cf. also \cite{wiki:image:projection_stereographique}
      and \cite[p. 350]{Munkres1975}.
      }.
    }
    {, cf. \cite[p. 350]{Munkres1975}.}
  %
  The topological superspace 
  $(\sphere[][][n], \mathcal{O}_{\sphere[][][n]})$
  of $(\dot S, \dot S \Cap \mathcal{O}_{\sphere[][][n]})$
  differs not much from the latter:
  The set 
  \begin{gather*}
    \sphere[][][n] 
    = 
    \dot S  \cup  \{(0,0,\dots,0,1)^{\tT}\}
  \end{gather*}

  contains just one point more than $\dot S$ and the
  topology 
  $
    \mathcal{O}_{\sphere[][][n]} 
      \supsetneq
      \dot S \Cap \mathcal{O}_{\sphere[][][n]}
  $
  differs from $\dot S \Cap \mathcal{O}_{\sphere[][][n]}$
  only by additionally containing the open neighborhoods of 
  the ''north pole`` $(0,0,\dots,0,1) \eqdef N$ as expressed by 
  \begin{align*}
    \mathcal{O}_{\sphere[][][n]}  
    ={}&  
    (\dot S \Cap \mathcal{O}_{\sphere[][][n]})
      \mathrel{\dot \cup}
      \{O \in  \mathcal{O}_{\sphere[][][n]}: N \in O\}
  \\
    ={}&
    (\dot S \Cap \mathcal{O}_{\sphere[][][n]})
      \mathrel{\dot \cup}
      \{
	\sphere[][][n] \setminus A :
	  A \in \mathcal{A}(\sphere[][][n]),
	  A \subseteq \dot S
      \}
  \\
    ={}&
    (\dot S \Cap \mathcal{O}_{\sphere[][][n]})
      \mathrel{\dot \cup}
      \{
	\sphere[][][n] \setminus K : 
	  K \in \mathcal{K}(\sphere[][][n]),
	  K \subseteq \dot S
      \}
  \\
    ={}&
    (\dot S \Cap \mathcal{O}_{\sphere[][][n]})
      \mathrel{\dot \cup}
      \{
	\sphere[][][n] \setminus K : 
	  K \in \mathcal{K}(\dot S)
      \}.
  \end{align*}
  
  Likewise we set 
  $\R^n_\infty \defeq \R^n \cup \{\infty\}$ with an additional point 
  $\infty \not \in \R^n$ and define
  \begin{gather*}
    \mathcal{O}^{\otimes n}_\infty 
    \defeq 
    (\mathcal{O}^{\otimes n})_\infty
    \defeq 
    \mathcal{O}^{\otimes n}
    \mathrel{\dot \cup }
    \{ \R^n_\infty \setminus K : K \in \mathcal{K}(\R^n)\}.
  \end{gather*}
  Then $(\R^n_\infty, \mathcal{O}^{\otimes n}_\infty)$ is a compact 
  topological space, called the 
  {\bf one-point compactification} of $(\R^n, \mathcal{O}^{\otimes n})$;
  it contains $(\R^n, \mathcal{O}^{\otimes n})$
  as dense subspace. Moreover the homeomorphism 
  $
    \pi: (\dot S, \dot S \Cap  \mathcal{O}_{\sphere[][][n]})
    \rarr
      (\R^n, \mathcal{O}^{\otimes n})
  $
  can be extended to a homeomorphism 
  $
    \topospace{\sphere[][][n]}
    \rarr
      (\R^n_\infty, \mathcal{O}^{\otimes n}_\infty)
  $
  by setting $\pi(N) \defeq \infty$.
  Setting $\|\infty\| \defeq \pinfty$ we then have for any sequence 
  of points $x_k$  from $(\R^n_\infty, \mathcal{O}^{\otimes n}_\infty)$ 
  the relation
  \begin{align*}
    x_k \rarr \infty 
    &\iseq 
    \pi^{-1}(x_k) \rarr \pi^{-1}(\infty)
  \\&\iseq 
    \pi^{-1}(x_k) \rarr N
  \\&\iseq 
    \|x_k\| \rarr \pinfty.
  \end{align*}
\end{example}

For general topological spaces $(X, \mathcal{O})$ the procedure is done
similarly by adding a new point $\infty$, resulting in the set
$X_\infty \defeq X \cup \{\infty\}$, and by equipping $\infty$ with 
an appropriate system of neighborhoods.
In the latter we have to be careful if $(X, \mathcal{O})$ is
not a Hausdorff space.
Namely, in this case it may happen that there
are compact subsets $K_1, K_2 \in \mathcal{K}(X, \mathcal{O})$
whose intersection $K_1 \cap K_2$ is no longer compact,
see Detail \ref{det:intersection_of_compact_sets_not_compact} in the 
Appendix;
we would therefore fail here, when we were trying to define the open neighborhoods of the 
new point $\infty$ as the sets
\begin{gather} \label{eq:failing_try_to_define_neighborhoods_of_infty}
  X_\infty \setminus K, \text{ with } K \in \mathcal{K}(X,\mathcal{O}),
\end{gather}
since the union of the ''open neighborhoods`` $X_\infty \setminus K_1$
and $X_\infty \setminus K_2$ is the set
$
  (X_\infty \setminus K_1) \cup (X_\infty \setminus K_2) 
  = 
  X_\infty \setminus (K_1 \cap K_2)
$
which is no longer a ''neighborhood`` of $\infty$. This problem 
is solved if we restrict us in 
\eqref{eq:failing_try_to_define_neighborhoods_of_infty} 
to those compact subsets $K$ of $(K, \mathcal{O})$ which are additionally
closed, see Detail
\ref{det:intersection_of_closed_and_compact_sets_is_again_closed_and_compact}
in the Appendix.
Choosing
\begin{gather} \label{eq:succeding_try_to_define_neighborhoods_of_infty}
  X_\infty \setminus K 
    \text{ with } 
    K \in \mathcal{KA}(X, \mathcal{O}) 
\end{gather}
as the open neighborhoods of $\infty$ indeed is the right idea.
Before we give the definition of the general one-point compactification 
in accordance to \eqref{eq:succeding_try_to_define_neighborhoods_of_infty}
we note that the sets $X_\infty \setminus K$ in 
\eqref{eq:failing_try_to_define_neighborhoods_of_infty} and
\eqref{eq:succeding_try_to_define_neighborhoods_of_infty}
coincide if $(X, \mathcal{O})$ is a Hausdorff space since in this case
we have 
$\mathcal{K}(X, \mathcal{O}) \subseteq \mathcal{A}(X, \mathcal{O})$
by part \upref{enu:compactness_implies_closeness} of
Theorem \ref{thm:relation_closed_compact}.
The following general definition as well as the subsequent 
Theorem \ref{thm:one_point_compactification} are, in essence,
taken from \cite[p. 150]{Kelley1955}.

\begin{definition}
  Let $(X, \mathcal{O})$ be a topological space and $\infty \not \in X$ an 
  additional point. The {\bf one-point compactification} of
  $(X, \mathcal{O})$ is the space 
  $(X, \mathcal{O})_\infty \defeq (X_\infty, \mathcal{O}_\infty)$,
  where 
  $
    X_\infty 
    \defeq 
    X \cup \{\infty\}
  $
  and 
  $
    \mathcal{O}_\infty 
    \defeq 
    \mathcal{O}
    \cup 
    \{ 
      X_\infty \setminus K : 
      K \in \mathcal{KA}(X, \mathcal{O}) 
    \}
  $.
\end{definition}

\begin{theorem} \label{thm:one_point_compactification}
  The one-point compactification $(X_\infty, \mathcal{O}_\infty)$ of a
  topological space $(X, \mathcal{O})$ is a compact topological space,
  which contains $(X, \mathcal{O})$ as subspace. 
  $(X_\infty, \mathcal{O}_\infty)$ is a Hausdorff space if and only if $X$ 
  is a locally compact Hausdorff space.
\end{theorem}

\section{Topologization of totally ordered sets and topological coercivity notions}
\addtointroscoll{

In this section's subsections 
\begin{itemize}
  \item 
    \ref{subsec:topology_for_linear_odered_sets}
    Three topologies for totally ordered sets    
  \item 
    \ref{subsec:topology_for_inf_complete_linear_odered_sets}
     The right order topology on an inf-complete totally ordered set
  \item 
    \ref{subsec:topological_coercivity_and_its_interpretation_as_continuity}
    Topological coercivity notions and continuity interpretations
  \item 
    \ref{subsec:topological_coercivity_and_boundedness_below}
    Topological coercivity and boundedness below
\end{itemize}
we introduce for a given totally ordered set $(Z, \leq)$ 
the right order topology (along with two other topologies),
give its very simple form in case of totally ordered sets,
use it to define topological coercivity notions
and show its good influence when investigating boundedness from 
below.

More precisely we introduce in the first subsection
three different topolgies for a given totally ordered set $(Z, \leq)$.
For us the most important of them is the right order topology 
$\mathcal{T}_\leq$, beeing the suited topology to investigate 
lower semicontinuity.
Also with regard to coercivity questions this topology is useful.

In the second subsection we will see that $(Z, \mathcal{T}_\leq)$
becomes very simple if the underlying totally ordered set 
is inf-complete.
The topology $\mathcal{T} = \mathcal{T}_\leq$
of the topological space 
$([\minfty, \pinfty], \mathcal{T})$
is an important example and will be studied in more detail in
\prettyref{sec:the_topological_space_intervalus_maximus_with_the_right_order_topology}.

In the third 
subsection 
the notions of  
topological (strong) coercivity towards a set and some boundedness notions 
are introduced. In 
\prettyref{thm:topological_coercivivity_is_eq_to_continuity_of_extension}
we will see that a mapping 
$f:(X, \mathcal{O}) \rarr (X', \mathcal{O}')$ is topological 
coercive (towards $\emptyset$) iff a certain extension 
$\widehat f : (X, \mathcal{O})_\infty \rarr (X', \mathcal{O}')_{\infty'}$
is continuous in the newly added point $\infty$. In case of a mapping 
$f: \R^n \rarr \R^m$ this later turns out to be equivalent
to the normcoercivity of $f$, see 
\prettyref{thm:characterize_normcoercivity}.
For a mapping $f:(X, \mathcal{O}) \rarr (Z, \mathcal{T}_\leq)$ another 
similar relation can be described if the totally ordered set 
$(Z, \leq)$ has a maximum $\widehat z$ and a minimum.
In this case $f: (X, \mathcal{O}) \rarr (Z, \mathcal{T}_\leq)$ is 
topological coercive towards $\{\widehat z\}$ iff another certain 
extension 
$\widehat f: (X, \mathcal{O})_\infty \rarr (Z, \mathcal{T}_\leq)$
is continuous in the newly added point $\infty$,  see
\prettyref{thm:topological_coercivity_towards_maximal_element_equivalent_to_continuous_property}.
In case of a mapping $f: \R^n \rarr [\minfty, \pinfty]$ this 
will turn out to be equivalent to the coercivity of $f$, see
\prettyref{thm:characterize_lsc_and_coercive}.

In the fourth and last 
subsection
we recall the usual global boundedness 
definition for functions $f: (X, \mathcal{O}) \rarr (Z, \leq)$ 
and add two less common, more easier to check, local
boundedness notions
and show  that the local ones imply the global one if 
$f: (X, \mathcal{O}) \rarr (Z, \mathcal{T}_\geq)$ is topological 
strongly coercive towards $\MAX_\leq (Z)$.
Note that here $Z$ is not equipped with the right order topology 
but really with the left order topology!

Finally we mention that 
the right order topology is a special case of the 
Scott topology for a partially ordered set $(Z, \sqsubseteq)$.
The latter topology is defined as the collection of all 
subsets $O$ of $Z$ which fulfill the following conditions:
\begin{enumerate}
  \item 
    Along with any $z \in O$ also the ``upper set'' 
    $\{\widetilde z \in Z: \widetilde z \sqsupseteq z\}$ belongs to 
    $O$;\\ i.e. -- more formally expressed -- the condition
    $
      \forall z \in O \; \forall \widetilde z \in Z : 
      \widetilde z \sqsupseteq z \implies \widetilde z \in O
    $
    holds true,
  \item
    Every directed subset $S$ of $(Z, \sqsubseteq)$ whose 
    supremum exists and belongs to $O$ has nonempty 
    intersection with $O$, i.e. fulfills 
    $S \cap O \neq \emptyset$,
\end{enumerate}
cf. \cite{Scott72} where Scott defined this topology using 
the name ``induced topology''.
\rem{Beweis, das Scott topologie auf total geordneter Menge 
gerade die rechte Ordnungstopologie ist, liegt als Scan bei und als
papierne Abheftung vor.}

}

\subsection{Three topologies for totally ordered sets}
\label{subsec:topology_for_linear_odered_sets}


Before defining topologies out of $\leq$ we remark that we use interval
notation just as for $\R$ endowed with the natural order. In addition
we introduce analogues for the unbounded real intervals like
$(\minfty, b]$.

\begin{definition}
  Let $(X, \leq)$ be a totally ordered set. We use the
  shortcuts
  \begin{align*}
    b) &\defeq \{x\in X: x   <  b \},  \\
    b] &\defeq \{x\in X: x \leq b \},  \\
    (a &\defeq \{x\in X: a   <  x \},  \\
    [a &\defeq \{x\in X: a \leq x \}. 
  \end{align*}
  If the totally ordered set is denoted with a decoration like 
  in $\leq'$ we feel free to adopt the notation accordingly and 
  write e.g. $('a$ instead of $(a$.
\end{definition}

Given a totally ordered set $(X, \leq)$ we consider three
different topologies for it,
namely two ``one sided'' topologies and one ``two sided'' topology.
We start with the ``one sided'' topologies, cf.
\cite[p. 74]{Steen1978}. But be aware that the definition there is not 
totally correct, 
see Detail \ref{det:definition_in_steen_not_totally_correct}
in the Appendix. 
A correct definition can be found in \cite{wiki:Ordnungstopologie}.
\rem{Die endlische Wikiseite enthaelt bei der Definition der 
Topologien einen Fehler; gleicher Fehler ($X$ musz auch offene Menge sein!)
auch in Counterexamples in topology..}
\begin{definition}
  Let $(X, \leq)$ be a totally ordered set. 
  The system of sets, which are $\emptyset, X$ or which can be written 
  as unions of sets of the form 
  $(a$, with $a \in X$, forms a topology. 
  It will be called 
  {\bf right order topology} for $(X, \leq)$ and will be
  denoted by $\mathcal{T}_\leq$.
  Analogously the {\bf left order topology} $\mathcal{T}_\geq$
  for $(X, \leq)$ is defined as system of sets which are
  $\emptyset, X$ or which can be written 
  as unions of sets of the form 
  $b)$, with $b \in X$.
\end{definition}

\begin{remark}
  \begin{enumerate}
    \item 
      The notations for the right order topology 
      and the left order topology for a totally ordered set 
      $(X, \leq)$ are consistent:
      Define the inverse order $\preccurlyeq$ on $X$ via 
      $x \preccurlyeq y \defiseq x \geq y$ for all $x, y \in X$.
      Then the left order topology $\mathcal{T}_\geq$ for 
      $(X, \leq)$ is indeed just the right order topology 
      $\mathcal{T}_{\preccurlyeq}$ for $(X, \preccurlyeq)$.
    \item 
      The above systems $\mathcal{T}_\leq$ and 
      $\mathcal{T}_\geq$ are really topologies on $X$: 
      By the first part of this remark it suffices to prove 
      that $\mathcal{T}_\leq$ is a topology.
      $\emptyset$ and $X$ belong to $\mathcal{T}_\leq$ 
      by definition.
      Clearly arbitrary unions of sets from $\mathcal{T}_\leq$ 
      belong again to $\mathcal{T}_\leq$ by definition of 
      this system. Finally also the intersection of two sets 
      $T, S \in \mathcal{T}_\leq$ again belongs to that system:
      If $T$ or $S$ is empty we have 
      $T \cap S = \emptyset \in \mathcal{T}_\leq$. Likewise $T \cap S$ 
      belongs to $\mathcal{T}_\leq$ if $T = X$ or $S = X$.
      In the remaining case $T = \bigcup_{i \in I}(t_i$ and
      $S = \bigcup_{j \in J}(s_j$ with any index sets $I, J$ and elements 
      $t_i, s_j \in X$ we finally have 
      \begin{align*}
	T \cap S 
      & = \left[ \bigcup_{i \in I} \, (t_i \right] 
	    \cap \left[ \bigcup_{j \in J} \, (s_j \right]
	=  \bigcup_{i \in I} \left[\, (t_i  
	    \cap \bigcup_{j \in J} \, (s_j \right]
      \\
      & =  \bigcup_{i \in I} \bigcup_{j \in J}  \bigg[\, (t_i  
	    \cap \, (s_j \bigg]
	=  \bigcup_{i \in I, j \in J} (\max{\{t_i,s_j\}}.
      \end{align*}
    Hence we have shown $T \cap S \in \mathcal{T}_\leq$ also in this case.
  \end{enumerate}
\end{remark}

Now the ``two-sided'' topology is introduced, cf.
\cite{wiki:Ordnungstopologie} and \cite[p. 22]{Qrnb2001}.
\begin{definition} \label{def:order_topology}
  Let $(X, \leq)$ be a totally ordered set. 
  The {\bf order topology} for $(X, \leq)$ is the system 
  $\mathcal{O}_\leq$ consisting of $\emptyset, X$ and the 
  ``open intervals''
  \begin{align*}
    & (a,b) && \text{or}  && (a  && \text{or} && b)
  \end{align*}
  where $a,b \in X$, and all unions of the open intervals.
\end{definition}

\begin{example}
  The order topology for $(\R, \leq)$ is the natural topology 
  of $\R$ which is induced by $|\cdot|$.
\end{example}

\begin{remark}
  The order topology for a totally ordered set $(X, \leq)$ really is
  a topology: In order to avoid dealing with many cases we 
  first represent 
  the sets from the system $\mathcal{O}_\leq$ in a unified way,
  which has been mentioned in \cite{wiki:Ordnungstopologie}.
  To this end let $\larr$ and $\rarr$ be two elements 
  which are not yet contained in $X$.
  Then set
  \begin{gather*}
	  \widehat X 
    \defeq \{\larr\}  \cup X \cup \{\rarr\}    
  \end{gather*}
  and extend the total order $\leq$ on $X$ to a total order 
  on $\widehat X$ (again denoted by $\leq$) by additionally setting 
  $\larr \leq x $ and $x \leq \rarr$ for all 
  $x \in \widehat X$.
  Then 
  \begin{align*}
            X &= (\larr, \rarr)   &  (a &= (a, \rarr)
  \\
    \emptyset &= (\rarr, \larr)   &  b) &= (\larr, b)
  \end{align*}
  for all $a, b \in X$ so that the sets from $\mathcal{O}_\leq$
  appear now simply as the unions of sets of the form 
  $(a,b)$ where $a,b \in \widehat X$.
  %
  This representation makes it clear
  that arbitrary unions of sets from $\mathcal{O}_\leq$ belong again 
  to $\mathcal{O}_\leq$. Moreover the intersection 
  of two arbitrary sets $O = \bigcup_{i \in I}(a_i, b_i)$ and 
  $P = \bigcup_{j \in J}(c_j, d_j)$ -- with 
  $a_i, b_i ,c_j, d_j \in \widehat X$ and any index sets $I,J$ -- 
  can be written in the form 
  \begin{gather*}  
      O \cap P
    = \bigcup_{
	\begin{smallmatrix}
	  i \in I \\
	  j \in J
	\end{smallmatrix}}
	\left[ (a_i, b_i) \cap (c_j, d_j) \right]
    = \bigcup_{
	\begin{smallmatrix}
	  i \in I \\
	  j \in J
	\end{smallmatrix}}
	(\max\{a_i,c_j\}, \min\{b_i, d_j\})
  \end{gather*}
  so that the intersection $O \cap P$ again belongs to 
  $\mathcal{O}_\leq$. Finally clearly 
  $X, \emptyset \in \mathcal{O}_\leq$ so that $\mathcal{O}_\leq$ 
  really is a topology on $X$.
\end{remark}

\begin{proposition} \label{prop:neighborhoods_of_minimal_element}
  Let a totally ordered space $(Z, \leq)$ be equipped with its
  right order topology $\mathcal{T}_\leq$.
  If $(Z, \leq)$ has some minimum $\widecheck z$ then 
  the only $(Z, \mathcal{T}_\leq)$-neighborhood
  of $\widecheck z$ is the whole space $Z$. 
  In particular 
  a mapping $f: (X, \mathcal{O}) \rarr (Z, \mathcal{T}_\leq)$
  is continuous in all points $x$ which are mapped to the minimal element.
  More formally expressed:
  $\mathcal{U}[\mathcal{T}_\leq] (\widecheck z) = \{Z\}$ and 
  $\forall x \in X: \big( f(x) = \widecheck z \implies f$ 
      is continuous in $x \big)$.
\end{proposition}
\begin{proof}
  Clearly the whole space $Z$ is a neighborhood of $\widecheck z$.
  It is also the only neighborhood of $\widecheck z$ since 
  this minimum is never contained in a set $(a$, $a \in Z$,
  and hence also not in unions of such sets.
  Let $x\in X$ be a point with $f(x) = \widecheck z$. For each neighborhood $U$ of $x$
  we trivially have $f[U] \subseteq Z$. Since $Z$
  is the only existing neighborhood of $\widecheck z = f(x)$, 
  this inclusion already shows that $f$ is continuous in $x$.
\end{proof} 

Recall in the next theorem that a mapping 
$f:(X, \leq) \rarr (X', \leq')$ between ordered sets
is called an {\bf order isomorphism} iff $f$ is bijective and 
fulfills $f(x_1) \leq' f(x_2) \iseq x_1 \leq x_2$ for all 
$x_1, x_2 \in X$.

\begin{theorem} \label{thm:monotonicity_vs_continuity}
  Let $\orderspace{X}{}$ and $\orderspace{X}{'}$ be
  totally ordered sets with their corresponding topological spaces
  $(X, \mathcal{T}_\leq)$ and $(X', \mathcal{T}_{\leq'})$, respectively.
  For a mapping 
  $f: X \rarr X'$ the following holds true:
  \begin{enumerate}
    \item \label{enu:continuity_in_point_implies_order_property}
      If $f: \ordertotopspace{X}{} \rarr \ordertotopspace{X}{'}$ 
      is continuous in $x_\ast$ then 
      $f(x) \geq' f(x_\ast)$ for all $x \geq x_\ast$ 
    \item \label{enu:continuity_implies_monotonicity}
      If $f: \ordertotopspace{X}{} \rarr \ordertotopspace{X}{'}$
      is continuous then 
      $f: \orderspace{X}{} \rarr \orderspace{X}{'}$ is monotonically
      increasing.
    \item \label{enu:homeomorphism_vs_order_isomorphism}
      $f: \ordertotopspace{X}{} \rarr \ordertotopspace{X}{'}$
      is a homeomorphism, iff
      $f: \orderspace{X}{} \rarr \orderspace{X}{'}$ is an order isomorphism.
  \end{enumerate}

\end{theorem}

\begin{proof}
  ~

  \upref{enu:continuity_in_point_implies_order_property}
    Let $f: \ordertotopspace{X}{} \rarr \ordertotopspace{X}{'}$ be 
    continuous in $x_\ast \in X$. For $x=x_\ast$ we trivially have
    $f(x) \geq' f(x_\ast)$. Assume that there is an $x > x_\ast$ such that
    $f(x) \not \geq' f(x_\ast)$. This means $f(x) <' f(x_\ast)$, 
    because $\leq'$ is a total order on $X'$. Hence 
    $f(x_\ast) \in ('f(x) \eqdef U'$. Since $f$ is continuous in 
    $x_\ast$ there is an
    open neighborhood $U \in \mathcal{U}(x_\ast)$ with 
    $f[U] \subseteq U' = ('f(x)$. Since $x > x_\ast$ assures $x \in U$ we
    would consequently get $f(x) \in f[U] \subseteq ('f(x)$ -- 
    a contradiction.
  
  \upref{enu:continuity_implies_monotonicity}
    This directly follows from 
    \upref{enu:continuity_in_point_implies_order_property}
  
  \upref{enu:homeomorphism_vs_order_isomorphism}
    Let $f:\ordertotopspace{X}{} \rarr \ordertotopspace{X}{'}$ be a
    homeomorphism. The continuity of 
    $f: \ordertotopspace{X}{} \rarr \ordertotopspace{X}{'}$ and
    $f^{-1}: \ordertotopspace{X}{'} \rarr \ordertotopspace{X}{}$
    yields the monotonicity of 
    $f: \orderspace{X}{} \rarr \orderspace{X}{'}$ and 
    $f^{-1}: \orderspace{X}{'} \rarr \orderspace{X}{}$, respectively,
    by part \upref{enu:continuity_implies_monotonicity}.
    Now let, to the contrary,
    $f: {\orderspace{X}{}} \rarr {\orderspace{X}{'}}$ be an order isomorphism.
    Then the bijective mapping $f$ gives a one to one correspondence between
    the open sets of $\ordertotopspace{X}{}$ and the open sets of
    $\ordertotopspace{X}{'}$ -- essentially by 
    $(a \leftrightarrow ('f(a)$. Thus $f$ is a homeomorphism between
    these two topological spaces.
\end{proof}

The following example shows that there are monotone functions between
totally ordered sets which are not continuous in the deduced topologies.

\begin{example}
  Consider the 
  totally ordered sets
  $\orderspace{X}{} = ([0,1], \leq)$ and 
  $\orderspace{X}{'} = (\{2,3\}, \leq')$, with the natural orders $\leq$ on
  $[0,1]$ and $\leq'$ on $\{2,3\}$.
  The mapping $f: {\orderspace{X}{}} \rarr \orderspace{X}{'}$, given by
  \[
    f(x) 
    \defeq 
    \begin{cases}
      2& \text{ if } x \in [0,1)  \\
      3& \text{ if } x = 1.
    \end{cases},
  \]
  is monotone;
  yet $f: \ordertotopspace{X}{} \rarr \ordertotopspace{X}{'}$ is 
  not continuous: The preimage of 
  $
    \{3\}
    = 
    \{ x' \in X' : x' >' 2 \} 
    = 
    ('2 
    \in \mathcal{T}_{\leq^{'}}$ is the set
  $\{1\} \subseteq [0,1]$. This nonempty set does not belong to 
  $\mathcal{T}_\leq$, because it is neither the full space $X$, 
  nor can it be written as union of intervals of the form $(x$ 
  where $x\in [0,1]$.
\end{example}

\subsection{The right order topology on an inf-complete totally ordered set}
  \label{subsec:topology_for_inf_complete_linear_odered_sets}

In this subsection we give a remark showing that the right order 
topology gets very simple if the underlying totally ordered set 
$(X, \leq)$ fulfills a property called inf-completeness which is 
defined as follows:

\begin{definition}
  We call a totally ordered set $(X,\leq)$ {\bf inf-complete}, iff
  each subset $\widecheck X \subseteq X$ possesses an infimum 
  $\inf \widecheck X \in X$.
\end{definition}

\begin{remark} \label{rem:definition_of_right_order_topology_for_inf_complete_totally_ordered_sets}
  The the right order topology becomes very simple if it is given 
  to a totally ordered set $(X, \leq)$ which is inf-complete:
  Consider the union of sets $(a_i$ with $a_i \in X$ where $i$ runs 
  through some nonempty index set $I$. Due to the inf-completeness of 
  $(X, \leq)$ we know that $\inf \{a_i: i \in I\} \eqdef a$ exists in 
  $X$ so that the union
  \begin{gather*}
    \bigcup_{i \in I} \, (a_i = (a
  \end{gather*}
  is again of the very same form as the original sets.
  In particular $\mathcal{T}_\leq$ just consists of $X, \emptyset$
  and the sets of the form $(a$ where $a \in X$.
\end{remark}

\begin{example}
  Consider the set $X \defeq (0,1) \cup (2,4)$ endowed 
  with the usual order $\leq$.
  The totally ordered set $(X, \leq)$ is not inf-complete 
  since the interval $(2,3) \subset X$ has many lower bounds in $X$ 
  but no infimum in $X$. 
  Setting $a_i \defeq 2+\tfrac1i$, $i \in \N$ we see that the union 
  \begin{gather*}
    \bigcup_{i \in \N} \, (a_i = (2,4)
  \end{gather*}
  is neither $\emptyset, X$ nor of the form 
  $(a$ with some $a \in X$.
\end{example}

\subsection{Topological coercivity notions and continuity interpretations }
\label{subsec:topological_coercivity_and_its_interpretation_as_continuity}

Recall that $\mathcal{K}(X, \mathcal{O})$ denotes the system
of compact subsets of a topological space $(X,\mathcal{O})$,
whereas the system of its compact and closed subsets is denoted by 
$\mathcal{KA}(X, \mathcal{O})$.
In the following we will need the following subsystems.
\begin{definition}
  Let $(X, \mathcal{O})$ be a topological space and $S \subseteq X$.
  Then we set
  \begin{align*}
             \bm{\mathcal{KA}_S(X, \mathcal{O})} 
    & \defeq \{K \in \mathcal{KA}(X, \mathcal{O}): K \cap S = \emptyset \},
  \\
             \bm{\mathcal{K}_S(X, \mathcal{O})} 
    & \defeq \{K \in \mathcal{K}(X, \mathcal{O}): K \cap S = \emptyset \}.
  \end{align*}
\end{definition}
Note that 
$\mathcal{KA}_{\emptyset}(X, \mathcal{O}) = \mathcal{KA}(X, \mathcal{O})$.
The main idea behind the first definition is to collect all those 
closed and compact subsets of $(X,\mathcal{O})$ in the set system 
$\mathcal{KA}_S(X, \mathcal{O})$, which are not allowed to hit 
the set $S$ but which might come ``arbitrary close'' to $S$.
The idea behind the second definition is similar.

\begin{lemma}
  \label{lem:compact_and_closed_subsets_in_right_order_topology}
  Let $(Z, \leq)$ be a totally ordered set which has a 
  minimum $\widecheck z$. Then  the following holds true:
  \begin{enumerate}
    \item \label{enu:closed_subsets_are_all_compact}
      All closed subsets of $(Z,\mathcal{T}_\leq)$ are compact;
      in particular 
      $
	\mathcal{KA}(Z, \mathcal{T}_\leq) 
	=
	\mathcal{A}(Z, \mathcal{T}_\leq)
      $.
    \item \label{enu:compact_closed_sets_which_do_not_hit_max_el_are_compl_of_its_open_neighborhoods}
      If $(Z, \leq)$ contains also a maximum $\widehat z$
      then 
      $
	\mathcal{KA}_{\{\widehat z\}}(Z, \mathcal{T}_\leq)
	=
	\{Z \setminus U': U' \in \mathcal{U}'(\widehat z) 
	    \cap \mathcal{T}_\leq\}
      $.
  \end{enumerate}
\end{lemma}

\begin{proof}
  \upref{enu:closed_subsets_are_all_compact}
  No open set $O \in \mathcal{T}_\leq$ contains 
  the minimum $\widecheck z$ except for $O = Z$.
  Except for the closed set $\emptyset = Z \setminus Z$, which 
  is anyway compact, every closed subset $A$ of $(Z, \mathcal{T}_\leq)$
  contains hence $\widecheck z$. In particular any open covering 
  $(T_i)_{i \in I}$ of such a set $A$ must have a member $T$,
  which is an open neighborhood of $\widecheck z$.
  However, the only neighborhood of this minimal element is 
  the full space $Z$ by definition of $\mathcal{T}_\leq$.
  So picking out $T = Z$ already gives a finite subcovering for 
  $A \subseteq Z$. Hence the nonempty closed subsets of 
  $(Z, \mathcal{T}_\leq)$ are compact.
  In particular 
  $
    \mathcal{KA}(Z, \mathcal{T}_\leq)
    =
    \mathcal{K}(Z, \mathcal{T}_\leq) \cap \mathcal{A}(Z, \mathcal{T}_\leq)
    =
    \mathcal{A}(Z, \mathcal{T}_\leq)
  $.
\\
  \upref{enu:compact_closed_sets_which_do_not_hit_max_el_are_compl_of_its_open_neighborhoods}
  Using the previous part we see that the system
  $\mathcal{KA}_{\{\widehat z\}}(Z, \mathcal{T}_\leq)$
  consists of exactly those closed subsets of 
  $(Z, \mathcal{T}_\leq)$ which do not contain $\widehat z$,
  i.e. of exactly the complements of those open sets which contain 
  $\widehat z$.
  In other words the system
  $\mathcal{KA}_{\{\widehat z\}}(Z, \mathcal{T}_\leq)$ consists 
  of exactly the complements of open neighborhoods of $\widehat z$.
  This is what the formula 
    $
    \mathcal{KA}_{\{\widehat z\}}(Z, \mathcal{T}_\leq)
	=
	\{Z \setminus U': U' \in \mathcal{U}'(\widehat z) 
	    \cap \mathcal{T}_\leq\}
    $
  expresses.
\end{proof}

The first parts of the following two definitions 
stem from \cite{wiki:CoerciveFunctions} 
where just the name ``coercive'' was used.
However we prefer the names ``topological coercive'' and 
``strongly topological coercive'' here.
The second parts of these definitions 
are new to the best of the author's knowledge.
After stating the definitions we give some remarks on them and 
point out a relation to the notions of 
\emph{normcoercivity} and \emph{coercivity}.

\begin{definition} \label{def:topological_coercivity}
  A genuine mapping $f: \decorspace{X}{} \rarr \decorspace{X}{'}$
  between topological spaces 
  $\decorspace{X}{}$ and $\decorspace{X}{'}$
  is called {\bf topological coercive}, iff
  for every closed compact subset $K'$ of $X'$ 
  there is a closed compact subset
  $K$ of $X$ such that $f[X \setminus K] \subseteq X' \setminus K'$;
  i.e. -- more formally expressed -- iff 
  \begin{gather*}
	\forall K' \in \mathcal{KA}(X', \mathcal{O}')
    \;\;  \exists K \in \mathcal{KA}(X, \mathcal{O}):
	  f[X \setminus K] \subseteq X' \setminus K'.
  \end{gather*}
  holds true.
\\
  More generally we say that $f$ is {\bf topological coercive towards}
  a set ${\bm S'} \subseteq X$ iff
  for every closed compact subset $K'$ of $X'$ which does not hit $S'$
  there is a closed compact subset
  $K$ of $X$ such that $f[X \setminus K] \subseteq X' \setminus K'$;
  i.e. -- more formally expressed -- iff 
  \begin{gather*}
	\forall K' \in \mathcal{KA}_{S'}(X', \mathcal{O}')
    \;\;  \exists K \in \mathcal{KA}(X, \mathcal{O}):
	  f[X \setminus K] \subseteq X' \setminus K'.
  \end{gather*}
  holds true.
\end{definition}

By replacing ``compact and closed'' in the codomain 
in the previous definition by ``compact'' 
we get the following definition:

\begin{definition} \label{def:topological_strong_coercivity}
  A genuine mapping $f: \decorspace{X}{} \rarr \decorspace{X}{'}$
  between topological spaces 
  $\decorspace{X}{}$ and $\decorspace{X}{'}$
  is called {\bf topological strongly coercive}, iff
  for every compact subset $K'$ of $X'$ 
  there is a closed compact subset
  $K$ of $X$ such that $f[X \setminus K] \subseteq X' \setminus K'$;
  i.e. -- more formally expressed -- iff 
  \begin{gather*}
	\forall K' \in \mathcal{K}(X', \mathcal{O}')
    \;\;  \exists K \in \mathcal{KA}(X, \mathcal{O}):
	  f[X \setminus K] \subseteq X' \setminus K'
  \end{gather*}
  holds true.
\\
  More generally we say that $f$ is 
  {\bf topological strongly coercive towards}
  a set ${\bm S'} \subseteq X$ iff
  for every compact subset $K'$ of $X'$ which does not hit $S'$
  there is a closed compact subset
  $K$ of $X$ such that $f[X \setminus K] \subseteq X' \setminus K'$;
  i.e. -- more formally expressed -- iff 
  \begin{gather*}
		    \forall K' \in \mathcal{K}_{S'}(X', \mathcal{O}')
		\;\;  \exists K \in \mathcal{KA}(X, \mathcal{O}):
		      f[X \setminus K] \subseteq X' \setminus K'
  \end{gather*}
  holds true.
\end{definition}

\begin{remark}
  A genuine mapping $f: (X, \mathcal{O}) \rarr (X', \mathcal{O}')$
  is topological coercive iff it is topological 
  coercive towards $\emptyset$.
  Likewise the mapping $f$ is topological strongly coercive iff it 
  is topological strongly coercive towards $\emptyset$.
\end{remark}

\begin{remark}
  The previous Definitions \ref{def:topological_coercivity} and 
  \ref{def:topological_strong_coercivity} coincide 
  if the codomain $(X',\mathcal{O}')$ is a topological space whose 
  compact sets are all closed, e.g. if $(X', \mathcal{O}')$ 
  is a Hausdorff space, cf. \prettyref{thm:relation_closed_compact}.
  In later applications however the codomain will be 
  a totally ordered set equipped with the right order topology 
  which contains compact sets that are not closed, so that the 
  definitions no longer coincide.
\end{remark}

\begin{remark}
  Although the notion of topological coercivity is defined in the context
  of any topological spaces $(X, \mathcal{O})$ and $(X', \mathcal{O}')$
  it is rather made for noncompact spaces $(X, \mathcal{O})$ and 
  $(X', \mathcal{O}')$; if one of these spaces is compact the notion of 
  of topological coercivity becomes uninteresting:
  If $(X, \mathcal{O})$ is compact then every genuine mapping 
  $f:(X, \mathcal{O}) \rarr (X', \mathcal{O}')$ from 
  $(X, \mathcal{O})$ to any topological space 
  $(X', \mathcal{O}')$ is trivially topological coercive since we 
  can always choose $K \defeq X$.
  If, on the other hand, the space $(X', \mathcal{O}')$ is compact
  we can choose $K' \defeq X'$ so that a genuine mapping 
  $f:(X, \mathcal{O}) \rarr (X', \mathcal{O}')$ is topological 
  coercive iff $(X, \mathcal{O})$ is compact.
\end{remark}
In Subsection \ref{subsec:coercivity_vs_continuity} we will define 
the notion normcoercive for mappings 
$f: \R^n \rarr \R^m$ and the notion coercive for mappings 
$f: \R^n \rarr [\minfty, \pinfty]$ and see that these notions are 
special cases of topological coercivity towards a set:
One the one hand a mapping $f: \R^n \rarr \R^m$ is normcoercive 
iff it is topological coercive, i.e. topological coercive 
towards $\emptyset$, see \prettyref{thm:characterize_normcoercivity}.
On the other hand a mapping $f: \R^n \rarr [\minfty, \pinfty]$ is 
coercive iff it is topological coercive towards 
$\max [\minfty, \pinfty] = \{\pinfty\}$, 
see \prettyref{thm:characterize_lsc_and_coercive}.
For proving these equivalences the subsequent two 
theorems will be helpful.
\\[0.7ex]
The first of these theorems states that the
topological coervivity of a mapping 
$f: (X, \mathcal{O}) \rarr (X', \mathcal{O}')$ 
can be viewed as continuity at ``infinity'':

\begin{theorem} 
  \label{thm:topological_coercivivity_is_eq_to_continuity_of_extension}
  Let $(X, \mathcal{O})$ and $(X', \mathcal{O}')$ be topological spaces 
  and $(X, \mathcal{O})_\infty$ and $(X', \mathcal{O}')_{\infty'}$ their 
  one-point compactifications.
  For a mapping $f: X \rarr X'$ and
  its extension 
  $\widehat f: X_\infty \rarr X'_{\infty'}$,
  given by
  \begin{gather*}
    \widehat f (x) \defeq 
    \begin{cases}
      f(x), & \text{if } x \in X \\
      \infty', & \text{if } x = \infty
    \end{cases}
  \end{gather*}
  the following are equivalent:
  \begin{enumerate}
    \item \label{enu:function_is_topological_coercive}
      $f: (X, \mathcal{O}) \rarr (X', \mathcal{O}')$ 
      is topological coercive.
    \item \label{enu:extened_function_is_continuous_at_infty}
      $\widehat f: 
	(X, \mathcal{O})_\infty \rarr (X', \mathcal{O}')_{\infty'}$
      is continuous at $\infty$.
  \end{enumerate}
\end{theorem}

\begin{proof}
  Using the definitions of topological coercivity and the 
  definition of the one point compactification we get
  \begin{align*}
   &\hphantom{~\iseq} 
	      f: (X, \mathcal{O}) \rarr (X', \mathcal{O}')
	      \text{ is topological coercive }  
   \\
   & \iseq
    	\forall K' \in \mathcal{KA}(X', \mathcal{O}')		   
    \;\;  \exists K \in \mathcal{KA}(X, \mathcal{O}):		   
	  f[X \setminus K] \subseteq X' \setminus K'
   \\
   & \iseq
    	\forall K' \in \mathcal{KA}(X', \mathcal{O}')
    \;\;  \exists K \in \mathcal{KA}(X, \mathcal{O}):  
	  f[X \setminus K] \cup \{\infty'\} 
	  \subseteq 
	  (X' \cup \{\infty'\}) \setminus K'
   \\
   & \iseq
    	\forall K' \in \mathcal{KA}(X', \mathcal{O}')
    \;\;  \exists K \in \mathcal{KA}(X, \mathcal{O}):  
	  \widehat f[X_\infty \setminus K] 
	  \subseteq 
	  X'_{\infty'} \setminus K'
   \\
   & \iseq
    	\forall U' \in \mathcal{U}'(\infty')		    
    \;\;  \exists U \in \mathcal{U}(\infty):
	  \widehat f[U] \subseteq U'
   \\
    & \iseq  
	    \widehat f :(X, \mathcal{O})_\infty 
		  \rarr (X', \mathcal{O}')_{\infty'}  
		  \text{ is continuous at the point } \infty.
  \end{align*}
\end{proof}

Regard now a mapping 
$f:(X, \mathcal{O}) \rarr (Z, \mathcal{T}_\leq)$
where $\mathcal{T}_\leq$ is the right order topology 
induced by some total order $\leq$ on $Z$.
If $Z$ has a minimum and a maximum
we can similar regard the topological coercivity of $f$
towards $\max X'$ as continuity at ``infinity'':

\begin{theorem} 
\label{thm:topological_coercivity_towards_maximal_element_equivalent_to_continuous_property}
  Let $(X, \mathcal{O})$ be a topological space and $(Z, \leq)$ a 
  totally ordered set which has a minimum $\widecheck z$ and a maximum 
  $\widehat z$.
  For a mapping $f: X \rarr Z$ and its extension 
  $\widehat f: X_\infty \rarr Z$ given by 
  \begin{gather}
    \widehat f(x) \defeq 
    \begin{cases}
      f(x) & \text{if } x \in X
    \\
      \widehat z & \text{if } x = \infty
    \end{cases}
  \end{gather}
  the following are equivalent:
  \begin{enumerate}
    \item \label{enu:mapping_top_coercive_towards_max}
      $f:(X, \mathcal{O}) \rarr (Z, \mathcal{T}_\leq)$ is 
      topological coercive towards 
      $\{ \widehat z \} = \{\max_\leq Z\}$.
    \item \label{enu:extended_mapping_in_same_totally_ordered_space}
      $\widehat f: (X, \mathcal{O})_\infty \rarr (Z, \mathcal{T}_\leq)$
      is continuous at the point $\infty$.
  \end{enumerate}  
\end{theorem}

\begin{proof}
  Using part \upref{enu:compact_closed_sets_which_do_not_hit_max_el_are_compl_of_its_open_neighborhoods}
  of \prettyref{lem:compact_and_closed_subsets_in_right_order_topology}
  and $\widehat f(\infty) = \widehat z$ we obtain
  \begin{align*}
   &\hphantom{~\iseq} 
	      f: (X, \mathcal{O}) \rarr (Z, \mathcal{T}_\leq)
	      \text{ is topological coercive towards } 
		  \{ \widehat z \}
   \\
   & \iseq
    	\forall K' \in \mathcal{KA}_{\{\widehat z\}}(Z, \mathcal{T}_\leq)		   
    \;\;  \exists K \in \mathcal{KA}(X, \mathcal{O}):		   
	  f[X \setminus K] \subseteq Z \setminus K'
   \\
   & \iseq
    	\forall U' \in \mathcal{U}'(\widehat z) \cap \mathcal{T}_\leq
    \;\;  \exists K \in \mathcal{KA}(X, \mathcal{O}):		   
	  f[X \setminus K] \subseteq Z \setminus (Z \setminus U')   \\
   & \iseq
    	\forall U' \in \mathcal{U}'(\widehat z) 
    \;\;  \exists K \in \mathcal{KA}(X, \mathcal{O}):		   
	  f[X \setminus K] \subseteq U'   \\
   & \iseq
    	\forall U' \in \mathcal{U}'(\widehat z)		    
    \;\;  \exists K \in \mathcal{KA}(X, \mathcal{O}):
	  \widehat f[X_\infty \setminus K] \subseteq U'
   \\
   & \iseq
    	\forall U' \in \mathcal{U}'(\widehat z)		    
    \;\;  \exists U \in \mathcal{U}(\infty):
	  \widehat f[U] \subseteq U'
   \\
    & \iseq  
	    \widehat f :(X, \mathcal{O})_\infty 
		  \rarr (Z, \mathcal{T}_\leq)  
		  \text{ is continuous at the point } \infty.    
  \end{align*}

\end{proof}

\subsection{Topological coercivity and boundedness below}
\label{subsec:topological_coercivity_and_boundedness_below}

In this subsection we deal with the relations
between one global and two local boundedness notions
and give a sufficient criteria when 
local boundedness implies the global boundedness,
cf. also \cite[p. 240f]{BrucknerEtAl2008real}.

We first give the definitions of the mentioned boundedness notions.
\begin{definition}
  Let $f: X \rarr Z$ be a genuine mapping from a topological space
  $(X, \mathcal{O})$ to some totally ordered set $(Z, \leq)$.
  We call $f$ {\bf bounded below}, if there is some $\widecheck z \in Z$
  such that  $f(x) \geq \widecheck z$ for all $x \in X$.
  We call $f$ {\bf locally bounded below}, iff every point $x_0 \in X$
  has a neighborhood $U \in \mathcal{U}(x_0)$ where $f|_U$ is bounded 
  below; i.e. -- more formally expressed -- iff
  \begin{gather*}
	\forall x_0 \in X
    \;\;  \exists U \in \mathcal{U}(x)
    \;\;  \exists \widecheck z \in Z
    \;\;  \forall x \in U : f(x) \geq \widecheck z
  \end{gather*}
  holds true.
  Similarly, we call $f$ {\bf compactly bounded below}, iff $f$ is bounded
  below on every compact subset of $X$; 
  i.e. -- more formally expressed -- iff
  \begin{gather*}
	\forall K \in \mathcal{K}(X, \mathcal{O})
    \;\;  \exists \widecheck z \in Z
    \;\;  \forall x \in K : f(x) \geq \widecheck z
  \end{gather*}
  holds true.
\end{definition}

The next proposition shows relations between these boundedness notions.
Note therein that the relation between "locally bounded below" and 
"compactly bounded below" is similar to the 
relation between the notions "locally uniform convergence"
and "compactly (uniform) convergence":
  Local boundedness below always implies 
  compact boundedness below; 
  in locally compact spaces the two notions even coincide.
Note further that all three boundedness notions for a mapping 
$f: (X, \mathcal{O}) \rarr (Z,\leq)$ coincide if 
$f: (X, \mathcal{O}) \rarr (Z, \mathcal{T}_\geq)$
is topological strongly coercive towards $\MAX_\leq(Z)$.

%

\begin{proposition} \label{prop:relation_between_boundedness_notions_for_mapping}
  The different boundedness notions for a function 
  $f: (X, \mathcal{O}) \rarr (Z, \leq)$ between a topological space and 
  a totally ordered space are related as follows:
  \begin{gather*}
    \label{diag:relation between the boundedness notions}
    \xymatrix{
      f \text{ bounded below } 
	  \ar@<-1ex>@{=>}[d]
	  \\
      f \text{ locally bounded below }
	  \ar@<-1ex>@{=>}[d]  
	  \\
      f \text { compactly bounded below} 
	  \ar@<-1ex>@{=>}[u]
	    _{(X, \mathcal{O}) \text{ loc. comp.} }
	  \ar@<-13ex>@/_2.0pc/@{=>}[uu]
	    _{f:(X, \mathcal{O}) \rarr (Z, \mathcal{T}_\geq) 
	      \text{ top. str. coerc. tow. } \MAX_\leq(Z)} 
    }
  \end{gather*}
\end{proposition}

\begin{proof}
  Clearly boundedness below implies locally boundedness below.
  Next, let $f: X \rarr Z$ be locally bounded below. For every $x \in X$
  there is then some -- without loss of generality open -- neighborhood
  $U_x$ of $x$ and some $z_x \in Z$ such that 
  \begin{gather*}
    f(\widetilde x) \geq z_x
  \end{gather*}
  for all $\widetilde x \in U_x$. 
  Let now $K$ be some -- without loss of generality nonempty -- compact 
  subset of $X$. Clearly the sets $U_x$, $x \in K$ form an open
  covering of $K$. By the compactness of $K$ there are finitely many 
  $x_1, x_2, \dots x_N \in K$ with 
  \begin{gather*}
    \bigcup_{n = 1}^{N} U_{x_n} \supseteq K.
  \end{gather*}
  Setting $\widecheck z \defeq \min\{z_{x_1}, z_{x_2}, \dots, z_{x_N}\}$
  we hence get $f(x) \geq \widecheck z$ for all $x \in K$,
  so that $f$ is indeed compactly bounded below.

  Assume now that $(X, \mathcal{O})$ is additionally locally compact 
  and let to the contrary $f$ be compactly bounded below.
  Every $x_0 \in X$ has some compact neighborhood $K \eqdef U.$
  For this compact set there is some $\widecheck z \in Z$ such that
  $f(x) \geq \widecheck z$ for all $x \in K= U$. 
  Thus $f$ is locally bounded below.
  Finally we consider a mapping 
  $f:(X, \mathcal{O}) \rarr (Z, \mathcal{T}_\geq)$ which is 
  topological strongly coercive towards $\MAX_\leq(Z)$ and show that
  $f: (X, \mathcal{O}) \rarr (Z, \leq)$ is already bounded below if
  it is compactly bounded below.
  Assuming the latter we reason dependent on the cardinality of $Z$.
  If $Z$ contains at most one element then $f$ anyway is bounded 
  below. Otherwise we choose any $z' \in Z \setminus \MAX_\leq(Z)$ 
  and consider the set
  \begin{gather*}
    K' \defeq \{ z \in Z : z \leq z' \}.
  \end{gather*}
  The set $K'$ is a compact subset of $(Z, \mathcal{T}_\geq)$ 
  by Detail \ref{det:closed_left_halfray_is_compact_in_opposite_topology}
  in the Appendix.
  Therefore and since $f: (X, \mathcal{O}) \rarr (Z, \mathcal{T}_\geq)$
  is topological strongly coercive towards $\MAX_\leq(Z)$
  there is a compact set $K \in \mathcal{K}(X, \mathcal{O})$ with
  $f[X \setminus K] \subseteq Z \setminus K'$, i.e.
  \begin{gather*}
    f(x) > z' \text{ for all } x \in X \setminus K.
  \end{gather*}
  Moreover the compactly lower bounded function 
  $f: (X, \mathcal{O}) \rarr (Z, \leq)$ is bounded below on $K$,
  i.e. there is a $z'' \in Z$ such that 
  \begin{gather*}
    f(x) \geq z'' \text{ for all } x \in K.
  \end{gather*}
  Summarizing we have $f(x) \geq \min \{ z', z''\}$ for all 
  $x \in X$, so that $f$ is indeed bounded below.
\end{proof}

\newcommand{\CaptionForContentsTheTopSpace}{The topological space $({[}-\infty, +\infty{]}, \mathcal{T})$}
\section[\CaptionForContentsTheTopSpace]
    {The topological space $\boldsymbol{([-\infty, +\infty], \mathcal{T})}$ } 
\label{sec:the_topological_space_intervalus_maximus_with_the_right_order_topology}

In subsections 
\begin{itemize}
  \item 
    \ref{subsec:a_topology_on_intervalus_maximus_suited_for_lower_semicontinuous_functions}
      A topology on ${[}-\infty, +\infty{]}$ suited for lower semicontinuous functions  
  \item 
    \ref{subsec:properties_of_the_topological_space_intervalus_maximus_with_right_order_topology}
     Properties of the topological space $({[}-\infty, +\infty{]}, \mathcal{T})$
  \item 
    \ref{subsec:known_properties_of_lsc_functions_revisited}
    Known properties of lower semicontinuous functions revisited
  \item 
    \ref{subsec:coercivity_vs_continuity}
    Coercivity properties versus continuity properties
  \item 
    \ref{subsec:continuous_arithmetic_operations_in_intervalus_maximus_with_right_order_topology}
    Continuous arithmetic operations in $([-\infty, +\infty], \mathcal{T})$
\end{itemize}
we equip $[\minfty, \pinfty]$ with the right order topology
$\mathcal{T} = \mathcal{T}_\leq$,
study some properties of the resulting topological space 
$([\minfty, \pinfty], \mathcal{T})$, allowing us to 
see known properties of lower semicontinuous functions 
in a topological light, show that coercivity can be regarded 
as continuity, and that there is a continuous addition 
on $[\minfty, \pinfty]$ if the topology $\mathcal{T}$ is 
installed on $[\minfty, \pinfty]$.

A key role for establishing a -- as far as the author knows -- new 
topological method for proving lower semicontinuity plus coercivity of 
a function is due to
\prettyref{thm:characterize_lsc_and_coercive},
which allows us to replace the task of proving the lower semicontinuity 
and coercivity of a function $h: \R^n \rarr [\minfty, \pinfty]$
by the task of showing that $h$ admits a certain continuous extension.

\newcommand{\CaptionForContentsAtopologyOn}{A topology on ${[}-\infty, +\infty{]}$ suited for lower semicontinuous functions}
\subsection[\CaptionForContentsAtopologyOn]
    {A topology on $\boldsymbol{[-\infty, +\infty]}$ suited for lower semicontinuous functions} 
\label{subsec:a_topology_on_intervalus_maximus_suited_for_lower_semicontinuous_functions}

In this subsection we search for a topology $\mathcal{T}$ for 
the interval 
$[\minfty, \pinfty]$ which is suited when dealing with 
lower semicontinuous functions. 


\begin{definition}
  A function $f: \R^n \rarr [\minfty,\pinfty]$ is called 
  {\bf lower semicontinuous} or {\bf lsc}, 
  iff it has one of the following equivalent properties:
  \begin{itemize}
    \item
      $
	\forall x, x_1, x_2, x_3, \dots \in \R^n: 
	x_l \rarr x 
	\implies
	f(x) \leq \liminf_{l\rarr \pinfty} f(x_l)
      $,
    \item 
      $
	f^-[[\minfty,\alpha]] 
	\text{ is closed for all } 
	\alpha \in (\minfty, \pinfty).	
      $
  \end{itemize}
\end{definition}

These conditions are really equivalent, cf.
\cite[Theorem 7.1]{Rockafellar1970}.


We start with a consideration which will lead us to the definition of 
our topology for $[\minfty,\pinfty]$.

Let $f:\R^n \rarr [\minfty,\pinfty]$ be a function.
Referring to the natural topology of
$\R^n$, when speaking about ``open'' and ``closed'' sets, we have
%
%
\begin{align}
  f \text{ is lsc} & \iseq f^-[[\minfty,\alpha]] \text{ is closed for all } \alpha \in (\minfty, \pinfty) \\
                   & \iseq f^-[(\alpha,\pinfty]] \text{ is open for all }   \alpha \in (\minfty, \pinfty). \label{eq:preimageSubbasissetIsOpen}
\end{align}



{{\bf Agreement.} \it
  In the rest of this thesis the interval $[\minfty, \pinfty]$ will 
  -- unless otherwise stated -- be equipped 
  with the topology created by taking the above sets
  $(\alpha,\pinfty], \alpha \in (\minfty, \pinfty)$ as subbasis, 
  i.e. with the topology
  \begin{equation*}
    \mathcal{T} 
    \coloneq
    \{ 
	\emptyset, [\minfty, \pinfty], 
	(\alpha, \pinfty]: \alpha \in (\minfty, \pinfty)
    \},
  \end{equation*}
  which is the right order topology $\mathcal{T}_\leq$ 
  for the inf--complete,
  totally ordered space ${([\minfty,\pinfty], \leq)}$, cf. Remark
  \ref{rem:definition_of_right_order_topology_for_inf_complete_totally_ordered_sets}.
  Only in 
  a 
  few situations we will equip $[\minfty, \pinfty]$ 
  with the ``just opposite'' topology
  \begin{equation*}
    \mathcal{T}_\geq
    =
    \{ 
	\emptyset, [\minfty, \pinfty], 
	[\minfty, \beta): \beta \in (\minfty, \pinfty)
    \}.
  \end{equation*}
}

By equivalence \prettyref{eq:preimageSubbasissetIsOpen} a function 
$f:\R^n \rarr [\minfty,\pinfty]$ is lower semicontinuous if 
and only if the preimages of all sets 
$(\alpha,\pinfty],\alpha \in (\minfty,\pinfty)$, are open sets.
Since the intervals $(\alpha,\pinfty],\alpha \in (\minfty,\pinfty)$ form
a subbasis of $\mathcal{T}$ we further have
\begin{align*}
  & f^- [(\alpha,\pinfty]] \text{ is open for all } \alpha \in (\minfty,\pinfty)
\\
	\iseq {}&  f^- [T] \text{ is open for all } T\in \mathcal{T}
\\
	\iseq {}& f:(\R^n,\mathcal{O}^{\otimes n}) \rarr ([\minfty, \pinfty],\mathcal{T})
	\text { is continuous}.
\end{align*}
%
In summary we obtain the following theorem, cf. 
\cite[Examples II -- 2.3 (3)]{GierzScottEtAl2003continuous}
\begin{theorem} \label{thm:characterize_lower_semi_continuous}
  For a mapping $f: \R^n \rarr [\minfty,\pinfty]$ the following are 
  equivalent:
  \begin{enumerate}
    \item 
      $f: \R^n \rarr [\minfty,\pinfty]$ is lower semicontinuous,
    \item 
      $f:(\R^n,\mathcal{O}^{\otimes n}) \rarr ([\minfty,\pinfty],\mathcal{T})$
      is continuous.
  \end{enumerate}

\end{theorem}

By this theorem the notion of lower semicontinuity can be extended 
to a broader class of functions,
while staying consistent with the definition for 
functions $f:\R^n \rarr [\minfty,\pinfty]$.

\begin{definition} \label{def:lowersemicontinuity_of_a_function_defined_on_a_topological_space}
  Let a set $X$ be endowed with some topology $\mathcal{O}_X$.
  A mapping $f: X \rarr [\minfty,\pinfty]$ is called 
  {\bf lower semicontinuous} iff 
  $f: (X, \mathcal{O}_X) \rarr ([\minfty,\pinfty], \mathcal{T})$
  is continuous.
\end{definition}

The topology $\mathcal{T}$ on $[\minfty,\pinfty]$ does not only 
allow to regard the notion of lower semicontinuity 
as continuity; also the notion of coercivity 
can be viewed as continuity property, 
see \prettyref{thm:characterize_lsc_and_coercive}.

\newcommand{\CaptionForContentsPropertiesOfTheTopSpace}{Properties of the topological space $({[}-\infty, +\infty{]}, \mathcal{T})$}
\subsection[\CaptionForContentsPropertiesOfTheTopSpace]
    {Properties of the topological space $\boldsymbol{([-\infty, +\infty], \mathcal{T})}$} 
\label{subsec:properties_of_the_topological_space_intervalus_maximus_with_right_order_topology}

The topology $\mathcal{T}$ is not induced by a metric on 
$[\minfty, \pinfty]$ since otherwise every two distinct points
would have non-overlapping neighborhoods, 
but this is obviously not the case; consider for
example the points $x_1 = 1$ and $x_2 = 2$ and any two neighborhoods 
$N_1$ and $N_2$ of $x_1$ and $x_2$, respectively -- the intersection 
$N_1 \cap N_2 \supseteq [2, \pinfty]$ is not empty.
Only by this fact that $([\minfty, \pinfty],\mathcal{T})$ 
is not a Hausdorff space
the following phenomena are possible:
\begin{enumerate}
  \item \label{enu:phenomena_of_several_limit_points}
    A sequence $(y_k)_{k \in \N}$ in $([\minfty,\pinfty], \mathcal{T})$
    can have several limit points at the same time.
    In particular, $\minfty$ is a limit point of any sequence
    $(y_k)_{k \in \N}$ in $([\minfty,\pinfty], \mathcal{T})$.
  \item \label{enu:phenomena_of_compact_subsets_that_are_not_closed}
    The space $([\minfty,\pinfty], \mathcal{T})$ contains 
    compact subsets that are not closed.
\end{enumerate}
Illustrations of these phenomena can be found in 
\prettyref{exa:more_than_one_limit_point} and 
\prettyref{exa:compact_not_closed_subset}, \linebreak[3] respectively.
Phenomena \upref{enu:phenomena_of_several_limit_points} 
is completely explained by 
\prettyref{thm:limits_in_intervalus_maximus_with_right_order_topology}.


\begin{example} \label{exa:more_than_one_limit_point}
  Consider the constant sequence $(y_n)_{n\in \N} = (1)_{n \in \N}$
  in the topological space $([\minfty,\pinfty], \mathcal{T})$.
  On the one hand every $y \in (1, \pinfty]$ is not a 
  $\mathcal{T}$-limit point of $(y_n)$; indeed, the neighborhood 
  $U \defeq (z, \pinfty]$ of $y$, where $z$ is any point between 
  $1$ and $y$, does not contain even one single sequence member.
  On the other hand every $y \in [\minfty, 1]$ is a 
  $\mathcal{T}$-limit point of $(y_n)$; indeed, any neighborhood 
  of $y$ contains the set $[y, \pinfty]$ and hence even all 
  sequence members.
\end{example}

More generally we have the following theorem:
\begin{theorem}[Limits of sequences in $({[}\minfty, \pinfty{]},\mathcal{T})$]
\label{thm:limits_in_intervalus_maximus_with_right_order_topology}
  Let $(y_n)_{n \in \N}$ be a sequence in $[\minfty,\pinfty]$.
  A point $y \in [\minfty,\pinfty]$ belongs to 
  $\tlim_{n\rarr \pinfty} y_n$, iff 
  $y \leq \liminf_{n \rarr \pinfty}  y_n $. 
  In particular the point $\minfty$ is 
  $\mathcal{T}$-limit point of 
  every sequence in $([\minfty,\pinfty],\mathcal{T})$.
\end{theorem}

\begin{proof}
  Consider first the case $y = \minfty$.
  Then clearly $y \leq \liminf_{n \rarr \pinfty} y_n$ and also
  $y \in \tlim_{n \rarr \pinfty} y_n$, 
  because the only $\mathcal{T}$-neighborhood of $y = \minfty$ is 
  $[\minfty, \pinfty]$ which contains trivially all $y_n$.
  Hence the claimed equivalence holds true in this case.
  In the other case $y \in (\minfty, \pinfty]$ we have 
  $y \not \in \tlim_{n \rarr \infty} y_n$ iff 
  there is some neighborhood $(a, \pinfty]$ of $y$ where 
  $a \in (\minfty, y)$ such that $y_n \not \in (a, \pinfty]$
  for infinitely many $n \in \N$, i.e. iff 
  $\liminf_{n \rarr \pinfty} y_n < y$ holds true.
  So the claimed equivalence holds true also in that case.
\end{proof}

\begin{theorem}[Compact subspaces of $({[}\minfty, \pinfty{]},\mathcal{T})$]
\label{thm:compactSubspaces}
  ~
   For nonempty subsets $K \subseteq [\minfty, \pinfty]$ the following are
   equivalent:
  \begin{enumerate}
    \item
      $(K, K \Cap \mathcal{T} )$ is a compact subspace of
      $([\minfty, \pinfty],\mathcal{T})$.
    \item
      $\inf K$ belongs to $K$.
  \end{enumerate}
  In particular the whole space $([\minfty, \pinfty],\mathcal{T})$
  is compact.
\end{theorem}

Before proving this theorem we give an example that shows that
the space $([\minfty, \pinfty], \mathcal{T})$ has compact subsets 
which are not closed.
It also illustrates that -- in contrast to the infimum -- 
the supremum of compact subsets of 
$([\minfty, \pinfty], \mathcal{T})$
needs not to belong to the compact set.

\begin{example} \label{exa:compact_not_closed_subset}
  Consider the set $K\defeq[0,1)$. $(K, K \Cap \mathcal{T})$ is compact 
  by \prettyref{thm:compactSubspaces};
  yet $K$ is not a closed subset of
  $([\minfty, \pinfty],\mathcal{T})$, since its complement
  $[\minfty,0) \cup [1,\pinfty]$ is obviously not an open set from
  $\mathcal{T}$. Furthermore $K$ does clearly not contain its supremum $1$.
\end{example}

This examples and part \upref{enu:closed_subsets_are_all_compact} of 
\prettyref{lem:compact_and_closed_subsets_in_right_order_topology} shows
$
  \mathcal{A}([\minfty, \pinfty], \mathcal{T}) 
  \subset
  \mathcal{K}([\minfty, \pinfty], \mathcal{T}) 
$.
Such a relation can never be true in Hausdorff spaces
$(X, \mathcal{O})$, where we rather have
$
  \mathcal{A}(X, \mathcal{T}) 
  \supseteq
  \mathcal{K}(X, \mathcal{T}) 
$,
due to part \upref{enu:compactness_implies_closeness} of 
\prettyref{thm:relation_closed_compact}
or even 
$
  \mathcal{A}(X, \mathcal{T}) 
  \supset
  \mathcal{K}(X, \mathcal{T}) 
$
if the space $(X, \mathcal{O})$ is not compact.

\begin{proof}[Proof of \prettyref{thm:compactSubspaces}]
  Let $(K, K \Cap \mathcal{T} )$ be any nonempty compact subspace and let
  $\widecheck k \in [\minfty, \pinfty]$ denote the infimum of $K$. 
  In the first case $\widecheck k = \pinfty$ the nonemptiness of $K$ yields $K=\{\pinfty\}$
  and thus $\widecheck k \in K$.
  In the second case $\widecheck k = \minfty$ we must have $\widecheck k \in K$,
  since otherwise the sets $(z,\pinfty] \in \mathcal{T}$, 
  $z \in \{-1, -2, -3, \dots\}$ would form an open covering of
  $K$ which can not be reduced to a finite subcover;
  so $K$ would not be compact.
    In the final third case $\widecheck k \in \R$ we similarly must have
  $\widecheck k \in K$ since otherwise the sets $\{\widecheck k + \frac{1}{n}\}$, 
  $n \in \N$ would form an open covering of $K$ which has no finite subcover.

  Let, to the contrary, $K$ now be a nonempty subset of
  $[\minfty, \pinfty]$  with $\widecheck k \defeq \inf K \in K$ and let
  $(T_i)_{i \in I}$ be an open covering of $K$ with sets $T_i$ from
  $\mathcal{T}$. Due to 
  \begin{equation*}
    \widecheck k \in K \subseteq \bigcup_{i \in I} T_i
  \end{equation*}
  there is an $i_{\ast} \in I$ with $\widecheck k \in T_{i_\ast}$. With this
  open set
  \begin{align*}
    T_{i_\ast} & \in \mathcal{T} \setminus \{\emptyset\} 
    \\
		& =  \{ 
		      [\minfty, \pinfty], (\alpha, \pinfty]: 
		      \alpha \in (\minfty, \pinfty] 
		    \}
  \end{align*}
  we already have found a finite subcover,
  because $T_{i_\ast} \supseteq [\widecheck k, \pinfty] \supseteq K$.
  So $(K, K \Cap \mathcal{T} )$ is a compact subspace of
  $([\minfty, \pinfty],\mathcal{T})$.

  Note finally that $[\minfty, \pinfty]$ 
  contains its infimum $\minfty$, so that
  $([\minfty, \pinfty],\mathcal{T})$ is compact 
  by the already proven equivalence.
\end{proof}

In the subsequent subsection we will
use Theorem \ref{thm:characterize_lower_semi_continuous}
and Theorem \ref{thm:compactSubspaces} to give a topological proof
of the known results that the composition 
of a continuous function with a lower semicontinuous function is again
lower semicontinuous and that a lower semicontinuous function 
takes its infimum on any nonempty compact set, 
respectively\DissVersionForMeOrHareBrainedOfficialVersion{, hoping that these two facts will appear in a new light.}{.}

\todo{NACH /BEIM erstem Korrekturlesen. sinvoll?? - Generalvorrausetzung: Top. Raeume nichtleer}

\subsection{Known properties of lower semicontinuous functions revisited} 
\label{subsec:known_properties_of_lsc_functions_revisited}

In this subsection we revisit known properties of lower semicontinous
functions. We will see that these properties stem from 
\prettyref{thm:characterize_lower_semi_continuous}
and the properties of the space $([\minfty,\pinfty], \mathcal{T})$.
The property we start with is the fact that every composition 
$g \circ f$ of a continuous mapping $f$ with some lower semicontinuous
mapping $g$ is lower semicontinous,
cf. \cite[1.40 Exercise]{RW04}.

\begin{theorem}
  Let $(X, \mathcal{O}_X)$ and $(Y, \mathcal{O}_Y)$
  be topological spaces,
  $f: (X, \mathcal{O}_X) \rarr (Y, \mathcal{O}_Y)$
  be a continuous mapping and 
  $g: Y \rarr [\minfty,\pinfty]$
  be a lower semicontinous mapping.
  Then the concatenation $h \defeq g \circ f: X \rarr [\minfty,\pinfty]$
  is again lower semicontinous.
\end{theorem}

\begin{proof}
  The mappings
  $f: (X, \mathcal{O}_X) \rarr (Y, \mathcal{O}_Y)$
  and 
  $g: (Y, \mathcal{O}_Y) \rarr ([\minfty,\pinfty], \mathcal{T})$
  are continuous by assumption and by definition, respectively.
  Hence their concatenation 
  $
    h = g \circ f : 
    (X, \mathcal{O}_X) \rarr ([\minfty,\pinfty], \mathcal{T})
  $
  is again continuous, i.e.
  $h: X \rarr [\minfty,\pinfty]$ is lower semicontinous.
\end{proof}

Phenomenon
\upref{enu:phenomena_of_several_limit_points} 
in \prettyref{subsec:properties_of_the_topological_space_intervalus_maximus_with_right_order_topology}
said that a sequence $(y_k)_{k \in \N}$ in $([\minfty,\pinfty], \mathcal{T})$
can have several limit points at the same time and that 
$\minfty$ is always a limit point.
The first part of this phenomenon is reflected also in the 
fact that lower semicontinous functions defined on punctured $\R^n$ 
can be usually continued in many ways to a lower semicontinous function
on whole $\R^n$, see 
Example \ref{exa:different_continuations_are_lsc}.
The second part of this phenomenon is reflected in the fact
that a function 
$f: (X,\mathcal{O}) \rarr ([\minfty, \pinfty],\mathcal{T})$
is automatically continuous in all preimage points of $\{\minfty\}$,
see \prettyref{lem:continuity_in_points_which_are_mapped_to_minfty}.

\begin{example} \label{exa:different_continuations_are_lsc}
  Consider the function $f:\R \setminus \{0\} \rarr [\minfty,\pinfty]$,
  given by $f(x) \defeq 1$.
  Setting $f(0) \defeq c$ with any $c \in [\minfty,1]$
  we obtain a lower semicontinous function 
  $f: \R \rarr [\minfty,\pinfty]$.
\end{example}

The following lemma is directly obtained as special case of 
\prettyref{prop:neighborhoods_of_minimal_element}.

\begin{lemma} \label{lem:continuity_in_points_which_are_mapped_to_minfty}
  Let $(X, \mathcal{O})$ be a topological space and 
  $f: (X,\mathcal{O}) \rarr ([\minfty, \pinfty],\mathcal{T})$
  a mapping. For every $x \in X$ we have
  \[
    f(x) = \minfty \implies f \text{ is continuous in } x.
  \]
\end{lemma}
%

\begin{proof}
  Let $x\in X$ be a point with $f(x) = \minfty$. For each neighborhood $U$ of $x$
  we trivially have $f[U] \subseteq [\minfty,\pinfty]$. Since $[\minfty,\pinfty]$
  is the only existing neighborhood of $\minfty = f(x)$, this inclusion already 
  shows that $f$ is continuous in $x$.
\end{proof}

The following theorem says that a lower semicontinous function 
attains a minimum on every nonempty compact subset, 
cf. \cite[1.10 Corollary]{RW04}.

\begin{theorem} \label{thm:lsc_function_takes_minimum}
  Let $(X, \mathcal{O}_X)$ be a topological space and 
  $f: X \rarr [\minfty,\pinfty]$ be lower semicontinous.
  Then $f$ attains its infimum on any nonempty compact subset of 
  $(X, \mathcal{O}_X)$.
\end{theorem}

\begin{proof}
  The mapping 
  $f: (X, \mathcal{O}_X) \rarr ([\minfty,\pinfty],\mathcal{T})$
  is continuous by \prettyref{def:lowersemicontinuity_of_a_function_defined_on_a_topological_space}.
  Hence every nonempty compact subset $K$ of $(X, \mathcal{O}_X)$ 
  is mapped by $f$ to a compact subset of 
  $([\minfty,\pinfty],\mathcal{T})$.
  This again compact image $f[K]$ contains its infimum by 
  \prettyref{thm:compactSubspaces}. 
\end{proof}

By \prettyref{thm:lsc_function_takes_minimum} a 
lower semicontinuous function $f: X \rarr [\minfty,\pinfty]$ 
on a topological space $(X, \mathcal{O}_X)$ takes 
its minima on every nonempty compact subset of this space.
However $f$ does not need to takes maxima on nonempty compact subsets
as the following example shows.

\begin{example}
  The function $f: \R \rarr [\minfty,\pinfty]$ given by
  \begin{gather*}
    f(x)
    \defeq 
    \begin{cases}
      - |x| & \text{ for } x \neq 0,
    \\
      - 1   & \text{ for } x = 0,
    \end{cases}
  \end{gather*}
  is a lower semicontinous function that does not 
  attain its supremum $0 = \sup_{x\in [-1,1]} f(x)$ on the compact 
  subset $[-1,1]$ of $(\R, \mathcal{O})$.  
\end{example}

We conclude this subsection by giving a table with 
some properties of the topological space $([\minfty, \pinfty], \mathcal{T})$
and corresponding properties of lower semicontinuous functions, 
i.e. continuous functions 
$(X, \mathcal{O}) \rarr ([\minfty, \pinfty], \mathcal{T})$.

\begin{center}
\begin{tabularx}{\textwidth}{p{0.28\textwidth} | p{0.48\textwidth} | p{0.17\textwidth}}
  Space $([\minfty, \pinfty],\mathcal{T})$  & Function $f:(X, \mathcal{O}) \xlongrightarrow{\text{cont.}} ([\minfty, \pinfty],\mathcal{T})$ & cf.                                                                     \\
  \hline
  A seq. $(y_k)_{k \in \N}$ can have many  
  limit points: \linebreak				     
  $
    y \in \tlim_{k \rarr \pinfty}\limits y_k 
    \implies {  }
    [\minfty, y] \subseteq 
    \tlim_{k\rarr \pinfty}y_k
  $                                         
                                            & Making a function value $f(x_0)$ smaller preserves lower semicontinuity:\linebreak                     
					      $\widetilde f(x) 
						\defeq \begin{cases}
							  f(x) &  \text { if } x \neq x_0 \\
							  c    &  \text { if } x = x_0
							\end{cases}
					      $
					      yields still continuous mapping 
					      $\widetilde f: (X, \mathcal{O}) \rarr ([\minfty, \pinfty], \mathcal{T})$ 
					      for $c \in [\minfty, f(x_0)]$
                                                                                                                                           & Thm. \ref{thm:limits_in_intervalus_maximus_with_right_order_topology} 
                                                                                                                                             \&  Ex. \ref{exa:different_continuations_are_lsc}                        \\
  \hline
  $
    \minfty \in \tlim_{k\rarr \pinfty}
    \limits y_k
  $    
  for all sequences $(y_k)_{k \in \N}$ in 
  $[\minfty, \pinfty]$.                   
                                            & $f(x_0) = \minfty \implies f$ cont. in $x_0$                                                 & Thm. \ref{thm:limits_in_intervalus_maximus_with_right_order_topology} 
																	     \&  Lem. \ref{lem:continuity_in_points_which_are_mapped_to_minfty}       \\
  \hline
  $K' \subseteq [\minfty, \pinfty]$ 
  is compact $\iseq \inf K' \in K'$       
                                            & $f$ takes a minimum on every compact set $K \subseteq X$                                     & Thm. \ref{thm:compactSubspaces} 
                                                                                                                                             \&  Thm. \ref{thm:lsc_function_takes_minimum}                            \\
\end{tabularx}
\\
\end{center}
\subsection{Coercivity properties versus continuity properties}
  \label{subsec:coercivity_vs_continuity}

In this subsection we define the notion of coercivity for 
functions $f: \R^n \rarr [\minfty,\pinfty]$ and 
see that $f$ is coercive and lower semicontinuous 
iff extending $f$ to the one point compactification of $\R^n$
by setting $\widehat f(\infty) \defeq \pinfty$ yields 
a continuous mapping $\widehat f: \R^n_\infty \rarr [\minfty, \pinfty]$,
see \prettyref{thm:characterize_lsc_and_coercive}.
This equivalence is the key for a 
-- as far as the author knows -- new technique
for proving coercivity plus lower semicontinuity.
See \prettyref{sec:compact_continuations} and 
\prettyref{sec:application_to_an_example} for more details.

We also define the notion of normcoercivity for mappings 
$f: \R^n \rarr \R^m$ and will see that this property 
is again equivalent to a continuity property of some
continuation of $f$ to the one point compactification of $\R^n$,
see \prettyref{thm:characterize_normcoercivity}.

We start with giving the definitions.
\begin{definition} \label{def:coercivity_for_mappings_from_Rn_into_Rdouble_extended}
  A function $f: \R^n \rarr [\minfty,\pinfty]$ is called
  {\bf coercive}, iff 
  \begin{equation*}
    f(x) \rarr \pinfty \text{  for  } \vectornorm{x} \rarr \pinfty
  \end{equation*}
\end{definition}
A related coercivity notion is given in 
\cite[Definition 1.12]{Maso1993}, cf. also
\cite[Example 1.14]{Maso1993}.
For the next definition 
cf. \cite[p. 134]{FacchineiANDPang2003}.
\begin{definition} \label{def:normcoercivity_for_mappings_between_Rn_and_Rm}
  A mapping $f: \R^n \rarr \R^m$ is called
  {\bf normcoercive}, iff
  \begin{equation*}
    \vectornorm{f(x)} \rarr \pinfty \text{ for } \vectornorm{x} \rarr \pinfty
  \end{equation*}
\end{definition}

For a mapping $f:\R^n \rarr \R$ we can speak both of coercivity and 
normcoercivity. Clearly coercivity implies normcoercivity.
The contrary holds not true as the following example shows:
\begin{example}
  The function $f: \R \rarr \R$ given by 
  $f(x) \defeq x$
  is clearly normcoercive. Considering the sequence of the numbers
  $x_k \defeq -k$ for $x \in \N$ we have
  $|x_k| \rarr \pinfty$ as $k \rarr \pinfty$ but 
  $f(x_k) = -k \rarr \minfty \neq \pinfty$ as $k \rarr \pinfty$
  so that $f$ is not coercive.
\end{example}

The following theorems show that coercivity properties of 
functions correspond to continuity properties of special 
continuations of them -- anticipating 
a name from
\prettyref{sec:compact_continuations} -- more precisely of special 
compact continuations of them.
The order topology for the interval $[\minfty, \pinfty]$
is denoted by $\mathcal{O}_\leq$, cf. 
\prettyref{def:order_topology}.

\begin{theorem} \label{thm:characterize_lsc_and_coercive}
  A mapping $f:\R^n \rarr [\minfty, \pinfty]$ and its continuation

  $\widehat f: \R^n_\infty \rarr [\minfty, \pinfty]$, given by
  $
    \widehat f(x) \defeq
    \begin{cases}
      f(x),    & \text{if } x  \in \R^n\\
      \pinfty, & \text{if } x =\infty
    \end{cases}
  $
  ,
  are connected by the following relations:
  \begin{enumerate}
    \item \label{enu:characterize_coercive}
      \begin{align*}
              & f: \R^n \rarr [\minfty, \pinfty] \text{ is coercive} 
	\\
	\iseq & f:
		(\R^n, \mathcal{O}^{\otimes n}) 
		\rarr 
		([\minfty, \pinfty], \mathcal{T})
		\text{ is topological coercive towards } \{\pinfty\}
	\\
        \iseq & \widehat f: 
		(\R^n_\infty, \mathcal{O}^{\otimes n}_\infty) 
		\rarr 
		([\minfty, \pinfty], \mathcal{T})
		\text{ is continuous at the point } \infty \in \R^n_\infty
	\\
        \iseq & \widehat f: 
		(\R^n_\infty, \mathcal{O}^{\otimes n}_\infty) 
		\rarr 
		([\minfty, \pinfty], \mathcal{O}_\leq)
		\text{ is continuous at the point } \infty.
      \end{align*}
    \item \label{enu:characterize_lsc_and_coercive}
      \begin{align*}
	      & f:\R^n \rarr [\minfty, \pinfty]
		\text{ is lower semicontinous and coercive} 
	\\
        \iseq & \widehat f: 
		(\R^n_\infty, \mathcal{O}^{\otimes n}_\infty) 
		\rarr 
		([\minfty, \pinfty], \mathcal{T})
		\text{ is continuous}.
      \end{align*}
  \end{enumerate}
\end{theorem}

Before proving this theorem we give an example 
to illustrate part \upref{enu:characterize_lsc_and_coercive}.
$\mathcal{O}$ denotes again the natural topology on $\R$.

\begin{example}
  The function $f: \R \rarr [\minfty, \pinfty]$, given by 
  \begin{gather*}
    f(x) \defeq x
  \end{gather*}
  is lower semicontinous but not coercive.
  In accordance to part \upref{enu:characterize_lsc_and_coercive}
  of \prettyref{thm:characterize_lsc_and_coercive}
  its continuation 
  $\widehat f: (\R_\infty, \mathcal{O}_\infty) \rarr ([\minfty, \pinfty], \mathcal{T})$,
  given by
  \begin{gather*}
    \widehat f(x) \defeq
    \begin{cases}
      f(x) = x   & \text{ if } x \in \R
    \\
      \pinfty 	 & \text{ if } x = \infty,
    \end{cases}    
  \end{gather*}
  is not continuous;
  more precise $\widehat f$ is not continuous in the newly added point 
  $\infty$ since 
  there is no compact subset $K$ of $\R$ such that a
  $\widehat f [ \R_\infty \setminus K ]$
  is contained in the neighborhood 
  $(3, \pinfty] \eqdef U$ of $\pinfty$ for the following reason:
  Any compact subst $K$ of $\R$ is bounded and hence contained in 
  some interval $[-N, N]$ with some $N \in \N$.
  Hence the image
  $
    \widehat f \big [\R_\infty \setminus K \big]
    \supseteq
    \widehat f \big[ \R_\infty \setminus [-N, N] \big]
    =
    (\minfty, -N) \cup (N, \pinfty]
    \supseteq 
    (\minfty, -N)
  $
  is not completely contained in $U = (3, \pinfty]$. 
\end{example}

\begin{proof}[Proof of \prettyref{thm:characterize_lsc_and_coercive}]
  ~
  \\
  \upref{enu:characterize_coercive}
  We have
  \begin{align*}
   & f \text{ is coercive }  
   \\
   \iseq {} &
	      f(x) \rarr \pinfty \text{ for } \vectornorm{x} \rarr \pinfty 
   \\
   \iseq {} &
	      \forall \alpha \in \R \;\; \exists R > 0 \;\; \forall x\in \R^n: 
	      \vectornorm{x} > R \Rarr  f(x) > \alpha
   \\
   \iseq {} &
	      \forall \alpha \in \R \;\; \exists R > 0 :
	      f[\R^n \setminus \closedball[R][\zerovec]] \subseteq (\alpha, \pinfty]
   \\
   \overset{(\ast)}{\iseq} {} &   
	      \forall \alpha \in \R \;\; \exists K \in \mathcal{KA}(\R^n) :
	      f[\R^n \setminus K] \subseteq (\alpha, \pinfty]
   \\
   \iseq {} &
	      \forall U \in \mathcal{U}'(+ \infty) \cap \mathcal{T} 
	      \;\; \exists K \in \mathcal{KA}(\R^n) :
	      f[\R^n \setminus K] \subseteq U'
   \\
   \overset{(\diamond)}{\iseq} {} &
	      \forall K' \in \mathcal{KA}_{\{\pinfty\}}([\minfty, \pinfty], \mathcal{T}) 
	      \;\; \exists K \in \mathcal{KA}(\R^n) :
	      f[\R^n \setminus K] \subseteq [\minfty, \pinfty] \setminus K'
   \\
   \iseq {} &
	      f: (\R^n, \mathcal{O}^{\otimes n}) 
		\rarr ([\minfty, \pinfty], \mathcal{T})
	      \text{ is topological coercive towards } \{\pinfty\}.
  \end{align*}
  Explanations for the equivalences in 
  $(\ast)$ and $(\diamond)$ are given in Detail
  \ref{det:explanation_equivalences_star_and_diamond_in_thm_characterize_lsc_and_coercive}
  in the Appendix.
  So we have proved the first of the claimed three equivalences.
  The second of the claimed equivalences is just a special case of 
  \prettyref{thm:topological_coercivity_towards_maximal_element_equivalent_to_continuous_property}.
  Finally the third of the claimed equivalences holds true since the 
  system $\{T \in \mathcal{T}: \pinfty \in T\}$ of open 
  $\mathcal{T}$--neighborhoods 
  of $\pinfty$ is both a $\mathcal{T}$--neighborhood basis of $\pinfty$ 
  and an $\mathcal{O}_\leq$--neighborhood basis for $\pinfty$;
  a detailed proof of the third equivalence can be found in 
  Detail \ref{det:continuity_notions_at_pinfty_are_equivalent}
  in the Appendix.
  \\

  \upref{enu:characterize_lsc_and_coercive}
  With \prettyref{thm:characterize_lower_semi_continuous} 
  and\DissVersionForMeOrHareBrainedOfficialVersion{ the just proven}{}
  part 
  \upref{enu:characterize_coercive} we get
  \begin{align*}
   \hphantom{\iseq} {}&
	      f \text{ is lsc and coercive }  
   \\
   \iseq {}&
	      f: (\R^n, \mathcal{O}^{\otimes n}) 
	      \rarr 
	      ([\minfty, \pinfty],\mathcal{T})
	      \begin{minipage}{\textwidth}
      		$\text{ is continuous in every } x \in \R^n$
		\\
		\text{ and coercive. }    
	      \end{minipage}
   \\
   \overset{\R^n \in \mathcal{O}_\infty^{\otimes n}}{\iseq } {}&
	      \widehat f: (\R^n_\infty, \mathcal{O}_\infty^{\otimes n}) 
	      \rarr 
	      ([\minfty, \pinfty],\mathcal{T})
	      \begin{minipage}{\textwidth}
	      $ \text{ is continuous in every } x \in \R^n$
	      \\
	      $\text{ and $f$ is coercive. }$
	      \end{minipage}
   \\
   \iseq {}&
	      \widehat f: (\R^n_\infty, \mathcal{O}_\infty^{\otimes n}) 
	      \rarr 
	      ([\minfty, \pinfty],\mathcal{T})
	      \begin{minipage}{\textwidth}
	      $ \text{ is continuous in every } x \in \R^n$
	      \\
	      $\text{ and in } x = \infty.$
	      \end{minipage}
   \\
   \iseq {}&
	      \widehat f : (\R^n_\infty, \mathcal{O}_\infty^{\otimes n} ) 
		  \rarr
		  ([\minfty, \pinfty],\mathcal{T})
		  \text{ is continuous. }
  \end{align*}
\end{proof}

Similarly we have the following theorem.
\begin{theorem}	\label{thm:characterize_normcoercivity}
  For a mapping $f: \R^n \rarr \R^m$ and its continuation
  $\widehat f: \R^n_{\infty} \rarr \R^m_{\infty}$, where 
  $\widehat f(\infty) \defeq \infty$, the following are 
  equivalent:
  \begin{enumerate}
    \item 
      $f: \R^n \rarr \R^m$ is normcoercive.
    \item
      $f:
	(\R^n, \mathcal{O}^{\otimes n})
	\rarr 
	(\R^m, \mathcal{O}^{\otimes m})
      $
      is topological coercive.
    \item 
      $
	\widehat f:(\R^n_\infty,\mathcal{O}_\infty^{\otimes n}) 
	\rarr 
	(\R^m_\infty,\mathcal{O}_\infty^{\otimes m})
      $
      is continuous in $\infty \in \R^n_\infty$.
  \end{enumerate} 
\end{theorem}

\begin{proof}
  Similar to the proof of part \upref{enu:characterize_coercive}
  in \prettyref{thm:characterize_lsc_and_coercive} we obtain
\MayChangesPartiallyPerformedVersionOrHareBrainedOfficialVersion{%
  \begin{align*}
   &\hphantom{~\iseq} 
	      f \text{ is normcoercive }  
   \\
   & \iseq 
	      \vectornorm{f(x)} \rarr \pinfty 
	      \text{ for } \vectornorm{x} \rarr \pinfty 
   \\
   & \iseq  
	      \forall r \in \R \;\; \exists R > 0 \;\; \forall x\in \R^n: 
	      \vectornorm{x} > R \Rarr  \vectornorm{f(x)} > r
   \\
   & \iseq  
	      \forall r \in \R \;\; \exists R > 0 :
	      f[\R^n \setminus \closedball[R][\zerovec]] 
	      \subseteq 
	      \R^m \setminus \closedball[r][\zerovec]
   \\
   &\overset{(\ast)}{\iseq}   
	      \forall r \in \R \;\; \exists K \in \mathcal{K}(\R^n) :
	      f[\R^n \setminus K] 
	      \subseteq 
	      \R^m \setminus \closedball[r][\zerovec]
   \\
   & \iseq  
	      \forall K' \in \mathcal{K}(\R^m) \;\; \exists K \in \mathcal{K}(\R^n) :
	      f[\R^n \setminus K] 
	      \subseteq 
	      \R^m \setminus K'
  \\
   & \iseq  
	      \forall K' \in \mathcal{KA}(\R^m) \;\; \exists K \in \mathcal{KA}(\R^n) :
	      f[\R^n \setminus K] 
	      \subseteq 
	      \R^m \setminus K'
    \\
   &\iseq    f:
	      (\R^n, \mathcal{O}^{\otimes n})
	      \rarr 
	      (\R^m, \mathcal{O}^{\otimes m})
	      \text{ is topological coercive}.
  \end{align*}  \may{Added one more equivalence (the second last statement is new)}
}
{
  \begin{align*}
   &\hphantom{~\iseq} 
	      f \text{ is normcoercive }  
   \\
   & \iseq 
	      \vectornorm{f(x)} \rarr \pinfty 
	      \text{ for } \vectornorm{x} \rarr \pinfty 
   \\
   & \iseq  
	      \forall r \in \R \;\; \exists R > 0 \;\; \forall x\in \R^n: 
	      \vectornorm{x} > R \Rarr  \vectornorm{f(x)} > r
   \\
   & \iseq  
	      \forall r \in \R \;\; \exists R > 0 :
	      f[\R^n \setminus \closedball[R][\zerovec]] 
	      \subseteq 
	      \R^m \setminus \closedball[r][\zerovec]
   \\
   &\overset{(\ast)}{\iseq}   
	      \forall r \in \R \;\; \exists K \in \mathcal{K}(\R^n) :
	      f[\R^n \setminus K] 
	      \subseteq 
	      \R^m \setminus \closedball[r][\zerovec]
   \\
   & \iseq  
	      \forall K' \in \mathcal{K}(\R^m) \;\; \exists K \in \mathcal{K}(\R^n) :
	      f[\R^n \setminus K] 
	      \subseteq 
	      \R^m \setminus K'
    \\
   &\iseq    f:
	      (\R^n, \mathcal{O}^{\otimes n})
	      \rarr 
	      (\R^m, \mathcal{O}^{\otimes m})
	      \text{ is topological coercive}.
  \end{align*}  
}%
  For the equivalence $(\ast)$ cf. 
  \prettyref{det:explanation_equivalences_star_and_diamond_in_thm_characterize_lsc_and_coercive}.
  So the equivalence of the first two statements from 
  \prettyref{thm:characterize_normcoercivity}
  is
  proved. The equivalence of the second and the third statement
  is just a special case of 
  \prettyref{thm:topological_coercivivity_is_eq_to_continuity_of_extension}.
\end{proof}

\newcommand{\CaptionForContentsContinuousArithmeticOp}{Continuous arithmetic operations in $([-\infty, +\infty], \mathcal{T})$}
\subsection[\CaptionForContentsContinuousArithmeticOp]
  {Continuous arithmetic operations in $\boldsymbol{([-\infty, +\infty], \mathcal{T})}$} 
\label{subsec:continuous_arithmetic_operations_in_intervalus_maximus_with_right_order_topology}

In this subsection we consider addition and multiplication on 
$[\minfty, \pinfty]$. In \prettyref{thm:addition_in_intervalus_maximus}
we show that there is a continuous addition
$
  +:([\minfty,\pinfty]^2, \mathcal{T}^{\otimes 2})
  \rarr
  ([\minfty,\pinfty], \mathcal{T})
$ 
on $([\minfty,\pinfty], \mathcal{T})$.
This is remarkable, since there is no continuous addition on the 
topological space
$([\minfty,\pinfty], \mathcal{O}_\leq)$,
no matter which value from $[\minfty,\pinfty]$ we choose for the critical
$\minfty + (\pinfty)$.
Regarding multiplication, however, things are more complicated.
For instance we will see in 
\prettyref{thm:multiplication_in_intervalus_maximus} that
multiplication with $\lambda \in (0,\pinfty)$ is continuous,
whereas multiplication with $\lambda \in (\minfty, 0)$ is not continuous 
-- but we should rather be happy about that: The just mentioned 
properties of the multiplication fit namely to the facts that 
multiplying a lower semicontinuous function with some 
$\lambda \in (0, \pinfty)$ gives again a lower semicontinuous function,
whereas multiplying with $\lambda \in (\minfty, 0)$ can result in 
a non lower semicontinuous function:
\begin{example}
  Consider the function $f: \R \rarr \R$ given by 
  \begin{gather*}
    f(x)
    \defeq
    \begin{cases}
      3  & \text{ for } x <0
    \\
      2  & \text{ for } x \geq 0
    \end{cases}.
  \end{gather*}
  Obviously $f$ is lower semicontinuous (but not upper semicontinuous).
  Multiplication of $f$ with $-1 \in (\minfty, 0)$ results 
  in the non lower semicontinuous function $-f$.
\end{example}

The next theorem shows that there is a continuous addition 
on $[\minfty, \pinfty]$.

\begin{theorem} \label{thm:addition_in_intervalus_maximus}
  Continuing the addition on $\R \cup \{\pinfty\}$, by setting 
  $ \pinfty + (\minfty) \defeq \minfty$ and
  $ \minfty + (\pinfty) \defeq \minfty,$ 
  we get a continuous function
  \begin{equation*}
    +:
    (
    [\minfty, \pinfty]\times[\minfty, \pinfty]
    ,
    \mathcal{T}       \otimes \mathcal{T}
    )
    \rarr
    ([\minfty, \pinfty],\mathcal{T}).
  \end{equation*}
  Setting $\pinfty + (\minfty)$ or $\minfty + (\pinfty)$ not to $\minfty$,
  but to any other value $c \in (\minfty, \pinfty ]$ would result 
  in a non-continuous mapping.
\end{theorem}

\begin{proof}
  We set $\minfty + (\pinfty)$ and 
  $\pinfty + (\minfty)$ to some values $c,d \in [\minfty,\pinfty]$, 
  respectively, and ask if the thereby extended addition
  $
    +:
    (
      [\minfty, \pinfty] \times [\minfty,\pinfty]
      ,
      \mathcal{T} \otimes \mathcal{T}
    )
    \rarr ([\minfty, \pinfty], \mathcal{T})
  $
  can be continuous at all in the points
  $(\minfty; \pinfty) \in [\minfty,\pinfty] \times [\minfty,\pinfty]$
  and
  $(\pinfty; \minfty) \in [\minfty, \pinfty] \times [\minfty, \pinfty]$,
  respectively.
  We deal first with the point 
  $(\minfty; \pinfty) \in [\minfty,\pinfty] \times [\minfty,\pinfty]$
  and consider the local mapping behavior of our extended addition near this
  point.
  To this end note that any neighborhood $U$ of that point contains a subset of
  the form $[\minfty,\pinfty] \times (\alpha , \pinfty]$, 
  where $\alpha \in [\minfty, \pinfty)$, and therefore is mapped to
  $+[U] = [\minfty, \pinfty]$ all the more. So we can achieve continuity in
  the point $(\minfty; \pinfty) \in [\minfty,\pinfty] \times [\minfty,\pinfty]$
  only by choosing a value $c$ whose only neighborhood is $[\minfty,\pinfty]$;
  clearly only $c=\minfty$ meets that demand.
  Analogously, setting $d=\minfty$ is the only chance to get an extended
  addition, which is continuous in the point
  $(\pinfty; \minfty) \in [\minfty,\pinfty] \times [\minfty,\pinfty]$.

  Now we prove that setting $\minfty + (\pinfty) \defeq \minfty$ and
  $\pinfty + (\minfty) \defeq \minfty$ really yields a continuous mapping
  \[
    +:
    (
    [\minfty, \pinfty]\times[\minfty, \pinfty]
    ,
    \mathcal{T}       \otimes \mathcal{T}
    )
    \rarr
    ([\minfty, \pinfty],\mathcal{T})
  \]
  To this end we show that all preimages
  \begin{align*}
    +^-[(c,\pinfty]] &= \{(a_1,a_2) \in [\minfty,\pinfty]^2 : a_1 + a_2 > c \}
    \\
		     &= \{(a_1,a_2) \in (\minfty,\pinfty]^2 : a_1 + a_2 > c \}
		      \eqdef A_c
  \end{align*}
  of the subbasis forming sets $(c,\pinfty], c \in [\minfty,\pinfty)$ are 
  again open sets. For this purpose we show that every $(a_1,a_2) \in A_c$ is an
  interior point of $A_c$, i.e. that there are neighborhoods
  $(\widecheck a_1,\pinfty]$ of $a_1$ and $(\widecheck a_2,\pinfty]$ of $a_2$
  with 
  \begin{equation*}
    \forall b_1 \in (\widecheck a_1,\pinfty], b_2 \in (\widecheck a_2,\pinfty]:
      b_1 + b_2 > c.
  \end{equation*}

  In the first case $a_1,a_2 \in \R$ we can choose 
  $\widecheck a_1 \defeq a_1 - \frac{1}{2}((a_1+a_2)-c) < a_1$ and
  $\widecheck a_2 \defeq a_2 - \frac{1}{2}((a_1+a_2)-c) < a_2$.
  In the second case $a_1 = a_2 = \pinfty$ the job is done by
  $\widecheck a_1 \defeq \frac{c}{2}$ and
  $\widecheck a_2 \defeq \frac{c}{2}$.
  In the third case $a_1 = \pinfty$ and $a_2 \in \R$
  we can choose any real $\widecheck a_2 < a_2$ and then set 
  $\widecheck a_1 \defeq c-\widecheck a_2 < \pinfty$.
  The remaining forth case $a_1 \in \R$ and $a_2 = \pinfty=a_1$ can be
  handled analogously by switching roles.
\end{proof}

Next we will consider multiplication. We start with the following lemma
which allows to transfer some of our results about 
addition to multiplication.
\begin{lemma} \label{lem:extention_of_exp}
  Extending the usual exponential function $x\mapsto e^x$ 
  via $e^\pinfty \defeq \pinfty$ and $e^\minfty \defeq 0$
  gives a homeomorphism
  $
    ([\minfty, \pinfty],\mathcal{T})
    \rarr
    ([0, \pinfty], [0,\pinfty] \Cap \mathcal{T} )
  $.
  It translates the, by means of
  $\pinfty+(\minfty) = \minfty+(\pinfty) = \minfty$, extended addition
  into the, by means of 
  $0 \cdot (\pinfty) = \pinfty \cdot 0 = 0$, extended multiplication;
  namely in virtue of
  \[
    \exp(x_1 + x_2) = \exp(x_1) \cdot \exp(x_2) 
  \]
  for all 
  $x_1, x_2 \in [\minfty,\pinfty]$.
\end{lemma}

\begin{proof}
  Since the extended exponential function is an order isomorphism
  between the totally ordered sets ${([\minfty,\pinfty], \leq)}$ and $
  ([0,\pinfty], \leq|_{[0,\pinfty]\times[0,\pinfty]})$
  we know, by \prettyref{thm:monotonicity_vs_continuity}, that
  \[
    \exp: 
    ([\minfty, \pinfty],\mathcal{T}) 
    \rarr 
    ([0, \pinfty], [0,\pinfty] \Cap \mathcal{T});
  \]
  is an homeomorphism; note here that the subspace topology 
  $[0,\pinfty] \Cap \mathcal{T}$ is the same as the 
  right order topology on $[0, \pinfty]$,
  generated by $\leq|_{[0,\pinfty]\times[0,\pinfty]}$.  
\end{proof}

\iftime{Mutiplikation vereinheitlichen; ausmisten; SONST: -- | eventuell inzwischen schon aufgeraumt}

The following theorem deals in its first block with multiplication on 
$[0, \pinfty]$
and with multiplication on $[\minfty, \pinfty]$. 
Since the results for the latter are not as 
satisfying as the results in 
\prettyref{thm:addition_in_intervalus_maximus}
we moreover deal in a second block 
with multiplication
\begin{equation*}
  m_\lambda:
  ([\minfty, \pinfty],\mathcal{T}) 
  \rarr 
  ([\minfty, \pinfty],\mathcal{T}), 
\quad 
  m_\lambda(x) \defeq \lambda x
\end{equation*}
\emph {by a factor} $\lambda \in [\minfty,\pinfty]$,
distinguishing the cases $\lambda \in (0, \pinfty)$, $\lambda = 0$,
$\lambda \in (\minfty, 0)$, $\lambda = \pinfty$ and $\lambda = \minfty$.
We will see that  
the continuity properties of $m_\lambda$
depend heavily on $\lambda$. For instance the following holds true for 
the mapping $m_\lambda:
  ([\minfty, \pinfty],\mathcal{T}) 
  \rarr 
  ([\minfty, \pinfty],\mathcal{T})$.
\begin{itemize}
  \item 
    For $\lambda \in (0, \pinfty)$ it is a homeomorphism
    and hence in particular continuous.
  \item 
    For $\lambda \in (\minfty, 0)$ it is discontinuous
    in every point of $[\minfty, \pinfty)$.
\end{itemize}
More precisely we have the following statements.

\begin{theorem} \label{thm:multiplication_in_intervalus_maximus}
  Considering multiplication as function of two variables the 
  following statements hold true:
  \begin{enumerate}
    \item \label{enu:multiplicationIsContinuousOnClosedQuarterplaneNE} 
      Continuing the multiplication of non-negative numbers, by setting
      the problematic cases
      $ 0 \cdot (\pinfty) \defeq 0$ and
      $ (\pinfty) \cdot 0 \defeq 0,$
      we get a continuous function
      \begin{equation*}
	\cdot:
	\big(
	[0, \pinfty]\times[0, \pinfty]
	,
	([0, \pinfty]\times[0, \pinfty]) \Cap (\mathcal{T} \otimes \mathcal{T})
	\big)
	\rarr
	\big([0, \pinfty], [0, \pinfty] \Cap \mathcal{T} \big)	
      \end{equation*}
      Setting $0 \cdot (\pinfty)$ or $(\pinfty) \cdot 0 $ not to $0$,
      but to any other value $d \in (0, \pinfty ]$, would result 
      in a non-continuous mapping.

    \item \label{enu:points_where_multiplication_is_continuous}
      Continuing the multiplication on $\R$, by setting
      each of the problematic cases
      $ 0 \cdot (\pinfty), (\pinfty) \cdot 0$
      and
      $ 0 \cdot (\minfty), (\minfty) \cdot 0$
      to any four values from $[\minfty,\pinfty]$,
      we get a function which is continuous in a point 
      $x \in [\minfty,\pinfty] \times [\minfty,\pinfty]$, iff
      \begin{align*}
	x 
	\in 
	\{
	  x \in [\minfty,\pinfty] \times [\minfty,\pinfty] :
	  x_1 > 0 \text{ and } x_2 > 0
	\}
	\\
	\cup
	\{
	  x \in [\minfty,\pinfty] \times [\minfty,\pinfty] :
	  x_1 \cdot x_2 = \minfty
	\}.
      \end{align*}
\enumeratext{For multiplication by a constant factor the following 
statements hold true:}
%

    \item \label{enu:multiplication_with_positive_factor_is_continuous}
      The multiplication $m_\lambda : x \mapsto \lambda x$ by a factor
      $\lambda \in (0, \pinfty)$ is a homeomorphism
      \begin{equation*}
	m_\lambda:
	([\minfty, \pinfty],\mathcal{T}) 
	\rarr 
	([\minfty, \pinfty],\mathcal{T})
      \end{equation*}
      and thus in particular continuous.

    \item \label{enu:multiplication_with_0_is_continuous}
      If we agree $0 \cdot x = x \cdot 0 = 0$ also for  
      $x = \minfty$ and $x = \pinfty$ then the multiplication by $0$
      is also a continuous mapping
      \begin{equation*}
	m_0:
	([\minfty, \pinfty],\mathcal{T}) 
	\rarr 
	([\minfty, \pinfty],\mathcal{T}).
      \end{equation*}
\rem{	Hier kann $0 \cdot \minfty$ auch auf irgendeinen Wert
      $c \in [\minfty,0]$ gesetzt werden, ohne dasz man die Stetigkeit verliert.
      Sinnvoll?
    }
%
    \item \label{enu:multiplication_with_negative_factor_is_not_continuous}
      The multiplication $m_\lambda : x \mapsto \lambda x$ with 
      $\lambda \in (\minfty, 0)$ is a mapping 
      \begin{equation*}
	m_\lambda: 
	([\minfty, \pinfty],\mathcal{T}) 
	\rarr
	([\minfty, \pinfty],\mathcal{T}),
      \end{equation*}

      which is discontinuous in each point $[\minfty,\pinfty)$; the point
      $\pinfty$ is the only one where this mapping is continuous.
    \item \label{enu:multiplication_with_pinfty}
      Extend the multiplication with $\pinfty$ by setting the problematic
      $(\pinfty)\cdot 0$ to some value $c\in [\minfty,\pinfty]$. 
      This
      extended multiplication
	\begin{align*}
	  m_{\pinfty}: 
	  ([\minfty, \pinfty],\mathcal{T})
	  &\rarr 
	  ([\minfty, \pinfty],\mathcal{T})
	  \\
	  x
	  &\mapsto
	  (\pinfty) \cdot x \defeq 
	  \begin{cases}
	    \pinfty & \text{ for } x > 0
	    \\
	    c	      & \text{ for } x = 0
	    \\
	    \minfty & \text{ for } x < 0
	  \end{cases}
	\end{align*}
      with the factor $\pinfty$ is then continuous in all $x > 0$ and in all
      $x < 0$. In the point $0$ it is continuous, iff we have set $c=\minfty$.
    \item \label{enu:multiplication_with_minfty}
      Extend the multiplication with $\minfty$ by setting the problematic
      $(\minfty) \cdot 0$ to some value $c \in [\minfty,\pinfty]$. The,
      in this way, extended multiplication
      \begin{align*}
	m_{\minfty}:
	([\minfty, \pinfty],\mathcal{T})
	&\rarr
	([\minfty, \pinfty],\mathcal{T})
	\\
	x 
	&\mapsto 
	(\minfty) \cdot x \defeq
	\begin{cases}
	  \minfty & \text { for } x > 0
	  \\
	  c	    & \text { for } x = 0
	  \\
	  \pinfty & \text { for } x < 0
	\end{cases}
      \end{align*}
      with the factor $\minfty$ is then continuous in all $x>0$, discontinuous
      in all $x < 0$.  In $0$ it is continuous, iff
      $c = \minfty$.
  \end{enumerate}   
\end{theorem}

\begin{proof}
  \upref{enu:multiplicationIsContinuousOnClosedQuarterplaneNE}
  With the help of the homeomorphism $\exp$ from 
  \prettyref{lem:extention_of_exp} and its higher dimensional 
  relative 
  \begin{align*}
    ([\minfty, \pinfty] \times [\minfty,\pinfty], \mathcal{T} \otimes \mathcal{T}) 
	& \rarr
    ([0, \pinfty] \times [0,\pinfty],  ([0, \pinfty] \times [0,\pinfty]) \Cap (\mathcal{T} \otimes \mathcal{T}) )
    \\
    (x_1,x_2) & \mapsto (\exp(x_1), \exp(x_2))
  \end{align*}
  we can translate our knowledge from 
  \prettyref{thm:addition_in_intervalus_maximus} 
  about the addition to the current
  \upref{enu:multiplicationIsContinuousOnClosedQuarterplaneNE},
  since those homeomorphisms yield a bijection $\beta$ between
  \[
    \mathcal{C}
    \left(
      ([\minfty, \pinfty],\mathcal{T})^{\otimes 2},([\minfty, \pinfty],\mathcal{T})
    \right)
  \]
  and
  \[
    \mathcal{C}
    \left(
      ([0, \pinfty],\mathcal{T})^{\otimes 2},([0, \pinfty],\mathcal{T})
    \right),
  \]
  namely via $\beta(f) \defeq g_f$, where 
  $
    g_f(y_1,y_2) 
    \defeq 
    \exp \left( 
		f \left(
			\exp^{-1}(y_1), \exp^{-1}(y_2)
		  \right) 
	 \right).
  $
  Choosing for $f$ the extended addition from 
  \prettyref{thm:addition_in_intervalus_maximus}
  we see that the continuous mapping 
  $\beta(+)$ is just our, by means of $0 \cdot \pinfty \defeq 0$ and
  $\pinfty \cdot 0 \defeq 0$, extended multiplication.
  It remains to show the uniqueness of that extension; assume that there is
  another continuous extension $\cdot'$ of the multiplication with, say,
  $d = 0 \cdot' (\pinfty) \in (0, \pinfty]$. Then $\beta^{-1}(\cdot ') \eqdef +'$
  would be a continuous extension of the addition with
  \[
    \minfty +' (\pinfty) 
    = 
    \exp^{-1}(0 \cdot' (\pinfty) )
    = 
    \exp^{-1}(d) \neq \minfty.
  \]
  But such a continuous extension of the addition does not exist 
  by \prettyref{thm:addition_in_intervalus_maximus}.

  \upref{enu:points_where_multiplication_is_continuous}
  We show that the (extended) multiplication is continuous in point $x$ with
  $x_1, x_2 > 0$. To this end let 
  $(\beta, \pinfty], \beta \in [\minfty, x_1 \cdot x_2)$ be some neighborhood of
  $x_1 \cdot x_2 > 0$. 
\MayChangesPartiallyPerformedVersionOrHareBrainedOfficialVersion{%
  [The case of the largest possible neighborhood $[\minfty, \pinfty]$ can be skipped since 
   $[\minfty, \pinfty] \supseteq (\beta, \pinfty]$.] \may{The sentence in brackets was added}
}
{}%
  Choose any $\widecheck x_1, \widecheck x_2 > 0 $ 
  with $\widecheck x_1 < x_1$ and
  $\widecheck x_2 < x_2$, $\beta < \widecheck x_1 \cdot \widecheck x_2$. Then 
  $(\widecheck x_1, \pinfty] \times (\widecheck x_2, \pinfty]$ is a neighborhood of
  $x$ which is mapped by the (extended) multiplication into 
  $(\beta, \pinfty]$. This shows the continuity in $x$.

  It remains to show that the (extended) multiplication is continuous in a point 
  \begin{gather*}
    x \in 
    ([\minfty,\pinfty] \times [\minfty,\pinfty] ) 
    \setminus 
    \{(x_1,x_2): x_1 > 0 \text{ and } x_2 > 0\},
  \end{gather*}
  iff $x_1 \cdot x_2 = \minfty$.
\MayChangesPartiallyPerformedVersionOrHareBrainedOfficialVersion{%
  \may{The justification of 
       ``multiplication continuous in $x \in $... $\iseq x_1\cdot x_2 = \minfty$ ''
       was rewritten in oder to remove some incorrect arguments}
  Let first the extended multiplication be continuous in a point
  $
    x \in 
    ([\minfty,\pinfty] \times [\minfty,\pinfty] ) 
    \setminus 
    \{(x_1,x_2): x_1 > 0 \text{ and } x_2 > 0\}
  $.
  Assume without loss of generality
  $x_1 \leq x_2$, so that $x_1 \leq 0$.
  Since any neighborhood $U$ of $x$ contains a subset of the form
  $(\widecheck x_1 , \pinfty] \times (\widecheck x_2, \pinfty]$
  with $\widecheck x_1 < 0$ and $\widecheck x_2 \leq x_2$ we see that 
  $\cdot [U] = [\minfty,\pinfty] $. Since the multiplication
  $\cdot$ is continuous in $x$ the latter equation means
  that $[\minfty,\pinfty]$ must be the only neighborhood of $x_1 \cdot x_2$;
  This is only the case if $x_1 \cdot x_2 = \minfty$.
  To the contrary \prettyref{lem:continuity_in_points_which_are_mapped_to_minfty} assures 
  that the multiplication is continuous in points 
  $x$ with $x_1 \cdot x_2 = \minfty$.
}
{Assume that 
  $
    x \in 
    ([\minfty,\pinfty] \times [\minfty,\pinfty] ) 
    \setminus 
    \{(x_1,x_2): x_1 > 0 \text{ and } x_2 > 0\}
  $.
  Due to the commutativity of the multiplication we may assume
  $x_1 \leq x_2$, without loss of generality, so that we have $x_1 \leq 0$.
  Since any neighborhood $U$ of $x$ contains a subset of the form
  $(\widecheck x_1 , \pinfty] \times (\widecheck x_2, \pinfty]$
  with $\widecheck x_1 < 0$ and $\widecheck x_2 \leq x_2$ we see that 
  $\cdot [U] = [\minfty,\pinfty] $. Since the multiplication
  $\cdot$ has to be continuous in $x$ the latter equation means
  that $[\minfty,\pinfty]$ must be the only neighborhood of $x_1 \cdot x_2$;
  This is only the case if $x_1 \cdot x_2 = \minfty$.
  Finally \prettyref{lem:continuity_in_points_which_are_mapped_to_minfty} assures 
  that the multiplication is continuous in points 
  $x$ with $x_1 \cdot x_2 = \minfty$.}%
  
  \upref{enu:multiplication_with_positive_factor_is_continuous} 
  The multiplication by a constant factor $\lambda \in (0, \pinfty)$ 
  is an order automorphism of $([\minfty,\pinfty], \leq)$.
  Because of part \upref{enu:homeomorphism_vs_order_isomorphism} of 
  \prettyref{thm:monotonicity_vs_continuity} it is therefore a homeomorphism 
  $([\minfty, \pinfty],\mathcal{T}) \rarr ([\minfty, \pinfty],\mathcal{T})$.

  \upref{enu:multiplication_with_0_is_continuous} A constant mapping between
  topological spaces is continuous.

  \upref{enu:multiplication_with_negative_factor_is_not_continuous}
  The continuity in $\pinfty$ is assured by 
  \prettyref{lem:continuity_in_points_which_are_mapped_to_minfty}.
  Let now $x \in [\minfty,\pinfty)$ and choose any $x_2 > x$. Since
  $\lambda x_2 < \lambda x_1$ we get the discontinuity of $m_\lambda$ in
  $x$ by part 
  \upref{enu:continuity_in_point_implies_order_property}
  of
  \prettyref{thm:monotonicity_vs_continuity}.

  \upref{enu:multiplication_with_pinfty} The continuity of $m_{\pinfty}$ in $x<0$
  is ensured by \prettyref{lem:continuity_in_points_which_are_mapped_to_minfty}.
  $m_{\pinfty}$ is also continuous in a point $x > 0$: Let $(\beta, \pinfty]$ be
  any neighborhood of $x$. 
\MayChangesPartiallyPerformedVersionOrHareBrainedOfficialVersion{%
  [The case of the largest possible neighborhood $[\minfty, \pinfty]$ can be skipped again since 
   $[\minfty, \pinfty] \supseteq (\beta, \pinfty]$.] \may{The sentence in brackets was added}
}
{}%
  Then $U \defeq (0, \pinfty]$ is a neighborhood of $x$
  which is mapped by $m_{\pinfty}$ into $(\beta, \pinfty]$.
  Consider the remaining point $x=0$. If we had set $c= \minfty$ we have continuity in
  $0$ again by \prettyref{lem:continuity_in_points_which_are_mapped_to_minfty}.
  If we had set $c > \minfty$ we can choose any neighborhood $(\beta, \pinfty]$
  of $c$. Since every neighborhood $U$ of $0$ contains an element $u < 0$ we have
  $\minfty \in m_{\pinfty}[U]$, so that we get
  $m_{\pinfty}[U] \not \subseteq (\beta, \pinfty]$ for all neighborhoods $U$ of $0$;
  i.e. $m_{\pinfty}$ is not continuous in $0$.

  \upref{enu:multiplication_with_minfty} The continuity of $m_\minfty$ in points
  $x > 0 $ is again ensured by
  \prettyref{lem:continuity_in_points_which_are_mapped_to_minfty}. Yet 
  in every point $x < 0$ this mapping is not continuous by part 
  \upref{enu:continuity_in_point_implies_order_property} of
  \prettyref{thm:monotonicity_vs_continuity},
  since for any $x_2 > 0 > x$ we have $m_\minfty(x_2) < m_\minfty(x)$.
  Consider now the remaining point $x=0$. If we had set $c= \minfty$ 
  we have continuity in $0$ once more by 
  \prettyref{lem:continuity_in_points_which_are_mapped_to_minfty}. If
  $c > \minfty $ we can just argue as before in the case $x<0$ to see that 
  $m_\minfty$ is not continuous in $0$.
\end{proof}


\section{Compact continuations} \label{sec:compact_continuations}
\addtointroscoll{

In this subsection we will introduce and deal with 
the notion of compact continuation of functions 
$f:(V,\mathcal{O}) \rarr (V',\mathcal{O}')$
between topological spaces $(V,\mathcal{O})$ and $(V',\mathcal{O}')$. 
This notion is, as far as the author knows, new.

Due to \prettyref{thm:characterize_lsc_and_coercive}
the lower semicontinuity and coercivity of a mapping 
$h: \R^n \rarr [\minfty, \pinfty]$ can be proven by checking that 
a certain extension $\widehat h: \R^n_\infty \rarr [\minfty, \pinfty]$ 
of $h$ is continuous,
i.e., in other words, if 
the mapping $\widehat h$ is a compact continuation of $h$.
If $h = g \circ f$, as in \prettyref{sec:application_to_an_example},
then the question arises if 
$h$ has that compact continuation provided that both $f$ and $g$ 
have according compact continuations.
An answer to this question is given in 
\prettyref{thm:concat_compactly_continuable_mappings}.

We remark here that this technique goes beyond the technique of 
proving coercivity of a mapping $h: \R^n \rarr [\minfty, \pinfty]$
by writing it as composition $h = g \circ f$ of a normcoercive mapping 
$f: \R^n \rarr \R^m$ and a coercive mapping 
$g: \R^m \rarr [\minfty, \pinfty]$, since the latter technique 
works only for 
decompositions of $h$ where the intermediate space is
$\R^m$, whereas the first technique can -- at least in principle -- work also 
for decompositions into functions $f: \R^n \rarr Y$ and 
$g: Y \rarr [\minfty, \pinfty]$ where the intermediate space can 
any topological space $Y$, like  e.g.
the product space 
$([\minfty, \pinfty] \times [\minfty, \pinfty], \mathcal{T} \otimes \mathcal{T})$
in the decomposition $g = g_2 \circ g_1$ in 
\prettyref{lem:application_of_compact_continuations_to_get_coercivity_and_lsc_of_sum}.

However our topological technique
has two disadvantages:
It can not be used to prove coercivity of a non lower semicontinuous
function and more important: Even if we have 
a straightforward choice of continuing each of the concatenated 
function in $h = g \circ f$ to functions 
$\widehat f: \hat X \rarr \hat Y$
and $\hill g : \hill Y \rarr \hill Z$
it might still be the case that
working with our technique could be somewhat cumbersome 
in cases where $\hat Y \neq \hill Y$.

\DissVersionForMeOrHareBrainedOfficialVersion
{Another more convenient approach would be to find new coercivity concepts
which include the old ones but which are not limited to decompositions 
$h = g \circ f$ where the intermediate space has to be $\R^m$.
In this approach we would not try to continue $g$ and $f$
but rather the ingredients used for measuring coercivity;
for instance norms could be (compactly?) continued to 
what could be named hypernorms. Norms could be also be replaced 
by more general ``levelfunctions'', cf. the Appendix
with the suggestions for further work.
%
}
{}

}

We now define the notion of compact continuation.
As far as the author knows this notion is new.
\begin{definition}
  A continuous mapping $f:(V,\mathcal{O}) \rarr (V',\mathcal{O}')$ between 
  topological spaces $(V,\mathcal{O})$, $(V', \mathcal{O}')$ 
  is called \textbf{compactly continuable} if there is a continuation 
  \begin{equation*}
    \widehat f: (\widehat V, \widehat{\mathcal{O}} ) \rarr (\widehat V', \widehat{\mathcal{O}}')
  \end{equation*}
  which fulfills:
  \begin{enumerate}
    \item
      $(\widehat V, \widehat {\mathcal{O}})$ is a compact topological space which contains
      $(V, \mathcal{O})$ as subspace,
    \item
      $(\widehat V', \widehat {\mathcal{O}'})$ is a topological space that contains
      $(V', \mathcal{O}')$ as subspace,
    \item 
      $\widehat f$ is continuous and fulfills $\widehat f(v) = f(v)$ for all $v \in V.$
  \end{enumerate}
  Each such continuation $\widehat f$ will be called \textbf{compact continuation}
  of $f$.
  If $\widehat f$ fulfills in addition
  $\widehat f [\widehat V  \setminus V] \subseteq \widehat V' \setminus V'$
  we call $\widehat f$ a \textbf{home leaving compact continuation} of $f$.

\end{definition}

\begin{theorem} \label{thm:concat_compactly_continuable_mappings}
  Assume that the two continuous mappings 
  $f:(X,\mathcal{O}_X) \rarr (Y,\mathcal{O}_Y),$
  $g:(Y,\mathcal{O}_Y) \rarr (Z,\mathcal{O}_Z)$
  have compact continuations
  \[
  \begin{alignedat}{1}
    \widehat f : (\widehat X, \mathcal{O}_{\widehat X}) \rarr 
	      (&\widehat Y, \mathcal{O}_{\widehat Y}),
    \\
    \hill g:  (&\hill Y, \mathcal{O}_{\hill Y}) \rarr
	       (\hill Z, \mathcal{O}_{\hill Z}).
  \end{alignedat}
  \]
  Then $g \circ f \eqcolon h$ has a compact continuation 
  \begin{equation*}
    \widehat h: (\widehat X, \mathcal{O}_{\widehat{X}} ) \rarr
	      (\hill Z, \mathcal{O}_{\hill Z})
  \end{equation*}
  if\DissVersionForMeOrHareBrainedOfficialVersion{ at least}{}
  one of the following conditions is fulfilled:
  \begin{enumerate}
    \item \label{enu:continue_id_in_easy_direction}
      $\id_Y: (Y,\mathcal{O}_Y) \rarr (Y,\mathcal{O}_Y) $ 
      has a compact continuation\\
      $
	\widehat {\id_Y} :
	(\widehat Y , \mathcal{O}_{\widehat Y}) \rarr (\hill Y, \mathcal{O}_{\hill Y})
      $.
    \item \label{enu:continue_id_in_hard_direction_general_case}
      $\id_Y: (Y,\mathcal{O}_Y) \rarr (Y,\mathcal{O}_Y) $ 
      has a compact continuation\\
      $
	\hill {\id_Y} : 
	(\hill Y, \mathcal{O}_{\hill Y}) \rarr (\widehat Y, \mathcal{O}_{\widehat Y})
      $ which, firstly, glues 
      $(\hill Y, \mathcal{O}_{\hill Y})$ to $(\widehat Y, \mathcal{O}_{\widehat Y})$
      and, secondly, fulfills 
      $
	\hill {\id_Y}(y_1) = \hill {\id_Y}(y_2)
	\implies
	\hill g (y_1) = \hill g(y_2)
      $, 
      for all $y_1,y_2 \in \hill Y$.
    
    %
    %
    \item \label{enu:continue_id_in_hard_direction_special_case}
            $\id_Y: (Y,\mathcal{O}_Y) \rarr (Y,\mathcal{O}_Y) $ 
      has a surjective compact continuation\\
      $
	\hill {\id_Y} : 
	(\hill Y, \mathcal{O}_{\hill Y}) \rarr (\widehat Y, \mathcal{O}_{\widehat Y})
      $ 
      where, firstly,  $(\widehat Y, \mathcal{O}_{\widehat Y})$ is a 
      Hausdorff space and, secondly, the condition
      $
	\hill {\id_Y}(y_1) = \hill {\id_Y}(y_2)
	\implies
	\hill g (y_1) = \hill g(y_2)
      $ 
      holds true for all $y_1,y_2 \in \hill Y$.
  \end{enumerate}
  If, in addition to 
  \upref{enu:continue_id_in_easy_direction} or  
  \upref{enu:continue_id_in_hard_direction_general_case} /
  \upref{enu:continue_id_in_hard_direction_special_case},
  respectively, both $\widehat f$ and 
  ($\,\,\widehat {\id_Y}$ or $\hill {\id_Y}$, respectively)
  are home leaving compact continuations then $\widehat h$ can be chosen such that
  it fulfills
  \begin{equation}
    \widehat h[\widehat X \setminus X]
    \subseteq 
    \hill g [\hill Y \setminus Y].
  \end{equation}

  
\end{theorem}

Before proving the theorem we show by an example that $g \circ f$,
does not need to have a compact continuation
$\topospace{\widehat X} \rarr \topospace{\hill Z}$, if none of the three 
conditions from the above theorem is fulfilled.
$\mathcal{O}$ denotes again the natural topology of $\R$.

\begin{example}

  Consider three copies
  \begin{gather*}
    (X,\mathcal{O}_X) 
    = 
    (Y,\mathcal{O}_Y) 
    = 
    (Z, \mathcal{O}_Z) 
    = 
    \left((0, 2\pi), (0, 2\pi) \Cap \mathcal{O}\right)
  \end{gather*}
  of the real open interval $(0, 2\pi)$ along with the identity mappings
  $f=\id_{(0, 2\pi)}: (X,\mathcal{O}_X) \rarr (Y,\mathcal{O}_Y)$ and
  $g=\id_{(0, 2\pi)}: (Y,\mathcal{O}_Y) \rarr (Z, \mathcal{O}_Z)$
  between them.
  We will extend the equal mappings $f$ and $g$
  in different ways to compact continuations
  $
    \widehat f: 
    (\widehat X, \mathcal{O}_{\widehat X})
    \rarr
    (\widehat Y, \mathcal{O}_{\widehat Y})
  $
  and
  $
    \hill g:
    (\hill Y, \mathcal{O}_{\hill Y})
    \rarr
    (\hill Z, \mathcal{O}_{\hill Z})
  $
  such that 
  $h \defeq g \circ f$ 
  can not be extended to a compact continuation 
  $
    \widehat h: 
    (\widehat X, \mathcal{O}_{\widehat X})
    \rarr 
    (\hill Z, \mathcal{O}_{\hill Z})
  $.
  Let
  \begin{alignat*}{5}
    \widehat f &\defeq \id_{(0, 2\pi) \cup \{ \infty \}}:{} &&
    \big( (0, 2\pi), (0,2\pi) \Cap \mathcal{O} \big)_\infty
    && \rarr
    \big( (0, 2\pi), (0,2\pi) \Cap \mathcal{O} \big)_\infty \,,
  \\ 
    \hill g & \defeq \id_{[0, 2\pi]}:{}  &&
    \big([0,2\pi], [0, 2\pi] \Cap \mathcal{O} \big) 
    && \rarr 
    \big([0,2\pi], [0, 2\pi] \Cap \mathcal{O} \big)
  \end{alignat*}
  and set
  \begin{alignat*}{3}
    (\widehat X, \mathcal{O}_{\widehat X})
    &\defeq
    (\widehat Y, \mathcal{O}_{\widehat Y})
    &&\defeq 
    \big( (0, 2\pi), (0,2\pi) \Cap \mathcal{O} \big)_\infty \,,
  \\
    (\hill Y, \mathcal{O}_{\hill Y})
    &\defeq 
    (\hill Z, \mathcal{O}_{\hill Z})
    &&\defeq
    ([0,2\pi], [0, 2\pi] \Cap \mathcal{O}).   
  \end{alignat*}
  The functions $\widehat f$ and $\hill g$ are compact continuations of 
  $f$ and $g$, respectively, but it is not possible to
  extend $h \defeq g \circ f = \id_{(0, 2\pi)}$  to a compact continuation 
  $
    \widehat h: 
    (\widehat X, \mathcal{O}_{\widehat X})
    \rarr 
    (\hill Z, \mathcal{O}_{\hill Z})
  $;
  indeed, if such a continuous mapping $\widehat h$ existed, it would have 
  to map its compact domain of definition to a compact subspace of 
  $(\hill Z, \mathcal{O}_{\hill Z})
    =
    ([0, 2\pi], [0, 2\pi] \Cap \mathcal{O})
  $;
  however
  \[
    \widehat h [\widehat X]
    =
    \widehat h [(0, 2\pi) \cup \{\infty\}]
    =
    (0, 2\pi) \cup \{\widehat h(\infty)\}
  \]
  would never be a compact subset of 
  $([0, 2\pi], [0, 2\pi] \Cap \mathcal{O})$ 
  -- regardless whether $\widehat h(\infty) = 0$, $\widehat h(\infty) = 2\pi$ or
  $\widehat h(\infty) \in (0,2\pi)$.

  So we know by the last theorem that none of the conditions 
  \upref{enu:continue_id_in_easy_direction},
  \upref{enu:continue_id_in_hard_direction_general_case}
  and
  \upref{enu:continue_id_in_hard_direction_special_case}
  can be fulfilled. We nevertheless verify this directly, 
  to complete our illustration of the preceding theorem.

  \upref{enu:continue_id_in_easy_direction} is not fulfilled as we just have
  shown by proving the nonexistence of a compact continuation 
  $
    \widehat h: 
    (\widehat X, \mathcal{O}_{\widehat X})
    \rarr 
    (\hill Z, \mathcal{O}_{\hill Z})
  $,
  i.e. of a compact continuation
  $
    \widehat \id_{(0, 2\pi)} =\widehat h: 
    (\widehat Y, \mathcal{O}_{\widehat Y})
    \rarr 
    (\hill Y, \mathcal{O}_{\hill Y})
  $.

  Furthermore 
  \upref{enu:continue_id_in_hard_direction_general_case}
  and 
  \upref{enu:continue_id_in_hard_direction_special_case}
  are not fulfilled, since any continuation of
  $\id_{(0,2\pi)}: (0, 2\pi) \rarr (0, 2\pi)$ to a mapping
  $
    \hill {\id_{(0, 2\pi)}}: 
    (0,2\pi) \cup \{0, 2\pi\}
    \rarr 
    (0,2\pi) \cup \{\infty\}     
  $
  is not injective any longer, 
  so that there is no chance for the injective mapping $\hill g$
  to fulfill 
  $ 
    \hill g (y_1)
    =
    \hill g (y_2)
  $
  in the occurring case that
  $
    \hill {\id_{(0, 2\pi)}}(y_1) 
    = 
    \hill {\id_{(0, 2\pi)}}(y_2)
  $
  for distinct points $y_1, y_2 \in \hill Y$.
\end{example}

\begin{proof}[Proof of \prettyref{thm:concat_compactly_continuable_mappings}]
  If \upref{enu:continue_id_in_easy_direction} holds,
  it suffices to take
  $
    \widehat h 
    \defeq 
    \hill g
    \circ
    \widehat {\id_Y} 
    \circ
    \widehat f
  $.

  Assume now that condition 
  \upref{enu:continue_id_in_hard_direction_general_case} holds. 
%
  The mapping $g': \topospace{\widehat Y} \rarr \topospace{\hill Z}$, given by
  \[
    g'(\widehat y) 
    \defeq 
    ``\hill g (\hill{\id_Y}^-[\{\widehat y\}])``
    \defeq
    \hill g (\hill y), 
    \text{ where } 
    \hill y 
    \text{ is any element of } 
    \hill Y 
    \text{ with }
    \hill{\id_Y}(\hill y)
    =
    \widehat y
  \]
  is well defined since
  $
	\hill {\id_Y}(y_1) = \hill {\id_Y}(y_2)
  $
  ensures
  $	
	\hill g (y_1) = \hill g(y_2)
  $, 
  for all $y_1,y_2 \in \hill Y$.
  The definition of $g'$ was done in such a way that 
  $\hill g = g' \circ \hill{\id_Y}$. 
  \begin{gather*}
    \xymatrix{
      & \topospace{\widehat Y}  \ar[dr]|-{g'}   \\
      & \topospace{\hill Y}   \ar[u]^{\hill {\id_Y}}  \ar[r]|-{\hill g}  & \topospace{\hill Z}    
    }
  \end{gather*}
  This implies, firstly, the continuity
  of $g'$, in virtue of 
  \prettyref{thm:mappings_from_ori_space_vs_mappings_from_glued_space},
  and, secondly, $g'(y) = \hill g (y)$ at least for all $y\in Y$.
  Thus we have found a compact continuation
  $g': \topospace{\widehat Y} \rarr \topospace{\hill Z}$ of
  $g: \topospace{Y} \rarr \topospace{Z}$. The concatenation 
  $\widehat h \defeq g' \circ \widehat f$ is the needed extension.

\MayChangesPartiallyPerformedVersionOrHareBrainedOfficialVersion{%
  The relation 
  $ \widehat h[\widehat X \setminus X]
    \subseteq 
    \hill g [\hill Y \setminus Y]
  $
  follows easily from the constructions of $\widehat h$ if the involved 
  mappings are homeleaving compact continuations.\may{This sentence was added}
}%
{}%
  Finally note that the assumptions in
  \upref{enu:continue_id_in_hard_direction_special_case}
  imply the assumptions in 
  \upref{enu:continue_id_in_hard_direction_general_case},
  in virtue of
  \prettyref{thm:sufficient_conditions_for_indentifying}.
\end{proof}

Next we deal with a special case of 
\prettyref{thm:concat_compactly_continuable_mappings}, where the 
``intermediate'' spaces $\topospace{Y}$, $\topospace{\hill Y}$, 
$\topospace{\widehat Y}$ and the continuation $\hill g$ have 
special forms,
which will occur, when applying the theory to our example
in \prettyref{sec:application_to_an_example}.
We start with the following preparatory lemma.

\begin{lemma} \label{lem:torus_like_and_one_point_compactification}
  For locally compact Hausdorff spaces 
  $\decorspace{Y}{'}$ and $\decorspace{Y}{''}$
  the following is true:
  \begin{enumerate}
    \item \label{enu:one_point_compactification_of_product_and_product_of_one_point_compactification_are_superspaces}
      Both $[\decorspace{Y}{'} \varotimes \decorspace{Y}{''} ]_\infty $ and 
      $\compdecorspace{Y}{'} \varotimes \compdecorspace{Y}{''}$ are
      compact Hausdorff spaces which contain 
      $\decorspace{Y}{'} \varotimes \decorspace{Y}{''}$
      as subspace.
    \item \label{enu:extension_of_id_on_product_space_to_a_homeleaving_compact_continuation}
      An extension of 
      $
	\id : 
	\decorspace{Y}{'} \varotimes \decorspace{Y}{''}
	\rarr
	\decorspace{Y}{'} \varotimes \decorspace{Y}{''}
      $
      to a surjective, homeleaving compact continuation
      $
	\overline{\id}: 
	\compdecorspace{Y}{'} \varotimes \compdecorspace{Y}{''}
	\rarr
	[\decorspace{Y}{'} \varotimes \decorspace{Y}{''} ]_\infty
      $ 
      is given by
      \[
	\overline{\id}(y', y'') 
	\defeq
	\begin{cases}
	  (y', y'')  & \text{, if } y' \in Y' \text{ and } y'' \in Y'' \\
	  \infty     & \text{, if } y' = \infty' \text{ or } y'' = \infty''.
	\end{cases}
      \]
  \end{enumerate}
\end{lemma}

\begin{proof}
  \upref{enu:one_point_compactification_of_product_and_product_of_one_point_compactification_are_superspaces}
  \prettyref{thm:one_point_compactification} ensures that both
  $\compdecorspace{Y}{'}$ and $\compdecorspace{Y}{''}$
  are compact Hausdorff spaces, which contain $\decorspace{Y}{'}$
  and $\decorspace{Y}{''}$, respectively, as subspace. Therefore their
  product space $\compdecorspace{Y}{'} \varotimes \compdecorspace{Y}{''}$
  is compact -- in virtue of Tichonov's Theorem
  \ref{thm:little_tichonov} -- and 
  contains 
  $
    \decorspace{Y}{'} \varotimes \decorspace{Y}{''}
    =
    (Y', Y' \Cap \mathcal{O}'_{\infty'}) \varotimes 
    (Y'', Y'' \Cap \mathcal{O}''_{\infty''})
  $ 
  as subspace, since 
  \prettyref{rem:productspace_of_supspaces_is_a_subspace_of_productspace}
  allows the reformulation
  \begin{align*}
    \big(Y', Y' \Cap \mathcal{O}'_{\infty'} \big) \varotimes 
    \big(Y'', Y'' \Cap \mathcal{O}''_{\infty''} \big)
    =
    \big (Y' \times Y'', (Y' \times Y'') \Cap 
	    (\mathcal{O}'_{\infty'} \varotimes \mathcal{O}''_{\infty''}) 
    \big).
  \end{align*} 
  By \prettyref{det:product_of_locally_compact_Hausdorff_is_of_same_type}
  the product space $\decorspace{Y}{'} \varotimes \decorspace{Y}{''}$ of
  two locally compact Hausdorff spaces is again a locally compact Hausdorff
  space, so that its one-point compactification 
  $[\decorspace{Y}{'} \varotimes \decorspace{Y}{''}]_\infty$ is a compact
  Hausdorff superspace of $\decorspace{Y}{'} \varotimes \decorspace{Y}{''}$, 
  by \prettyref{thm:one_point_compactification}.
  This show the first part of the lemma.

  \upref{enu:extension_of_id_on_product_space_to_a_homeleaving_compact_continuation}
  As core part for proving that $\overline{\id}$ is a 
  surjective, homeleaving compact continuation of $\id$ we have to show
  that $\overline{\id}$ is continuous; it easy to see, by $\overline{\id}$'s
  definition, that it fulfills the remaining properties, we had to show.
  In order to prove the continuity of $\overline{\id}$ 
  we will use that the projections 
  $
    \pi' : 
    \decorspace{Y}{'} \varotimes \decorspace{Y}{''}
    \rarr
    \decorspace{Y}{'}
  $
  and
  $
    \pi'' : 
    \decorspace{Y}{'} \varotimes \decorspace{Y}{''}
    \rarr
    \decorspace{Y}{''}
  $
  to the first and second component, respectively, are continuous and
  therefore map compact subsets of 
  $\decorspace{Y}{'} \varotimes \decorspace{Y}{''}$ to compact subsets of
  $\decorspace{Y}{'}$ and $\decorspace{Y}{''}$, respectively.
  In every point $(y', y'')$ of
  the open subset 
  $
    Y' \times Y'' 
    \in 
    \mathcal{O}'_{\infty'} \varotimes \mathcal{O}''_{\infty''}
  $
  the mapping $\overline{\id}$ is clearly continuous.
  It remains to show that $\overline{\id}$ is continuous in all points
  of the form $(\infty', y'')$ or $(y', \infty'')$ where 
  $y' \in Y'_{\infty'}$ and 
  $y''\in Y''_{\infty''}$. In order to show the continuity in all 
  these preimage points of $\infty$ we consider any neighborhood
  \[
    V = (Y' \times Y'')_\infty \setminus K
  \]
  of $\infty$, with arbitrary 
  $
    K 
    \in 
    \mathcal{KA}(Y' \times Y'') 
    \xlongequal{Thm. \ref{thm:relation_closed_compact}}
    \mathcal{K}(Y' \times Y'')
  $
  and convince ourselves that the set
  $
    U 
    \defeq
    ( Y'_{\infty'} \times Y''_{\infty''} )
    \setminus
    ( \pi'[K] \times \pi''[K] )  
    =
      (Y'_{\infty'} \setminus \pi'[K] )
      \times
      Y''_{\infty''}
    \;
    \cup
    \;
      Y'_{\infty'}
      \times
      (Y''_{\infty''} \setminus \pi''[K] )
  $,
  firstly,
  fulfills 
  $
    \overline{\id}[U] 
    =
    (Y' \times Y'')_\infty \setminus ( \pi'[K] \times \pi''[K] )
    \subseteq
    V
  $
  and, secondly, is an open neighborhood of all our preimage points of 
  $\infty$.
  This shows the second part of the lemma.
\end{proof}

Using this Lemma we are now  going to prove the announced
special case of
\prettyref{thm:concat_compactly_continuable_mappings}:

\begin{theorem} \label{thm:concat_compactly_continuable_mappings_special_case}
  Let $\decorspace{Y}{'}$ and $\decorspace{Y}{''}$ be locally compact 
  Hausdorff spaces and let
  two continuous mappings 
  $
    f:(X,\mathcal{O}_X) 
      \rarr \decorspace{Y}{'} \varotimes \decorspace{Y}{''}
  $,
  $
    g:\decorspace{Y}{'} \varotimes \decorspace{Y}{''} 
      \rarr (Z,\mathcal{O}_Z)
  $
  have compact continuations
  \[
  \begin{alignedat}{1}
    \widehat f : (\widehat X, \mathcal{O}_{\widehat X}) \rarr 
	      [\decorspace{Y}{'} &\varotimes \decorspace{Y}{''} ]_\infty \,,
    \\
    \hill g:  \compdecorspace{Y}{'} &\varotimes \compdecorspace{Y}{''} 
	    \rarr (\hill Z, \mathcal{O}_{\hill Z}).
  \end{alignedat}
  \]
  Then $g \circ f \eqcolon h$ has a compact continuation 
  \begin{equation*}
    \widehat h: (\widehat X, \mathcal{O}_{\widehat{X}} ) \rarr
	      (\hill Z, \mathcal{O}_{\hill Z}),
  \end{equation*}
  if $\hill g$ fulfills
  \begin{equation} 
  \label{eq:special_case_extension_of_g_identifies_less_than_id}
    \hill g(\infty', y'')
    =
    \hill g(y', \infty'')
  \end{equation}
  for all $y' \in Y'_{\infty'}$ and $y'' \in Y''_{\infty''}$.
  If, in addition, 
  $\widehat f[\widehat X \setminus X] \subseteq \{\infty\}$ 
  then $\widehat h$ can be chosen such that it fulfills
  \begin{equation}
    \widehat h[\widehat X \setminus X]
    \subseteq 
      \hill{g}[\{\infty'\}\times Y'']
      \cup
      \hill{g}[Y' \times \{\infty''\}]
      \cup
      \hill{g}[\{(\infty', \infty'')\}].
  \end{equation}
\end{theorem}

\begin{proof}
  Setting 
  $
    (Y,\mathcal{O}_Y) 
    \defeq \decorspace{Y}{'} \varotimes \decorspace{Y}{''}
  $,
  $
    (\widehat Y, \mathcal{O}_{\widehat Y})
    \defeq [\decorspace{Y}{'} \varotimes \decorspace{Y}{''} ]_\infty
  $
  and
  $
    (\hill Y, \mathcal{O}_{\hill Y})
    \defeq \compdecorspace{Y}{'} \varotimes \compdecorspace{Y}{''} 
  $
  we get the theorem as special case of
  \prettyref{thm:concat_compactly_continuable_mappings}, since
  all assumptions of its condition
  \upref{enu:continue_id_in_hard_direction_special_case}
  and its additional condition hold true, in virtue of
  \prettyref{lem:torus_like_and_one_point_compactification} and the condition 
  \prettyref{eq:special_case_extension_of_g_identifies_less_than_id}.
\end{proof}

\section{Application of the theory to an example} \label{sec:application_to_an_example}

We agree 
$\minfty + (\pinfty) = \pinfty + (\minfty) = \minfty$ in the
following example. 
Although the 
assumptions in this 
example prevent the occurrence of the value $\minfty$ 
we nevertheless need the stated agreement
in order to obtain a continuous addition 
on $[\minfty,\pinfty]$, 
cf. \prettyref{thm:addition_in_intervalus_maximus}.

\begin{lemma} \label{lem:application_of_compact_continuations_to_get_coercivity_and_lsc_of_sum}
  Assume that the following mappings are given:
  \begin{enumerate}
    \item 
      Two matrices / linear mappings 
      $H:\R^n \rarr \R^d, K: \R^n \rarr \R^e$ with
      \begin{gather*}
        \mathcal{N}(H) \cap \mathcal{N}(K) = \{\zerovec\}.
      \end{gather*}
    \item 
      Two proper, lower semicontinuous and coercive mappings
      $
	\phi: \R^d \rarr [\minfty, \pinfty]
      $
      and
      $
	\psi: \R^e \rarr [\minfty, \pinfty]
      $.
  \end{enumerate}
  Then the mapping $h: \R^n \rarr [\minfty, \pinfty]$, given by
  \begin{equation}
    x \mapsto \phi(Hx) + \psi(Kx)
  \end{equation}
  is lower semicontinuous and coercive. In particular, the mapping $h$
  attains its infimum $\inf h \in [\minfty,\pinfty]$ at some point in $\R^n$.
\end{lemma}

\begin{proof}
  Due to part \upref{enu:characterize_lsc_and_coercive}
  in \prettyref{thm:characterize_lsc_and_coercive})
  our task of proving that $h$ is coercive and lower 
  semicontinuous can be done by showing that setting
  $\widehat h(\infty) \defeq \pinfty$ gives a continuous continuation
  $
    \widehat h: 
    (\R^n_\infty, \mathcal{O}^{\otimes n}_\infty) 
      \rarr 
      ([\minfty,\pinfty], \mathcal{T})
  $
  of $h$.
  We will do this in three steps:
  Firstly we write $h$ as composition $h = g \circ f$ of easier
  functions $g$ and $f$ and extend them to compact continuations
  $\hill{g}$ and $\widehat f$.
  Secondly, a compact continuation $\hill h$ of $h$ is obtained
  from $\hill{g}$ and $\widehat f$ by applying
  \prettyref{thm:concat_compactly_continuable_mappings_special_case}.
  Thirdly we convince us that $\widehat h = \hill{h}$.
\\[0.3ex]
  The mapping $h$ can be written as composition
  $h= \underbrace{g_2 \circ g_1}_{\eqcolon g} \circ f$
  of mappings
  \begin{alignat*}{3}
    f:\R^n \rarr \;  
	  & \R^d \times \R^e, 
		&& 
		      &&
    \\
    g_1:\;
	  & \R^d \times \R^e \rarr 
		&& \;[\minfty,\pinfty] \times [\minfty,\pinfty] 
		      &&,
    \\
	  & \adapt{\R^d \times \R^e \rarr }{g_2:} 
		&&\;[\minfty, \pinfty] \times [\minfty,\pinfty]
		      && \rarr [\minfty, \pinfty]
    \shortintertext{which are given by}
    f: x \mapsto  
	  & \begin{pmatrix}
	      Hx\\
	      Kx
	     \end{pmatrix}, 
		&& 
		      &&
    \\
    g_1:\;
	  & \begin{pmatrix}
	      y_1\\
	      y_2
	     \end{pmatrix}
	     \mapsto
		&& \;\begin{pmatrix}
		       \phi(y_1)\\
		       \psi(y_2)
		     \end{pmatrix}, 
		      &&    
    \\
	  & \adapt{\R^d \times \R^e \rarr }{g_2:} 
		&& \;\begin{pmatrix}
		       a_1\\
		       a_2
		     \end{pmatrix}
		     \mapsto a_1 + a_2.  &&
  \end{alignat*}
  After equipping the vector spaces 
  $\R^n \eqcolon X$ and $\R^d \times \R^e \eqcolon Y$ 
  with their natural topology, 
  the interval $[\minfty, \pinfty] \eqcolon Z$ 
  with the Topology $\mathcal{T}$, and 
  $[\minfty,\pinfty]\times[\minfty,\pinfty]$ with the corresponding 
  product topology $\mathcal{T}\otimes\mathcal{T}$ we have continuous
  mappings $f, g_1, g_2$ and $g = g_2 \circ g_1$.
  Due to $\mathcal{N}(K) \cap \mathcal{N}(L) = \{\zerovec\}$ the mapping $f$
  is normcoercive and hence the mapping
  \begin{alignat*}{1}
    \widehat f&: (\R^n_\infty, \mathcal{O}_\infty^{\otimes n}) 
	     \rarr 
	     \big( (\R^d \times \R^e)_\infty , (\mathcal{O}^{\otimes d} \otimes \mathcal{O}^{\otimes e})_\infty \big)
    \\
    \widehat f&(x) \defeq \begin{cases}
			f(x)   & \text{ if } x\in \R^n \\
			\infty & \text{ if } x = \infty
                      \end{cases}
  \end{alignat*}
  is a 
%
  compact continuation by 
  \prettyref{thm:characterize_normcoercivity};
  furthermore $\widehat f$ fulfills clearly 
  $\widehat f[\R^n_\infty \setminus \R^n] \subseteq \{\infty\}$ 
  by definition.
  Similar, by part \upref{enu:characterize_lsc_and_coercive} in 
  \prettyref{thm:characterize_lsc_and_coercive},
  we obtain compact continuations 
  $
    \hill \phi:
    (\R^d_\infty, \mathcal{O}^{\otimes d}_\infty)
    \rarr 
    ([\minfty,\pinfty], \mathcal{T})
  $ 
  and 
  $\hill \psi: 
    (\R^e_\infty, \mathcal{O}^{\otimes e}_\infty)
    \rarr 
    ([\minfty,\pinfty], \mathcal{T})
  $ 
  of $\phi$ and $\psi$
  by setting $\hill \phi (\infty) \coloneq \pinfty$ and
  $\hill \psi (\infty) \defeq \pinfty$, respectively. These two mappings 
  form a compact continuation 
  $
    \hill {g_1}:
      (\R^d_\infty \times \R^e_\infty, 
	\mathcal{O}^{\otimes d}_\infty  \otimes \mathcal{O}^{\otimes e}_\infty)
      \rarr 
      ([\minfty,\pinfty]^2, \mathcal{T}^{\otimes 2})
  $
  of $g_1$.
  Then
  \begin{alignat*}{1}
    \hill g& : 
      (\R^d_\infty \times \R^e_\infty,  
      \mathcal{O}^{\otimes d}_\infty  \otimes \mathcal{O}^{\otimes e}_\infty)
      \rarr 
      ([\minfty,\pinfty], \mathcal{T})\\
    \hill g& \defeq g_2 \circ \hill{g_1}
    = 
    \hill \phi + \hill \psi    
  \end{alignat*}
  is a compact continuation of $g$.
  In order to apply 
  \prettyref{thm:concat_compactly_continuable_mappings_special_case}
  we note that
  $(Y', \mathcal{O}') \defeq   (\R^d, \mathcal{O}^{\otimes d})$ and
  $(Y'', \mathcal{O}'') \defeq (\R^e, \mathcal{O}^{\otimes e})$ are
  surely locally compact Hausdorff spaces, and that the mappings 
  $
    \widehat f: (\R^n_\infty, \mathcal{O}_\infty^{\otimes n}) 
      \rarr 
      \big[ (Y', \mathcal{O}') \otimes (Y'', \mathcal{O}'') \big]_\infty
  $,
  $
    \hill g : 
      (Y', \mathcal{O}')_\infty \otimes (Y'', \mathcal{O}'')_\infty
      \rarr 
      ([\minfty,\pinfty], \mathcal{T})    
  $
  have the needed form, where $\hill{g}$ fulfills
  $
    \hill{g}(\infty, y'') 
    = 
    \hill{\phi}(\infty) + \hill{\psi}(y'')
    = 
    \pinfty 
    =
    \hill{\phi}(y') + \hill{\psi}(\infty)
    =
    \hill{g}(y',\infty)
  $
  for all $y' \in Y'_\infty$ and $y'' \in Y''_\infty$,
  because $\phi$ and $\psi$ are proper.
  Applying the theorem we obtain
  a compact continuation 
  \begin{equation*}
    \hill h:(\R^d \times \R^e)_\infty \rarr [\minfty,\pinfty]
  \end{equation*}
  of $h$ with 
  \begin{equation*}
    \hill {h}[\{\infty\}]
    =
    \hill h [\R^n_\infty \setminus \R^n]
    \subseteq
    \hill{g}[\{\infty\} \times Y'']
      \cup \hill{g}[Y' \times \{\infty\}]
      \cup \hill{g}[\{(\infty,\infty)\}]
    =
    \{\pinfty\},
  \end{equation*}
  i.e. $\hill{h}(\infty) = \pinfty = \widehat{h}(\infty)$.
  So $\widehat h = \hill h$ is indeed a continuous mapping 
  $
    \big((\R^d \times \R^e)_\infty, \mathcal{O}^{\otimes (d+e)}_\infty \big)
    \rarr 
    \big( [\minfty,\pinfty], \mathcal{T} \big)
  $.
\end{proof}

\chapter{Coercivity of a sum of functions} \label{chap:coercivity_of_sum_of_functions}

\minitoc

\addtointroscoll{

In this chapter we develop a tool 
(\prettyref{thm:sum_coercive_on_certain_subspaces}) which gives
information on which subspaces a sum $F+G$ of certain functions 
is coercive. The coercivity assertion of 
\prettyref{lem:application_of_compact_continuations_to_get_coercivity_and_lsc_of_sum}
is contained as special case in the coercivity assertion of
\prettyref{thm:sum_coercive_on_certain_subspaces} 
if we set 
$F = \phi(Hx) = F_1 \sdirsum 0_{X_2}$ and 
$G = \psi(Kx) = G_1 \sdirsum 0_{Y_2}$
with $F_1 \defeq F|_{X_1}$ and $G_1 \defeq G|_{Y_1}$,
where $X_1 \defeq \mathcal{R}(H^*)$, $X_2 \defeq \mathcal{N}(H)$
and $Y_1 \defeq \mathcal{R}(K^*)$, $Y_2 \defeq \mathcal{N}(K)$,
see 
\prettyref{det:coercivity_assertion_in_example_is_special_case_of_thm_sum_coercive_on_certain_subspaces}
in the Appendix.

In contrast to the previous chapter we restrict us in this chapter
to coercivity notions without regarding e.g. lower semicontinuity
at the same time.
Moreover the coercivity notions in this chapter are rather based
on norms instead of compact (or compact and closed) sets.
In case of vector spaces of finite dimension there is however 
a strong relation between topological coercivity notions 
from the previous chapter and the coercivity notions that 
will be given in this chapter, 
see \prettyref{lem:coervivity_vs_top_strong_coercivity_in_normed_spaces}
and cf. \prettyref{thm:characterize_lsc_and_coercive}.
For linear mappings between vector spaces of finite dimension 
normcoercivity is equivalent to injectivity, see 
\prettyref{thm:linear_mapping_coercive_iff_injective}.

}

\section{Extension of coercivity notions to broader classes
of functions}

So far we introduced the notions of coercivity and normcoercivity only 
for mappings $f: \R^n \rarr \R^m$, cf. the definitions on page 
\pageref{def:coercivity_for_mappings_from_Rn_into_Rdouble_extended}.
We now extend the notion of coercivity and normcoercivity to 
broader classes of functions, show that they behave well under 
concatenation and that the 
``Cartesian product'' of normcoercive mappings is again normcoercive,
see \prettyref{thm:concatenation_normcoercive_and_dots} and
\prettyref{lem:normcoercivity_of_SoToSay_cartesian_product_of_functions},
respectively.
Then another extension of the notion of coercivity is performed
by replacing the codomain $[\minfty, \pinfty]$ by a general 
totally ordered set $(Z, \leq)$.
In
\prettyref{lem:coervivity_vs_top_strong_coercivity_in_normed_spaces}
we will see that this coercivity notion is related to topological 
coercivity notions involving the spaces 
$(Z, \mathcal{T}_\geq)$ and $(Z, \mathcal{T}_\leq)$.
Finally a variant of 
\prettyref{prop:relation_between_boundedness_notions_for_mapping}
is given in 
\prettyref{thm:coercive_mapping_in_finite_dim_space_bounded_below_already_when_locally_bounded_below},
saying that a coercive mapping 
$F: (X, \|\cdot\|_X) \rarr (Z, \leq)$ from normed space 
of finite dimension into a totally ordered set is
already bounded below if it is locally bounded below.

\newcommand{\norm}[1]{\left\Vert #1\right\Vert }
\newcommand{\seq}{\subseteq}
\begin{definition}
  Let $\left(X,\norm{\cdot}_{X}\right)$ be a normed space with 
  nonempty 
  subset $\check{X}\seq X$. We call a mapping 
  $f:\check{X}\rightarrow [\minfty, \pinfty] $\textbf{
  coercive}, if and only if \[
  \lim_{\begin{smallmatrix}\norm{\check{x}}_{X}\rightarrow\pinfty\\
  \check{x}\in\check{X}\end{smallmatrix}}f\left(\check{x}\right)=\pinfty.\]
\end{definition}
  \begin{definition}
    Let $\left(X,\norm{\cdot}_{X}\right)$ and $\left(Y,\norm{\cdot}_{Y}\right)$
    be normed spaces with
    nonempty 
    subsets $\check{X}\seq X,$
    $\check{Y}\seq Y$. We call a mapping $f:\check{X}\rightarrow\check{Y}$
    \textbf{normcoercive}, if and only if \[
    \lim_{\begin{smallmatrix}\norm{\check{x}}_{X}\rightarrow\pinfty\\
    \check{x}\in\check{X}\end{smallmatrix}}\norm{f\left(\check{x}\right)}_{Y}=\pinfty.\]
    (I.e. $\check{x}\mapsto\norm{f\left(\check{x}\right)}_{Y}$ is
    coercive.)
  \end{definition}

Note in theses definitions that functions $f$ are vacuously coercive
respectively normcoercive, if the domain of definition $\widecheck X$ 
is bounded:
The -- more explicitly formulated -- 
defining conditions for coercivity and normcoercivity
\begin{align*}
 &\text{For all sequences } (\widecheck x^{(k)})_{k \in \N} 
    \text{ in } \widecheck X \text{ with } 
    \|\widecheck x^{(k)}\| \rarr \pinfty \text{ we have }
    f(\widecheck x^{(k)}) \rarr \pinfty,
  \\
 &\text{For all sequences } (\widecheck x^{(k)})_{k \in \N} 
    \text{ in } \widecheck X \text{ with } 
    \|\widecheck x^{(k)}\| \rarr \pinfty \text{ we have }
    \|f(\widecheck x^{(k)})\| \rarr \pinfty
\end{align*}
are namely both trivially fulfilled in that case
since a bounded set $\widecheck X$ contains no sequences
$(\widecheck x^{(k)})_{k\in \N}$ with 
$\|\widecheck x^{(k)}\| \rarr \pinfty$ as $k \rarr \pinfty$.
The following tool is obtained directly from the definitions.
\begin{theorem} \label{thm:concatenation_normcoercive_and_dots}
  The following concatenation statements hold:
  \begin{enumerate}
    \item \label{enu:concatenation_normcoercive_normcoercive}
      The concatenation of normcoercive mappings
      is again normcoercive. 
    \item \label{enu:concatenation_normcoercive_coercive}
      The concatenation of a normcoercive mapping 
      $E:\check{X}\rightarrow\check{Y}$
      with a coercive mapping 
      $
	F:\check{Y} \rightarrow [\minfty,\pinfty]	 
      $
      is coercive.
  \end{enumerate}
\end{theorem}



In the following lemma we equip the product spaces
of $X \times Y$ with the norm 
$
  \| \cdot \| 
  \defeq 
  \|\cdot \|_{X \times Y} 
  \defeq \|\cdot \|_X + \|\cdot \|_Y
$
(or any equivalent norm).

\begin{lemma} 
\label{lem:normcoercivity_of_SoToSay_cartesian_product_of_functions}
  Let $(X, \|\cdot \|_X)$, $(Y, \|\cdot\|_Y)$ and 
  $(Z, \|\cdot \|_Z)$, $(W, \| \cdot \|_W)$ 
  be normed spaces and let
  $\widecheck F: \widecheck X \rarr Z$, 
  $\widecheck G: \widecheck Y \rarr W$
  be normcoercive mappings, defined on subsets 
  $\widecheck X$ and $\widecheck Y$ of $X$ and $Y$,
  respectively.
  Then the function 
  $
    \widecheck A: 
    \widecheck X \times \widecheck Y
    \rarr 
    Z \times W
  $,
  given by 
  \begin{gather*}
    \widecheck A (\widecheck x, \widecheck y)
    \defeq 
    \begin{pmatrix}
      \widecheck F(\widecheck x)
      \\
      \widecheck G(\widecheck y)
    \end{pmatrix}
  \end{gather*}
  is also normcoercive.
\end{lemma}

\begin{proof}
  In order to prove that 
  $\widecheck A: \widecheck X \times \widecheck Y \rarr Z \times W$
  is normcoercive consider an arbitrary sequence 
  $(\widecheck x_n, \widecheck y_n)_{n \in \N}$ in 
  $\widecheck X \times \widecheck Y$ with 
  \begin{gather} \label{eq:sequence_in_product_space_tends_to_infty}
    \|(\widecheck x_n, \widecheck y_n)\|_{X \times Y} \rarr \pinfty
  \end{gather}
  as $n \rarr \pinfty$.
  We have to show that for any $C \in \R$ there is an $N \in \N$ such
  that for all natural $n \geq N$ the inequality 
  $\|\widecheck A(\widecheck x_n, \widecheck y_n)\| \geq C$ holds true.
  Assume that the latter statement is not true; then there is a $C > 0$ and a
  subsequence 
  $(\widecheck x_{n_k}, \widecheck y_{n_k})_{k \in \N}$ such that 
  $
    C > 
      \|\widecheck A (\widecheck x_{n_k}, \widecheck y_{n_k}) \|_{Z \times W}
     =
      \|\widecheck F(\widecheck x_{n_k})\|_Z
     +
      \|\widecheck G (\widecheck y_{n_k})\|_W
  $
  for all $k \in \N$.
  In particular we had
  \begin{align} \label{eq:functions_bounded_on_subsequence}
       \|\widecheck F(\widecheck x_{n_k})\|_Z < C
     &&\text{and}
     &&\|\widecheck G(\widecheck y_{n_k})\|_W < C
  \end{align}
  for all $k \in \N$. Consequently both 
  $(\|\widecheck x_{n_k}\|)_{k \in \N}$ and 
  $(\|\widecheck y_{n_k}\|)_{k \in \N}$
  would be bounded above by some $B > 0$, see Detail
  \ref{det:both_subsequences_bounded_above}
  in the Appendix.
  We thus would obtain
  $
    \|(\widecheck x_{n_k}, \widecheck y_{n_k})\|_{X \times Y}
    =
    \|\widecheck x_{n_k}\|_X + \|\widecheck y_{n_k}\|_Y
    \leq 
    2B
  $
  for all $k \in \N$ and hence a contradiction to
  \eqref{eq:sequence_in_product_space_tends_to_infty}.
\end{proof}

\begin{definition}
  Let $(X, \|\cdot\|)$ be a normed space
  and let $(Z, \leq)$ be a totally ordered set.
  A mapping $f: (X, \|\cdot\|) \rarr (Z, \leq)$ is called 
  {\bf coercive} iff for any $z$ which is not a maximum of 
  $(Z, \leq)$ there is an $R > 0$ such that $f(x) > z$ for all 
  $x \in X$ with $\|x\| > R$, 
  i.e -- more formally expressed -- iff
  \begin{gather*}
	  \forall z \in Z \setminus \MAX_\leq(Z)
    \;\;  \exists R > 0
    \;\;  \forall x \in X:
	  \|x\| > R \implies f(x) > z
  \end{gather*}
  holds true.
\end{definition}

Note in the following lemma that we really mean 
``$\mathcal{T}_\geq$'' in the second condition and that it is not 
a typo.

\begin{lemma} \label{lem:coervivity_vs_top_strong_coercivity_in_normed_spaces}
  Let $(X, \| \cdot \|)$ be a real normed space of finite dimension and
  let $(Z, \leq)$ be a totally ordered set. Equip $X$ with the 
  topology $\mathcal{O}$ which is induced by $\| \cdot \|$.
  For a mapping $f: X \rarr Z$ the following are equivalent:
  \begin{enumerate}
    \item \label{enu:mapping_from_normed_space_is_coercive}
      $f: (X, \|\cdot\|) \rarr (Z, \leq)$ is coercive.
    \item \label{enu:mapping_from_normed_space_is_topological_strongly_coercive_towards_MAX_in_opposite_topology}
      $f: (X, \mathcal{O}) \rarr (Z, \mathcal{T}_\geq)$ is topological strongly 
      coercive towards $\MAX_\leq(Z)$.
    \enumeratext{If $(Z, \leq)$ contains a minimum and a maximum the above conditions are also equivalent to}
    \item \label{enu:mapping_from_normed_space_is_topological_coercive_towards_MAX_in_right_order_topology}
      $f: (X, \mathcal{O}) \rarr (Z, \mathcal{T}_\leq)$ is topological coercive 
      towards $\MAX_\leq(Z)$.
  \end{enumerate}
\end{lemma}

\begin{proof}
  If $Z$ contains less than two elements all three statements are 
  clearly all true and hence equivalent.
  In the following we may hence assume that $Z$ contains at least 
  two elements.
\\
  ``\upref{enu:mapping_from_normed_space_is_coercive}
  $\implies$ 
  \upref{enu:mapping_from_normed_space_is_topological_strongly_coercive_towards_MAX_in_opposite_topology}'':
  Let $f: (X, \|\cdot\|) \rarr (Z, \leq)$ be coercive and let 
  $K' \in \mathcal{K}_{\MAX_\leq(Z)}(Z, \mathcal{T}_\geq)$.
  There is some $b \in Z \setminus \MAX_\leq(Z)$ with 
  $K' \subseteq b]$, see Detail
  \ref{det:compact_subset_contained_in_closed_half_ray}
  in the Appendix.
  Since $f:(X, \|\cdot\|) \rarr (Z, \leq)$ is coercive there is 
  for that $b \in Z \setminus \MAX_\leq(Z)$ an $R > 0$ such that
  $f(x) > b $ for all $x \in X$ with $\|x\| > R$.
  In other words 
  \begin{gather*}
    f[X \setminus \closedball[R][][][\|\cdot\|]] \subseteq Z \setminus b].
  \end{gather*}
  Setting $K \defeq \closedball[R][][][\|\cdot\|]$ we hence have found 
  a compact and closed subset of $(X, \|\cdot\|)$ with 
  $f[X \setminus K] \subseteq X' \setminus b] \subseteq X' \setminus K'$.
\\
  ``\upref{enu:mapping_from_normed_space_is_topological_strongly_coercive_towards_MAX_in_opposite_topology}
  $\implies$ 
  \upref{enu:mapping_from_normed_space_is_coercive}'':
  Let $f: (X, \mathcal{O}) \rarr (Z, \leq)$ be topological strongly 
  coercive towards $\MAX_\leq(Z)$.
  Let $z \in Z \setminus \MAX_\leq(Z)$.
  The set $z] \eqdef K'$ is a compact subset of $(Z, \mathcal{T}_\geq)$, 
  cf. Detail
  \ref{det:closed_left_halfray_is_compact_in_opposite_topology}
  in the Appendix.
  Moreover $K' \cap \MAX_\leq(Z) = \emptyset$ so that 
  $K' \in \mathcal{K}_{\MAX_\leq(Z)}(Z, \mathcal{T}_\geq)$.
  Since $f: (X, \mathcal{O}) \rarr (Z, \mathcal{T}_\geq)$ is 
  topological strongly coercive towards $\MAX_\leq(Z)$ there 
  is hence a $K \in \mathcal{K}(K, \mathcal{O})$ such that 
  \begin{gather*}
    f[X \setminus K] \subseteq Z \setminus K'.
  \end{gather*}
  Let $R > 0$ be so large that 
  $\closedball[R][][][\|\cdot\|] \supseteq K$.
  Then 
  \begin{gather*}
    f[X \setminus \closedball[R][][][\|\cdot\|] ]
    \subseteq f[X \setminus K] \subseteq Z \setminus K' 
    = Z \setminus z] = (z.
  \end{gather*}
  In other words we know that for $x \in X$ 
  the inequality $\|x\| > R$ implies $f(x) > z$.
  So $f:(X, \mathcal{O}) \rarr (Z, \leq)$ is coercive.
  Finally assume now additionally that $(Z, \leq)$ contains both 
  a minimum $\widecheck z$ and a maximum $\widehat z$.
  Set $S' \defeq \MAX_{\leq}(Z) = \{\widehat z\}$.
  We have to prove that the two statements
  \begin{alignat}{3}
    \label{eq:mapping_strong_coercive__towards_MAX_in_left_ordere_topology}
    	\forall K' &\in \mathcal{K}_{S'}(Z, \mathcal{T}_\geq)
    \;\; &&  \exists K \in \mathcal{KA}(X, \mathcal{O}):
	  f[X \setminus K] \subseteq Z \setminus K',
  \\
    \label{eq:mapping_coercive_towards_MAX_in_right_order_topology}
    	\forall L' &\in \mathcal{KA}_{S'}(Z, \mathcal{T}_\leq)
    \;\; &&  \exists K \in \mathcal{KA}(X, \mathcal{O}):
	  f[X \setminus K] \subseteq Z \setminus L'
  \end{alignat}
  are now equivalent.
  In order to prove that 
  \eqref{eq:mapping_coercive_towards_MAX_in_right_order_topology}
  implies \eqref{eq:mapping_strong_coercive__towards_MAX_in_left_ordere_topology}
  it is clearly sufficient to show that for any 
  $K' \in \mathcal{K}_{S'}(Z, \mathcal{T}_\geq)$ there is some 
  $L' \in \mathcal{KA}_{S'}(Z, \mathcal{T}_\leq)$
  with $Z \setminus L' \subseteq Z \setminus K'$, i.e. with 
  $L' \supseteq K'$.
  For the inverse implication it is likewise sufficient to show that for 
  any $L' \in \mathcal{KA}_{S'}(Z, \mathcal{T}_\leq)$ there is some 
  $K' \in \mathcal{K}_{S'}(Z, \mathcal{T}_\geq)$ with 
  $K' \supseteq L'$.
  Let first $K' \in \mathcal{K}_{S'}(Z, \mathcal{T}_\geq)$.
  As we have seen in part 
  ``\upref{enu:mapping_from_normed_space_is_coercive}
  $\implies$ 
  \upref{enu:mapping_from_normed_space_is_topological_strongly_coercive_towards_MAX_in_opposite_topology}''
  there is some $b \in Z \setminus S'$ with $K' \subseteq b]$.
  Clearly $L' \defeq b] = Z \setminus (b$ is a closed subset of 
  $(Z, \mathcal{T}_\leq)$.
  Moreover $L' = [\widecheck z, b]$ is surely a compact subset of 
  $(Z, \mathcal{T}_\leq)$; cf. 
  Detail \ref{det:closed_left_halfray_is_compact_in_opposite_topology}
  with reversed order or note that 
  $\widecheck z \in L'$ ca be covered by no open set from 
  $\mathcal{T}_\leq$, except for the whole space $Z \supseteq L'$.
  So $L' = b]$ fulfills both 
  $L' \in \mathcal{KA}_{S'}(Z, \mathcal{T}_\leq)$
  and $L' \supseteq K'$.
  Let to the contrary 
  $L' \in \mathcal{KA}_{S'}(Z, \mathcal{T}_\leq)$.
  Then $Z \setminus L'$ is an open neighborhood of $\widehat z$ and 
  contains hence a set of the form $(a$ with some 
  $a \in Z \setminus S' = Z \setminus \{\widehat z\}$.
  Building complements transforms $(a \subseteq Z \setminus L'$
  into $L' \subseteq a] \eqdef K'$. Again $K'$ is a compact subset of 
  $(Z, \mathcal{T}_\geq)$, 
  cf. \prettyref{det:closed_left_halfray_is_compact_in_opposite_topology}
  in the Appendix.
  Moreover $K'$ does not hit 
  $S'$ so that it fulfills both 
  $K' \in \mathcal{K}_{S'}(Z, \mathcal{T}_\geq)$ and 
  $K' \supseteq L'$.
\end{proof}

\begin{theorem} \label{thm:coercive_mapping_in_finite_dim_space_bounded_below_already_when_locally_bounded_below}
  Let $(X, \|\cdot\|_X)$ be a normed space of finite dimension and 
  $(Z, \leq)$ a totally ordered set. A coercive mapping 
  $F: (X, \|\cdot\|_X) \rarr (Z, \leq)$ is already bounded below if it is 
  locally bounded below.
\end{theorem}

\begin{proof}
  If $\dim X = 0$, the image $F[X] = F[\{\zerovec\}]$ consists of 
  just one single point, so that $F$ is bounded below by that value.
  If $n \defeq \dim X \in \N$ we may without loss of generality 
  assume that $(X, \|\cdot \|_X) = (\R^n, \|\cdot\|)$ with some 
  norm $\|\cdot \|$ on $\R^n$.
  After equipping the totally ordered space 
  $(Z, \leq)$ with the left order topology 
  $\mathcal{T}_\geq$ the coercivity of the mapping
  $
    F:(\R^n, \mathcal{O}^{\otimes n} )
      \rarr (Z, \leq)
  $
  corresponds to the topological strong coercivity of
  $
    F:(\R^n, \mathcal{O}^{\otimes n} )
      \rarr (Z, \mathcal{T}_\geq)
  $ 
  towards $\MAX_\leq(Z)$
  by \prettyref{lem:coervivity_vs_top_strong_coercivity_in_normed_spaces}.
  Hence \prettyref{prop:relation_between_boundedness_notions_for_mapping}
  ensures that the locally bounded below mapping 
  $
    F:(\R^n, \mathcal{O}^{\otimes n} )
      \rarr (Z, \leq)
  $
  is even bounded below.
\end{proof}

\section{Normcoercive linear mappings}

A linear mapping defined in any finite dimensional 
space is normcoercive if and only if it is injective:

\begin{theorem} \label{thm:linear_mapping_coercive_iff_injective} 
  A linear mapping $\alpha:X\rightarrow Y$
  of a finite-dimensional normed space $\left(X,\|\cdot\|_{X}\right)$
  into a normed space $\left(Y,\|\cdot\|_{Y}\right)$ is normcoercive
  if and only if its nullspace $\mathcal{N}$ just consists of $\zerovec_{X}$.
\end{theorem}

\begin{proof}
  In the case $X = \{\zerovec\}$ we clearly have 
  $\mathcal{N} = \{\zerovec\}$; moreover there is no 
  sequence $(x_n)_{n\in \N}$ with $\|x_n\| \rarr \pinfty$,
  as $n \rarr \pinfty$, so that $f$ is trivially normcoercive.
  Consider now the case $X \supset \{\zerovec\}$.
  If $\mathcal{N}$ contains an element $x\neq\zerovec_{X}$ then $\alpha$
  is not normcoercive since the sequence $\left(x_{n}\right)_{n\in\mathbb{N}}$
  defined by $x_{n}\defeq nx$ fulfills $\|x_{n}\|_{X}\rightarrow\pinfty$
  but $\|\alpha\left(x_{n}\right)\|_{Y}=n\|\alpha\left(x\right)\|_{Y}=0\nrightarrow\pinfty$
  for $n\rightarrow\pinfty$. To show the other direction we assume
  that $\mathcal{N}=\left\{ \zerovec_{X}\right\}$.
  Then the sphere 
    $
      \mathbb{S}
      \defeq 
      \left\{ \tilde{x}\in X:\|\tilde{x}\|_{X}=1\right\}
    $
  is mapped by $\alpha$ to a set which omits $\zerovec_{Y}.$ For this and
  by the compactness of the nonempty set $\mathbb{S}$ 
  we find a point $\check{x}\in\mathbb{S}$ with
  \[
    \min_{\tilde{x}\in\mathbb{S}}\|\alpha\left(\tilde{x}\right)\|_{Y}
    =
    \|\alpha\left(\check{x}\right)\|_{Y}>0.
  \]
  By scaling with a positive number $\lambda>0$ we see that
  \begin{eqnarray*}
    \min_{\|x\|_{X}=\lambda}\|\alpha\left(x\right)\|_{Y} 
    & = & \min_{x\in\lambda\mathbb{S}}\|\alpha\left(x\right)\|_{Y}\\
    & = & \lambda\min_{x\in\lambda\mathbb{S}}\|\alpha\left(\frac{x}{\lambda}\right)\|_{Y}\\
    & = & \lambda\min_{\tilde{x}\in\mathbb{S}}\|\alpha\left(\tilde{x}\right)\|_{Y}\\
    & = & \lambda\|\alpha\left(\check{x}\right)\|_{Y}.
  \end{eqnarray*}
  This means that 
  $
    \|\alpha\left(x\right)\|_{Y}
    \geq
    \|x\|_{X}
    \underbrace{\|\alpha\left(\check{x}\right)\|_{Y}}_{>0}\rightarrow\pinfty$
    for $\|x\|_{X}\rightarrow\pinfty
  $,
  i.e. $\alpha$ is normcoercive.
\end{proof}

\begin{corollary} \label{cor:Hinreichende_Bed_an_nullraueme}
  Let 
  $
    \begin{cases}
      \alpha': & X\rightarrow Y'\\
      \alpha'': & X\rightarrow Y''
    \end{cases}
  $ be linear mappings of a finite-dimensional normed space 
  $\left(X,\|\cdot\|_{X}\right)$
  into normed spaces $\left(Y',\|\cdot\|_{Y'}\right)$,
  $\left(Y'',\|\cdot\|_{Y''}\right)$.
  If their nullspaces
  $\mathcal{N}'$, $\mathcal{N}''$ have only $\zerovec_{X}$ in
  common then the linear mapping 
  $\alpha:X\rightarrow Y'\times Y''$,
  given by
  \[
    \alpha\left(x\right)
    \defeq
    \left(\begin{smallmatrix}
	\alpha'\left(x\right)\\
	\alpha''\left(x\right)
    \end{smallmatrix}\right),
  \]
  is normcoercive.
\end{corollary}


\begin{proof}
Since the nullspace $\mathcal{N}$ of $\alpha$ fulfills 
$\mathcal{N}=\mathcal{N}'\cap\mathcal{N}''=\{\underbrace{\left(\begin{smallmatrix}\zerovec_{Y'}\\
\zerovec_{Y''}\end{smallmatrix}\right)}_{=\zerovec_{Y}}\}$ we obtain the statement by applying \prettyref{thm:linear_mapping_coercive_iff_injective}.
\end{proof}




\begin{definition}
  Let $X = X_1 \oplus X_2$ be a direct 
  decomposition of a real vector space $X$.
  The linear mapping $\pi_{X_1,X_2}: X \rarr X_1$, given by
  \begin{gather*}
    \pi_{X_1,X_2}(x) 
    = 
    \pi_{X_1,X_2}(x_1 + x_2)
    \defeq 
    x_1
  \end{gather*}
  is called {\bf projection to $X_1$ along $X_2$}.
  If $X$ is equipped with some inner 
  product $\langle \cdot, \cdot \rangle$ such that
  $X_2 = X_1^\perp$ we will also shortly write $\pi_{X_1}$.
\end{definition}



\begin{lemma} \label{lem:nullspace_of_projections}
  Let $X = X_1 \oplus X_2$ and $X = W_1 \oplus  W_2$
  be direct 
  decompositions of 
  a real vector space $X$.
  The following holds true:
  \begin{enumerate}
    \item \label{enu:nullspace_of_projection}
      The nullspace of $\pi_{X_1,X_2}$ is 
      $\mathcal{N}(\pi_{X_1,X_2}) = X_2$.
      In particular, for any subspace $\widetilde {X_1}$ 
      of $X$ which is 
      also complementary to $X_2$,
      the restriction
      $\pi_{X_1,X_2}|_{\widetilde {X_1}}: \widetilde{X_1} \rarr X_1$ 
      is a vector space isomorphism between 
      $\widetilde{X_1}$ and $X_1$.      
    \item \label{enu:where_duo_of_proj_is_injective}
      The linear mapping 
      $\alpha : X \rarr X_1 \times W_1$, given by
      \begin{gather*}
        \alpha(z) \defeq 
	\begin{pmatrix}
	  \pi_{X_1,X_2}(z)
	\\
	  \pi_{W_1,W_2}(z)
	\end{pmatrix}
      \end{gather*}
      has nullspace $X_2 \cap W_2$;
      in particular restricting $\alpha$
      to any complementary subspace $Z_1$ of $X_2 \cap W_2$
      yields an injective mapping 
      $\alpha|_{Z_1}: Z_1 \rarr X_1 \times W_1$.
    \item \label{enu:where_duo_of_orth_proj_is_injective}
      If $\langle \cdot, \cdot \rangle$ is some inner product on 
      $X$ such that $X_2 = X_1^\perp$ and $W_2 = W_1^\perp$ then
      the linear mapping 
      $\alpha : X \rarr X_1 \times W_1$, given by
      \begin{gather*}
        \alpha(z) \defeq 
	\begin{pmatrix}
	  \pi_{X_1}(z)
	\\
	  \pi_{W_1}(z)
	\end{pmatrix}
      \end{gather*}
      has nullspace 
      $X_1^\perp \cap W_1^\perp$.
      In particular the restriction 
      $\alpha|_{X_1+W_1}: X_1+W_1 \rarr X_1 \times W_1$ is injective.
  \end{enumerate}
\end{lemma}

\begin{proof}
  \upref{enu:nullspace_of_projection}
  Writing an arbitrarily chosen $x \in X$ in the form $x = x_1 + x_2$
  with uniquely determined $x_1 \in X_1$ and $x_2 \in X_2$
  we obtain
  \begin{gather*}
	  x \in \mathcal{N}(\pi_{X_1,X_2})
    \iseq \pi_{X_1,X_2}(x_1+x_2) = \zerovec
    \iseq x_1 = \zerovec 
    \iseq x = x_2
    \iseq x \in X_2
  \end{gather*}
  so that $\mathcal{N}(\pi_{X_1,X_2}) = X_2$. 
  This implies that the restricted mapping 
  $\pi_{X_1,X_2}|_{\widetilde {X_1}}: \widetilde {X_1} \rarr X_1$
  is injective for any subspace $\widetilde {X_1}$,
  which is also complementary to $X_2$, since
  \begin{gather*}
      \mathcal{N}(\pi_{X_1,X_2}|_{\widetilde {X_1}})
    = \mathcal{N}(\pi_{X_1,X_2}) \cap \widetilde {X_1}
    = X_2 \cap \widetilde {X_1}
    = \{\zerovec\}.
  \end{gather*}
  Due to 
  \begin{gather*}
      \pi_{X_1,X_2}|_{\widetilde {X_1}}[\widetilde {X_1}]
    = \pi_{X_1,X_2}[\widetilde {X_1}] 
    = \pi_{X_1,X_2}[\widetilde {X_1} \oplus X_2]
    = \pi_{X_1,X_2}[X]
    = X_1
  \end{gather*}
  the linear mapping $\pi_{X_1,X_2}|_{\widetilde X_1}$ is also 
  surjective and hence a vector space isomorphism.
\\
  \upref{enu:where_duo_of_proj_is_injective}
  Applying the just proven part twice we obtain for any 
  $x \in X$ the equivalences
  \begin{gather*}
	  \alpha(x) = \zerovec
    \iseq \pi_{X_1,X_2}(x) = \zerovec \wedge \pi_{W_1,W_2}(x) = \zerovec
    \iseq x \in X_2 \wedge x \in W_2
    \iseq x \in X_2 \cap W_2,
  \end{gather*}
  so that $\mathcal{N}(\alpha) = X_2 \cap W_2$.
  Likewise as in the already proven part 
  \upref{enu:nullspace_of_projection}
  this implies that the restricted mapping 
  $\alpha|_{Z_1}: Z_1 \rarr X_1 \times W_1$
  is injective for any subspace $Z_1$ of 
  $X$ which is complementary to $X_2 \cap W_2$.
\\
  \upref{enu:where_duo_of_orth_proj_is_injective}
  By the just proven previous part 
  \upref{enu:where_duo_of_proj_is_injective} we have 
  $
      \mathcal{N}(\alpha) 
    = X_2 \cap W_2 
    = X_1^\perp \cap W_1^\perp
  $.
  Therefore and since 
  $(X_1^\perp \cap W_1^\perp) \cap (X_1 + W_1) = \{\zerovec\}$,
  see Detail 
  \ref{det:intersection_between_sum_of_spaces_and_intersection_of_their_orthogonal_complements_is_trivial},
  we obtain
  \begin{gather*}
      \mathcal{N}(\alpha|_{X_1+W_1})
    = \mathcal{N}(\alpha) \cap (X_1 + W_1)
    = (X_1^\perp \cap W_1^\perp) \cap (X_1 + W_1) 
    = \{\zerovec\}.
  \end{gather*}
  Hence $\alpha|_{X_1 + W_1}$ is injective.
\end{proof}

  \section{Semidirect sums and coercivity} \label{sec:semidirect_sums_and_coercivity}


In this subsection we consider functions 
$F,G: \R^n \rarr \R \cup \{\pinfty\}$, which allow a certain 
decomposition into coercive and locally bounded from below parts 
$F_1: X_1 \rarr \R \cup \{\pinfty\}$, 
$G_1: Y_1 \rarr \R \cup \{\pinfty\}$ and bounded from below parts 
$F_2: X_2 \rarr \R \cup \{\pinfty\}$,
$G_2: Y_2 \rarr \R \cup \{\pinfty\}$
and prove a sufficient criteria for 
$F + G$ beeing coercive on a subspace $Z_1$. 
The exact result is stated in 
\prettyref{thm:sum_coercive_on_certain_subspaces}.

The mentioned decomposability of $F$ means more precisely 
that $F$ can be written as some, to be introduced, 
{\emph semidirect sum} $F = F_1 \sdirsum F_2$.
The demanded boundedness assumptions for $F_2$ and $G_2$
allows us to replace $F_2$ and $G_2$ by the constant 
zero functions $0_{X_2}$ and $0_{Y_2}$.
Working with the simpler direct decompositions
$F_1 \sdirsum 0_{X_2}$ and $G_1 \sdirsum 0_{Y_2}$
is the core of the proofs in this subsection.

\begin{definition}
  Let $X = X_1 \oplus X_2$ be a direct decomposition of a real 
  vector space $X$. The   
  {\bf semi-direct sum}   
  of functions $F_1: X_1 \rarr \R \cup \{\pinfty\}$ and 
  $F_2: X_2 \rarr \R \cup \{\pinfty\}$ is the function 
  $F_1 \sdirsum F_2 : X \rarr \R \cup \{\pinfty\}$, given by 
  \begin{gather*}
    (F_1 \sdirsum F_2 ) (x_1 + x_2) \defeq F_1(x_1) + F_2(x_2)
  \end{gather*}
\end{definition}

\begin{remark}
  Although the notation $X_1 \oplus X_2$ for the underlying spaces 
  suggests the similar notation $F_1 \oplus F_2$ 
  for a pair of functions defined on $X_1$ and $X_2$, respectively,
  we prefer the notation 
  $F_1 \sdirsum F_2$ for the following reason:
  If $\widetilde{F_1}: X_1 \rarr \R \cup \{\pinfty\}$ and 
  $\widetilde F_2 : X_2 \rarr \R \cup \{\pinfty\}$ are mappings with 
  $F_1 \sdirsum F_2 = \widetilde{F_1} \sdirsum \widetilde F_2$ 
  we can in general not conclude that
  $F_1 = \widetilde{F_1}$ and $F_2 = \widetilde F_2$;
  for real-valued functions we can conclude only that there is a constant 
  $C \in \R$ such that $F_1 = \widetilde{F_1} + C$ and
  $F_2 = \widetilde F_2 - C$,
  see Detail \ref{det:sum_unique_only_up_to_constant} 
  -- moreover not even the latter is 
  in general true, if one of the four functions takes the 
  value $\pinfty$, see Detail 
  \ref{det:sum_not_even_unique_up_to_constant}.
  But at least we have 
  \begin{gather}
  \label{eq:chancel_rule_for_semi_direct_sum}
    F_1 \sdirsum F_2 = \widetilde{F_1} \sdirsum F_2
    \implies 
    F_1 = \widetilde{F_1},
  \end{gather}
  if $F_2$ is real-valued; note here that $F_1$ and $\widetilde{F_1}$ 
  have
  the same domain of definition!

\end{remark}

  \begin{lemma} \label{lem:semidirect_sum_with_zerofunction}
  Let $\R^n = X_1 \oplus X_2 = Y_1 \oplus Y_2$ be decompositions of 
  $\R^n$ into subspaces and let 
  $F_1: X_1 \rarr \R \cup \{\pinfty\}$,
  $G_1: Y_1 \rarr \R \cup \{\pinfty\}$
  be mappings. The following holds true:
  \begin{enumerate}
    \item \label{enu:semidirect_sum_with_zero_independet_from_choice_of_complement}
      For every subspace $\widetilde {X_1}$ of $\R^n$ which is 
      also complementary to $X_2$ there is exactly one mapping 
      $\widetilde{F_1}: \widetilde{X_1} \rarr \R \cup \{\pinfty\}$
      with 
      \begin{gather*}
        \widetilde{F_1} \sdirsum 0_{X_2}
	=
	F_1 \sdirsum 0_{X_2},
      \end{gather*}
      namely the function
      $
	\widetilde{F_1} = F_1 \circ \pi_{X_1,X_2}|_{\widetilde{X_1}}
	= (F_1 \sdirsum 0_{X_2})|_{\widetilde{X_1}}
      $.
      In particular $F_1$ is coercive iff 
      $\widetilde F_1$ is coercive.
    \item \label{enu:sum_of_semidir_sum_with_zero_is_of_the_same_structure}
      For any subspace $Z_1$ of $\R^n$ which is complementary to 
      $X_2 \cap Y_2 \eqdef Z_2$ we have
      \begin{gather*}
        H
	\defeq
	(F_1 \sdirsum 0_{X_2}) + (G_1 \sdirsum 0_{Y_2})
	=
	H_1 \sdirsum 0_{X_2 \cap Y_2},
      \end{gather*}
      where 
      $
	H_1 
	\defeq 
	H|_{Z_1} 
	= 
	F_1 \circ \pi_{X_1,X_2}|_{Z_1} + G_1 \circ \pi_{Y_1,Y_2}|_{Z_1}
      $.
      If $X_1 \perp X_2$ and $Y_1 \perp Y_2$ holds true in addition
      we can choose $Z_1 = X_1 + Y_1$.
  \end{enumerate}

\end{lemma}

\begin{proof}
  \upref{enu:semidirect_sum_with_zero_independet_from_choice_of_complement}
  We first show the uniqueness of $\widetilde{F_1}$.
  To this end let 
  $\Phi_1: X_1 \rarr \R \cup \{\pinfty\}$ be a mapping with 
  $\Phi_1 \sdirsum 0_{X_2} = \widetilde{F_1} \sdirsum 0_{X_2}$.
  Clearly the mapping $0_{X_2}$ is real-valued so that we get 
  $\Phi_1 = \widetilde{F_1}$ by \eqref{eq:chancel_rule_for_semi_direct_sum}.
  Next we show that 
  $\widetilde{F_1} = F_1 \circ \pi_{X_1,X_2}|_{\widetilde{X_1}}$
  fulfills the claimed equality 
  $ \widetilde{F_1} \sdirsum 0_{X_2} = 	F_1 \sdirsum 0_{X_2}$.
  To this end we write an arbitrarily chosen $x \in \R^n$ in the forms
  $x = x_1 + x_2 = \widetilde x_1 + x_2'$ with 
  $x_1 \in X_1, \widetilde x_1 \in \widetilde{X_1}$ and 
  $x_2, x_2' \in X_2$. Then
  $
      \pi_{X_1,X_2}(\widetilde x_1) 
    = \pi_{X_1,X_2}(x_1 + (x_2 - x_2') )  
    = x_1 
  $,
  so that 
  $
      \widetilde{F_1} (\widetilde x_1) 
    = F_1(\pi_{X_1,X_2}(\widetilde x_1)) 
    = F_1(x_1)
  $.
  Therefrom we obtain
  \begin{align*}
    & (F_1 \sdirsum 0_{X_2})(x)
    = (F_1 \sdirsum 0_{X_2})(x_1 + x_2) 
    = F_1(x_1) + 0
    = \widetilde{F_1}(\widetilde x_1) + 0
    = (\widetilde{F_1} \sdirsum 0_{X_2}) (\widetilde x_1 + x_2')
  \\
   = {}& (\widetilde{F_1} \sdirsum 0_{X_2}) (x)
  \end{align*}
  as well as
  $
    F_1 \circ \pi_{X_1,X_2}|_{\widetilde{X_1}} 
    = 
    (F_1 \sdirsum 0_{X_2})|_{\widetilde{X_1}}
  $
  since
  \begin{align*}
    & F_1 \circ \pi_{X_1,X_2}|_{\widetilde{X_1}}(\widetilde x_1)
    = F_1(x_1)
    = F_1(x_1) + 0_{X_2}(x_2-x_2')
    = (F_1 \sdirsum 0_{X_2})(x_1+x_2-x_2')
    \\
    = {}& (F_1 \sdirsum 0_{X_2})|_{\widetilde{X_1}}(\widetilde x_1).
  \end{align*}
  It remains to show that $F_1$ is coercive iff 
  $\widetilde F_1 = F_1 \circ \pi_{X_1,X_2}|_{\widetilde{X_1}}$ is
  coercive. To this end note that 
  \begin{gather*} 
    \pi 
    \defeq 
    \pi_{X_1,X_2}|_{\widetilde{X_1}}:  \widetilde {X_1} \rarr X_1
  \end{gather*}
  is a vector space isomorphism by part \upref{enu:nullspace_of_projection} 
  of \prettyref{lem:nullspace_of_projections}. 
  Since the spaces $\widetilde {X_1}$ and $X_1$ are of finite dimension 
  the mapping $\pi$ is even a 
  bicontinuous vector space isomorphism. In particular
  the equivalence
  \begin{gather*}
        \|\widetilde {x_1}^{(n)}\| \rarr \pinfty
    \iseq
    \|\pi(\widetilde {x_1}^{(n)})\| \rarr \pinfty
  \end{gather*}
  holds true for all sequences
  $(\widetilde {x_1}^{(n)})_{n \in \N}$ in $\widetilde {X_1}$
  so that 
  \begin{alignat*}{3}
    & \widetilde{F_1}(\widetilde{x_1}) \rarr \pinfty 
      &&\text{ as } \|\widetilde {x_1}\| \rarr \pinfty, 
      \; \widetilde {x_1} \in \widetilde{X_1}
  \\ \iseq {}&
    F_1(\pi(\widetilde{x_1})) \rarr \pinfty 
      &&\text{ as } \|\pi(\widetilde{x_1})\| \rarr \pinfty,
      \; \widetilde{x_1} \in \widetilde{X_1}
  \\ \iseq {}&
    F_1(x_1) \rarr \pinfty 
      &&\text{ as } \|x_1\| \rarr \pinfty, 
      \; x_1 \in X_1.
  \end{alignat*}
\\
  \upref{enu:sum_of_semidir_sum_with_zero_is_of_the_same_structure}
  We first show that 
  $
      H_1 
    \defeq H|_{Z_1} 
    = F_1 \circ \pi_{X_1,X_2}|_{Z_1} + G_1 \circ \pi_{Y_1,Y_2}|_{Z_1}
  $.
  Writing an arbitrarily chosen $z_1' \in Z_1$ in the forms
  $z_1' = x_1' + x_2' = y_1' + y_2'$, where $x_1' \in X_1, x_2' \in X_2$
  and $y_1' \in Y_1, y_2' \in Y_2$, we indeed get 
  \begin{align*}
      H_1(z_1') 
   &= (F_1 \sdirsum 0_{X_2})(z_1') + (G_1 \sdirsum 0_{Y_2}) (z_1')
    = (F_1 \sdirsum 0_{X_2})(x_1'+x_2') + (G_1 \sdirsum 0_{Y_2}) (y_1'+y_2')
  \\
   &= F_1(x_1') + G_1(y_1')
    = F_1(\pi_{X_1,X_2}(x_1' + x_2') ) + G_1(\pi_{Y_1,Y_2}(y_1' + y_2') )
  \\
   &= [ F_1 \circ \pi_{X_1,X_2} + G_1 \circ \pi_{Y_1,Y_2} ](z_1').
  \end{align*}
  In order to prove 
  $
    (F_1 \sdirsum 0_{X_2}) + (G_1 \sdirsum 0_{Y_2})
    =
    H_1 \sdirsum 0_{X_2 \cap Y_2}
  $
  we write an arbitrarily chosen $x \in \R^n$ in the forms
  $x = x_1 + x_2 = y_1 + y_2 = z_1 + z_2$ where each vector 
  is an element of the similar denoted subspace.
  Using 
  $\pi_{X_1,X_2}(z_1) = \pi_{X_1,X_2}(x_1 + (x_2 - z_2) ) = x_1$,
  $\pi_{Y_1,Y_2}(z_1) = y_1$ and the previous calculation we obtain
  \begin{align*}
      H_1 \sdirsum 0_{X_2 \cap Y_2}(x)
   &= H_1(z_1) + 0
    = F_1(\pi_{X_1,X_2}(z_1) ) + G_1(\pi_{Y_1,Y_2}(z_1) )
    = F_1(x_1) +0  + G_1(y_1) + 0
  \\
   &= (F_1 \sdirsum 0_{X_2})(x_1 + x_2) + (G_1 \sdirsum 0_{Y_2})(y_1+y_2)
    = H(x).
  \end{align*}
  If $X_1 \perp X_2$ and $Y_1 \perp Y_2$ we can choose $Z_1 = X_1 + Y_1$
  since $(X_1 + Y_1 )^\perp = X_1^\perp \cap Y_1^\perp = X_2 \cap Y_2$
  so that in particular $\R^n = (X_1 + Y_1 ) \oplus (X_2 \cap Y_2)$.
\end{proof}

\begin{theorem} \label{thm:coercivity_does_not_depend_on_complementary_subspace}
  Let $\R^n = X_1 \oplus X_2$ be a direct decomposition of $\R^n$
  into subspaces $X_1$ and $X_2$ and let 
  $F_1 : X_1 \rarr \R \cup \{\pinfty\}$ be coercive and 
  $F_2 : X_2 \rarr \R \cup \{\pinfty\}$ be bounded below.
  Every function $F: \R^n \rarr \R \cup \{\pinfty\}$ with 
  $F \geq F_1 \sdirsum F_2$ is then  coercive on all those subspaces 
  $\widetilde {X_1}$ of $\R^n$ which are complementary to $X_2$, 
  i.e. which give a direct decomposition
  $\widetilde {X_1} \oplus  X_2 = \R^n = X_1 \oplus X_2$.
\end{theorem}

\begin{proof}
  Since $F_2$ is bounded below there is a constant $m \in \R$ with
  \begin{gather*}
    F_2(x_2) \geq m
  \end{gather*}
  for all $x_2 \in X_2$. Due to
  $F \geq F_1 \sdirsum F_2 \geq (F_1 \sdirsum 0_{X_2}) + m$
  it suffices to show that 
  $F_1 \sdirsum 0_{X_2} : \R^n \rarr \R \cup \{\pinfty\}$
  is coercive on every subspace $\widetilde {X_1}$ of $\R^n$
  which is complementary to $X_2$.
  The latter however follows from part
  \upref{enu:semidirect_sum_with_zero_independet_from_choice_of_complement}
  of \prettyref{lem:semidirect_sum_with_zerofunction} after fixing any 
  subspace $\widetilde {X_1}$ and setting 
  $\widetilde{F_1} \defeq (F_1 \sdirsum 0_{X_2})|_{\widetilde {X_1}}$.
\end{proof}

As word of warning note that, in contrast to part 
\upref{enu:semidirect_sum_with_zero_independet_from_choice_of_complement}
in \prettyref{lem:semidirect_sum_with_zerofunction}, 
the previous theorem states no equivalence between the coercivity of 
$F_1$ and $\widetilde F_1 \defeq F_1|_{\widetilde {X_1}}$
but states only that the coercivity of $F_1$ carries over to 
$\widetilde F_1$ if the assumptions of the previous theorem are fulfilled.
If $F_2$ is not constant zero the reverse implication is in general
not true as the following example shows:

\begin{example}
  Consider the direct decompositions 
  $\R^2 = X_1 \oplus X_2 = \widetilde {X_1} \oplus X_2$ with the one
  dimensional subspaces 
  $X_1 \defeq \R (1,0)^\tT$, $X_2 \defeq \R (0,1)^\tT$  
  and $\widetilde{X_1} \defeq \R (1,1)^\tT$.
  Consider the functions $F_1: X_1 \rarr \R$,
  $F_2: X_2 \rarr \R$ and $\widetilde{F_1}: \widetilde {X_1} \rarr \R $
  given by
  \begin{align*}
    F_1 &\defeq 0_{X_1},
  &&F_2(x_2) \defeq \|x_2\|_2^2,
  \\
  \widetilde {F_1} & \defeq (\underbrace{F_1 \sdirsum F_2}_{\eqdef F})|_{\widetilde {X_1}}.
  \end{align*}
  Clearly $F_2$ is bounded below. Moreover 
  $\widetilde {F_1}: \widetilde {X_1} \rarr \R$ is 
  coercive since 
  $
    \widetilde {F_1} ( (\xi, \xi)^\tT ) 
    = F((\xi, \xi)^\tT) 
    = \xi^2 \rarr \pinfty
  $
  as $\|(\xi, \xi)^\tT\|_2 \rarr \pinfty$.
  However the function $F_1$ is clearly not coercive.
  Note that this does not contradict the previous theorem since it
  is not even possible to write 
  $F = F_1 \sdirsum F_2$ in the form 
  $F = \widetilde {F_1} \sdirsum \Phi_2$ with a function
  $\Phi_2: X_2 \rarr \R \cup \{\pinfty\}$;
  if that would be possible the function 
  $\Phi_2$ would actually be finite and 
  the mapping 
  \begin{gather*}
    g: x_2 \mapsto F\big((1,1)^\tT + x_2\big) - F\big((0,0)^\tT + x_2\big) 
    = \widetilde{F_1}\big((1,1)^\tT\big) - \widetilde{F_1}\big((0,0)^\tT\big) 
  \end{gather*}
  would be constant on whole $X_2$. That is however clearly not the case;
  for instance we have
  $g((0,0)^\tT) = F((1,1)^\tT) - F((0,0)^\tT) = 1-0 = 1$
  and
  $g((0,3)^\tT) = F((1,4)^\tT) - F((0,3)^\tT) = 16-9 = 7$.

\end{example}

\begin{theorem} \label{thm:sum_coercive_on_certain_subspaces}
  Let $\R^n = X_1 \oplus X_2 = Y_1 \oplus Y_2$ be direct decompositions
  of $\R^n$ into subspaces and let
  $F_1: X_1 \rarr \R \cup \{\pinfty\}$, 
  $G_1: Y_1 \rarr \R \cup \{\pinfty\}$
  be both coercive and locally bounded below and let
  $F_2: X_2 \rarr \R \cup \{\pinfty\}$,
  $G_2: Y_2 \rarr \R \cup \{\pinfty\}$
  be bounded below. 
  Then 
      the sum $F+G: \R^n \rarr \R \cup \{\pinfty\}$ of functions 
      $F \geq F_1 \sdirsum F_2$ and $G \geq G_1 \sdirsum G_2$
      is coercive on all those vector subspaces $Z_1$ of $\R^n$
      with $\R^n = Z_1 \oplus (X_2 \cap Y_2)$.
      In particular $F+G$ is coercive on $X_1 + Y_1$, if 
      $X_1 \perp X_2$ and $Y_1 \perp Y_2$ hold additionally true.
\end{theorem}

Before proving the theorem we give a remark on two important assumptions.
\begin{remark} \label{rem:important_assumptions_for_coercivity_of_sum}
    It is important to demand locally boundedness of $F_1$ and $G_1$, 
    see \prettyref{exa:locally_boundedness_important}.
    In case of a non-orthogonal decomposition there is no guarantee
    that $F+G$ is coercive on $X_1 + Y_1$ as
    \prettyref{exa:non_orthogonal_decompositions} shows.
\end{remark}

\begin{proof}[Proof of \prettyref{thm:sum_coercive_on_certain_subspaces}]
  Since $F_2$ and $G_2$ are bounded below there is a constant 
  $m_2 \in \R$ with
  \begin{align*}
    F_2(x_2) \geq m_2,
    &&
    G_2(y_2) \geq m_2
  \end{align*}
  for all $x_2 \in X_2$, $y_2 \in Y_2$.
  Hence
  $
    F + G 
    \geq 
    (F_1 \sdirsum 0_{X_2}) + (G_1 \sdirsum0_{Y_2}) + 2m_2
  $,
  so that it suffices to show that
  $(F_1 \sdirsum 0_{X_2}) + (G_1 \sdirsum 0_{Y_2}) \eqdef H$
  is coercive on any subspace $Z_1$ which is complementary to
  $(X_2 \cap Y_2) \eqdef Z_2$.
  Concerning the domains of definition $X_1,Y_1$ and $Z_1$ of 
  the mappings $F_1, G_1$ and $H_1 \defeq H|_{Z_1}$, respectively, we may, 
  without loss of generality, assume
  $X_1 = X_2^\perp$, $Y_1 = Y_2^\perp$ and 
  $
    Z_1 
    = 
    Z_2^\perp 
  $,
  respectively, see Detail \ref{det:orthogonality_can_be_assumed_wlog}.
  In order to prove that 
  $H_1$ is coercive let any sequence
  $(z_k)_{k\in \N}$ in 
  $
    Z_1 
    = 
    (X_2 \cap Y_2)^\perp 
    = 
    X_2^\perp + Y_2^\perp
    =
    X_1 + Y_1
  $
  with $\|z_k\| \rarr \pinfty$ for $k \rarr \pinfty$
  be given. 
  The claimed  $H_1(z_k) \rarr \pinfty$ as $k \rarr \pinfty$
  holds trivially true, if there is a $K \in \N$ such that 
  $H_1(z_k) = \pinfty$ for all $k \geq K$. If there is no
  such $K$ we may without loss of generality assume
  $H_1(z_k) \in \R$ for all $k \in \N$.
  Since both $F_1$ and $G_1$ are bounded below, see  
  Detail \ref{det:coercive_and_locally_bounded_below_imply_bounded_below},
  there is a constant $m_1 \in \R$ such that
  \begin{align*}
    F_1(x) \geq m_1,
    &&
    G_1(y) \geq m_1
  \end{align*}
  for all $x \in X_1$, $y \in Y_1$.
  Therefore and by part 
  \upref{enu:sum_of_semidir_sum_with_zero_is_of_the_same_structure}
  of Lemma \ref{lem:semidirect_sum_with_zerofunction}
  we obtain
\rem{Beweis wuerde wohl ab hier 
  leichter und viel klarer, wenn Normbegriff zu 
  ``Hypernorm'' ausgedeht wuerde, vgl. Idee im Anhang}
  \begin{align*}
    H_1(z_k) 
    &= 
      F_1\big(\pi_{X_1}(z_k)\big) + G_1\big(\pi_{Y_1}(z_k)\big)
  \\
    &\geq 
    \max
    \big\{
      F_1\big(\pi_{X_1}(z_k)\big), G_1\big(\pi_{Y_1}(z_k)\big)
    \big\}
    +
    m_1
  \\
    &=
    \left\| 
    \begin{pmatrix}
      F_1\big(\pi_{X_1}(z_k)\big)
      \\
      G_1\big(\pi_{Y_1}(z_k)\big)
    \end{pmatrix}
    \right\|_\infty
    + m_1
  \\
    &=
    \left\| 
      \widecheck A(\widecheck \alpha(z_k))
    \right\|_\infty
    + m_1,
  \end{align*}
  where
  \begin{align*}
    \widecheck A(x,y) 
    \defeq 
    \begin{pmatrix}
      F_1(x)
    \\
      G_1(y)
    \end{pmatrix},
  &&
    \widecheck \alpha(z) 
    \defeq 
    \begin{pmatrix}
      \pi_{X_1}(z)
    \\
      \pi_{Y_1}(z)
    \end{pmatrix};
  \end{align*}
  the mappings 
  $\widecheck A: D_{\widecheck A} \rarr \R^2$ and 
  $\widecheck \alpha: D_{\widecheck \alpha} \rarr D_{\widecheck A}$, 
  are here defined on the nonempty sets
  \begin{gather*}
    D_{\widecheck A} 
    \defeq 
    \{ 
	(x_1, y_1) \in X_1 \times Y_1 :
	F_1(x_1), G_1(y_1) \in \R
    \}
    \subseteq X_1 \times Y_1    
  \shortintertext{and}
    D_{\widecheck \alpha}
    \defeq
    \{z_1 \in Z_1 = X_1 + Y_1 : 
      (\pi_{X_1}(\widecheck z_1), \pi_{Y_1}(\widecheck z_1))^\tT \in D_{\widecheck A}\}
    \subseteq
    X_1 + Y_1 
    \subseteq 
    \R^n,
  \end{gather*}
  respectively. The mappings $\widecheck A$ and 
  $\widecheck \alpha$ are restrictions of the likewise defined 
  mappings
  $
    A: X_1 \times Y_1 
      \rarr (\R \cup \{\pinfty\}) \times (\R \cup \{\pinfty\})
  $ and 
  $\alpha: \R^n \rarr X_1 \times Y_1$, respectively.
  Due to the previous estimate it suffices to show that 
  $\widecheck A \circ \widecheck \alpha: D_{\widecheck \alpha} \rarr \R^2$
  is normcoercive.
  Part \upref{enu:where_duo_of_orth_proj_is_injective} of 
  \prettyref{lem:nullspace_of_projections} ensures that 
  $\alpha|_{X_1+Y_1}$ is injective.
  The normcoercivity of $\alpha|_{X_1+Y_1}$ is hence obtained by 
  \prettyref{thm:linear_mapping_coercive_iff_injective} and carries over to 
  $\widecheck \alpha = \alpha|_{D_{\widecheck \alpha}}$. 
  In order to prove the normcoercivity of $\widecheck A$ we write 
  its domain of definition in the form 
  \begin{align*}
    D_{\widecheck A}
    &=
    \{(x_1,y_1) \in X_1\times Y_1: F_1(x_1) \in \R, \, G_1(y_1) \in \R \}
  \\&=
    \underbrace{\{x_1 \in X_1 : F_1(x_1) \in \R\} }_{\eqdef \widecheck X}
    \times
    \underbrace{\{y_1 \in Y_1 : G_1(y_1) \in \R\} }_{\eqdef \widecheck Y}
  \end{align*}
  and restrict the coercive and hence normcoercive functions 
  $F_1$ and $G_1$ to 
  $F_1|_{\widecheck X} \eqdef \widecheck F$ and 
  $G_1|_{\widecheck Y} \eqdef \widecheck G$, respectively.
  Applying Lemma
  \ref{lem:normcoercivity_of_SoToSay_cartesian_product_of_functions}
  to 
  \begin{gather*}
    \widecheck A(\cdot, \bullet) 
    = 
    \begin{pmatrix}
      \widecheck F (\cdot)
    \\
      \widecheck G (\bullet)
    \end{pmatrix}
  \end{gather*}
  gives then the normcoercivity of $\widecheck A$.
  Finally the concatenation 
  $\widecheck A \circ \widecheck \alpha$ of the normcoercive
  mappings is again normcoercive by Theorem
  \ref{thm:concatenation_normcoercive_and_dots}.
\end{proof}

\begin{example} \label{exa:locally_boundedness_important}
  Consider the functions $F, G: \R^2 \rarr \R$ given by
  \begin{align*}
    F(x_1, x_2) 
    \defeq
    \begin{cases}
      x_1^2 - \frac{1}{x_1^4} & \text{ for } x_1 \neq 0
    \\
      0  		      & \text{ for } x_1 = 0
    \end{cases},
  &&
    G(x_1, x_2) 
    \defeq
    \begin{cases}
      x_2^2  & \text{ for } x_2 \neq 0
    \\
      0	     & \text{ for } x_2 = 0
    \end{cases}.
  \end{align*}
  Setting 
  \begin{align*}
    X_1 \defeq &{} \linspan{e_1}, 
  &
    Y_1 \defeq & \linspan{e_2} = X_2,
  \\
    X_2 \defeq &{} \linspan{e_2},
  &
    Y_2 \defeq &{} \linspan{e_1} = X_1,
  \\
    F_1 \defeq &{} F|_{X_1}, 
  &
    G_1 \defeq &{} G|_{Y_1} = G|_{X_2}, 
  \\
    F_2 \defeq &{} 0_{X_2},
  &
    G_2 \defeq &{} 0_{Y_2} = 0_{X_1},
  \end{align*}
  we can write $F$ and $G$ as semidirect sums
  \begin{align*}
    F 
    =
    F_1 \sdirsum F_2 
  &&
    G
    =
    G_1 \sdirsum G_2.
  \end{align*}
  Clearly all assumptions of Theorem 
  \ref{thm:sum_coercive_on_certain_subspaces}
  are fulfilled -- except for one:
  The function $F_1$ fails to be locally bounded below, because of 
  the exceptional point $(0,0) \in X_1$.
  Setting
  \begin{gather*}
    x^{(n)} \defeq (x_1^{(n)}, x_2^{(n)}) \defeq (\tfrac1n, n)
  \end{gather*}
  gives a sequence $(x^{(n)})_{n \in \N}$ with
  $\|x^{(n)}\| \rarr \pinfty$ as $n\rarr \pinfty$ for which
  \begin{gather*}
    F(x^{(n)}) + G(x^{(n)})
    =
    \Big(\frac{1}{n}\Big)^2 - \frac{1}{(\tfrac1{n})^4} + n^2
    =
    - n^4 + n^2 + \tfrac1{n^2}
    \rarr 
    \minfty	
    \neq
    \pinfty
  \end{gather*}
  as $n \rarr \pinfty$.
  In particular the sum $F+G$ is not coercive on the 
  complementary subspace $X_1 + Y_1 = \R^2$ of $X_2 \cap Y_2 = \{\zerovec\}$.
\end{example}

\begin{example} \label{exa:non_orthogonal_decompositions}
  Consider the function $H: \R^2 \rarr \R$, given by 
  $
    H(x_1,x_2) \defeq x_1^2
  $
  and regard it with respect to the decompositions 
  \begin{gather*}
    \R^2 
      = 
      \underbrace{\Linspan{e_1}}_{\eqdef X_1}
      \oplus 
      \underbrace{\Linspan{e_2}}_{\eqdef X_2}
      =
      \underbrace{\Linspan{e_1+e_2}}_{\eqdef Y_1}
      \oplus 
      \underbrace{\Linspan{e_2}}_{\eqdef Y_2},
  \end{gather*}
  the first beeing an orthogonal one 
  and the second beeing a non orthogonal
  one. Clearly $H$ is coercive both on $X_1$ and $Y_1$. Moreover $H$ is
  bounded below on $X_2 = Y_2$ since it is even constant there.
  Setting
  \begin{align*}
    F_1 \defeq {}& H|_{X_1},  	      &  G_1 \defeq {}& H|_{Y_1}, 
    \\
    F_2 \defeq {}& H|_{X_2} \equiv 0, &  G_2 \defeq {}& H|_{Y_2} \equiv 0
  \intertext{we can write the functions $F \defeq H$ and $G \defeq H$ as semidirect sums}
    F = {}& F_1 \sdirsum F_2,     &   G = {}& G_1 \sdirsum G_2.
  \end{align*}
  In accordance with the previous theorem we see that 
  $F+G = 2H$ is coercive on any subspace $Z_1$ of $\R^2$ 
  with $\R^2 = Z_1 \oplus (X_2 \cap Y_2)$.
  However $X_1 + Y_1 = \R^2$ is none of these subspaces and 
  $F+G = 2H$ is clearly not coercive on 
  $X_1 + Y_1 = \R^2 \supseteq \linspan{e_2}$.
  
\end{example}

\clearpage   

\chapter{Penalizers and constraints in convex problems}
\label{chap:penalizers_and_constraints_in_convex_problems}

\minitoc
\section{Unconstrained perspective versus constrained perspective} 


\addtointroscoll{

This section consists of three subsections.
In subsections 
\ref{subsec:definition_of_zero_times_infty} and 
\ref{subsec:definition_of_argmin}, respectively,
different possibilities of defining
$0\cdot(+\infty)$ and the set $\argmin F$ of minimizers 
of a function $F: \R^n \rarr \R \cup \{\pinfty\}$
are dicussed among their pros and cons, respectively. 
We finially choose the definitions
\begin{align*}
  0\cdot(+\infty) & \defeq 0  
\shortintertext{and} 
  \argmin F 
  &\defeq
  \{ \check x \in \R^n: 
        F(\check x) \leq F(x) \text{ for all } x \in \R^n
  \}. 
\end{align*}
These definitions are suggested when regarding minimizations problems 
of the form 
\begin{gather*}
  F_1 + \lambda F_2 \rarr \min
\end{gather*}
from an ``uncounstrained perspecitive'', which we prefer to take 
instead of the alternative ``constrained perspecitive''.

Subsection \ref{subsec:a_kind_of_dilemma}
serves as introduction to the already discussed Subsections
\ref{subsec:definition_of_zero_times_infty} and 
\ref{subsec:definition_of_argmin},
giving a summarizing and connecting overview of the main ideas
presented there, along with 
our concept to keep the gap between the two different 
perspectives as closed as possible in the following sections.

%
%

We finally mention that we use quite often quotation 
marks in this section, usually at places where,
sometimes hidden, unanswered questions lurk. 
However these implicit questions can be ignored
when regarding this section just as motivation 
for our way of defining $0\cdot(+\infty)$ and $\argmin F$.

}

\subsection{A kind of dilemma}
\label{subsec:a_kind_of_dilemma}

Consider for a possibly empty, fixed subset $C \subseteq \R^n$ 
those pairs of mappings
\begin{align*}
  F: \R^n \rightarrow \R \cup \{+ \infty \},
  &&
  f: C \rightarrow \R,
\end{align*}
which are related in a one to one manner by
$\dom F = C$ and $F|_{\dom F} = f$.
We will also write $F = \hat f$ and $f = \check F$ to indicate
that $F$ and $f$ are related in that manner.
Two things need to be defined: $\argmin F$ and $0\cdot(+\infty)$.
If we want to take an ``unconstrained perspective'' we should define 
\begin{align*}
  \argmin F 
  \defeq
  \{ \check x \in \R^n: 
       F(\check x) \leq F(x) \text{ for all } x \in \R^n
  \},
  &&
  0\cdot (+\infty) \defeq 0.
\end{align*}
If we prefer to take a ``constrained perspective'' we should define 
\begin{align*}
  \argmin F 
  \defeq
  \{ \check x \in \dom F:
	F(\check x) \leq F(x) \text{ for all } x \in \dom F  
  \},
  &&
  0\cdot (+\infty) \defeq +\infty.
\end{align*}

The decision we have to take will turn out to be in a way a dilemma:
On the one hand we would like the minimization problems
$\argmin F$ vs. $\argmin f$
and ``especially'' the minimization problems
$\argmin F = \argmin (\Phi + \lambda \Psi)$
vs. $\argmin f = \argmin (\phi + \lambda \psi)$, 
$\lambda \in [0, +\infty)$, to be always equivalent. To this end we
should choose the definitions fitting to the constrained perspective.
On the other hand we would like to avoid a clash with a definition 
of $\argmin$ in a general situation and -- even more important -- 
want the equation $\Phi + 0 \Psi = \Phi$ to hold true.
To that end we should, however, choose the definitions from 
the unconstrained perspective.
\\[1ex]
We are aware that it is unfortunately not uncommon to define 
$\argmin$ fitting to the constrained perspective and $0 \cdot (+ \infty) \defeq 0$ fitting
to the unconstrained perspective. We try to avoid this mixture of, 
in general not equivalent, perspectives at the 
level of definitions. Instead we will follow the unconstrained 
perspective here and pursue the strategy of imposing conditions 
in our theorems that
ensure at least a weak form of equivalence between the unconstrained
and the constrained perspective.
For instance conditions like 
$\dom \Phi \cap \dom \Psi \not = \emptyset$ in 
\prettyref{thm:constraint_vs_nonconstraint} ensure  
$F\defeq \Phi + \lambda \Psi \not \equiv +\infty$ for 
$\lambda \in [0, \infty)$, so that 
the unconstrained and the constrained perspective 
of the minimization problem are equivalent here, 
at least in the sense of 
$\argmin F = \argmin f$; 
for $\lambda \in (0, + \infty )$ we even have equivalence in a 
stronger sense, since 
\begin{gather*}
  (\Phi + \lambda \Psi){\check {}} = \check \Phi + \lambda \check \Psi.
\end{gather*}
holds in addition.
This is, however, no longer true for $\lambda = 0$, if 
$\dom \Psi \not \supseteq \dom \Phi$.
It is the price we have to pay to ensure $\Phi + 0 \Psi = \Phi$
without putting further assumptions like 
$\dom \Psi \supseteq \dom \Phi$. Note that a more general version 
of this inclusion, was assumed by Rockafellar in his chapter on 
Ordinary Convex Problems and Lagrange multipliers, 
cf. \cite[p. 273]{Rockafellar1970}.
\\[1ex]
\rem{Es ist unklar was ``equivalent'' hier genau bedeuten soll.
Verwandt und ebenfalls unklar: Wie formalisiert man ``Minimierungsproblem''
gescheit? [Einige Ansaetze und Ideen: Handzettel vom 20. 11. 2012];
AUCH: SCAN aus Optimierungsbuch}
\rem{Noch nicht scharf herausgearbeitet und mit obiger Bem. zusammenhaengend: 
Die Vorstellung, der ``unrestringierten Sichtweise'' vs. 
der ``restringierten Sichtweise'' und die Zugehoerigkeit der Def.
$0 (+\infty) \defeq 0$ zur unrestringierten Sichtweise, sowie der
Def. von $\argmin F\defeq{..\in \dom F:..}$ zur restringierten Sichtweise}
%
%
%
    
The following table gives a summarized overview. Some details can be found in the next subsections.
\\ 
\begin{tabu}{l | l | l } \label{tab:unconstrained_vs_constrained_perspecive}
                                          & unconstrained perspective                                  & constrained perspective  \\
  \hline                        
  Definition of $0\cdot(+\infty)$         & $0$                                                        & $ + \infty$\\
  Definition of $\argmin F$               & $\{\check x\in \R^n:$                                      & $\{\check x\in \dom F: $\\
                                          & $ \forall x \in \R^n: F(\check x) \leq F(x) \}$            & $\forall x \in \dom F: F(\check x) \leq F(x) \}$\\
  \hline  
  $\argmin F = \argmin f$                 & for $F \not \equiv +\infty$                                & always\\
  $\argmin \{F_1 + 
      \iota_{{\rm lev}_{\tau}F_2}\} =$    & for $\dom F_1 \cap {\rm lev}_{\tau}F_2 \not = \emptyset$   & always\\
  $\argmin \{F_1 \st
      F_2 \leq \tau \}$&                                                      & \\
  \hline
  $(F_1 + \lambda F_2)\check {} 
      = \check {F_1} + \lambda \check {F_2}$      & for $\lambda \in \R \setminus \{0\}$               & for every $\lambda \in \R$\\  
  $F_1 + 0F_2 = F_1$                      & always true                                                & only true if $\dom F_2 \supseteq \dom F_1$\\
  \hline
  $F$ lsc $\Rightarrow$ $\lambda F$ lsc   & for $\lambda \in [0, +\infty)$                             & in general only for $\lambda \in (0, +\infty)$\\
  \hline
\end{tabu}

\subsection[Definition of $0\cdot(+\infty)$]
  {Definition of $\boldsymbol{0\cdot(+\infty)}$} \label{subsec:definition_of_zero_times_infty}
Let $\phi: C_\phi \rightarrow \R$ and $\psi: C_\psi \rightarrow \R$
be mappings with domains $C_\phi \subseteq \R^n$ and 
$C_\psi \subseteq \R^n$, respectively,
and let $\Phi \defeq \hat\phi$ and $\Psi \defeq \hat \psi$ denote their
natural continuations to functions 
$\R^n \rightarrow \R \cup \{+\infty \}$.
In the constrained perspective we want
$\Phi + \lambda \Psi$ to be the ``exact'' twin of 
$\phi + \lambda \psi$ for all $\lambda \in [0, \infty)$,
i.e. we want 
\begin{gather*}
  (\Phi + \lambda \Psi){\check {}} 
  = 
  \check \Phi + \lambda \check \Psi
  = 
  \phi + \lambda \psi
\end{gather*}
to hold true.
For $\lambda \in (0, + \infty)$ this equation is always fulfilled.
For $\lambda = 0$ it is however in general only true, if we would set
$0\cdot(+ \infty)$ to be $+ \infty$; choosing any other value 
from $[0, + \infty)$ for this product, let us say the value $0$,
would cause the domain of definition of 
$(\Phi + 0 \Psi){\check {}}$ to be different from 
the domain of definition of $\check \Phi + 0 \check \Psi$, if
$\dom \Psi \not \supseteq \dom \Phi$:
Here the domain of definition of 
$(\Phi + 0 \Psi){\check {}} = \check \Phi$ 
equals 
$C_\phi = \dom \Phi$, whereas the domain of definition of
$\check \Phi + 0 \check \Psi$ is 
$C_\phi \cap C_\psi = \dom \Phi \cap \dom \Psi \subset \dom \Phi$.
\\[1ex]
In the unconstrained perspective we concede $\Phi + \lambda \Psi$
a mode of being that is beyond being a copy of 
$\phi + \lambda \psi$, made up for technical purposes;
Here we consider $\Phi, \Psi$ and $\Phi + \lambda \Psi$ in first line
``really'' as mappings $\R^n \rightarrow \R \cup \{+ \infty\}$
which all have the same domain of definition. This allows us to achieve
$\Phi + 0 \Psi = \Phi$ by setting
\begin{gather*}
  0\cdot(+\infty) \defeq 0.
\end{gather*}
With this definition we accept that the identity
$
  (\Phi + \lambda \Psi){\check {}} 
  = 
  \check \Phi + \lambda \check \Psi
  = 
  \phi + \lambda \psi
$
may fail for $\lambda = 0$.
\\[1ex]
Finally we remark that our definition of $0\cdot (+\infty)$
seems to be the ``correct'' one from the viewpoint of lower semicontinous
functions: If $\Psi : \R^n \rightarrow \R \cup \{+ \infty\}$ is
lower semicontinuous then so is $\lambda \Psi$ for all 
$\lambda \in (0,+\infty)$ and also for $\lambda = 0$, thanks to 
our definition $0\cdot(+ \infty) \defeq 0$.
Note that lower semicontinuity would, however, in general not be 
preserved, if we had chosen $0\cdot(+ \infty)$ to be $+ \infty$
in the constrained perspective's sense:
Consider the function $\psi: (0, +\infty) \rightarrow \R$, given by
$\psi(x) \defeq \frac{1}{x}$. Its natural continuation 
$\Psi \defeq \hat \psi : \R^n \rightarrow \R \cup \{+ \infty\}$ is lower 
semicontinous, but its product $0\cdot \Psi$ 
(in the constrained perspective's sense!) would not be lower 
semicontinous, since its epigraph would be the non-closed set 
$(0, + \infty)\times [0, + \infty)$.

\subsection[Definition of $\argmin$]
    {Definition of $\boldsymbol{\argmin}$} \label{subsec:definition_of_argmin} 

Let $f: C \rightarrow \R$ be some real-valued function,
defined on some subset $C \subset \R^n$ and let 
$F \defeq \hat f$ be its natural
continuation to a function $\R^n \rightarrow \R \cup \{+\infty\}$. 
\\[1ex]
In the constrained perspective we regard $F$ as a kind of 
working copy of $f$; in particular we want the equation 
$\argmin F = \argmin f$
to hold always true. Defining $\argmin F$ as
$
  \{ \check x \in \dom F:
	F(\check x) \leq F(x) \text{ for all } x \in \dom F  
  \}
$
would do the job.
\\[1ex]
In the unconstrained perspective we, however, want to minimize $F$
``really'' over $\R^n$, its whole domain of definition, so that we
define
\begin{gather*}
  \argmin F 
  \defeq
  \{ \check x \in \R^n: 
        F(\check x) \leq F(x) \text{ for all } x \in \R^n
  \}  
\end{gather*}
We then still have $\argmin F = \argmin f$, except for the 
particular case $F \equiv + \infty$ where we unfortunately
get $\argmin F = \R^n \not = \emptyset = \argmin f$.
\\[1ex]
Despite this small disadvantage we nevertheless define $\argmin F$
according to the unconstrained perspective -- not only because 
we had already decided us for this perspective when defining
$0\cdot(+\infty) \defeq 0$ but also for the sake of consistency
with the definition of $\argmin$ in the following more general situation:
Assume we want to define $\argmin H$ for a quite general function
$H: X \rightarrow Y$  between a (possibly empty) set $X$ and a 
totally ordered set $(Y, \leq_Y)$. 
The natural choice for defining the 
(possibly empty) set of minimizers seems to be 
\begin{gather*} \label{eq:def_of_argmin}
  \argmin H 
  \defeq 
  \{ \check x \in X : H(\check x) \leq_Y H(x) \text{ for all } x \in X\}.
\end{gather*}
Our de facto definition of $\argmin F$ appears then just as a special 
case for $X = \R^n$, $Y = (-\infty, + \infty]$ with the natural order
and $H = F$. In contrast, the rejected, constrained perspective way
of defining $\argmin F$ would clash to the general definition 
for $F \equiv +\infty$.
\\[1ex]
We conclude this section with a remark
to the constrained optimization problem
\begin{gather*}
  \argmin \{F_1 \st F_2 \leq \tau \}
  \defeq
  \{\check x \in \R^n : F_2(\check x) \leq \tau \text{ and } 
  F_1(\check x) \leq F_1(x) \text{ for all } x\in {\rm lev}_{\tau}F_2\},
\end{gather*}
where $\tau \in \R$ and
$F_1,F_2: \R^n \rightarrow \R \cup \{+\infty\}$. In the 
constrained perspective we can always rewrite it to
$\argmin \{F_1 + \iota_{{\rm lev}_{\tau}F_2}\} $. In the 
unconstrained perspective we can do this however only if 
$F_1 + \iota_{{\rm lev}_{\tau}F_2} \not \equiv +\infty$, 
i.e. if the overlapping condition
$\dom F_1 \cap {\rm lev}_{\tau}F_2 \not = \emptyset$ is fulfilled.
A similar condition which ensures a stronger overlapping 
between $\dom F_1$ and ${\rm lev}_{\tau}F_2$ is used in part i) of
\prettyref{thm:constraint_vs_nonconstraint}.
The question is also if we should at all speak of the 'constrained
problem' $\argmin \{F_1 \st F_2 \leq \tau \}$, 
defined as above, in the context of our unconstrained perspective,
or if we should consider just the problem 
$\argmin F_1 + \iota_{{\rm lev}_{\tau}F_2}$ instead.


\section{Penalizers and constraints} \label{sec:penalizers_and_constraints}

\addtointroscoll{

This section consists of three subsections:
In the first subsection we review 
general relations between the constrained problem 
\begin{align} \label{eq:repeatet_constraint} 
  &(P_{1,\tau}) \qquad \argmin_{x \in \mathbb R^n} \left\{ \Phi(x) \st \Psi(x) \le \tau \right\}
\intertext{and the unconstrained, penalized problem}
\label{eq:repeated_nonconstraint} 
  & (P_{2,\lambda}) \qquad \argmin_{x \in \mathbb R^n} \{ \Phi(x) +  \lambda \Psi(x) \}, \;\; \lambda \ge 0.
\end{align}
This relation is stated in Detail in
\prettyref{thm:constraint_vs_nonconstraint}.
In the second subsection we add to a primal problem, which 
can be the constrained or the penalized problem,
the corresponding Fenchel Dual problem along with 
conditions that characterize their solutions.
In the third subsection we discuss  
\prettyref{thm:constraint_vs_nonconstraint}. In particular a 
relation between one of its assumptions and Slater's Constraint 
Qualification is given.

}

\subsection{Relation between solvers of constrained and 
  penalized problems  
  } 

\MayChangesPartiallyPerformedVersionOrHareBrainedOfficialVersion{%
\diss{Bei finalem Notationsabgleich: 
Bezeichnungen in Lemma: Statt $x^*$ besser $\hat x$ ?}
}%
{}%
In this subsection there are two lemmas and one theorem along with
their proofs and some examples.
The first Lemma \ref{aux_lemma} is an auxiliary lemma for
the second Lemma \ref{level_sets}. The latter lemma gives 
a relation between the subgradients $\partial \Psi(x^*)$ and 
$\partial \iota_S(x^*)$, where $S\defeq {\rm lev}_{\Psi(x^*)}\Psi$.
This relation is used to prove
\prettyref{thm:constraint_vs_nonconstraint},
which gives relations between solvers of ${{\rm SOL}}(P_{1,\tau})$
and ${\rm SOL}(P_{2,\lambda})$.
For comments on this subsection see  
\prettyref{sec:Notes_to_pvc_relation_theorem}.
%
\begin{lemma} \label{aux_lemma}
Let $\Psi: \R^n \rightarrow \R \cup \{+\infty\}$ be a proper and 
convex function, 
$x^* \in {\rm dom} \Psi$ and
$S \defeq {\rm lev}_{\Psi(x^*)} \Psi $. 
Let $p \in \R^n$ such that the half-space
$H_{p,\alpha}^{\le}$ 
with $\alpha \defeq \langle p,x^* \rangle$ contains $S$.
Then we have the equality
\begin{gather} \label{an_equality}
\inf_{ x \in H_{p,\alpha}^= } \Psi(x) = \Psi(x^*),
\end{gather}
if $x^* \in {\rm int}({\rm dom} \Psi)$ or 
if both $x^* \in {\rm ri}({\rm dom} \Psi)$ and $S$ is not completely
contained in $H_{p,\alpha}^{=}$.
\end{lemma}
%
{\bf Proof.} 
For $n = 0$ the assertion of the Lemma is trivially true.
Without loss of generality we may therefore assume $n \geq 1$ 
in the following.
We first consider the case $x^* \in {\rm int}({\rm dom} \Psi)$.
Assume that there exists $y \in H_{p,\alpha}^=$ 
such that $\Psi(y) < \Psi(x^*)$. Since $y,x^* \in \dom \Psi$, 
we see by the convexity of $\Psi$ that
$$
\Psi(\underbrace{\lambda y + (1-\lambda) x^*}_{\eqdef x_\lambda}) \le \lambda \Psi(y) +  (1-\lambda) \Psi( x^*) < \Psi( x^*) 
$$
for all $\lambda \in (0,1)$.
Since 
$x^* \in {\rm int}({\rm dom} \Psi)$
we have 
$x_\lambda \in {\rm int}({\rm dom} \Psi)$ for $\lambda$ small enough.
Since $\Psi$ is continuous on ${\rm int}({\rm dom} \Psi)$, 
there exists $\varepsilon>0$ such that the Euclidean ball
$\closedball[\varepsilon][x_\lambda]$ centered at $x_\lambda$ 
with radius $\varepsilon$ fulfills $\closedball[\varepsilon][x_\lambda] \subseteq {\rm int}(\dom \Psi)$ and
$\Psi(x) < \Psi(x^*)$ for all $x \in \closedball[\varepsilon][x_\lambda]$.
Hence we obtain by the assumption on $S$ and $p$ the inclusion
$\closedball[\varepsilon][x_\lambda] \subseteq S \subseteq H_{p,\alpha}^{\le}$ 
so that
$\closedball[\varepsilon][x_\lambda] \cap H_{p,\alpha}^{>} = \emptyset$.
This contradicts $x_\lambda \in H_{p,\alpha}^=$. 
\\
The remaining case can be reduced to this argument:
Without loss of generality we may assume $x^*$
to be the point of origin, so that $H_{p,\alpha}^{=}$ and 
${\rm aff}(\dom \Psi) \eqdef U$  are vector subspaces of $\R^n$;
note herein 
$x^* \in {\rm ri}(\dom \Psi) \subseteq {\rm aff}(\dom \Psi)$.
For simplicity of perception we may without loss of generality assume further,
that $p$ is of the form $p=(0,\dots, 0, 1)$, i.e.
$H_{p,\alpha}^{=} = \R^{n-1} \times \{0\}$ and
$H_{p,\alpha}^{\leq} = \R^{n-1} \times (-\infty,0]$.
The level set $S \subseteq H_{p,\alpha}^{\leq}$ is not completely contained in
$H_{p,\alpha}^{=}$. Therefore $H_{p,\alpha}^{=}$, or rather 
$H^{=} \defeq H_{p,\alpha}^{=} \cap U$, must separate $U$ in an 
upper part $H^{\geq} \defeq H_{p,\alpha}^{\geq} \cap U$ and a 
lower part $H^{\leq} \defeq H_{p,\alpha}^{\leq} \cap U$;
note here that $H^{\geq}$ is a hyperplane in $U = {\rm aff}(\dom \Psi)$
by Detail 
\ref{det:hyperplane_intersection_with_subspace_yields_here_again_a_hyperplane}. 
Due to $H^{\leq} \supseteq S$ and since
$
  \inf_{ x \in H_{p,\alpha}^= } \Psi(x) 
  = 
  \inf_{ x \in H^= } \Psi(x) 
$
we can consider $\Psi$ only on $U = {\rm aff}(\dom \Psi)$ 
and then argue just as before in this vector subspace, using
$x^*$ to be an interior point of $S$ (considered of course as subset of $U$).
\hfill $\Box$
\\[2ex]

\begin{remark} \label{rem:infimum_on_hyperplane_touching_levelset}
~
  \begin{enumerate}
    \item \label{enu:touching_point_must_be_in_interior_of_domain}
      In cases where ${\rm aff}(\dom \Psi)$ is the full space $\R^n$,
      i.e. where ${\rm int}(\dom \Psi) = {\rm ri}(\dom \Psi)$, 
      the condition $x^* \in {\rm int}({\rm dom} \Psi)$ is, in general,
      really necessary to get the equality \eqref{an_equality}
      as Fig. \ref{bound_1} illustrates.
      \begin{figure}[htpb] 
      \begin{center}
	\includegraphics[width=0.35\textwidth]{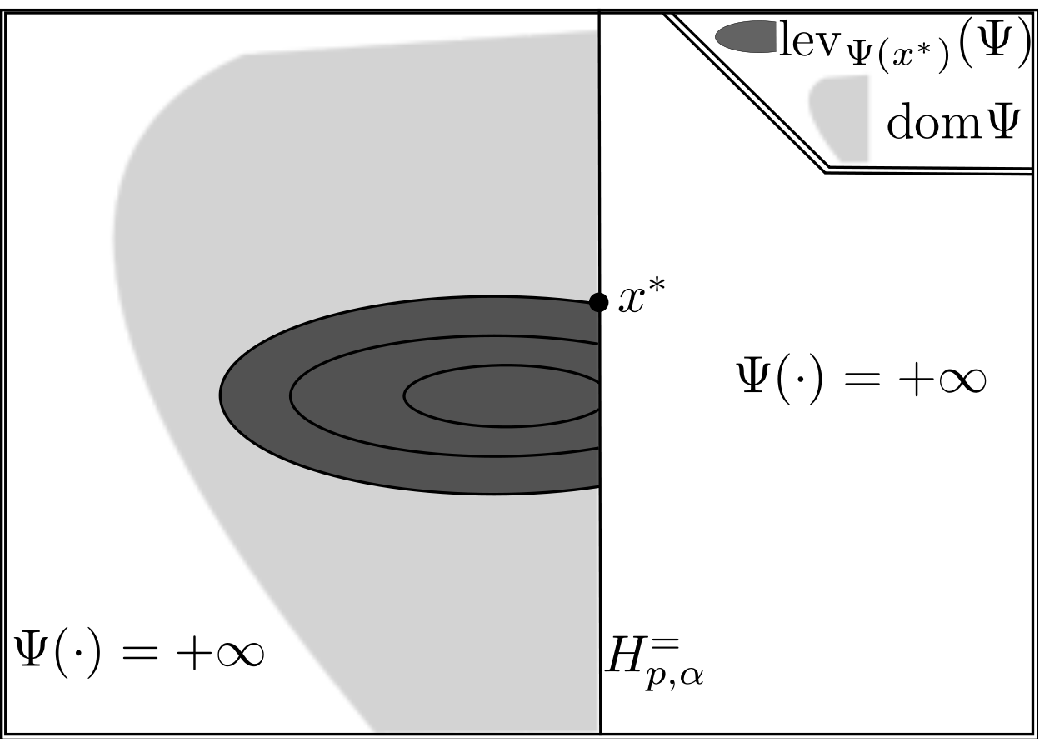}
	\caption{Illustration that relation \eqref{an_equality} is in general not valid for $x^* \in {\rm dom} \Psi \backslash {\rm int}(\dom \Psi)$. \label{bound_1}}
	\label{fig:inf_not_attained_when_POI_not_in_int}
      \end{center}
      \end{figure}
    \item \label{enu:touching_point_must_be_in_ri_of_domain_and_hyperplane_must_be_placed_well}
      In cases where ${\rm aff}(\dom \Psi) \subset \R^n$, i.e. where
      ${\rm int}(\dom \Psi) = \emptyset$, the condition 
      $x^* \in {\rm ri}(\dom \Psi) $ in general really needs to be 
      complemented by the condition 
      $S \not \subseteq H^{=}_{p,\alpha}$ 
      to get the equality \eqref{an_equality}, see the second part of
      Remark
      \ref{rem:relation_between_subdifferentials_not_full}
    or make the following gedankenexperiment:
    Look at Figure
    \ref{fig:inf_not_attained_when_POI_not_in_int}
    and regard the two dimensional effective domain of $\Psi$
    as $x_1$-$x_2$-plane of $\R^3$, i.e.
    extend the there sketched function 
    $\Psi: \R^2 \rightarrow \R \cup \{+\infty\}$
    to a function $\hat \Psi: \R^3 \rightarrow \R \cup \{+\infty\}$ 
    by setting 
    \begin{gather*}
      \hat \Psi(x_1,x_2,x_3) \defeq 
      \begin{cases}
	\Psi(x_1,x_2) & \text{ if } x_3 = 0\\
	+ \infty      & \text{ if } x_3 \not = 0.
      \end{cases}
    \end{gather*}
    Move now $x^*$ and the line  $H^=_{p,\alpha}$ to some place in 
    ${\rm ri}(\dom \Psi) \setminus \argmin \Psi$
    but change the direction of $H^=_{p,\alpha}$, if necessary,
    in such a way that we still have
    $S \defeq {\rm lev}_{\Psi(x^*)} \Psi \subseteq H^\leq_{p,\alpha}$.
    Consider finally the line $H^=_{p,\alpha}$ as part of a plane
    $\hat H^=_{\hat p, \hat \alpha}$ with 
    $\hat p \in \R^3 \setminus \{\zerovec\}$ and 
    $\hat \alpha \defeq \langle \hat p,  x^*\rangle$. 
    As long as we consider only such planes 
    $\hat H^=_{\hat p, \hat \alpha}$ which are not identical to 
    the $x_1$-$x_2$-plane ${\rm aff}(\dom \Psi)$,
    but intersect this plane
    only in $H^=_{p,\alpha}$, everything keeps essentially the same as
    before: Also $\hat H^=_{\hat p, \hat \alpha}$ separates $\dom \Psi$
    at $x^* \in {\rm ri}(\dom \Psi)$ into two parts, such that 
    $S$ is completely contained in 
    $\hat H^\leq_{\hat p, \hat \alpha}$.
    Such a separation is, however, no longer performed by 
    $\hat H^=_{\hat p, \hat \alpha}$ if it is identical to the 
    $x_1$-$x_2$-plane. In this case equation 
    \eqref{an_equality} is clearly no longer fulfilled.

  \end{enumerate}
\end{remark}
\vspace{0.2cm}

The following lemma will be used in our proof of \prettyref{thm:constraint_vs_nonconstraint}.
%
\begin{lemma} \label{level_sets}
Let $\Psi: \R^n \rightarrow \R \cup \{ +\infty\}$ be a proper, convex function,  $x^* \in {\rm dom} \Psi$ and $S \defeq {\rm lev}_{\Psi(x^*)} \Psi$. 
Then we have 
\begin{equation} \label{wichtig}
\R_0^+ \, \partial \Psi(x^*) \subseteq \partial \iota_{S} (x^*).
\end{equation}
If $x^*$ is not a minimizer of $\Psi$ we moreover have
\begin{alignat}{3}
\label{eq:subgradient_relation_with_indicatorfunction_for_ri_point}
  \partial \iota_S(x^*)
  =
  \overline{\R_0^+ \partial\Psi(x^*)} 
    & \; \text { if } x^* \in {\rm ri}(\dom \Psi),
\\
\label{eq:subgradient_relation_with_indicatorfunction_for_interior_point}
  \partial \iota_S(x^*)
  =
  \R_0^+ \partial\Psi(x^*) 
    & \; \text { if } x^* \in {\rm int}(\dom \Psi), 
	  \text { or in other words } 
\\ 
    & \;    \text { if } x^* \in {\rm ri}(\dom \Psi)
	    \text{ and } {\rm aff}(\dom \Psi) = \R^n.   
\nonumber
\end{alignat}
\end{lemma}
%
\vspace{0.2cm}
A proof of a similar lemma for finite functions $\Psi: \mathbb R^n \rightarrow \R $ based on cone relations 
can be found, e.g., in \cite[p. 245]{HL93}.
Here we provide a proof which uses the {\it epigraphical projection}, also known as {\it inf-projection} as defined in \cite[p. 18+, p. 51]{RW04}.
For a function $f: \mathbb R^n \times \mathbb R^m \rightarrow \mathbb R \cup \{ +\infty \}$, the inf-projection is defined by
$
\nu (u) \defeq \inf_x f(x,u).
$
The name 'epigraphical projection' is due to the following fact:
$\epi \nu$ is the image of $\epi f$ under the projection $(x,u,\alpha) \mapsto (u,\alpha)$, if
$\argmin _x f(x,u)$ is attained for each $u \in \dom \nu$.
(Note that this is not the projection onto epigraphs as used, e.g., in \cite[p. 427]{BC11}.)
The inf-projection is convexity preserving, i.e.,
if $f$ is convex, then $\nu$ is also convex, cf. \cite[Proposition 2.22]{RW04}.
\\[2ex]
{\bf Proof.}  
1. First we show that $\R_0^+ \, \partial \Psi(x^*) \subseteq \partial \iota_{S} (x^*)$.
By definition of the subdifferential we obtain 
\begin{eqnarray*} \label{lhs}
 q \in \partial \Psi(x^*) 
&\Longleftrightarrow& \quad \forall x \in \R^n: \;\, \langle q, x-x^* \rangle \le \Psi(x) - \Psi(x^*),\\
&\Longrightarrow& \quad \forall x \in S: \quad \langle q, x-x^* \rangle \le 0 \nonumber
\end{eqnarray*}
Hence we obtain the above inclusion by
\begin{equation} \label{rhs}
 p \in \partial \iota_{S} (x^*) \; \Longleftrightarrow \; \forall x \in S: \; \langle p, x-x^* \rangle \le 0.
\end{equation}
2. Next we prove 
$\partial \iota_{S} (x^*) \subseteq \R_0^+ \, \partial \Psi(x^*) $ 
if $x^*$ is not a minimizer of $\Psi$ and the additional assumptions in 
\eqref{eq:subgradient_relation_with_indicatorfunction_for_interior_point} 
are fulfilled, so that $x^* \in {\rm int}({\rm dom} \Psi)$. 
Let $p \in \partial \iota_{S} (x^*)$. If $p$ is the zero vector, then we are done since $\partial \Psi(x^*) \not = \emptyset$.
In the following we assume that $p$ is not  the zero vector. 
It remains to show that there exists $h>0$ such that $\frac{1}{h} p \in \partial \Psi(x^*)$.
We can restrict our attention to $p =(0,\ldots,0,p_n)^\tT$ with $p_n >0$.
(Otherwise we can perform a suitable rotation of the coordinate system.)
Then \eqref{rhs} becomes
\begin{equation} \label{drei}
p \in \partial \iota_{S} (x^*) \; \Longleftrightarrow \; \forall x = (\bar x,x_n) \in S: \; p_n x_n \le p_n x_n^*.
\end{equation}
Hence we can apply lemma \ref{aux_lemma} with $p =(0,\ldots,0,p_n)^\tT$ and obtain 
$$
\inf_{ \{ x\in \mathbb R^n:x_n = x_n^*\}} \Psi (x) = \Psi(x^*).
$$
Introducing the inf-projection
$\nu: \R \rightarrow \R\cup \{ \pm \infty \}$ by
$$
\nu (x_n) \defeq \inf_{\bar x \in \R^{n-1}} \Psi(\bar x,x_n).
$$
this can be rewritten as
\begin{equation} \label{show_1}
 \nu (x_n^*) =  \Psi(x^*).
\end{equation}
Therefore we have
\begin{align*}
  \frac{1}{h} p = (0,\ldots,0,\frac{1}{h} p_n)^\tT \in \partial \Psi(x^*) 
   &\;  \Longleftrightarrow \;
      \forall x \in \mathbb R^n: \Psi(x) \ge \nu(x_n^*) + \frac1h p_n(x_n-x_n^*) 
      \\
   &\;  \Longleftrightarrow \;
      \forall x_n \in \mathbb R: \nu(x_n)  \ge \nu(x_n^*) + \frac1h p_n(x_n-x_n^*)  
      \\
   &\;  \Longleftrightarrow \;
      \frac{1}{h}p_n \in \partial \nu(x_n^*),
\end{align*}
so that it remains to show that $\partial \nu(x_n^*)$ contains a positive number.
By \eqref{show_1} we verify that $\nu(x_n^*)$ is finite. Moreover, 
$x^* \in {\rm int}({\rm dom} \Psi)$ implies 
 $x_n^* \in {\rm int}({\rm dom}  \nu)$.
Therefore $\partial \nu(x_n^*) \not = \emptyset$. 
Let $q_n \in \partial \nu(x_n^*)$, i.e.,
$$
q_n(x_n - x_n^*) \le \nu(x_n) - \nu(x_n^*) 
$$
for all $x_n \in \R.$
Since $x^*$ is not a minimizer of $\Psi$, there exists $y \in \R^n$ with 
$\Psi(y) < \Psi(x^*)$
and we get by \eqref{drei} that $y_n \le x_n^*$.
Since $y_n = x_n^*$ would by \eqref{show_1} 
imply that
$\Psi(x^*) = \nu(y_n) \le \Psi(y)$, we even have $y_n < x_n^*$.
Thus
$$
q_n(y_n -  x_n^*) \le \nu(y_n) - \nu(x_n^*) \le \Psi(y) - \Psi(x^*) < 0
$$
implies $q_n > 0$ and we are done. 
\\
3. Next we prove 
$\partial \iota_{S} (x^*) \subseteq \overline{\R_0^+ \, \partial \Psi(x^*)} $ 
if $x^*$ is not a minimizer of $\Psi$ and $x^* \in {\rm ri}({\rm dom} \Psi)$;
then taking closures in
$
  \R_0^+ \, \partial \Psi(x^*)
  \subseteq
  \partial \iota_{S} (x^*) 
  \subseteq 
  \overline{\R_0^+ \, \partial \Psi(x^*)} 
$ 
gives the wanted 
$ 
  \partial \iota_S(x^*) 
  = 
  \overline{\R_0^+ \partial\Psi(x^*)} 
$
since $\partial \iota_S(x^*)$ is closed.
\\
We have $x^* \in \dom {\iota_S} = S \subseteq \dom \Psi$, so that 
both effective domains are in particular contained in 
${\rm aff}(\dom \Psi) \eqdef A$. Applying Theorem 
\ref{thm:subdifferential_for_functions_with_not_full_dim_dom}
two times yields hence
\begin{alignat*}{2}
  \partial\iota_S(x^*) &= \partial(\iota_S|_A)(x^*) &&+ U^\perp,
\\
  \partial\Psi(x^*) &= \partial(\Psi|_A)(x^*) &&+ U^\perp,
\end{alignat*}
where $U$ is the difference space of $A$.
By part 2. of the proof we know 
$\partial(\iota_S|_A)(x^*) = \R_0^+\partial(\Psi|_A)(x^*)$.
So the claimed
$\partial\iota_S(x^*) \subseteq \overline{\R_0^+\partial\Psi(x^*)}$
is equivalent to 
$
  \R_0^+\partial(\Psi|_A)(x^*)
  + 
  U^\perp
  \subseteq
  \overline{\R_0^+[\partial(\Psi|_A)(x^*) + U^\perp]}
$
and can hence be proved by showing that the relation 
\begin{gather*}
  \R_0^+ B + W \subseteq \overline{\R_0^+ (B+W)}
\end{gather*}
holds true for any subsets $B, W$ of $\R^n$ with $\R_0^+W = W$.
To this end let $\lambda \in \R_0^+$, $b\in B$ and $w \in W$ be 
given. In case $\lambda \not = 0$ we have
$
  \lambda b + w 
  = \lambda(b + \lambda^{-1}w) 
  \in 
  \R_0^+(B+W)
  \subseteq
  \overline{\R_0^+(B+W)}
$.
In case $\lambda = 0$ we have
$
  \lambda b + w
  =
  0b + w
  =\lim_{k\rightarrow \infty} \frac{1}{k}(b + k w) 
  \in 
  \overline{ \R_0^+ (B+W)}
$.
Thus $\R_0^+B + W \subseteq \overline{\R_0^+(B+W)}$ really holds true.
\hfill $\Box$
\\

\begin{remark} \label{remark_1}
The condition that $x^*$ is not a 
minimizer of $\Psi$ is essential to have equality in \eqref{wichtig}
as the following example illustrates.
The function $\Psi$ given by $\Psi(x) = x^2$ is minimal at
$x^* = 0 \in {\rm int}(\dom \Psi)$. 
We have $S \defeq {\rm lev}_{\Psi(0)} \Psi = \{0\}$ so that
$$
\R_0^+ \partial \Psi(x^*) = \{ 0 \} \subset \R = \partial \iota_{S} (x^*).
$$
\end{remark}

\begin{remark} \label{rem:relation_between_subdifferentials_not_full}
{\rm i)} The condition $x^* \in \dom \Psi$ is not sufficient to get equality in \eqref{wichtig}.
Consider the proper, convex, lower semicontinuous function $\Psi$ given by
$$
\Psi(x) \defeq 
\left\{
\begin{array}{ll}
-\sqrt{x}& {\rm if} \; x\ge 0,\\
+\infty & {\rm if} \; x<0.
 \end{array}
\right.
$$

The point $x^* = 0$ is not a minimizer of $\Psi$ and
belongs to $\dom \Psi$ but not to ${\rm ri} (\dom \Psi)$.
Using $S \defeq {\rm lev}_{\Psi(0)} \Psi = \R_0^+$
we see that
$$
\R_0^+ \partial \Psi(x^*) = \emptyset \subset  (-\infty,0] = \partial \iota_{S} (x^*).
$$
{\rm ii)} Even the condition $x^* \in {\rm ri}(\dom \Psi)$ is 
  not sufficient to guarantee equality in \eqref{wichtig}, 
  if ${\rm aff}(\dom \Psi)$ is not the full space $\R^n$:
  Consider the proper, convex and lower semicontinuous 
  function $\Psi:\R^2 \rightarrow \R \cup \{+\infty\}$, 
  given by
  \begin{gather*}
    \Psi(x_1,x_2) \defeq
    \begin{cases}
      x_1     & \text{ if } x_2      = 0, \\
      +\infty & \text{ if } x_2 \not = 0.
    \end{cases}
  \end{gather*}
  The affine hull ${\rm aff}(\dom \Psi) = \R \times \{0\}$
  is a proper subset of $\R^2$. We have 
  $S\defeq{\rm lev}_{\Psi(x^*)} \Psi = (-\infty, x^*_1] \times \{0\}$
  for arbitrarily chosen  
  $
    x^* = (x^*_1,0) 
    \in 
    \R \times \{0\}
    = 
    {\rm aff}(\dom \Psi)
    =
    {\rm ri }(\dom \Psi) 
  $.
  Applying Theorem 
  \ref{thm:subdifferential_for_functions_with_not_full_dim_dom}
  to ${\rm aff}(\dom \Psi)\eqdef A \eqdef U$ yields  
  \begin{align*}
    \partial \iota_S(x^*)
    &=
    \partial (\iota_S|_A)(x^*) + U^\perp
    =
    \R_0^+(1,0)^T + \R(0,1)^T
    \\
    &=
    \{(p_1,p_2)^T: p_1 \in [0, +\infty), p_2 \in (-\infty, +\infty) \}
  \end{align*}
  and
  $
    \partial\Psi(x^*)
    =
    (1,0)^\tT + \R(0,1)^\tT
    =
    \{(1,p_2)^\tT: p_2 \in (-\infty, +\infty) \}
  $
  so that
  \begin{gather*}
    \R_0^+\partial\Psi(x^*)
    =
    \{(0,0)^\tT\}
    \cup
    \{(p_1,p_2)^\tT: p_1 \in (0, +\infty), p_2 \in (-\infty, +\infty) \}.
  \end{gather*}
  We see that the closure $\overline{\R_0^+\partial \Psi}$
  is just the closed half-plane $\partial \iota_S(x^*)$,
  as guaranteed by Lemma \upref[]{level_sets}.
  However we only have
  $\R_0^+ \partial \Psi(x^*) \subset \partial \iota_S(x^*)$.
  \\[1ex]
  Concerning Lemma \ref{aux_lemma} we note 
  that equation \eqref{an_equality} holds true here if and only if
  $S$ is not completely contained in the straight line 
  $H^=_{p,\alpha(p)}$, where $\alpha (p)  \defeq \langle p, x^* \rangle$:
  Choosing any $p = (p_1, p_2)$ with $p_1 > 0$  we see that the 
  line $H^=_{p,\alpha(p)}$ intersects ${\rm aff}(\dom \Psi)$ only in 
  $x^*$, so that we clearly have
  $\inf_{ x \in H_{p,\alpha}^= } \Psi(x) = \Psi(x^*)$.
  However this equation is no longer fulfilled if we choose 
      $p$ in such a way that ${\rm aff}(\dom \Psi) \subseteq H^=_{p,\alpha}$,
      say e.g. $p=(0,1)$.

\end{remark} 

Using Lemma  \ref{level_sets} it is not hard to prove 
the following \prettyref{thm:constraint_vs_nonconstraint} on the correspondence between 
the constrained problem ($P_{1,\tau}$)  in \eqref{eq:intro_constraint} 
and the penalized problem ($P_{2,\lambda}$) in \eqref{eq:intro_nonconstraint}.
%
The core part of the theorem has been restated in 
\prettyref{cor:constraint_vs_nonconstraint} for 
proper, convex functions 
$\Phi, \Psi: \R^n \rightarrow \R \cup \{+\infty\}$,
where the effective domain $\dom \Psi$ is open and contains $\dom \Phi$.
In that case the theorem basically states, on the one hand,
in its second part the following: For $\lambda > 0$ any 
$\hat x \in {\rm SOL}(P_{2,\lambda})$ which does not minimize 
$\Phi$, belongs also to 
${{\rm SOL}}(P_{1,\tau})$ exactly for $\tau = \Psi(\hat x)$.
In its first part, on the other hand, it then states something converse: 
For any $\hat x \in  {\rm SOL} (P_{1,\tau})$ which does neither minimize $\Phi$ nor $\Psi$, 
there exists $\lambda >0$ such that
$\hat x \in {\rm SOL} (P_{2,\lambda})$.
To determine this $\lambda$ we will later use duality considerations.
In the following theorem we give the rigorous statement and take also 
the case $\lambda = 0$ into account in both parts of the theorem; 
note here however that the second part of the theorem
does not state that for given $\hat x \in {\rm SOL}(P_{2,0})$
there actually is a $\tau \in \R$ with
$\hat x \in {{\rm SOL}}(P_{1,\tau})$, cf. Remark 
\ref{rem:for_lambda_eq_0_there_needs_to_be_no_tau}.
Before proving the theorem we give also one remark to part i) and one 
remark to part ii), noting that, on the one hand,
several Lagrange Multiplier values
for $\lambda$ can correspond to the same levelparameter $\tau$,
and that, on the other hand, several levelparameters $\tau$
can correspond to one and the same Lagrange Multiplier value $\lambda$.
%
\rem{Folgende Alternative koennte vielleicht Vereinfachung 
bewirken:
Fall $\dom \Phi \cap {\rm lev}_\tau \Psi = \emptyset$
separat behandeln, und nur im anderen Fall, dann auf die 
Indikatorfunktion zurueckgreifen;
allerdings beist sichs mit unconstrained perspective, vgl.
Ende der Perspectivesection}
\begin{theorem} \label{thm:constraint_vs_nonconstraint}
{\rm i)}
  Let $\Phi,\Psi: \mathbb R^n \rightarrow \R \cup \{+\infty\}$ be proper, convex functions.
  Consider ${\rm (P_{1,\tau})}$ for a $\tau \in (\inf \Psi, + \infty)$
  with 
  $
      {\rm ri}(\dom \Phi) \cap {\rm ri}({\rm lev}_\tau \Psi) 
      \not = 
      \emptyset
  $ 
  and let $\hat x$ be a minimizer of {\rm ($P_{1,\tau}$)}, which 
  is situated in ${\rm int}(\dom \Psi)$.
  Then  there exists a real parameter $\lambda \ge 0$ such that $\hat x$ 
  is also a minimizer of {\rm ($P_{2,\lambda}$)}.
  This parameter $\lambda$ is positive, 
  if $\hat x$ is in addition not a minimizer of $\Phi$.
  %
%
%
\\[1ex]
{\rm ii)} 
  For proper $\Phi,\Psi: \mathbb R^n \rightarrow \R \cup \{+\infty\}$
  with $\dom \Phi \cap \dom \Psi \not = \emptyset$,
  let $\hat x$  be a minimizer of {\rm ($P_{2,\lambda}$)}. 
  For $\lambda = 0$ and $\tau \in OP(\Phi, \Psi)$ the point
  $\hat x$ is also a minimizer of {\rm ($P_{1,\tau}$)}
  if and only if $\tau \ge \Psi(\hat x)$.
  If $\lambda > 0$, then $\hat x$ is also a minimizer of {\rm ($P_{1,\tau}$)} 
  for $\tau \defeq \Psi(\hat x) \in OP(\Phi, \Psi)$. 
  Moreover, if $\Phi,\Psi$ are proper, convex functions and $\hat x \in {\rm int}({\dom \Psi})$,
  this $\tau$ is unique among all values in $OP(\Phi, \Psi)$ if and only if $\hat x$ is not a minimizer of $\Phi$. 
\rem{Das umstaendliche beschraenken von $\tau$ ist noetig, um 
       das Auseinananderklaffen der ``restringierten Sichtweise''
       und der ``unrestringierten Sichtweise'' auf die 
       Minimierungsproblem mit Nebenbed. zu vermeiden.}

\end{theorem}
This theorem implies directly the following  
\begin{corollary} \label{cor:constraint_vs_nonconstraint}
  Let $\Psi$ be a proper and convex function with open effective domain
  and let $\Phi$ be another proper and convex function with 
  $\dom \Phi \subseteq \dom \Psi$. For those 
  $\widehat x \in \R^n$ which do neither belong to $\argmin \Phi$ nor 
  to $\argmin \Psi$ the following holds true:
  \begin{enumerate}
    \item 
      If $\widehat x \in {\rm SOL}(P_{1, \tau})$ for some 
      $\tau \in (\inf \Psi, \pinfty)$ then also 
      $\widehat x \in {\rm SOL}(P_{2,\lambda})$ for some $\lambda > 0$.
    \item 
      If $\widehat x \in {\rm SOL}(P_{2,\lambda})$ for some $\lambda > 0$
      then there is exactly one $\tau \in OP(\Phi, \Psi)$ such that
      $\widehat x \in {\rm SOL}(P_{1, \tau})$, namely $\tau = \Psi(\widehat x)$.
  \end{enumerate}

\end{corollary}

Before proving \prettyref{thm:constraint_vs_nonconstraint} 
we give the announced remarks.

\begin{remark} \label{remark_2}
Part {\rm i)} of the theorem is not constructive. In general, there may exist
various parameters $\lambda$ corresponding to the same parameter $\tau$
as the following example with $m<-2$ and $\tau=1$ shows: 
Consider the proper and convex functions $\Phi, \Psi : \R \rightarrow \R$ given by
$\Psi(x)\defeq|x|$ and 
  \begin{equation*}
    \Phi(x)\defeq
    \begin{cases}
      (x-2)^2      & {\rm if } \; x \geq 1, \\
       m(x-1) + 1  & {\rm if } \; x <    1,
    \end{cases}
  \end{equation*}
  where $m \leq -2$. Note that $\Phi$ is differentiable for $m=-2$.
  Since $\argmin_{x\in\R}\Phi(x)=\{2\}$ we obtain 
  $c \defeq \min_{x \in \argmin \Phi} |x| = 2$.
  Having a look at the graph of $\Phi$ and noting that it is strictly monotonic decreasing
  on $(0,c)=(0,2)$ we see that 
  \begin{equation*}
    \qquad \argmin_{x \in \mathbb R} 
             \left\{ \Phi(x) \st |x| \le \tau \right\}
    = \{\tau\}
  \end{equation*}
  for all $\tau \in (0,2)$.
 On the other hand, we get
  \begin{equation*}
    \qquad \argmin_{x \in \mathbb R}  \left\{ \Phi(x) +\lambda|x| \right\} =
    \begin{cases}
      \{2-\frac{\lambda}{2}\}  &  {\rm if } \; \lambda \in [0,2),       \\
      \{1\}                    &  {\rm if } \; \lambda \in [2,-m),      \\
        [0,1]                  &  {\rm if } \; \lambda = -m,            \\
      \{0\}                    &  {\rm if } \; \lambda \in (-m,+\infty),  
    \end{cases}
  \end{equation*}
so that $\tau = 1$ corresponds to $\lambda \in [2,-m]$. It is known that the set of Lagrange multipliers
$\lambda$ is a bounded, closed interval under certain assumptions, 
see \cite[Corollary 29.1.5]{Rockafellar1970}
\end{remark}
\begin{remark} \label{remark_2_1}
Concerning part {\rm ii)} of the theorem 
in case that there are different minimizers
of $(P_{2,\lambda})$, say $\hat x_1$ and  $\hat x_2$, 
we notice that 
$\Psi(\hat x_1) \not = \Psi(\hat x_2)$ 
can appear as the following example shows:
For $\Phi(x)\defeq|x-2|$ and $\Psi(x)\defeq|x|$ and 
$\lambda=1$ we have 
  \begin{equation*}
 (P_{2,1}) \qquad   
 \Phi(x)+\Psi(x)=
 \left\{
    \begin{array}{cl}
      -2(x-1)  &  {\rm if } \; x<0 ,      \\
           2   &  {\rm if } \; x\in[0,2], \\
      +2(x-1)  &  {\rm if } \; x>2,
      \end{array}
  \right. 
 \end{equation*}
  i.e., $\argmin_{x \in \R} \left\{ \Phi(x) + \Psi(x) \right\} = [0,2]$. 
  Hence we can choose, e.g., $\hat{x}_{1}=1$ and $\hat {x}_{2}=2$ and obtain 
  $\Psi(\hat{x}_{1}) = 1 \neq 2 = \Psi(\hat{x}_{2})$.
\end{remark}

\begin{remark} \label{rem:for_lambda_eq_0_there_needs_to_be_no_tau}
  As warning we finally note that part 
  ii) of the theorem needs to be carefully read
  in case $\lambda = 0$, since the assertion 
  $
    \forall \tau \in OP(\Phi, \Psi):
    \hat x \in {\rm SOL}(P_{1,\tau})
    \Leftrightarrow 
    \tau \geq \Psi(\hat x)
  $
  does not state that there actually is a \emph{real} $\tau$
  with $\hat x \in {\rm SOL}(P_{1,\tau})$.
  This can be concluded if and only if $\hat x \in \dom \Psi$. 
  In our chosen ``unconstrained perspective'', however, the occurrence of 
  $\hat x \not \in \dom \Psi$ can indeed happen.
  Consider for example the proper, convex and
  lower semicontinuous functions
  $\Phi: \R \rightarrow \R$,
  $\Psi: \R \rightarrow \R \cup \{+ \infty\}$
  given by
  \begin{align*}
    \Phi(x) \defeq [x-(-1)]^2,
    &&
    \Psi(x) \defeq 
    \begin{cases}
      -\sqrt{x}     & \text{ if } x \geq 0, \\
      +\infty       & \text{ if } x < 0.
    \end{cases}
  \end{align*}
  Clearly $\dom \Phi \cap \dom \Psi \not = \emptyset$ is fulfilled. 
  For $\lambda =0$, we see that
  $\hat x = -1$ is the unique minimizer of $\Phi + 0\Psi = \Phi$.
  Since $\hat x\not \in \dom \Psi$ we have in particular
  $\hat x \not \in {\rm SOL}(P_{1,\tau})$ 
  for all $\tau \in \R = OP(\Phi,\Psi)$.
\end{remark}

\begin{proof}[Proof of \prettyref{thm:constraint_vs_nonconstraint}]
  i) Let $\hat x \in {\rm SOL}(P_{1,\tau}) \cap {\rm int}(\dom \Psi)$,
  where $\tau \in (\inf \Psi, +\infty)$.
  Then $\Psi(\hat x) \le \tau$ holds true.
  In case $\Psi(\hat x) < \tau$ the continuity of
  $\Psi$ in ${\rm int}(\dom \Psi)$ assures 
  $\Psi(x) < \tau$ in a neighborhood of $\hat x$.
  Consequently $\hat x$ is a local minimizer of $\Phi$ and hence also
  a global minimizer of this convex function.
  In particular $\hat x$ is a solution of ${\rm SOL}(P_{2,0})$.
  In case $\Psi(\hat x) = \tau$, we get by
  Fermat's rule, the regularity assumption,
  $\partial \Psi (\hat x) \not = \emptyset$ and 
  Lemma \ref{level_sets} the relation
  \begin{equation*} 
    0 
    \in \partial\big(\Phi  + \iota_{{\rm lev}_\tau \Psi} \big)(\hat x) 
    = 
    \partial \Phi(\hat x)  + \partial \iota_{{\rm lev}_\tau \Psi} (\hat x)
    =
    \partial \Phi(\hat x) + \R_0^+ \partial \Psi(\hat x). 
  \end{equation*}
  This means that there exists $\lambda \ge 0$ such that
  $\zerovec \in \partial \Phi(\hat x) + \lambda \partial \Psi(\hat x) \subseteq \partial \big( \Phi + \lambda \Psi\big)(\hat x)$
  so that by Fermat's rule $\hat x$ is a minimizer of ($P_{2,\lambda}$).
  If $\hat x$ is not a minimizer of $\Phi$, then clearly $\lambda > 0$.
  \\[1ex]
  ii) 
Let $\hat x \in {\rm SOL}(P_{2,\lambda})$.
If $\lambda = 0$ we have to distinguish 
-- at least in our taken unconstrained perspective --
two cases:
In case $\hat x \not \in \dom \Psi$ and any 
$\tau \in OP(\Phi, \Psi) \subseteq \R$
neither the point $\hat x$ is a minimizer of $(P_{1,\tau})$
nor is $\tau \geq +\infty = \Psi(\hat x)$. So the claimed
equivalence holds true in this case.
In case $\hat x \in \dom \Psi$ this equivalence holds also true
for any $\tau \in OP(\Phi, \Psi):$
For real $\tau < \Psi(\hat x)$ neither 
$\hat x \in {{\rm SOL}}(P_{1,\tau})$ holds true nor does 
$\tau \geq \Psi(\hat x)$.
For real $\tau \geq \Psi(\hat x)$ we have 
$\hat x \in {\rm SOL}(P_{2,0}) = \argmin \Phi$, 
so that also $\hat x \in {{\rm SOL}}(P_{1,\tau})$ is fulfilled.

If $\lambda >0$, we have $\hat x \in \dom \Phi \cap \dom \Psi$ and get
$\hat x \in {\rm SOL}(P_{1,\tau})$ at least for 
$\tau = \Psi(\hat x) \in OP(\Phi, \Psi)$ by the following reason:
if there would exist $\tilde x$ with 
$\Phi(\tilde x) < \Phi(\hat x) < +\infty$ and 
$\Psi(\tilde x) \le \tau < + \infty$, then we can conclude
$\Phi(\tilde x) + \lambda \Psi(\tilde x) < \Phi(\hat x) + \lambda \Psi(\hat x)$,
since only finite values occur.
This contradicts $\hat x \in {\rm SOL}(P_{2,\lambda})$.
Finally, let in addition $\Phi, \Psi$ be convex and 
$\hat x \in {\rm int}({\dom \Psi})$.
If $\hat x$ is a minimizer of $\Phi$ then $\tau = \Psi(\hat x)$
is not the only value in $OP(\Phi, \Psi)$ with 
$\hat x \in {{\rm SOL}}(P_{1,\tau})$, since clearly every 
$\tau \geq \Psi(\hat x)$ belongs all the more to $OP(\Phi,\Psi)$
while $\hat x \in {{\rm SOL}}(P_{1,\tau})$ keeps fulfilled.
If $\hat x$ is not a minimizer of $\Phi$ then there can not exist
another value $\tilde \tau \not = \Psi(\hat x)$ from 
$OP(\Phi,\Psi)$ with $\hat x \in {{\rm SOL}}(P_{1,\tilde \tau})$:
For $\tilde\tau > \Psi(\hat x)$
the condition $\hat x \in {\rm int}(\dom \Psi)$ would imply
$\hat x \in \argmin \Phi$, as we already have seen in part i) 
of the proof,
whereas 
for $\tilde \tau < \Psi(\hat x)$ the point $\hat x$ would not even 
fulfill the constraint condition.
\end{proof}

\subsection{Fenchel duality relation} 

Using duality arguments we will specify the relations between ($P_{1,\tau}$) and ($P_{2,\lambda}$) 
for a more specific class of problems in 
\prettyref{sec:homegeneous_penalizers_and_constraints}. In particular, 
we want to determine $\lambda$ in part i) of 
\prettyref{thm:constraint_vs_nonconstraint}.
To this end, we need the following known Fenchel duality relation, compare, e.g.,
\cite[p. 505]{RW04}.
\begin{lemma} \label{fenchel-dual-hom}
Let $\Phi \in \Gamma_0(\R^n)$, $\Psi \in \Gamma_0(\R^m)$,
$L \in \R^{m,n}$ and $\mu > 0$.
Assume that the following conditions are fulfilled.
\begin{itemize}
\item[{\rm i)}]
${\rm ri} (\dom \Phi) \cap {\rm ri}(\dom \Psi(\mu L \cdot)) \not = \emptyset$,
\item[{\rm ii)}]
$ \mathcal{R}(L) \cap {\rm ri} (\dom \Psi(\mu \cdot)) \not = \emptyset$,
\item[{\rm iii)}]
${\rm ri}(\dom \Phi^*(-L^* \cdot)) \cap {\rm ri}(\dom \Psi^*(\frac{\cdot}{\mu})) \not = \emptyset$,
\item[{\rm iv)}]
$\mathcal{R}(-L^*) \cap {\rm ri} (\dom \Phi^*) \not = \emptyset$.
\end{itemize}
Then, the primal problem
\begin{equation} \label{primal}
(P) \qquad  \argmin_{x \in \R^n} \left\{ \Phi(x) + \Psi(\mu Lx) \right\}, \qquad \mu > 0,
\end{equation}
has a solution if and only if the dual problem 
\begin{equation} \label{dual}
(D) \qquad  \argmin_{p \in \R^m} \big\{ \Phi^* (-L^* p) + \Psi^*\left(\frac{p}{\mu} \right) \big\}
\end{equation}
has a solution. Furthermore
$\hat x \in \R^n$ and $\hat p \in \R^m$ are solutions of the primal and the dual problem, respectively, if and only if 
\begin{equation} \label{primal-dual-relation}
\frac{1}{\mu} \hat p \in \partial \Psi(\mu L \hat x) \quad {\rm and} \quad -L^* \hat p \in \partial \Phi(\hat x).
\end{equation}
\end{lemma} 

\begin{proof}
  Assumptions i) and ii) assure that we can apply 
  \cite[Theorem 23.8]{Rockafellar1970} and \cite[Theorem 23.9]{Rockafellar1970}.
  Using these theorems, Fermat's Rule and \cite[Corollary 23.5.1]{Rockafellar1970}
  we obtain on the one hand
  \begin{eqnarray*}
    && {\rm SOL}(P) \not = \emptyset,
    \\ 
    &\Leftrightarrow&
    \exists \hat x \in \R^n \quad \mbox{such that} \quad
    \zerovec \in \partial \big( \Phi(\cdot) + \Psi(\mu L \cdot) \big)(\hat x)
    \; = \; 
    \partial \Phi(\hat x) + \mu L^* \partial \Psi(\mu L \hat x),
    \\ 
    &\Leftrightarrow&
    \exists \hat x \in \R^n \; \exists \hat p \in \R^m \quad \mbox{such that} \quad
    \hat p \in \mu \partial \Psi(\mu L \hat x)  
    \quad {\rm and} \quad
    -L^* \hat p \in \partial \Phi(\hat x),
    \\  
    &\Leftrightarrow&
    \exists \hat x \in \R^n \; \exists \hat p \in \R^m \quad \mbox{such that} \quad
    \mu L \hat x \in \partial \Psi^*\left(\frac{\hat p}{\mu} \right)
    \quad {\rm and} \quad
    \hat x \in \partial \Phi^*(-L^* \hat p).
  \end{eqnarray*}
  Due to the assumptions iii) and iv) we similarly obtain
  \begin{align*} 
    &{\rm SOL}(D) \not = \emptyset,
  \\ 
    \Leftrightarrow {}& 
    \exists \hat p \in \R^m \;\; \mbox{such that} \;\;
    \zerovec \in \partial \Big( \Phi^*(-L^* \cdot) + \Psi^*\left(\frac{\cdot}{\mu} \right)\Big)(\hat p)
    \; = \; 
    -L \partial \Phi^*(-L^* \hat p) + \frac{1}{\mu} \partial \Psi^* \left( \frac{\hat p}{\mu} \right),
    \\ 
    \Leftrightarrow {}& 
    \exists \hat p \in \R^m \; \exists \hat x \in \R^n \quad \mbox{such that} \quad
    \hat x \in \partial \Phi^*(- L^* \hat p)
    \quad \mbox{and} \quad
    \mu L \hat x \in \partial \Psi^*\left(\frac{\hat p}{\mu} \right),
  \end{align*}
  on the other hand.
\end{proof}

\subsection{Notes to \prettyref{thm:constraint_vs_nonconstraint} and to some technical assumptions} 
\label{sec:Notes_to_pvc_relation_theorem}

In this subsection we discuss mainly
\prettyref{thm:constraint_vs_nonconstraint} with respect to 
\DissVersionForMeOrHareBrainedOfficialVersion{four}{two} 
aspects:
In the fist part we deal with the condition 
$\hat x \in {\rm int}(\dom \Psi)$
and illustrate its importance -- at least in the 
``unconstrained perspective'' -- by two examples.
\DissVersionForMeOrHareBrainedOfficialVersion{A longer}{The second} 
part is dedicated to the regularity assumptions
used in \prettyref{thm:constraint_vs_nonconstraint} 
and in \cite[Theorem 2.4]{CiShSt2012} and their relation
to Slater's Constraint Qualification.
\DissVersionForMeOrHareBrainedOfficialVersion{
  In the third part we discuss differences in the 
  appearance of \prettyref{thm:constraint_vs_nonconstraint}
  and of \cite[Theorem 2.4]{CiShSt2012}.
  A note on the taken restriction $\tau \in OP(\Phi, \Psi)$
  concludes this subsection and the current section.}
{}

\subsection*{The condition \boldmath$\hat x \in {\rm int}(\dom \Psi)$
in \prettyref{thm:constraint_vs_nonconstraint}} 

Concerning part i) of \prettyref{thm:constraint_vs_nonconstraint}
we note that the condition 
$\hat x \in {\rm int}(\dom \Psi)$ is essential
-- at least in our chosen ``unconstrained perspective'':
It can not be omitted as the next example shows.
We will also see that it can not even be replaced by 
the weaker condition $\hat x \in {\rm ri}(\dom \Psi)$.

\begin{example}  ~
  \begin{enumerate}
    \item 
      Consider the proper, convex and lower semicontinuous functions
      $\Phi: \R \rightarrow \R$,
      $\Psi: \R \rightarrow \R \cup \{+ \infty\}$
      given by
      \begin{align*}
	\Phi(x) \defeq [x-(-1)]^2,
	&&
	\Psi(x) \defeq 
	\begin{cases}
	  -\sqrt{x}     & \text{ if } x \geq 0, \\
	  +\infty       & \text{ if } x < 0.
	\end{cases}
      \end{align*}
      We have 
      $
	{\rm ri}(\dom \Phi) \cap 
	    {\rm ri}({\rm lev}_{\tau}\Psi)
	= 
	(\tau^2, + \infty)
	\not =
	\emptyset
      $
      for every $\tau \in (-\infty, 0]=(\inf \Psi, \sup \Psi]$.
      Furthermore
      $
	\argmin \{\Phi \st \Psi \leq \tau \}
	= \{\tau^2\} \eqdef \{\hat x_\tau\}
      $
      does not intersect $\{-1\} = \argmin \Phi$
      for all these $\tau$.
      In case $\tau \in (-\infty, 0)$
      we have $\hat x_\tau \in {\rm int}(\dom \Psi)$
      and -- as guaranteed by part i) of the previous theorem --
      there is indeed a $\lambda \geq 0$ with 
      $\hat x_\tau \in \argmin( \Phi + \lambda \Psi)$ i.e. with 
      $\Phi'(\tau^2) + \lambda \Psi'(\tau^2) = 0$, namely 
      $\lambda = -4 \tau (\tau^2 + 1) > 0$.
      In case $\tau = 0$, however, such a \emph{real}
      $\lambda \geq 0$ does not exist: For $\lambda = 0$ we 
      have 
      $
	\hat x_\tau 
	= 
	0 
	\not \in \{-1\} 
	=
	\argmin(\Phi)
	= 
	\argmin(\Phi + 0\Psi)
      $
      -- in our unconstrained perspective -- 
      and for $\lambda \in (0, +\infty)$ we have 
      $	 
	0 
	\not \in 
	\emptyset 
	= \partial (\Phi + \lambda \Psi)(\hat x_\tau)
      $
      so that $\hat x_\tau \not \in \argmin (\Phi + \lambda \Psi)$
      as well.
    \item 
      Consider the proper, convex and lower semicontinuous functions 
      $\Phi: \R^2 \rightarrow \R$, 
      $\Psi: \R^2 \rightarrow \R \cup \{+\infty \}$ given by
      \begin{align*}
	\Phi(x_1, x_2) \defeq x_1^2 + (x_2-1)^2,
	&&
	\Psi(x_1,x_2) \defeq
	\begin{cases}
	  x_1     & \text{ if } x_2      = 0, \\
	  +\infty & \text{ if } x_2 \not = 0.
	\end{cases}
      \end{align*}
      For any $\tau \in (\inf \Psi, + \infty) = \R$ we have
      $
	{\rm ri}(\dom \Phi) \cap {\rm ri}({\rm lev}_{\tau}\Psi) 
	= \R^2 \cap [(\minfty, \tau)\times \{0\}]
	\not =
	\emptyset
      $.
      Consider
      \begin{align*}
	  \hat x_\tau
	&\in 
	  \argmin \{\Phi \st \Psi \leq \tau \}
	=
	  \argmin_{x\in(-\infty, \tau]\times \{0\}} \Phi(x)
	=
	  \left[
	    \argmin_{x_1 \in (-\infty,\tau] } x_1^2 + 1
	  \right] \times \{0\} \\
	&=
	  \begin{cases}
	    \{(\tau, 0)^T\} & \text{ for } \tau < 0 \\
	    \{(0,0)^T\}     & \text{ for } \tau \geq 0.
	  \end{cases}
      \end{align*}
      In case $\tau < 0$ there is even a $\lambda \in (0, +\infty)$ 
      with 
      \begin{gather*}
	  (\tau, 0)^T
	= 
	  \hat x_\tau  
	\in 
	  \argmin \{\Phi + \lambda \Psi\}
	\overset{\lambda \not = 0}{=}
	  \left[
	    \argmin_{x_1 \in \R} (x_1^2 + \lambda x_1)
	  \right]
	  \times \{0\}
	=
	  \{(-\tfrac{\lambda}{2}, 0)^T\},
      \end{gather*}
      namely $\lambda = -2\tau > 0$.
      In case $\tau \geq 0$, however, there is no $\lambda \geq 0$ 
      with $(0,0)^T = \hat x_\tau \in \argmin(\Phi + \lambda \Psi)$:
      On the one hand any $\lambda > 0$ can not do the job, since
      $
	\argmin(\Phi + \lambda \Psi) = \{(-\tfrac{\lambda}{2}, 0)^T\}
	\not \ni (0,0)^T
      $
      for all $\lambda \in (0, +\infty)$.
      On the other hand also $\lambda = 0$ can not do the job, since
      $
	\argmin(\Phi+ 0 \Psi) 
	= 
	\argmin \Phi 
	= \{(0,1)^T\}
	\not \ni 
	(0,0)^T 
      $.
  \end{enumerate}

\end{example}

\subsection*{Regularity assumptions and the related Slater Condition} 
\MayChangesPartiallyPerformedVersionOrHareBrainedOfficialVersion{%

\diss{Bei finalem Notationsabgleich: Bezeichnungen fuer die
Funktionen hier besser anders?}
}%
{}%
In part i) of \prettyref{thm:constraint_vs_nonconstraint} the condition 
\begin{gather*}
      {\rm ri}(\dom \Phi) \cap {\rm ri}({\rm lev}_\tau \Psi) 
      \not = 
      \emptyset,
\end{gather*}
from \cite[Theorem 23.8]{Rockafellar1970} was used as regularity assumption
to ensure a certain amount of overlapping between 
the sets $\dom \Phi$ and ${\rm lev}_{\tau}\Psi$.
In \cite{CiShSt2012} we used a different condition 
which, however, implies our used condition; that condition was:
\begin{center}
  ``Assume that there exists a point in  $\dom \Phi \cap {\rm lev}_\tau \Psi$ 
  \\where one of the functions $\Phi$ or $\iota_{{\rm lev}_\tau \Psi}$ 
  is continuous.''\footnote{
  We took that condition from the book of Ekeland -- T{\'e}mam,
  cf. \cite[Proposition 5.6 on p. 26]{EkTe1999} 
  with caution in case $F_1 = + \infty$ and $F_2 = - \infty$.}  
\end{center}

Another related regularity assumptions is Slater's Constraint Qualification
\begin{gather*}
  \exists x_0 \in \dom \Phi: \Psi(x_0) < \tau.
\end{gather*}

We will shortly discuss the relation between this Slater Condition and 
the first condition for functions 
$\Psi$ which additionally have an open effective domain $\dom \Psi$.
This additional assumption has the following effect on part i) of 
\prettyref{thm:constraint_vs_nonconstraint}:
All minimizers of {\rm ($P_{1,\tau}$)} are now
automatically situated in ${\rm int}(\dom \Psi)$ 
and for real $\tau > \inf \Psi$ the regularity condition 
$
    {\rm ri}(\dom \Phi) \cap {\rm ri}({\rm lev}_\tau \Psi) 
    \not = 
    \emptyset
$ 
is equivalent to Slater's Constraint Qualification, by
the subsequent lemma. 
In this case, the existence of a Lagrange multiplier $\lambda \ge 0$ 
is also assured by \cite[Corollary 28.2.1]{Rockafellar1970}\footnote{
after resetting $\Phi$ to $+ \infty$ outside of $\dom \Psi$ in order 
to achieve $\dom \Phi \subseteq \dom \Psi$ as demanded by Rockafellar 
on p. 273; his other demand ${\rm ri}(\dom \Phi) \subseteq {\rm ri}(\dom \Psi$) 
is then automatically fulfilled since $\dom \Psi$ is open here.}
if we note \cite[Theorem 28.1]{Rockafellar1970}.
\\[1ex]
Dropping this additional assumption again and returning to 
our general setting in \prettyref{thm:constraint_vs_nonconstraint}
we note that it still might be possible to replace the
first regularity assumption by this Slater Condition;
however the latter does in general no longer imply
the first regularity assumption:
The condition 
$\Psi(x_0) < \tau$ in itself does not ensure 
$ x_0 \in {\rm ri}( {\rm lev}_\tau \Psi )$
as Fig.
\ref{fig:boundary_of__lower_levelset_tau__can_contain_point_with_value_smaller_than_tau}
shows.
Imaging that we choose $\Phi$ now in a way such that 
$\dom \Phi$ is a closed triangle which has $x_0$ as one of its 
vertices and that $\dom \Phi$ intersects the sketched $\dom \Psi$
only in $x_0$. In particular  
$\dom \Phi \cap {\rm lev}_\tau \Psi = \{x_0\}$
so that Slater's Condition is fulfilled here,
but our first regularity condition 
$
  {\rm ri}(\dom \Phi) \cap {\rm ri}({\rm lev}_\tau \Psi) 
  \not = 
  \emptyset
$ does not hold, since $x_0 \not \in {\rm ri}({\rm lev}_\tau \Psi)$ here.
However, in situations where $x_0 \in {\rm int}(\dom \Psi)$ 
holds true in addition, we have
$
  x_0 
  \in 
  {\rm int}(\dom \Psi) \cap {\rm lev}_{<\tau}\Psi 
  = 
  {\rm ri}({\rm lev}_{\tau}\Psi)
$,
by Theorem \ref{thm:ri_of_levelset} and we could 
state the \prettyref{thm:constraint_vs_nonconstraint}
also with the extended slater condition 
\begin{gather*}
  \exists x_0 \in \dom \Phi : x_0 \in {\rm int}(\dom \Psi) 
  \text { and } \Psi(x_0) < \tau
\end{gather*}
by the following Lemma:

\begin{lemma} \label{lem:slaters_cond_is_equiv_if_dom_is_open}
  Let $\Phi, \Psi: \R^n \rightarrow \R \cup \{+ \infty\}$ be proper and 
  convex functions and let ${\rm int}(\dom \Psi) \not = \emptyset$.
  %
  %
  Then, for any $\tau \in \R$, the following statements are equivalent:
  \begin{enumerate}
    \item \label{enu:reg_cond_here}
      $\tau > \inf \Psi$ and 
      $
	{\rm ri}(\dom \Phi) \cap {\rm ri}({\rm lev}_{\tau}\Psi) 
	\not = 
	\emptyset
      $
    \item \label{enu:reg_cond_in_published_paper}
      $\tau > \inf \Psi$ and there exists an 
      $x' \in \dom \Phi \cap {\rm lev}_{\tau}\Psi$
      where $\Phi$ or $\iota_{{\rm lev}_{\tau}\Psi}$ is continuous.
    \item \label{enu:extended_slater_condition}
      There is an $x_0 \in \dom \Phi$ with 
      $x_0 \in {\rm int}(\dom \Psi)$ and $\Psi(x_0) < \tau$.
    \item \label{enu:interior_of_levelset_meets_dom}
      $\tau > \inf \Psi$ and 
      $\dom \Phi \cap {\rm int}({\rm lev}_{\tau}\Psi) \not = \emptyset$.
  \end{enumerate}

\end{lemma}

\begin{proof}
  
  \upref{enu:interior_of_levelset_meets_dom}
  $\Rightarrow$ \upref{enu:extended_slater_condition} :
  Let $x_0 \in \dom \Phi \cap {\rm int}({\rm lev}_{\tau}\Psi)$.
  Then $x_0 \in \dom \Phi$ holds banally true. 
  Due to ${\rm int}(\dom \Psi) \not = \emptyset$ 
  we know that $\dom \Psi$ has full dimension $n$, 
  so that Theorem \ref{thm:ri_of_levelset} yields
  $
    x_0 
    \in 
    {\rm int}({\rm lev}_{\tau}\Psi)
    =
    {\rm ri}({\rm lev}_{\tau}\Psi)
    =
    {\rm ri}(\dom \Psi) \cap {\rm lev}_{<\tau}\Psi
    =
    {\rm int}(\dom \Psi) \cap {\rm lev}_{<\tau}\Psi
  $.
  \upref{enu:extended_slater_condition}
  $\Rightarrow$ \upref{enu:reg_cond_in_published_paper} :
    Let there exist $x_0 \in \dom \Phi \cap {\rm int}(\dom \Psi)$ with $\Psi(x_0) < \tau.$
    This assures directly $\tau > \inf \Psi$. To see the
    continuity of $\iota_{{\rm lev}_{\tau}\Psi}$ in $x_0 \eqdef x'$, 
    note that the convex function $\Psi$ is continuous in
    $x_0 \in 
    {\rm int}(\dom \Psi) $,
    assuring $\Psi(x) < \tau$ in a whole neighborhood of $x_0$.
  \upref{enu:reg_cond_in_published_paper} 
  $\Rightarrow$ \upref{enu:reg_cond_here} :
    Let $\Phi$ or $\iota_{{\rm lev}_{\tau}\Psi}$ be continuous in a point
    $x' \in \dom \Phi \cap {\rm lev}_{\tau}\Psi$. Then at least one 
    of the nonempty, convex sets $A = \dom \Phi$ or 
    $B = {\rm lev}_{\tau}\Psi = \dom \iota_{{\rm lev}_{\tau}\Psi}$
    contains that common point in its interior; say 
    $x' \in {\rm int}(A)$ without loss of generality. 
    Choosing any point $y' \in {\rm ri} (B)$, as permitted by Theorem
    \ref{thm:nonemtpy_convex_set_has_nonempty_ri}, we have
    \begin{gather*}
      z_\lambda \defeq (1 - \lambda)y' + \lambda x' \in {\rm ri} (B)
    \end{gather*}
    for all $\lambda \in [0,1)$, due to Theorem
    \ref{thm:raypoint_from_ri__leaving_closure__never_comes_back}.
    So we achieve $z_\lambda \in {\rm ri} (B) \cap {\rm int}(A)$
    by choosing $\lambda \in [0,1)$ close enough to $1$.
    In particular ${\rm ri} (A) \cap {\rm ri} (B) \not = \emptyset$
    holds true.
  \upref{enu:reg_cond_here} 
  $\Rightarrow$ \upref{enu:interior_of_levelset_meets_dom}:
    Let 
    $x_0 \in {\rm ri}(\dom \Phi)\cap {\rm ri}({\rm lev}_{\tau}\Psi)$,
    where $\tau > {\inf \Psi}$.
    Then $x_0 \in \dom \Phi$ holds banally true. Using Theorem 
    \ref{thm:ri_of_levelset} we also obtain 
    $
      x_0 
      \in 
      {\rm ri} ({\rm lev}_{\tau}\Psi)
      =
      {\rm int}({\rm lev}_{\tau}\Psi)
    $,
    again due to the fact that ${\rm lev}_{\tau}\Psi$ 
    has the same full dimension $n$ as $\dom \Psi$.
\end{proof}

\begin{figure}[h!]
  \begin{center}
   \includegraphics[width=5cm]{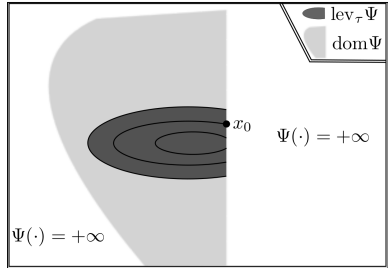}
    \caption{Example where $\Psi(x_0) < \tau$ does not imply $ x_0 \in {\rm ri}( {\rm lev}_\tau \Psi )$. }
  \label{fig:boundary_of__lower_levelset_tau__can_contain_point_with_value_smaller_than_tau}
\end{center}
\end{figure}

\DissVersionForMeOrHareBrainedOfficialVersion
{
\subsection*{A remark on differences between 
\prettyref{thm:constraint_vs_nonconstraint} and 
Theorem 2.4 in \cite{CiShSt2012}} 

\prettyref{thm:constraint_vs_nonconstraint} and 
Theorem 2.4 in \cite{CiShSt2012} seem quite different at first glace.
This is however only partially true as we will see soon.
For convenience of the reader we repeat 
Theorem 2.4 in \cite{CiShSt2012} here:

\begin{theorem*} 
  {\rm i)} 
  Let $\Phi,\Psi: \mathbb R^n \rightarrow \R \cup \{+\infty\}$ be proper, convex functions.
  Assume that there exists a point in  $\dom \Phi \cap {\rm lev}_\tau \Psi$ 
  where one of the functions $\Phi$ or $\iota_{{\rm lev}_\tau \Psi}$ is continuous.
  Let $\hat x \in {\rm int}(\dom \Psi)$ be a minimizer of {\rm ($P_{1,\tau}$)}, where $\hat x$ is not a minimizer of $\Psi$ in case $\Psi(\hat x) = \tau$.
  Then  there exists a parameter 
  $\lambda \ge 0$ 
  such that
  $\hat x$ 
  is also a minimizer of {\rm ($P_{2,\lambda}$)}.
  If $\hat x$ is in addition not a minimizer of $\Phi$, then $\lambda >0$.

  {\rm ii)} For proper $\Phi,\Psi: \mathbb R^n \rightarrow \R \cup \{+\infty\}$
  with $\dom \Phi \cap \dom \Psi \not = \emptyset$,
  let $\hat x$  be a minimizer of {\rm ($P_{2,\lambda}$)}. 
  If $\lambda = 0$, then $\hat x$ is also a minimizer of {\rm ($P_{1,\tau}$)}
  if and only if $\tau \ge \Psi(\hat x)$.
  If $\lambda > 0$, then $\hat x$ is also a minimizer of {\rm ($P_{1,\tau}$)} 
  for $\tau \defeq \Psi(\hat x)$. 
  Moreover, if $\Phi,\Psi$ are proper, convex functions and $\hat x \in {\rm int}({\dom \Psi})$,
  this $\tau$ is unique if and only if $\hat x$ is not a minimizer of $\Phi$. 
\end{theorem*}

Besides the difference in the already discussed regularity condition
there is another difference in part i) of the theorems: The condition 
``where $\hat x$ is not a minimizer of $\Psi$ in case 
$\Psi(\hat x) = \tau$''
was replaced by the condition $\tau \in (\inf \Psi, + \infty)$,
which clearly ensures the first condition.
But is our \prettyref{thm:constraint_vs_nonconstraint} not
weaker than the older Theorem 2.4, since the case
$\tau \leq \inf \Psi$ was excluded? It is not for the following reason:
Let the regularity condition of the cited Theorem 2.4 
be fulfilled.
In case $\tau = \inf \Psi$ every minimizer $\hat x$ of 
${\rm (P_{1,\tau})}$ belongs therefore\footnote{
also in our taken unconstrained perspective!} to 
${\rm lev}_{\inf \Psi} \Psi$, 
so that $\Psi(\hat x) = \inf \Psi$,
making it impossible to fulfill the assumption
``where $\hat x$ is not a minimizer of $\Psi$ in case 
$\Psi(\hat x) = \tau$''.
In case $\tau < \inf \Psi$ we have
$\dom \Phi \cap {\rm lev}_\tau \Psi \ = \emptyset$, 
i.e. not even the regularity condition is fulfilled.
So part i) of that Theorem 2.4 is not applicable for $\tau \leq \inf \Psi$
anyway. 
Finally we give a comment to a restriction not 
(explicitly) mentioned in cited Theorem 2.4, namely the 

\subsection*{Restriction of 
\boldmath$\tau$ to \boldmath$OP(\Phi, \Psi)$
in \prettyref{thm:constraint_vs_nonconstraint}} 
In our taken ``unconstrained perspective'' we must 
demand $\dom \Phi \cap {\rm lev}_\tau \Psi \not = \emptyset$,
in order to ensure that not every $x \in \R^n$ is 
minimizer of ${\rm SOL}(P_{1,\tau})$ and to 
ensure
$
  \argmin \{\Phi \st \Psi \leq \tau \}
  =
  \argmin \{\Phi + \iota_{{\rm lev}_{\tau}} \Psi \}
$.
To this end we restricted $\tau$ to $OP(\Phi, \Psi)$.} 
{}


\section{Assisting theory with examples} \label{sec:assisting_theory}

\addtointroscoll{

This section provides tools which allow to 
transfer and refine the general relation between
${\rm SOL}(P_{1, \tau})$ and 
${\rm SOL}(P_{2,\lambda})$, 
as stated in 
\prettyref{thm:constraint_vs_nonconstraint} resp.
\prettyref{cor:constraint_vs_nonconstraint},
to the more special setting in 
\prettyref{sec:homegeneous_penalizers_and_constraints}
with homogeneous penalizers and constraints,
resulting in our main 
\prettyref{thm:theo_lambda_tau} of the last
\prettyref{sec:homegeneous_penalizers_and_constraints}.

Among this current section's subsections 
\begin{itemize}
  \item 
    \ref{subsec:convex_functions_and_their_periods_space}
    Convex functions and their periods space
  \item 
    \ref{subsec:operations_that_preserve_essentially_smoothness}
    Operations that preserve essentially smoothness
  \item 
    \ref{subsec:operations_that_preserve_decomposability_into_an_innerly_striclty_convex_and_a_constant_part}
    Operations that preserve decomposability into a innerly strictly convex and a constant part
  \item 
    \ref{subsec:existence_and_direction_of_argmin_of_sum}
    Existence and direction of $\argmin(F+G)$ for certain
    classes of functions
\end{itemize}
the last one is the most important one for that transferring;
roughly speaking its  
\prettyref{thm:direction_of_argmin__essentially_smooth__strict_convex_sharpend}
ensures, for given $\lambda > 0$, that the value 
$\tau = \Psi(\widehat x) = \|L \widehat x\|$ is independent 
from the choice of $\widehat x \in {\rm SOL}(P_{2,\lambda})$,
if $\Phi$ is additionally 
essentially smooth and (essentially) strictly convex on some 
affine subset $\widecheck A$ of ${\rm aff}(\dom \Phi)$.
Demanding such essentially smoothness and (essentially) strictness 
properties on $\Phi$ is done in the setting of the next section,
so that we can apply directly
\prettyref{thm:direction_of_argmin__essentially_smooth__strict_convex_sharpend}
for the primal problems in \prettyref{subsec:setting}.

For the corresponding dual problems we likewise, for given 
$\tau$, would like the value 
$\lambda = \|\widehat p\|_*$ to be independent from the choice 
of $\widehat p \in {\rm SOL}(D_{1,\tau})$.
However we can not directly apply
\prettyref{thm:direction_of_argmin__essentially_smooth__strict_convex_sharpend}
for the dual problems since here the more complicated, concatenated function
$p \mapsto \Phi^*(-L^* p) \eqdef \widetilde{\Phi}(p)$ needs to be considered.
In
\prettyref{sec:homegeneous_penalizers_and_constraints}
we will see that $\Phi^*$ has similar essentially smoothness and 
strictness properties as $\Phi$. So the question remains 
if concatenation with a (not necessarily invertible) linear mapping 
preserve these properties.
Luckily this is the case if certain conditions hold true, see 
\prettyref{thm:concatenation_ess_smooth_with_linear_mapping} and
\prettyref{thm:concatenation_stricly_convex_on_ri_with_linear_mapping}
in the second and third subsection, respectively.

For the proof of that helpful
\prettyref{thm:concatenation_stricly_convex_on_ri_with_linear_mapping}
or rather its 
\prettyref{lem:decompostion_into_periodspace_part_and_strict_convex_part_carries_over_to_affine_subset}
we will use Theorems and Lemmata developed in 
\prettyref{subsec:convex_functions_and_their_periods_space}.
%
%

}

\subsection{Convex functions and their periods space} 
\label{subsec:convex_functions_and_their_periods_space}

In this subsection we define and deal with the periods space of 
a convex functions.
The notion of periods space is closely related to semidirect sums 
discussed in the previous chapter:
For a convex funtion $F: \R^n \rarr \R\cup\{\pinfty\}$
and any decomposition $\R^n = X_1 \oplus X_2$ 
of its domain of definition into some subspace $X_2 \subseteq P[F]$
and some complementary subspace $X_1$ we can write $F$ in the form 
$F = F_1 \sdirsum 0_{X_2}$ with 
$F_1 = F|_{X_1}$.
In subsection 
\ref{subsec:operations_that_preserve_decomposability_into_an_innerly_striclty_convex_and_a_constant_part}
it will be convenient to allow $X_1$ to be also an affine subset of 
$\R^n$. To this end we extend the definition of semidirect sums 
from \prettyref{sec:semidirect_sums_and_coercivity} as follows:
\begin{definition} \label{def:more_general_definition_of_semidirect_sums}
  Let a nonempty subset $X \subseteq \R^n$ have a direct decomposition 
  $X = X_1 \oplus X_2$ into subsets $X_1, X_2 \subseteq \R^n$.
  The semi-direct sum of functions 
  $F_1 : X_1 \rarr \R \cup \{\pinfty\}$,
  $F_2 : X_2 \rarr \R \cup \{\pinfty\}$ is the function 
  $F_1 \sdirsum F_2$, given by 
  \begin{gather*}
    (F_1 \sdirsum F_2 ) (x_1 + x_2) \defeq F_1(x_1) + F_2(x_2)
  \end{gather*}
\end{definition}

The next theorem shows that the periods of a convex function
form a vector space. This space is equal to the
constancy space, defined by Rockafellar, see 
\cite[p. 69]{Rockafellar1970}.

\begin{theorem}[and {\bf Definition}] \label{thm:periods_space}
  Let $X$ be a nonempty affine subset of $\R^n$ with 
  underlying difference space 
  $U \subseteq \R^n$
  and let $F: X \rightarrow \R \cup \{+\infty\}$ be a convex function. 
  The set
  \begin{align*}
    P[F] \defeq {}&
      \{p \in U: F(x+p) = F(x) \text{ for all } x\in X\} 
    \\
    = {}&
    \{p \in U: F(x+p) = F(x) \text{ for all } x\in {\rm aff}(\dom F)\} 
  \end{align*}
  of all periods of $F$ then forms a vector subspace of $U$. 
  We will call it {\bf periods space} of $F$. 
\end{theorem}

\begin{proof}
  The sets are equal; note herein that in case 
  $x \not \in {\rm aff}(\dom F)$ the equation
  $F(x+p) = F(x)$ is anyway fulfilled for all $p \in U$,
  since then neither $x$ nor $x+p$ belong to ${\rm aff}(\dom F)$,
  so that $F(x) = + \infty = F(x+p)$.
  Next we prove that $P[F]$ is a subspace of $U$ by the 
  Subspace Criterion.
  Clearly $\zerovec \in P[F]$. Furthermore $P[F]$ is closed under addition:
  Let $p', p \in P[F]$ be arbitrarily chosen. Then
  $F(x' + p' + p) = F(x'+p') = F(x')$ for all $x' \in X$ and therefore 
  $p' + p \in P[F]$. Finally $P[F]$ is closed under scalar multiplication:
  Let $p \in P[F]$ and $x \in X$ be arbitrarily chosen.
  We have to show that $F(x + \lambda p) = F(x)$ for all $\lambda \in \R$,
  i.e. that the function $f: \R \rightarrow \R \cup \{+\infty\}$, given by
  $f(\lambda) \defeq F(x + \lambda p)$ is constant.
  In case $f \equiv +\infty$ this is clearly true.
  In case $f \not \equiv +\infty$ we choose any $\lambda_0 \in \dom f$.
  Since $p$ is a period of $F$ all values $f(\lambda_0+k)$, where $k\in \Z$,
  equal $f(\lambda_0) < +\infty$. In particular we have 
  $\lambda_0 + k \in \dom f$ for $k \in \Z$. 
  Part 
  \upref{enu:convex_function_constant_on_line_segment_spaned_by_3_points}
  of Lemma \ref{lem:convexity_condition_inside_and_outside},
  applied to 
  $a_n = \lambda_0 - n$, $b_n = \lambda_0$ and $c_n = \lambda_0 + n$,
  where $n \in \N$, now just says that the convex function $f$ is constant
  on all Intervals $[\lambda_0-n, \lambda_0+n]$, where $n \in \N$, and hence
  on whole $\R$.
\end{proof}

\begin{lemma} 
\label{lem:decomposition_into_some_set_and_a_subspace_of_periods_space}
  Let $X$ be a nonempty affine subset of $\R^n$ and let 
  $E: X \rightarrow \R \cup \{+\infty\}$ be a proper and convex function.
  For any decomposition ${\rm aff}(\dom E) = \check A \oplus \check P$
  of ${\rm aff}(\dom E) \eqdef A$ into some affine set $\check A \subseteq \R^n$
  and some subspace $\check P$ of the periods space $P[E]$
  the following holds true:
  \begin{align}
    {\rm aff}(\dom E|_{\check A}) &= \check A,
  &
    \dom E &= \dom E|_{\check A} \oplus  \check P
    \label{eq:aff_of_dom_of_restricted_function__and__decomposition_of_dom}
  \\
    {\rm int}_{\check A}(\dom E|_{\check A})
    &=
    {\rm ri}(\dom E|_{\check A}), 
  &
    {\rm int}_A(\dom E) 
    &= 
    {\rm int}_{\check A}(\dom E|_{\check A}) \oplus  \check P
    \label{eq:int_dom_of_restricted_function_is_ri_dom__and__decomposition_of_int_dom}
  \end{align}
  Moreover all the sets in these equations are nonempty.
\end{lemma}

\begin{proof}
  Since $E$ is proper we have 
  $\emptyset \not = {\rm aff}(\dom E) = \check A \oplus \check P$ so that
  $\check A \not = \emptyset$ and $\check P \not = \emptyset$ as well.
  The inclusion 
  $
    \dom E|_{\check A} \oplus \check P
    \subseteq
    \dom E
  $
  holds true since 
  $E(\check a + \check p) = E(\check a) = E|_{\check A}(\check a) < +\infty$
  for all $\check a \in \dom E|_{\check A}$ and all 
  $\check p \in \check P \subseteq P[E]$.
  The reverse inclusion $\dom E \subseteq \dom E|_{\check A} \oplus \check P$
  holds also true, since every 
  $x \in \dom E \subseteq {\rm aff}(\dom E) = \check A \oplus \check P$ 
  can be written in the form 
  $x=\check a + \check p$ with some $\check p \in \check P$ and
  $\check a \in {\rm aff}(\dom E|_{\check A})$, where we even have 
  $\check a \in \dom E|_{\check A}$, because 
  $E|_{\check A}(\check a) = E(\check a + \check p) = E(x) < + \infty$.
  Altogether we have 
  \[
    \dom E = \dom E|_{\check A} \oplus \check P,
  \] 
  where $E \not \equiv \pinfty$
  guarantees $\dom E \not = \emptyset $, so that $\dom E|_{\check A}$
  is nonempty, as well.
  Due to the banal $\dom E|_{\check A} \subseteq \check A$ we get the 
  inclusion 
  $
    {\rm aff}(\dom E|_{\check A}) 
    \subseteq 
    {\rm aff}(\check A)
    =
    \check A
  $,
  where actually equality holds true, since (a slightly transposed) 
  equation
  \eqref{eq:aff_of_direct_sum_with_affine_subset} in 
  \prettyref{thm:operations_that_interchange_with_direct_sum}
  gives on the one hand
  \begin{align*}
    {\rm aff}(\dom E|_{\check A}) \oplus \check P
    =
    {\rm aff}(\dom E|_{\check A} \oplus \check P)
    =
    {\rm aff}(\dom E)
    = \check A \oplus \check P
  \end{align*}
  -- whereas the assumption ${\rm aff}(\dom E|_{\check A}) \subset \check A$
  would, on the other hand,  result in the strict subset relation
  $
    {\rm aff}(\dom E|_{\check A}) \oplus \check P 
    \subset 
    \check A \oplus \check P
  $,
  due to $\check P \not = \emptyset$.
  The therewith proven 
  \[
    {\rm aff}(\dom E|_{\check A}) = \check A
  \]
  gives now directly 
  \[
    {\rm int}_{\check A}(\dom E|_{\check{A}})
    =
    {\rm ri}(\dom E|_{\check A}),
  \]
  where these sets are nonempty by Theorem 
  \ref{thm:nonemtpy_convex_set_has_nonempty_ri}
  Using the latter equation and equation
  \eqref{eq:ri_of_direct_sum_with_affine_subset} from Theorem 
  \ref{thm:operations_that_interchange_with_direct_sum}
  we finally obtain
  \begin{equation*}
    {\rm int}_{A}(\dom E) 
    = 
    {\rm ri}(\dom E)
    =
    {\rm ri}(\dom E|_{\check A} \oplus \check P)
    =
    {\rm ri}(\dom E|_{\check A}) \oplus \check P
    =
    {\rm int}_{\check A}(\dom E|_{\check A}) \oplus \check P,
  \end{equation*}
  where ${\rm int}_A(\dom E) \not = \emptyset$ ensures
  that also ${\rm int}_{\check A}(\dom E|_{\check A})$ is non empty.  
\end{proof}

\begin{theorem} 
\label{thm:restricting_function_to_equiv_affine_sets_give_ess_the_same_function}
  Let $F: \R^n \rightarrow \R \cup \{+\infty\}$ be a convex function,
  $\check P$ a subspace of the periods space $P[F]$ and 
  $\check A, \tilde A \subseteq \R^n$ affine sets with 
  $\check A \oplus \check P = \tilde A \oplus  \check P$.
  Then $\check F \defeq F|_{\check A}: \check A \rightarrow \R \cup \{+\infty\}$
  and $\tilde F \defeq F|_{\tilde A}: \tilde A \rightarrow \R \cup \{+\infty\}$
  are the same mapping, except for an affine transformation between 
  their domain of definition:
  There is a bijective affine mapping 
  $\tilde \alpha: \check A \rightarrow \tilde A$ with 
  $\check F = \tilde  F \circ \tilde \alpha$, namely the mapping given by 
  $\tilde \alpha(\check a) = \tilde \alpha(\tilde a + \check p) \defeq \tilde a$.
\end{theorem}

\begin{proof}
  Due to $\check A \oplus \check P = \tilde A \oplus  \check P$
  every $\check a \in \check A$ can be written in the form 
  $\check a = \check a + \zerovec = \tilde a(\check a) + \check p (\check a)$
  with uniquely determined $\tilde a(\check a) \in \tilde A$ and 
  $\check p(\check a) \in \check P$.
  Setting $\tilde \alpha (\check a) \defeq \tilde a(\check a)$
  gives hence a well defined mapping 
  $\tilde \alpha: \check A \rightarrow \tilde A$.
  Geometrically speaking each $\check a \in \check A$ is projected
  parallel to $\check P$ to a point 
  $\tilde a = \tilde \alpha(\check a) \in \tilde A$.
  This mapping is bijective, since it is both injective and surjective:
  Let $\tilde \alpha (\check a_1 ) = \tilde \alpha (\check a_2)$ 
  for $\check a_1, \check a_2 \in \check A$.
  Then 
  $
    \check a_1 - \check a_2 
    =
    (\tilde \alpha (\check a_1) + \check p(\check a_1))
      -
      (\tilde \alpha (\check a_2) + \check p(\check a_2))
    =
    \zerovec + \check p(\check a_1) - \check p(\check a_2) 
    \eqdef 
    \check p
    \in \check P
  $,
  so that $\check a_2 + \check p = \check a_1 + \zerovec$.
  The directness of the sum $\check A \oplus  \check P$ gives thus 
  $p = \zerovec$, i.e. $\check a_2 = \check a_1$.
  This shows that $\tilde \alpha$ is injective.
  In order to prove the surjectivity of 
  $\tilde \alpha $ let $\tilde a \in \tilde A$ be given.
  Thanks to $\tilde A \oplus  \check P = \check A \oplus \check P$
  we can write $\tilde a$ in the form 
  $\tilde a = \tilde a + \zerovec = \check a_* + \check p_*$ with some
  $\check a_* \in \check A$ and $\check p_* \in \check P$.
  Rearranging the latter to $\check a_* = \tilde a - \check p_*$
  gives $\tilde a = \tilde \alpha (\check a_*)$.
  It remains to show that $\tilde \alpha : \check A \rightarrow \tilde A$
  is affine.
  To this end let $t \in \R$ and write arbitrarily chosen
  $\check a_1, \check a_2 \in \check A$ in the form 
  \begin{align*}
    \check a_1 = \tilde a_1 + \check p_1,
    &&
    \check a_2 = \tilde a_2 + \check p_2
  \end{align*}  
  with $\tilde a_1, \tilde a_2 \in \tilde A$ and 
  $\check p_1, \check p_2  \in \check P$.
  Then their affine combination
  \begin{gather*}
    \check a_1 + t (\check a_2 - \check a_1) 
    =
    \tilde a_1 + t (\tilde a_2 - \tilde a_1) 
      +
      \check p_1 + t (\check p_2 - \check p_2)
  \end{gather*}
  is of the same form with 
  $\tilde a_1 + t (\tilde a_2 - \tilde a_1) \in \tilde A$ and
  $\check p_1 + t (\check p_2 - \check p_2) \in \check P$,
  so that 
  $
    \tilde \alpha (\check a_1 + t (\check a_2 - \check a_1))
    =
    \tilde a_1 + t (\tilde a_2 - \tilde a_1)
    =
    \tilde \alpha (\check a_1) 
      + 
      t (\tilde \alpha (\check a_2) - \tilde \alpha (\check a_1))
  $
  really holds true.
\end{proof}

\begin{remark}
  Let $F:\R^n \rightarrow \R \cup \{+\infty\}$ be a convex function.
  Every $p \in P[F]$ fulfills $\dom F + p = \dom F$.
\end{remark}

The previous remark gave a necessary condition for $p \in P[F]$.
The following lemma gives a sufficient condition. It says that, 
in case of a proper, lower semicontinuous and convex function, we do not have to 
check the condition $F(x+p) = F(x)$ for all $x \in \R^n$ in order
to prove $p\in P[F]$: 
It already suffices to find only one single $a\in \dom F$ 
such that $F(x+p) = F(x)$ for all $x\in a + {\rm span}(p) $.
We note that it is even sufficient to find one single 
$a \in \dom F$ such that $F$ is bounded above on the line 
$a + {\rm span}(p)$ by some real $\alpha$; this is ensured 
by \cite[Corollary 8.6.1]{Rockafellar1970}, which contains the next lemma 
as special case.

\iftime{ALLES HIER VON $\R^n$ auf $X$ umstellen oder alles auf 
$\R^n$ zur\"uck?, incl. 
Dinge aus ANHANG [EDIT: 18. Maerz 2013: Ist dies bei der jetzigen 
Unterteilung in Unterabschnitte ueberhaupt nich noetig?]}

\begin{lemma} \label{lem:belonging_to_periods_space}
  Assume that a function $F: \R^n \rightarrow \R \cup \{+\infty\}$
  from $\Gamma_0(\R^n)$
  is constant on a line or point $a+ {\rm span}( p ) \subseteq \R^n$
  which intersects $\dom F$. 
  Then $p \in P[F]$.
\end{lemma}

\begin{proof}
  In case $p=\zerovec$ the assertion is clearly fulfilled. 
  In the main case $p \not = \zerovec$ we have to show that 
  $F$ is constant on every straight line 
  $x + {\rm span}(p)$ parallel, but not identical 
  to $a + {\rm span}(p)$.
  In case of $F \equiv + \infty$ we are done.
  In the remaining case 
  $F|_{x + {\rm span}(p)} \not \equiv + \infty$
  we consider $F$ on the affine plane spanned by the 
  non-identical, parallel straight lines 
  $x + {\rm span}(p)$ and $a + {\rm span}(p)$, 
  or rather only on the closed strip
  \begin{gather*}
    S_x 
    \defeq 
    {\rm co}([x + {\rm span}(p)] \cup [ a + {\rm span}(p)])
  \end{gather*}
  bounded by these lines. We perform our task in two steps:
  Firstly we will show that $F$ is constant on every straight line 
  $y + {\rm span}(p)$ in 
  $
    {\rm ri} (S_x) 
    = 
    S_x \setminus ([x + {\rm span}(p)] \cup [ a + {\rm span}(p)])
  $.
  Secondly we carry this knowledge over to the bounding line 
  $x + {\rm span}(p)$ of $S_x$.
  The whole straight line $a + {\rm span}(p)$ belongs to 
  $\dom (F)$ as well as at least one point
  $x' \in x + {\rm span}(p)$,
  since $F|_{x + {\rm span}(p)} \not \equiv + \infty$.
  Using the convexity of $\dom F$ we hence obtain
  \begin{gather*}
    \dom F
    =
    {\rm co}(\dom F) 
    \supseteq
    {\rm co}(\{x'\} \cup [a + {\rm span}(p)])
    \supseteq 
    {\rm ri} (S_x),
  \end{gather*}
  i.e. $F$ takes only finite values on every straight line 
  $y + {\rm span}(p) \subseteq {\rm ri}(S_x)$.
  Assume that $F$  is not constant on some line 
  $y + {\rm span}(p) \subseteq {\rm ri}(S_x)$,
  i.e. that there were parameters $\check t, \hat t \in \R$
  with $F(y + \check t p ) < F(y + \hat t p)$.
  Defining the function $F_{y,p}: \R \rightarrow \R$ via
  $F_{y,p}(t) \defeq F(y + t p)$ this reads
  $F_{y,p}(\check t) < F_{y,p}(\hat t)$. 
  Since $F_{y,p}$ is convex the equation
  \eqref{ineq:convexity_condition_outside_of_two_points_of_dom} 
  from Lemma
  \ref{lem:convexity_condition_inside_and_outside}
  would yield 
  \begin{gather*}
    F(y + [(1- \lambda)\check t + \lambda \hat t]p)
    =
    F_{y,p}((1- \lambda)\check t + \lambda \hat t)
    \geq 
    F_{y,p}( \check t) 
      + \lambda (F_{y,p}(\hat t) - F_{y,p}(\check t))
    \rightarrow + \infty
  \end{gather*}
  as $\lambda \rightarrow + \infty$.
  In particular there would exist $t_1, t_2 \in \R$
  such that 
  \begin{gather*}
    F(\underbrace{y + t_2 p}_{\eqdef y_2})
    >
    F(\underbrace{y + t_1 p}_{\eqdef y_1})
    \geq
    F(a)
    =
    F(a + t p)
  \end{gather*}
  for all $t \in \R$. So ${\rm lev}_{F(y_1)}(F)$ would contain 
  not only the point $y_1$ but also the straight line 
  $a + {\rm span}(p)$.
  The convexity of ${\rm lev}_{F(y_1)}(F)$ would therefore give
  \begin{gather*}
    {\rm lev}_{F(y_1)}(F)
    =
    {\rm co} ({\rm lev}_{F(y_1)}(F))
    \supseteq 
    {\rm co} (\{y_1\} \cup [a+ {\rm span}(p)])
    \supseteq 
    {\rm ri}(S_{y_1})
    =
    {\rm ri}(S_y)
  \end{gather*}
  with the nonempty closed strip 
  $S_y = {\rm co}([y + {\rm span}(p)] \cup [ a + {\rm span}(p)])$.
  The lower-semicontinuity of $F$ ensures the closeness of  
  ${\rm lev}_{F(y_1)}(F)$ so that 
  \begin{gather*}
    {\rm lev}_{F(y_1)}(F)
    =
    \overline{ ({\rm lev}_{F(y_1)}(F)) }
    \supseteq
    \overline{{\rm ri}(S_y)}
    \supseteq 
    y + {\rm span}(p)
    \ni
    y_2,
  \end{gather*}
  yielding $F(y_2) \leq F(y_1)$ which contradicts
  $F(y_2) > F(y_1)$.
  So $F$ is constant on every straight line
  $y + {\rm span}(p)$ in ${\rm ri}(S_x)$,
  i.e. $F(y + tp) = F(y)$ for all $t \in \R$.
  Applying Theorem 
  \ref{thm:limit_behaviour_of_proper_convex_lsc_functions}
  to $a \in \dom F$ and an arbitrarily chosen 
  $x^* = x + t^*p \in x + {\rm span}(p)$ we see that
  \begin{gather*}
    F(x^*) 
    =
    \lim_{\mu \uparrow 1} F( (1-\mu) a + \mu x^*)
    =
    \lim_{\mu \uparrow 1} 
    F( \underbrace{(1-\mu)a + \mu x}_{\eqdef y_\mu} 
	  + \underbrace{\mu t}_{\eqdef t_\mu} p ).
  \end{gather*}
  The point $y_\mu$ belongs to the relatively open strip 
  ${\rm ri}(S_x)$ for all $\mu \in (0,1)$,
  so that $F$ is constant on the straight line
  $y_\mu + {\rm span}(p)$.
  Therewith and by Theorem 
  \ref{thm:limit_behaviour_of_proper_convex_lsc_functions}
  we obtain
  \begin{gather*}
    F(y_\mu + t_\mu p)
    =
    F(y_\mu)
    =
    F((1-\mu)a + \mu x )
    \rightarrow
    F(x)
  \end{gather*}
  as $\mu \uparrow 1$.
  Altogether we have $F(x^*) = F(x)$ for all 
  point $x^* \in x + {\rm span}(p)$.
\end{proof}

\begin{remark}
  Demanding that $F$ is lower semicontinuous is important to ensure
  $p \in P[F]$ as the following example shows:
  Consider the function $F:\R^2 \rightarrow \R \cup \{+\infty\}$ given by 
  \begin{gather*}
    F(x_1,x_2) \defeq 
    \begin{cases}
      + \infty & \text{ for } x_1 < 0
    \\
      x_2^2    & \text{ for } x_1 = 0
    \\
      0	       & \text{ for } x_1 > 0
    \end{cases}.
  \end{gather*}
  and regard e.g. $a=(3,0)$ and $p = (0,4)$. Then all assumptions are 
  fulfilled, except for the lower semicontinuity of $F$.
  Moreover the closed right half plane $\dom F$ clearly fulfills 
  $\dom F = \dom F + p$;
  however $p \not \in P[F]$ since $F(\zerovec+p) \not = F(\zerovec)$.
\end{remark}

\begin{theorem} \label{thm:periods_space_of_restricted_function}
  Let $E: X \rightarrow \R \cup \{+\infty\}$ be a convex function, 
  defined on an affine subset $X$ of $\R^n$.
  For any affine subset $A \subseteq X$ and its difference space $U$ we have
  \begin{gather*}
    P[E] \cap U \subseteq P[E|_A].
  \end{gather*}
  We actually have $P[E] \cap U = P[E|_A]$, if in addition
  $E \in \Gamma_0[X]$ and $A \cap \dom E \not = \emptyset$.
\end{theorem}

Before proving this theorem we show by two examples that both
the lower semicontinuity of E and the condition 
$A \cap \dom E \not = \emptyset$ are essential to get the equality
$P[E] \cap U = P[E|_A]$.

\begin{example}
  ~
  \begin{enumerate}
    \item 
      Consider the function $E: \R^3 \rightarrow \R \cup \{+\infty\}$ given by
      \begin{gather*}
        E(x_1,x_2,x_3) 
	\defeq 
	\begin{cases}
	  x_3   	& \text{ if } x_3 > 0,
	\\
	  0		& \text{ if } x_3 = 0 \text{ and } x_2 = 0,
	\\
	  + \infty	& \text{ else }. 
	\end{cases}
      \end{gather*}
      $E$ is obtained from the mapping $\R^3 \rightarrow \R, x \mapsto x_3$
      by restricting its effective domain to the non-closed set
      $\dom E = H^>_{e_3,0} \cup \langle e_1 \rangle$.
      The proper and convex function $E$ is not
      lower semicontinuous, so that $E \not \in \Gamma_0(\R^3)$.
      Both the $x_1 x_2$ plane ${\rm span} \{e_1,e_2\} \eqdef A$ and 
      its translate $A_0 + e_3 \eqdef A'$ are affine subsets of $\R^3$
      that intersect $\dom E$. 
      Although they have the same difference space $U=A$
      the periods spaces $P[E|_A]$ and $P[E|_{A'}]$ are different;
      more precisely 
      \begin{gather*}
        P[E|_A] = P[E] \cap U \subset P[E|_{A'}]
      \end{gather*}
      holds true:
      Clearly 
      $P[E] \cap U = {\rm span}(e_1) \cap A = {\rm span}(e_1) = P[E|_A]$.
      However 
      $
	P[E] \cap U  
	= 
	{\rm span}(e_1)
	\subset {\rm span}(e_1, e_2) 
	= 
	P[E|_{A'}]
      $.
    \item 
      Consider the function $E: \R^3 \rightarrow \R \cup \{+ \infty\}$
      given by
      \begin{gather*}
        E(x_1,x_2,x_3) 
	\defeq 
	\begin{cases}
	  x_3   	& \text{ if } x_3 \leq 0 \text{ and } x_2 = 0,
	\\
	  + \infty	& \text{ else }. 
	\end{cases}
      \end{gather*}
      $E$ is obtained from the mapping $\R^3 \rightarrow \R, x \mapsto x_3$
      by "restricting" it to the closed half-plane
      $\dom E = \{(x_1,0, x_3) \in \R^3: x_1 \in \R, x_3 \leq 0\}$.
      Defining $A,A'$ and $U$ as above we have 
      $A \cap \dom E = {\rm span}(e_1) \not = \emptyset$ but
      $A' \cap \dom E = \emptyset$.
      Clearly $P[E] \cap U = {\rm span}(e_1) = P[E|_A]$. However 
      $
	P[E] \cap U
	=
	{\rm span}(e_1) 
	\subset 
	{\rm span}(e_1,e_2) 
	= 
	U 
	= 
	P[E|_{A'}]
      $,
      since $E|_{A'} \equiv +\infty$.      
  \end{enumerate}
\end{example}

\begin{proof}[Proof of Theorem \ref{lem:restriction_of_ess_smooth_function_on_good_affine_set}]
  Let $p \in P[E] \cap U$. Then 
  \begin{gather*}
    E(x+p) = E(x)
  \end{gather*}
  for all $x \in X$.
  For all $x \in A$ we have $x+ p \in A$ and hence 
  \begin{gather*}
    E|_A(x+p) = E(x+p) = E(x) = E|_A(x),
  \end{gather*}
  for all $x \in A \subseteq X$.
  This shows  $P[E] \cap U \subseteq P[E|_A]$.
  Let now the additional assumptions be fulfilled and let
  $p \in P[E|_A]$.
  Then $p \in U$. Since $E|_A \not \equiv + \infty$, 
  and $E(x+p)=E(x)$ for all $x \in A$ we see by part 
  \upref{enu:convex_function_constant_on_line_segment_spaned_by_3_points}
  of Lemma \ref{lem:convexity_condition_inside_and_outside} that
  $E$ is in particular constant on any line 
  $a + {\rm span}(p)$, $a \in A$ which intersects the 
  nonempty set $\dom E|_A$.
  Lemma \ref{lem:belonging_to_periods_space} gives thus
  $p \in P[E]$ so that $p \in P[E] \cap U$.
  This shows that also the reversed inclusion
  $P[E|_A] \subseteq P[E] \cap U$
  holds true under the additional assumptions.  
\end{proof}

\subsection{Operations that preserve essentially smoothness} 
\label{subsec:operations_that_preserve_essentially_smoothness}

Roughly speaking essential smoothness is preserved when 
performing the following operations on an essentially smooth function 
$H: A \rarr \R \cup \{\pinfty\}$, defined on some affine subspace 
$A$ of $\R^n$:
\begin{itemize}
  \item Restrictions $H|_{\widecheck A}$	
	to an affine subspace $\widecheck A$ of 
	$A$ which intersects ${\rm ri}({\dom H})$
  \item Extensions $F$ of $H$ of the form $F = H \sdirsum 0_{\widecheck P}$
  \item Forming concatenations $F = H \circ M$ with a linear mapping
	whose range intersects ${\rm ri}({\dom H})$,
\end{itemize}
see 
\prettyref{lem:restriction_of_ess_smooth_function_on_good_affine_set},
\prettyref{lem:essentially_smoothness_does_not_depend_on_periodspace}
and
\prettyref{thm:concatenation_ess_smooth_with_linear_mapping}.

\begin{lemma}
\label{lem:restriction_of_ess_smooth_function_on_good_affine_set}
  Let $A$ be an affine subspace of $\R^n$ and 
  $F:A \rightarrow \R \cup \{+\infty\}$ be essentially smooth.
  The restriction $F|_{\check A}$ of $F$ to an affine set 
  $\check A \subseteq A$ stays essentially smooth, if 
  $\check A$ intersects ${\rm ri}(\dom F) [= {\rm int_A}(\dom F)].$
\end{lemma}

The condition
$\check A \cap {\rm ri}(\dom F) \not = \emptyset$ is essential
to preserve the essential smoothness when restricting $F$ to $\check A$.
Cf. example 
\ref{exa:essentially_smoothness_gets_lost_by_concatenation_with_projection}.

\begin{proof}[Proof of Lemma \ref{lem:restriction_of_ess_smooth_function_on_good_affine_set}]
  By definition of ``essentially smooth'', cf. \cite[p. 251]{Rockafellar1970}
  and nearby explanations, see \cite[Lemma 26.2]{Rockafellar1970} and cf.
  \cite[p. 213]{Rockafellar1970} we have
  \begin{itemize}
    \item [$a)$]
      ${\rm int}_A(\dom F) \not = \emptyset$,
    \item [$b)$]
      $F$ is differentiable in every 
      $x \in {\rm int}_A(\dom F) = {\rm ri}(\dom F)$ and
    \item [$c)$]
      the directional derivative
      $F'(x + \lambda(a-x);a-x) \rightarrow -\infty$ as 
      $\lambda \searrow 0$ for every 
      $x \in \partial_A(\dom F) = {\rm rb}(\dom F)$ and every 
      $a \in {\rm int}_A(\dom F) = {\rm ri}(\dom F)$
  \end{itemize}
  Set $\check F \defeq F|_{\check A}$. Then 
  $\dom \check F = \check A \cap \dom F$, so that equation
  \eqref{eq:aff_restricted_to_affine_subset} in Theorem
  \ref{thm:restricting_set_operations_to_an_affine_subset}
  gives 
  $
    {\rm aff}(\dom \check F)
    = 
    \check A \cap {\rm aff}(\dom F)
    =
    \check A \cap A
    =
    \check A
  $,
  ensuring 
  $
    {\rm int}_{\check A}(\dom \check F)
    =
    {\rm ri}(\dom \check F)
  $
  and thus
  $
    \partial_{\check A}(\dom \check F)
    =
    {\rm rb}(\dom \check F)
  $.
  Equation \eqref{eq:ri_restricted_to_affine_subset} of the same 
  theorem gives
  \begin{itemize}
    \item [$\check a)$]
    $
      {\rm int}_{\check A}(\dom \check F)
      =
      {\rm ri}(\dom \check F)
      =
      {\rm ri}(\check A \cap \dom F)
      =
      \check A \cap {\rm ri}(\dom F)
      \not =
      \emptyset
    $.
  \end{itemize}
  Due to 
  $
    {\rm int}_{\check A}(\dom \check F)
    =
    \check A \cap {\rm ri}(\dom F)
    \subseteq
    {\rm ri}(\dom F)
    =
    {\rm int}_A(\dom F)
  $ 
  we know that
  \begin{itemize}
    \item [$\check b)$]
      $\check F = F|_{\check A}$ is differentiable in every
      $x \in {\rm int}_{\check A}(\dom \check F)$.
  \end{itemize}
  Since equation \eqref{eq:rb_restricted_to_affine_subset} from 
  Theorem \ref{thm:restricting_set_operations_to_an_affine_subset}
  ensures
  $
    \partial_{\check A}(\dom \check F)
    =
    {\rm rb}(\dom \check F)
    =
    {\rm rb}(\check A \cap \dom F)
    =
    \check A \cap {\rm rb}(\dom F)
    \subseteq
    {\rm rb}(\dom F)
    =
    \partial_A(\dom F)
  $
  we finally -- still -- have
  \begin{itemize}
    \item [$\check c)$]
      $
        \check F'(x + \lambda(a-x);a-x) 
	=
	F'(x + \lambda(a-x);a-x)
	\rightarrow -\infty
      $ 
      as 
      $\lambda \searrow 0$ 
      for every 
      $
	x \in \partial_{\check A}(\dom \check F) 
	\subseteq
	\partial_A(\dom F)
      $ 
      and every 
      $
	a \in {\rm int}_{\check A}(\dom \check F) 
	\subseteq
	{\rm int}_A(\dom F)
      $.
  \end{itemize}
  Therefore $F|_{\check A} = \check F$ is essentially smooth.
\end{proof}

\begin{lemma} \label{lem:essentially_smoothness_does_not_depend_on_periodspace}
  Let $F: \R^n \rightarrow \R \cup \{+\infty\}$ be a convex function 
  and let ${\rm aff}(\dom F)$ be decomposed as direct sum 
  ${\rm aff}(\dom F) = \check A \oplus \check P$
  of some affine subspace $\check A$ of $\R^n$ and some vector subspace 
  $\check P$ of the periods space $P[F]$. Then the following are equivalent:
  \begin{enumerate}
    \item 
      $F$ is essentially smooth on 
      $\check A \oplus \check P = {\rm aff}(\dom F)$.
    \item 
      $F$ is essentially smooth on $\check A$.
  \end{enumerate}

\end{lemma}

\begin{proof}
  Assume without loss of generality that $\check A$ is placed in a way that
  it even is a vector subspace of $\R^n$ and set 
  $A \defeq \check A \oplus \check P = {\rm aff}(\dom F) = {\rm span}(\dom F)$,
  $f \defeq F|_A$ and $\check f \defeq F|_{\check A}$. We are going to show the following:
  \begin{align}
  \label{eq:nonemptyness_of_interior_is_equivalent}
    {\rm int}_A(\dom f) \not = \emptyset 
    &\Leftrightarrow
    {\rm int}_{\check A}(\dom \check f) \not = \emptyset,
  \\
    f \text{ is differentiable in} & \text{ every } a \in {\rm int}_A(\dom f)
  \nonumber
  \\ 
    \label{eq:differentiability_is_equivalent}
    &\Updownarrow 
  \\
    \check f \text{ is differentiable in} & \text{ every } 
			\check a \in {\rm int}_{\check A}(\dom \check f).
  \nonumber
  \intertext{In case that $f$ and $\check f$ are differentiable in
	      ${\rm int}_A(\dom f)$ and in 
	      ${\rm int}_{\check A}(\dom \check f)$, respectively,
	      we will finally show}
    \|\mathcal{D}f|_{a_k}\|_{A \rightarrow \R} 
    \rightarrow + \infty
    &\text{ for all } 
    (a_k)_k \in BS(\dom f)
  \nonumber
  \\ 
  \label{eq:presence_of_infinit_steep_wall_is_equivalent}
    &\Updownarrow 
  \\
    \|\mathcal{D}\check f|_{\check a_k}\|_{\check A \rightarrow \R} 
    \rightarrow + \infty
    &\text{ for all } 
    (\check a_k)_k \in BS(\dom \check f)
  \nonumber
  \end{align}
  where 
  $BS(\dom f)$ consists of those convergent sequences in 
  ${\rm int}_A(\dom f)$ whose limit point belongs to the 
  relative boundary
  $\partial_A(\dom f)$.
  $BS(\dom \check f)$ is defined accordingly.
  \\
  Note first  that
  $\dom f = \dom \check f \oplus \check P$ gives by Theorem
  \ref{thm:operations_that_interchange_with_direct_sum}
  the equality
  $
    \check A \oplus \check P = A = {\rm aff}(\dom f) 
    = {\rm aff}(\dom \check f) \oplus \check P
  $. Using $\check A \supseteq {\rm aff}(\dom \check f)$ 
  we hence get $\check A = {\rm aff}(\dom \check f)$.
  By the very same theorem we obtain analogously
  $
    \partial_A(\dom f)
    =
    {\rm rb} (\dom \check f \oplus \check P )
    =
    {\rm rb} (\dom \check f) \oplus \check P
    = 
    \partial_{\check A}(\dom \check f)
    \oplus 
    \check P
  $
  and 
  $
    {\rm int}_A(\dom f) 
    =
    {\rm ri}(\dom \check f \oplus \check P)
    =
    {\rm ri}(\dom \check f) \oplus \check P
    = 
    {\rm int}_{\check A}(\dom \check f)
    \oplus 
    \check P
  $.
  The latter equality already shows that 
  \eqref{eq:nonemptyness_of_interior_is_equivalent} 
  is true.
  In order to prove 
  \eqref{eq:differentiability_is_equivalent}
  we will make use of unique decompositions
  $a=\check a + p$ and $h = \check h + q$
  of $a, h \in A$ into $\check a, \check h \in \check A$ and 
  $p, q \in \check P$.
  Assume first the differentiability of $f$ in an arbitrarily chosen
  $a \in {\rm int}_A(\dom f)$; i.e. that there exists a (unique) linear mapping
  $\mathcal{D}f|_a : A \rightarrow \R$ and a function 
  $r_a: A \rightarrow \R$, which is both continuous in $\zerovec$
  and fulfills $r_a(\zerovec) = 0$, such that
  \begin{gather*}
    f(a+h) = f(a) + \mathcal{D}f|_a(h) + r_a(h)\|h\|
  \end{gather*}
  for all sufficiently small $h \in A$.
  For any $\check a \in {\rm int}_{\check A}(\dom \check f)$ we have 
  $
    \check a 
    = 
    \check a + \zerovec 
    \in 
    {\rm int}_{\check A}(\dom \check f) \oplus \check P
    =
    {\rm int}_A(\dom f)
  $.
  So the latter formula stays valid for $a= \check a$ and all 
  sufficiently small $h = \check h \in \check A \subseteq A$.
  Therefore $\check f = f|_{\check A}$ is also differentiable with 
  $
    \mathcal{D} \check f|_{\check a}(\check h) 
    = 
    \mathcal{D} f_a(\check h)
  $
  for all $\check h \in \check A$.
  Assume now to the contrary the differentiability of $\check f$ in an 
  arbitrarily chosen 
  $\check a \in {\rm int}_{\check A}(\dom \check f)$; i.e. that there is a 
  (unique) linear mapping 
  $\mathcal{D}\check f|_{\check a} : \check A \rightarrow \R $ and a function
  $\check r_{\check a}: \check A \rightarrow \R$, which is both continuous 
  in $\zerovec$ and fulfills $\check r_{\check a}(\zerovec) = 0$, such that
  \begin{gather*}
    \check f(\check a+\check h) 
    = 
    \check f(\check a) + \mathcal{D}\check f|_{\check a}(\check h) 
    + \check r_{\check a}(\check h)\|\check h\|
  \end{gather*}
  for all sufficiently small $\check h \in \check A$.
  Any $a \in {\rm int}_A(\dom f) = {\rm int}_{\check A}(\dom \check f)\oplus \check P$
  can be written uniquely as $a = \check a + p$ with 
  $\check a \in {\rm int}_{\check A}(\dom \check f)$.
  For $h \not = \zerovec$ the translational symmetry of $f$ in directions of
  $\check P$ therefore gives
  \begin{align*}
    f(a+h) 
    &= \check f(\check a + \check h)
    \\
    &=   
    \check f(\check a) + \mathcal{D}\check f|_{\check a}(\check h) 
    + \check r_{\check a}(\check h)\|\check h\|
    \\
    &=
    f(a) 
      +
      \underbrace{
      \mathcal{D}\check f|_{\check a}(\check h)
      }_{\eqdef L_a(\check h + q)= L_a(h)}
      +
      \underbrace{
      \check r_{\check a}(\check h) \frac{\|\check h\|}{\|h\|} 
      }_{\eqdef r_a(\check h +q)= r_a(h)}
      \|h\|.
  \end{align*}
  Clearly $L_a: A \rightarrow \R$ is a linear mapping; so we need only to 
  show that extending $r_a: A \setminus \{\zerovec\} \rightarrow \R$ via
  $r_a(\zerovec) \defeq 0 $ yields a function $A \rightarrow \R$ which is 
  continuous in $\zerovec$. Lemma 
  \ref{lem:same_blow_up_factor_gets_all_attachted_affines_strangers_out_of_ball}
  says that there is a constant $C > 0$ such that 
  $\frac{\|\check h\|}{\|h\|} = \frac{\|\check h\|}{\|\check h + q\|}\leq C$.
  Consequently
  $
    |r_a(h)| 
    = 
    |r_a(\check h +q)| 
    \leq 
    C |\check r_{\check a}(\check h)|
    \rightarrow 0
  $
  as $h \rightarrow \zerovec$ 
  (i.e. as the components $\check h, q \rightarrow \zerovec$).
  Thus $f$ is differentiable in $a$ and 
  \begin{gather*}
    \mathcal{D}f|_a(h)
    =
    L_a(\check h + q) 
    = 
    \mathcal{D}\check f|_{\check a}(\check h).
  \end{gather*}
  We finally proof that 
  \eqref{eq:presence_of_infinit_steep_wall_is_equivalent}
  holds true (under the there stated differentiability assumption). For these 
  purpose we will use the found relation between the derivatives of 
  $f$ and $\check f$. For any 
  $
    a 
    = 
    \check a + p 
    \in 
    {\rm int}_{\check A}(\dom \check f) \oplus \check P
    = 
    {\rm int}_A(\dom f)
  $
  we have
  $
    |\mathcal{D}\check f|_{\check a}(\check h)|
    =
    |\mathcal{D}f|_a(\check h)|
  $
  for all $\check h \in \check A$ with $\|\check h\| = 1$.
  In particular 
  $
    \|\mathcal{D}\check f|_{\check a}\|_{\check A \rightarrow \R}
    \leq
    \|\mathcal{D}f|_a\|_{A \rightarrow \R}
  $
  on the one hand.
  Using again the inequality $\|\check h\| \leq C\|\check h + p\| = C \|h\|$ 
  from Lemma 
  \ref{lem:same_blow_up_factor_gets_all_attachted_affines_strangers_out_of_ball}
  we get 
  $
    |\mathcal{D}f|_a(h)| 
    =
    |\mathcal{D}\check f|_{\check a}(\check h)|
    \leq
    \|\mathcal{D}\check f|_{\check a}\|_{\check A \rightarrow \R} \|\check h\|
    \leq 
    C \|\mathcal{D}\check f|_{\check a}\|_{\check A \rightarrow \R}\|h\|
  $
  for all $h \in A$, so that 
  $
    \|\mathcal{D}f|_a\|_{A\rightarrow\R}
    \leq
    C \|\mathcal{D}\check f|_{\check a}\|_{\check A \rightarrow \R}
  $
  on the other hand. Noting that the constant $C$ does not depend on the 
  choice of $a$ we have in total 
  \begin{gather*}
    \|\mathcal{D}\check f|_{\check a}\|_{\check A \rightarrow \R}
    \leq
    \|\mathcal{D}f|_a\|_{A \rightarrow \R}
    \leq
    C \|\mathcal{D}\check f|_{\check a}\|_{\check A \rightarrow \R}
  \end{gather*}
  for all $a= \check a + p \in A$.
  Therefrom and by using 
  $
    \partial_A(\dom f) 
    = 
    \partial_{\check A}(\dom \check f)
    \oplus 
    \check P
  $
  and 
  $
    {\rm int}_A(\dom f) = {\rm int}_{\check A}(\dom \check f) \oplus \check P
  $
  we finally obtain 
  \eqref{eq:presence_of_infinit_steep_wall_is_equivalent}.
\end{proof}

\begin{theorem} \label{thm:concatenation_ess_smooth_with_linear_mapping}
  Let the convex function $E: \R^n \rightarrow \R \cup \{+\infty\}$ 
  be essentially smooth on 
  ${\rm aff}(\dom E)$ and let $M:\R^m \rightarrow \R^n$ be a linear mapping 
  whose range $\mathcal{R}(M)$ intersects ${\rm ri}(\dom E)$. Then the 
  concatenation $F \defeq E \circ M : \R^m \rightarrow \R \cup \{+\infty\}$
  is convex and essentially smooth on ${\rm aff}(\dom F)$.
\end{theorem}

\begin{proof}
  The linearity of $M$ transfers the convexity of $E$ to $F$.
  Consider the restricted functions
  $\check E \defeq E|_{\mathcal{R}(M)}$ and 
  $\check F \defeq F|_{\mathcal{R}(M^*)}$.
  Since $\mathcal{R}(M) \cap {\rm ri}(\dom E) \not = \emptyset$ we can
  apply Lemma \ref{lem:restriction_of_ess_smooth_function_on_good_affine_set}
  to see that $\check E$ is essentially smooth on 
  \begin{gather*}
    \check A_E 
    \defeq 
    \mathcal{R}(M) \cap {\rm aff}(\dom E)
    =
    {\rm aff}(\mathcal{R}(M) \cap \dom E)
    =
    {\rm aff}(\dom \check E),
  \end{gather*}
  where 
  $ 
    \mathcal{R}(M) \cap {\rm aff}(\dom E)
    =
    {\rm aff}(\mathcal{R}(M) \cap \dom E)
  $ 
  holds true by Theorem 
  \ref{thm:restricting_set_operations_to_an_affine_subset}.
  The equation 
  \begin{gather*}
    \check F = \check E \circ \check M,
  \end{gather*}
  where $\check M \defeq M|_{\mathcal{R}(M^*)}$, elucidates
  that $\check F$ and $\check E$ are the very same mapping 
  -- except for the bijective linear transformation 
  $\check M: \mathcal{R}(M^*) \rightarrow \mathcal{R}(M)$
  between their domains of definition.
  Hence $\check F$ is likewise essentially smooth on 
  \begin{gather*}
    \check M^{-1}[\check A_E] 
    =
    \check M^{-1}[{\rm aff}(\dom \check E)]
    =
    {\rm aff}(\check M^{-1}[\dom \check E])
    =
    {\rm aff}(\dom \check F)
    \eqdef
    \check A_F.
  \end{gather*}  
  Applying Lemma 
  \ref{lem:essentially_smoothness_does_not_depend_on_periodspace}
  to $F$, 
  $
    \check A_F 
    \defeq 
    {\rm aff}(\dom \check F) 
    = 
    \mathcal{R}(M^*) \cap {\rm aff}(\dom F)
  $ 
  and $\check P \defeq \mathcal{N}(M)$
  we finally see that $F$ is essentially smooth on 
  ${\rm aff}(\dom F) = {\rm aff}(\dom \check F) \oplus \mathcal{N}(M)$,
  since $F|_{\check A} = \check F$ is essentially smooth on 
  ${\rm aff}(\dom \check F)$; note here that the validity of
  $
    {\rm aff}(\dom F)
    =
    {\rm aff}(\dom \check F \oplus \mathcal{N}(M))
    = 
    {\rm aff}(\dom \check F) \oplus \mathcal{N}(M)
  $
  is guaranteed by Theorem 
  \ref{thm:operations_that_interchange_with_direct_sum}.
\end{proof}

We give two related examples to illustrate the role
of the assumption $\mathcal{R}(M)\cap {\rm ri}(\dom E) \not = \emptyset$.
Although we start with an example where this assumption is not fulfilled but where 
$E \circ M$ is never the less again essentially smooth, we will see in the second example
that we in general can not replace that assumption by the weaker assumption 
$\mathcal{R}(M) \cap \dom E \not = \emptyset$. We use the notations
$\mathbb{H}$ and $Q$ for the {\bf open upper half plane}
$\{w \in \R^2: w_2 > 0\} \subseteq \R^2 = \C$ and the 
{\bf first open quadrant} 
$\{z \in \R^2 : z_1 > 0, z_2 > 0\} \subseteq \C$, respectively.

\begin{figure}[htpb]
\begin{center}
\includegraphics[width=0.35\textwidth]{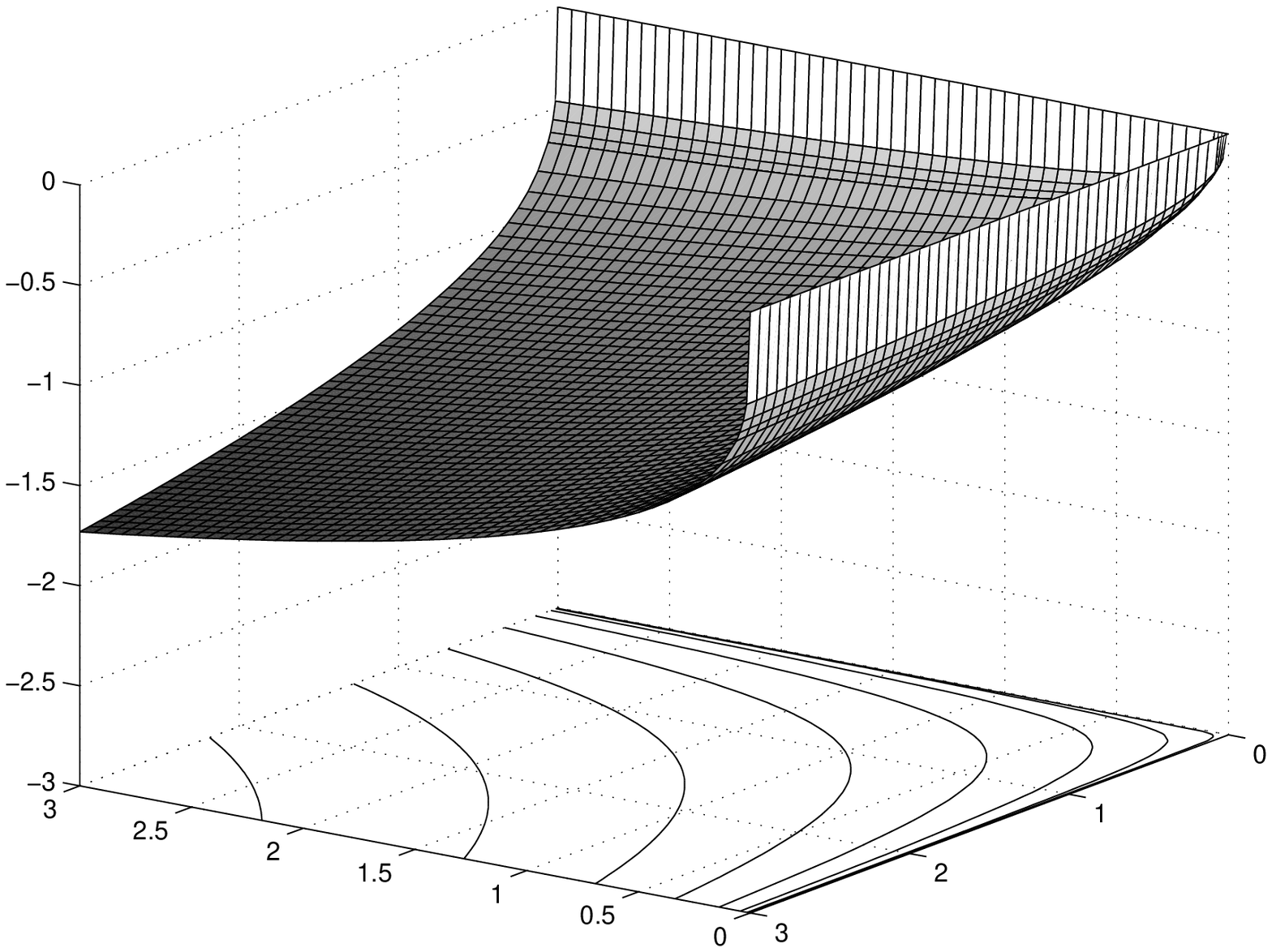}
\includegraphics[width=0.35\textwidth]{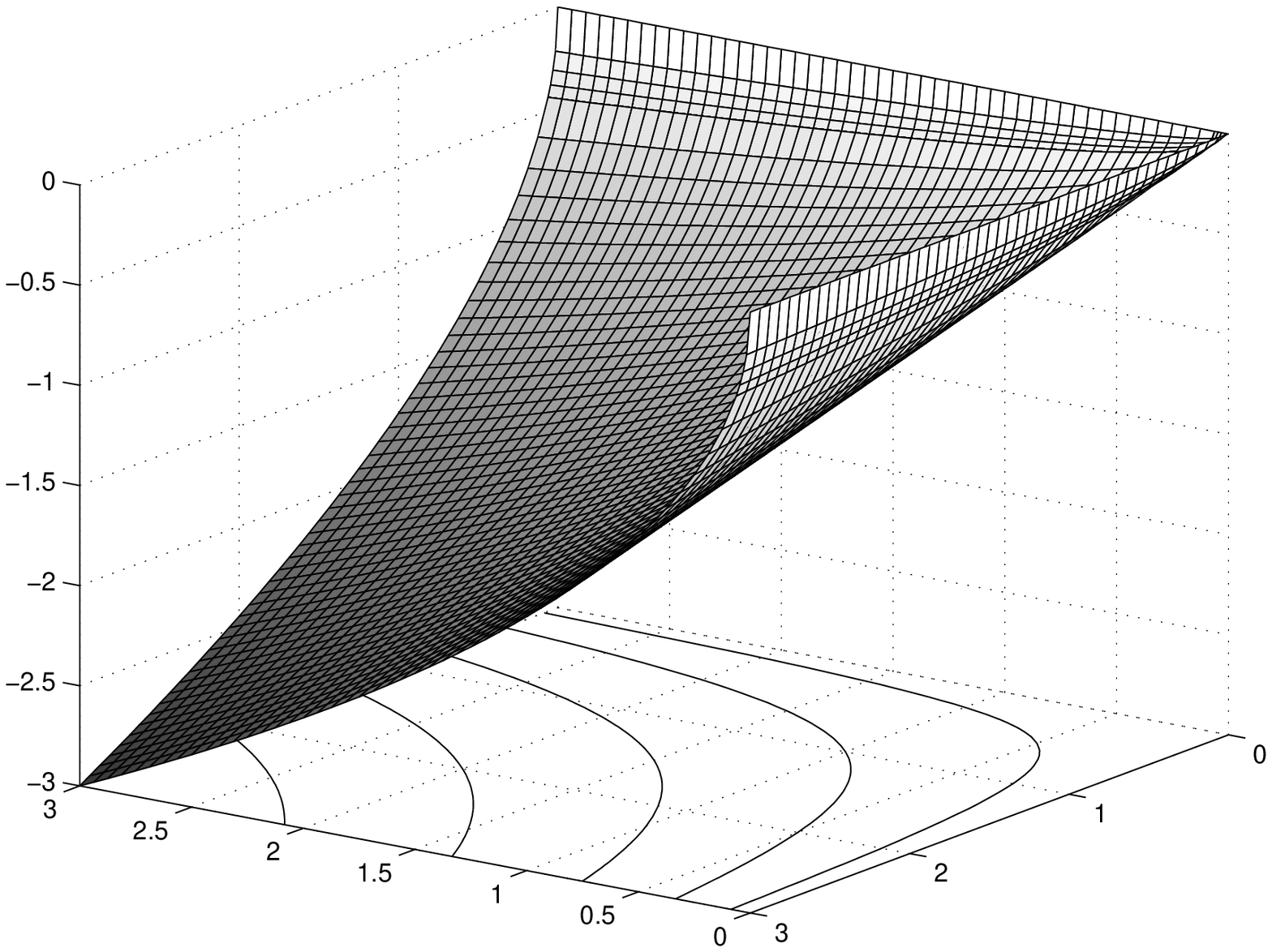}
\caption{
	  Graphs and contour lines of $E_\alpha$ or rather $g_\alpha$.
	  Left for $\alpha = \frac{1}{4} \in (0, \frac{1}{2})$ and right for the border case 
	  $\alpha = \frac{1}{2}$, where 
	  $
	    \|\nabla g_\alpha (z^{(k)})\|_{2}
	    \rightarrow 
	    +\infty
	  $
	  for all boundary points 
	  $z=\lim_{k \rightarrow +\infty} z^{(k)}$ 
	  of $\dom E_{\frac{1}{2}}$, except the origin $(0,0)$.
	  For better quality of the plot a smaller step size was
	  used near the X-axis and the Y-axis, where the norm of the 
	  gradient of $g_\alpha$ is large.
	  \label{fig:essentially_smoothness_get_lost_by_restriction}
	}
\end{center}
\end{figure}

\begin{example} 
\label{exa:essentially_smoothness_gets_lost_by_concatenation_with_projection}
  Consider first the function 
  $\tilde g_\alpha : \overline{\mathbb{H}} \rightarrow \R \cup \{+\infty\}$
  on the closed upper half plane $\overline{\mathbb{H}}$, defined by 
  $\tilde g_\alpha (w) \defeq -w_2^\alpha = -( \Im (w))^\alpha$, with some parameter 
  $\alpha \in (0, + \infty)$.
  Continuing $\tilde g_\alpha$ by setting 
  \begin{gather*}
    \tilde E_\alpha (w) 
    \defeq
    \begin{cases}
      \tilde g_\alpha(w) = -w_2^\alpha & \text{ for } w \in \overline{\mathbb{H}} 
      \\
      +\infty                          & \text{ for } w \in \R^2 \setminus \overline{\mathbb{H}}
    \end{cases}
  \end{gather*}
  we obtain a function $\tilde E_\alpha : \R^2 \rightarrow \R \cup \{+\infty\}$,
  which is convex and essentially smooth for $\alpha \in (0,1)$.
  Concatenation with the linear projection $M:\R^2 \rightarrow \R \times \{0\}$,
  $M(z) \defeq (z_1,0)$ yields the mapping $\tilde F_\alpha = \tilde E_\alpha \circ M$;
  here $\tilde F_\alpha(z) = 0$ for all $z \in \R^2$ elucidates
  that $\tilde F_\alpha$ is both convex and essentially smooth, 
  although $\mathcal{R}(M)$ does not intersect 
  $\overline{\mathbb{H}} = {\rm ri}(\dom \tilde E_\alpha)$.

  The essentially smoothness will, however, be no longer preserved by 
  concatenation with $M$ if we transform $\tilde g_\alpha$'s 
  domain of definition, i.e. the upper closed half plane 
  $\overline{\mathbb{H}} \subseteq \R^2 = \C$, to the first closed quadrant
  $\overline{Q}$ by means of the bijective mapping 
  $h: \overline{Q} \rightarrow \overline{\mathbb{H}}$, given by
  $h(z) \defeq \frac{1}{2} z^2 = (\frac{1}{2}(z_1^2 - z_2^2), z_1 z_2)$:
  The function 
  $g_\alpha \defeq \tilde g_\alpha \circ h : \overline{Q} \rightarrow \overline{\mathbb{H}}$,
  where $g_\alpha(z) = -(z_1 z_2)^\alpha = -z_1^\alpha z_2^\alpha$ and 
  $\alpha \in (0, +\infty)$,
  is infinitely differentiable in $Q$ and continuous in $\overline Q$. Its Hessian
  \begin{gather*}
    H_\alpha |_z
    =
    \alpha z_1^{\alpha -2} z_2^{\alpha -2}
    \begin{pmatrix}
      (1 - \alpha) z_2^2 &   - \alpha z_1 z_2
      \\
      - \alpha z_1 z_2   & (1 - \alpha) z_1^2 
    \end{pmatrix}
  \end{gather*}
  is positive definite for all $z \in Q$, if $\alpha \in (0, \frac{1}{2})$
  by virtue of Sylvester's criterion.
  Therefore the continuous function $g_\alpha $ is strictly convex in $Q$ and
  convex in $\overline{Q}$ for $\alpha \in (0, \frac{1}{2})$.
  For these $\alpha $ we furthermore have
  $\|\nabla g_\alpha (z^{(k)})\|_{2} \rightarrow + \infty$  as 
  $k \rightarrow + \infty$ for any sequence 
  $(z^{(k)})_{k \in \N}$ in $Q$, converging to some boundary point 
  $z^{(\infty)}$ of $Q$,
  see \prettyref{det:gradient_gets_large_when_approaching_boundary}.
  Altogether we see that continuing $g_\alpha $
  by setting 
  \begin{gather*}
    E_\alpha (z) \defeq 
    \begin{cases}
      g_\alpha (z) = -z_1^\alpha z_2^\alpha & \text { if } z \in \overline{Q}
      \\
      +\infty      & \text{ if } z \in \R^2 \setminus \overline{Q}
    \end{cases}
  \end{gather*}
  leads to a function $E_\alpha: \R^2 \rightarrow \R \cup \{+\infty\}$
  which is convex and essentially smooth for $\alpha \in (0, \frac{1}{2}) $.
  However $F_\alpha = E_\alpha \circ M = \iota_{[0, +\infty)\times \R}$
  is not essentially smooth; here $\mathcal{R}(M)$ indeed intersects only 
  $\dom E_\alpha $ but not the relative interior of this effective domain,
  which is consistent with Theorem 
  \ref{thm:concatenation_ess_smooth_with_linear_mapping}.
\end{example}

\subsection{Operations that preserve decomposability into a innerly strictly convex and a constant part} 
\label{subsec:operations_that_preserve_decomposability_into_an_innerly_striclty_convex_and_a_constant_part}

Before giving an overview over the current subsection we need to 
introduce a manner of speaking,
in which we use the extension of semidirect sums $F_1 \sdirsum F_2$
from \prettyref{def:more_general_definition_of_semidirect_sums}.

\begin{definition}
  Let $X_1$ be a nonempty affine subset of $\R^n$.
  We call a function $E_1: X_1 \rarr \R \cup \{\pinfty\}$
  {\bf innerly stricly convex} iff $E_1$ is strictly convex in 
  ${\rm ri}(\dom E_1) = {\rm int}_{{\rm aff}(\dom E_1)}(\dom E_1)$.
  Any semi-direct sum $E = E_1 \sdirsum E_2: X \rarr \R \cup \{\pinfty\}$
  of an innerly strictly convex function 
  $E_1: X_1 \rarr \R \cup \{\pinfty\}$ and some constant function 
  $E_2: X_2 \rarr \R$, defined on some vector subspace $X_2$
  will also be called 
  {\bf decomposition} of $E$ {\bf into an innerly strictly convex part} 
  $E_1$ {\bf and a constant part} $E_2$.
\end{definition}

Roughly speaking we show in this subsection that the following operations
on a proper convex and lower semicontinuous function 
$E: X \rarr \R \cup \{\pinfty\}$ yield a new function which still has 
a decomposition into an innerly strictly convex part and a constant part:
\begin{itemize}
  \item Restrictions $E|_B$	
	to an affine subspace $B \subseteq A \eqdef {\rm aff}(\dom E)$ which intersects ${\rm ri}({\dom E})$,
  \item Forming concatenations $F = E \circ M$ with a linear mapping
	whose range intersects ${\rm ri}({\dom E})$,
\end{itemize}
see 
\prettyref{lem:decompostion_into_periodspace_part_and_strict_convex_part_carries_over_to_affine_subset}
and 
\prettyref{thm:concatenation_stricly_convex_on_ri_with_linear_mapping},
respectively.

\begin{lemma} \label{lem:decompostion_into_periodspace_part_and_strict_convex_part_carries_over_to_affine_subset}
  Let $E: X \rightarrow \R \cup \{+\infty\}$ be 
  a proper, convex and lower semicontinuous
  function on some nonempty affine subset 
  $X \subseteq \R^n$ and let there exist a decomposition 
  ${\rm aff}(\dom E) = \check A \oplus \check P$
  of ${\rm aff}(\dom E) \eqdef A$ into a subspace $\check P$ of $P[E]$
  and an affine subspace $\check A \subseteq \R^n$ such that $E$ is 
  strictly convex on
  ${\rm int}_{\check A}(\dom E|_{\check A})$. Then 
  \begin{enumerate}
    \item \label{enu:subspace_of_periods_space_is_whole_space}
      In fact we even have $\check P = P[E]$.
    \item \label{enu:lem:decompostion_into_periodspace_part_and_strict_convex_part_carries_over_to_affine_subset} 
      Any affine subset $B \subseteq A$ that intersects ${\rm ri}(\dom E)$
      has a decomposition $B= \check B \oplus \check Q$
      into a vector subspace $\check Q \subseteq \check P = P[E]$ 
      and some affine subspace $\check B \subseteq \R^n$ such that
      $E$ is strictly convex on ${\rm int}_{\check B}(\dom E|_{\check B})$.
  \end{enumerate}
  Moreover 
  ${\rm int}_{\check A}(\dom E|_{\check A}) = {\rm ri}(\dom E|_{\check A})$
  and
  ${\rm int}_{\check B}(\dom E|_{\check B}) = {\rm ri}(\dom E|_{\check B})$
  are nonempty sets.

\end{lemma}

\begin{proof}
  Since $E$ is proper and convex we have
  $\check A \not = \emptyset$ and 
  $
  {\rm int}_{\check A}(\dom E|_{\check A})
  =
  {\rm ri}(\dom E|_{\check A})
  \not =
  \emptyset
  $
  by Lemma
  \ref{lem:decomposition_into_some_set_and_a_subspace_of_periods_space}.
  Due to $B \cap {\rm ri}(\dom E) \not = \emptyset$ the function $E|_B$ 
  is still proper and convex;
  so the same Lemma gives also
  $\check B \not = \emptyset$ and 
  $
  {\rm int}_{\check B}(\dom E|_{\check B})
  =
  {\rm ri}(\dom E|_{\check B})
  \not =
  \emptyset
  $.
  \\[0.7ex]
  \upref{enu:subspace_of_periods_space_is_whole_space} Since $\check P$ is 
  a subspace of the periods space $P[E]$ we clearly have 
  $\check P \subseteq P[E]$. The reverse inclusion $P[E] \subseteq \check P$
  also holds true:
  Let $p \in P[E]$ and chose any 
  $a_0 \in {\rm int}_{\check A}(\dom E|_{\check A})$ and think of it as 
  new origin. Since $E(a_0+p) = E(a_0) < +\infty$ we have
  $a_0 + p \in \dom E \subseteq {\rm aff}(\dom E) =\check A \oplus \check P$,
  so that $a_0 + p = \check a + \check p$, for some $\check a \in \check A$
  and $\check p \in \check P$. Hence 
  $\check a - a_0 = p - \check p \in P[E]$.
  The affine combination
  $a_0 + \lambda (\check a - a_0)$ still belongs to $\check A$ for all 
  $\lambda \in \R$ and hence even to ${\rm int}_{\check A}(\dom E|_{\check A})$ for
  all sufficiently small chosen $\lambda > 0$.
  Choose such a $\lambda > 0 $ and consider the possibly degenerated 
  line segment 
  $
    {\rm co}(a_0, a_0 + \lambda (\check a - a_0)) 
    \subseteq 
    {\rm int}_{\check A}(\dom E|_{\check A})
  $.
  On the one hand $E$ is strictly convex on the latter set and hence on our 
  line segment. On the other hand $\check a - a_0 \in P[E]$
  means that $E$ is constant on this line segment.
  Both can be true only if our line segment is degenerated to one single 
  point, i.e. if $a_0 = a_0 + \lambda (\check a - a_0)$.
  This gives $\zerovec = \check a - a_0 = p - \check p$, 
  so that indeed $p = \check p \in \check P$.
  \\[0.7ex]
  \upref{enu:lem:decompostion_into_periodspace_part_and_strict_convex_part_carries_over_to_affine_subset}
  Let 
  $
    b_0 \in B \cap {\rm ri}(\dom E) 
  = 
    {\rm ri}(\dom E \cap B)
  =
    {\rm ri}(\dom E|_{B})
  $, where we used Theorem 
  \ref{thm:restricting_set_operations_to_an_affine_subset}.
  Without loss of generality we may assume $b_0 = \zerovec$;
  otherwise we could replace $E$ by $E(\cdot - b_0)$ 
  without changing the truth value of the other assumptions and 
  assertions of the lemma.
  By Theorem
  \ref{thm:periods_space_of_restricted_function} and the already proven part   
  \upref{enu:subspace_of_periods_space_is_whole_space} we then have
  \[
    \check Q \defeq P[E|_B] = P[E] \cap B \subseteq P[E] = \check P.
  \]
  Choose now firstly any subspace $\check B$ of $B$ with 
  $B = \check B \oplus \check Q$,
  then some subspace $Q'$ of $P[E] = \check P$ with 
  $\check P = \check Q \oplus Q'$ 
  and finally some subspace $B'$ of $A$ with 
  \[
    A 
  = 
    B' \oplus (\check B \oplus \check Q \oplus Q')
  =
    \underbrace{B' \oplus \check B}_{\eqdef\tilde A}
      \oplus 
      \underbrace{\check Q \oplus Q'}_{= \check P}.
  \]
  By Theorem 
  \ref{thm:restricting_function_to_equiv_affine_sets_give_ess_the_same_function}
  we know that $\check E \defeq E|_{\check A}$ and 
  $\tilde E \defeq E|_{\tilde A}$ are the very same mapping, except 
  for a bijective affine transformation 
  $\tilde \alpha : \check A \rightarrow \tilde A$
  between their domains of definitions, which links these functions via 
  $\check E = \tilde E \circ \tilde \alpha$.
  Consequently $\check E$ is strictly convex on a subset 
  $\check S \subseteq \check A$ if and only if $\tilde E$
  is strictly convex on $\tilde \alpha[\check S] \eqdef \tilde S$.
  Choosing
  $
    \check S 
    \defeq 
    {\rm int}_{\check A}(\dom \check E)
    =
    {\rm int}_{\check A}(\dom E|_{\check A})
  $
  we see that $\tilde E$ is strictly convex on 
  $
    \tilde \alpha [{\rm int}_{\check A}(\dom \check E)]
    =
    {\rm int}_{\tilde A}(\tilde \alpha [\dom \check E])
    =
    {\rm int}_{\tilde A}(\dom \tilde E)
    =
    {\rm int}_{\tilde A}(\dom E|_{\tilde A})
  $.
  So $B = \check B \oplus \check Q$ would give the needed decomposition, if 
  $
    {\rm int}_{\check B}(\dom E|_{\check B})
    \subseteq 
    {\rm int}_{\tilde A}(\dom E|_{\tilde A})
  $
  can be verified.
  Due to $\check B \subseteq \tilde A$ it suffices to show 
  $
    {\rm int}_{\check B}(\dom E|_{\check B})
    =
    {\rm ri}(\dom E) \cap \check B
  $
  and 
  $
    {\rm int}_{\tilde A}(\dom E|_{\tilde A})
    =
    {\rm ri}(\dom E) \cap \tilde A
  $.
  In order to prove the first equation we note that 
  $\check B$ intersects ${\rm ri}(\dom E)$ in $b_0 = \zerovec$
  so that equation \eqref{eq:aff_restricted_to_affine_subset}
  in Theorem 
  \ref{thm:restricting_set_operations_to_an_affine_subset}
  gives 
  $
    {\rm aff}(\dom E|_{\check B})
    =
    {\rm aff}(\dom E \cap \check B)
    =
    {\rm aff}(\dom E) \cap \check B
    =
    \check B
  $.
  Therefore and by equation 
  \eqref{eq:ri_restricted_to_affine_subset} in Theorem
  \ref{thm:restricting_set_operations_to_an_affine_subset}
  we indeed get 
  $
    {\rm int}_{\check B}(\dom E|_{\check B})
    =
    {\rm ri}(\dom E|_{\check B})
    =
    {\rm ri}(\dom E \cap \check B)
    =
    {\rm ri}(\dom E) \cap \check B
  $.
  Just analogously we obtain 
  $
    {\rm int}_{\tilde A}(\dom E|_{\tilde A})
    =
    {\rm ri}(\dom E) \cap \tilde A
  $.
\end{proof}

\begin{theorem} \label{thm:concatenation_stricly_convex_on_ri_with_linear_mapping}
  Let $E: \R^n \rightarrow \R \cup \{+\infty\}$ be 
  proper,
  convex as well as lower semicontinuous
  and let
  $M:\R^m \rightarrow \R^n$ be a linear mapping whose range $\mathcal{R}(M)$
  intersects ${\rm ri}(\dom E)$. 
  Assume further that there exists a decomposition 
  \begin{gather*}
    {\rm aff}(\dom E) = \check A_E \oplus \check P_E
  \end{gather*}
   of 
  ${\rm aff}(\dom E)$ into a subspace $\check P_E$ of 
  $P[E]$ and an 
  affine subspace
  $\check A_E \subseteq \R^n$ 
  such that $E$ is strictly convex on 
  $
    {\rm int}_{\check A_E}(\dom E|_{\check A_E})
  $.
  Then
  the function 
  $F \defeq E \circ M : \R^m \rightarrow \R \cup \{+\infty\}$
  is 	again 
  proper, convex and lower semicontinuous
  and
  there exists a decomposition 
  \begin{gather*}
    {\rm aff}(\dom F) = \check A_F \oplus \check P_F
  \end{gather*}
    of 
  ${\rm aff}(\dom F)$ into a subspace $\check P_F$ of $P[F]$ and 
  an 
  affine 
  subspace $\check A_F \subseteq \R^m$ such that $F$ is strictly convex on 
  $
    {\rm int}_{\check A_F}(\dom F|_{\check A_F}) 
  $.
\end{theorem}

\todo{Checken:{\bf HAT SICH ER-LEDIGT?! WICHTIG}: Noch Beweis, dasz 1) jeweils int = ri ist 
dazu und dasz 2) ri hier auch nicht leer ist}

\begin{remark}
  Note that Lemma 
  \ref{lem:decomposition_into_some_set_and_a_subspace_of_periods_space}
  implies that all sets that occur in the above theorem
  are nonempty.
\end{remark}

\begin{proof}[Proof of \prettyref{thm:concatenation_stricly_convex_on_ri_with_linear_mapping}]
  The mapping $F = E \circ M$ is surely again convex and lower 
  semicontinuous. Due to 
  $
    \mathcal{R}(M) \cap \dom E 
    \supseteq 
    \mathcal{R}(M) \cap {\rm ri}(\dom E)
    \neq \emptyset
  $
  it is also again proper.
  Since $\R^m = \mathcal{R}(M^*) \oplus \mathcal{N}(M)$ and since clearly
  $\mathcal{N}(M) \subseteq P[F]$ we have 
  \begin{gather*}
    \dom F = \dom F|_{\mathcal{R}(M^*)} \oplus \mathcal{N}(M).
  \end{gather*}
  It suffices to prove that there is a decomposition 
  \begin{gather} \label{eq:decomposition_of_affinely_restricted_dom}
    {\rm aff}(\dom F|_{\mathcal{R}(M^*)})
    =
    \check A_F \oplus Q_F    
  \end{gather}
  with a subspace $Q_F \subseteq P[F|_{\mathcal{R}(M^*)}]$ and some 
  affine subset $\check A_F \subseteq \R^m$,
  such that $F$ is strictly convex on 
  ${\rm int}_{\check A_F}(\dom F|_{\check A_F})$,
  since this decomposition then yields, by virtue of equation
  \eqref{eq:aff_of_direct_sum_with_affine_subset}
  in Theorem
  \ref{thm:operations_that_interchange_with_direct_sum},
  the needed decomposition
  \begin{align*}
    {\rm aff}(\dom F)
    &= 
    {\rm aff}(\dom F|_{\mathcal{R}(M^*)} \oplus \mathcal{N}(M))
  \\&=
    {\rm aff}(\dom F|_{\mathcal{R}(M^*)}) \oplus \mathcal{N}(M)
  \\&=
    \check A_F \oplus 
      \underbrace{Q_F \oplus \mathcal{N}(M)}_{\eqdef\check P_F};
  \end{align*}
  note herein that not only $\mathcal{N}(M)$ is a subspace of $P[F]$
  but also $Q_F$:
  Let $q \in Q_F \subseteq P[F|_{\mathcal{R}(M^*)}]$ and write every 
  $x' \in {\rm aff}(\dom F)$ in the form 
  $x' = a' + q' + n'$ with $a' \in \check A_F, q' \in Q_F$
  and $n' \in \mathcal{N}(M) \subseteq P[F]$.
  Since $a' + q' \in {\rm aff}(\dom F|_{\mathcal{R}(M^*)})$
  we then indeed obtain 
  \begin{align*}
    F(x'+q) 
    &= 
    F(a' + q' + q + n') 
    =
    F(a' + q' + q)
    \\
    &=
    F(a' + q')
    =
    F(a' + q' + n')
    =
    F(x')
  \end{align*}
  for every $x' \in {\rm aff}(\dom F)$,
  i.e. $q \in P[F]$.
  In order to prove that a decomposition as in 
  \eqref{eq:decomposition_of_affinely_restricted_dom}
  really exists we consider the restricted functions 
  $\tilde F \defeq F|_{\mathcal{R}(M^*)}$, $\tilde E \defeq E|_{\mathcal{R}(M)}$
  and $\tilde M \defeq M|_{\mathcal{R}(M^*)}$.
  The equation $\tilde F = \tilde E \circ \tilde M$ then elucidates
  that $\tilde F$ and $\tilde E$ are the very same mapping 
  -- except for the linear homeomorphism 
  $\tilde M : \mathcal{R}(M^*) \rightarrow \mathcal{R}(M)$
  between their domains of definition.
  Hence our task to prove that there is a decomposition as in 
  \eqref{eq:decomposition_of_affinely_restricted_dom}
  is equivalent to prove that there exists a decomposition 
  \begin{gather*}
    {\rm aff}(\dom E|_{\mathcal{R}(M)})
    =
    \check B_E \oplus Q_E
  \end{gather*}
  of ${\rm aff}(\dom E|_{\mathcal{R}(M)})$ into 
  a subspace $Q_E \subseteq P[E|_{\mathcal{R}(M)}]$ and
  some affine subset $\check B_E \subseteq \R^n$ 
  such that $E$ is strictly convex on 
  ${\rm int}_{\check B_E}(\dom E|_{\check B_E})$.
\todo{checken: Ist hier wirklich ri = int $\check A_E$?}
  To this end we set 
  \begin{align*}
    B \defeq{}& {\rm aff}(\dom E|_{\mathcal{R}(M)})
       = {\rm aff}(\mathcal{R}(M) \cap \dom E)
       = \mathcal{R}(M) \cap {\rm aff}(\dom E)
       \subseteq    
       {\rm aff}(\dom E)
       \eqdef A,
  \end{align*}
  where we have used again equation
  \eqref{eq:aff_restricted_to_affine_subset}.
  The decomposition $A = \check A_E \oplus \check P_E$ fulfills the
  assumption of Lemma 
  \ref{lem:decompostion_into_periodspace_part_and_strict_convex_part_carries_over_to_affine_subset}.
  Part 
  \upref{enu:lem:decompostion_into_periodspace_part_and_strict_convex_part_carries_over_to_affine_subset}
  of this lemma gives now a decomposition 
  \begin{gather*}
    {\rm aff}(\dom E|_{\mathcal{R}(M)}) = \check B \oplus \check Q,
  \end{gather*}
  where $\check B \subseteq \R^n$ is an affine subset such that 
  $E$ is strictly convex on 
  ${\rm int}_{\check B_E}(\dom E|_{\check B_E})$
  and where $\check Q \subseteq \check P_E \subseteq P[E]$.
  Setting $\check B_E \defeq B$ and $Q_E \defeq \check Q$ we are done,
  since the demanded $\check Q \subseteq P[E|_{\mathcal{R}(M)}]$
  really holds true:
  Due to the banal 
  $
    b + \check Q 
    \subseteq 
    \check B \oplus \check Q 
    \subseteq 
    {\rm aff}(\dom E|_{\mathcal{R}(M)})
    =
    B
  $
  for any $b \in \check B \subseteq B$
  we see that $\check Q$ is a subspace of $B$'s difference space 
  $B - b \eqdef U$. Thereby and by 
  \prettyref{thm:periods_space_of_restricted_function}
  we now indeed obtain
  $\check Q = \check Q \cap U \subseteq P[E] \cap U \subseteq P[E|_B]$.
\end{proof}

\subsection[Existence and direction of $\argmin(F+G)$ for certain classes of functions]
  {Existence and direction of \bm{$\argmin(F+G)$} for certain
   classes of functions}
\label{subsec:existence_and_direction_of_argmin_of_sum}

The next lemma gives a necessary criterion in order to ensure
that a function of the form $F+G$ has a minimizer.
The core of the proof consists in showing that the convex function 
$F+G$ has a bounded nonempty level set, i.e. is coercive.
The inequality from Lemma
\ref{lem:same_blow_up_factor_gets_all_attachted_affines_strangers_out_of_ball}
helps in part \upref{enu:intersection_of_lower_level_sets_bounded}


\begin{lemma} \label{lem:levelsets_and_existence_of_minimizer_of_sum}
  Let $\R^n$ be decomposed as direct sums 
  $\R^n = U_1  \oplus U_2$ and $\R^n = V_1  \oplus V_2$ 
  of vector subspaces $U_1, U_2$ and $V_1, V_2$, respectively. 
  Let $\ F , G \in \Gamma_0(\R^n)$ be functions which inhere
  the translation invariances 
  \begin{align*}
     F (x) &=  F (x+u_2),     \\
    G(x)    &=    G(x+v_2)
  \end{align*}
  for all $x\in\R^n$, $u_2\in U_2$ and $v_2\in V_2$. 
  Then the following holds true for levels $\alpha, \beta \in \R$:
  \begin{enumerate}
    \item \label{enu:intersection_of_lower_level_sets_empty_or_unbounded}
      ${\rm lev}_\alpha ( F ) \cap{\rm lev}_\beta(G)$ is empty or unbounded, if 
      $U_2 \cap V_2 \supset \{\zerovec\}$.
    \item \label{enu:intersection_of_lower_level_sets_bounded}
      ${\rm lev}_\alpha ( F ) \cap{\rm lev}_\beta(G)$ is bounded (possibly empty), if
      $U_2 \cap V_2 = \{\zerovec\}$,
      and
      ${\rm lev}_\alpha ( F |_{U_1})$, ${\rm lev}_\beta(G|_{V_1})$ are bounded.
    \item \label{enu:sum_takes_it_minimum}
      $ F  + G$ takes its minimum in $\R$, if 
      $\dom F \cap \dom G \not = \emptyset$,
      $U_2 \cap V_2 = \{\zerovec\}$
      and
      ${\rm lev}_{\alpha} ( F |_{U_1})$, ${\rm lev}_{\beta}(G|_{V_1})$ 
      are nonempty and bounded.
      Moreover the set $\argmin(F+G)$ of minimizers is compact in this
      case.
  \end{enumerate}
\end{lemma}

\begin{proof}
  We use the abbreviations $ f  \defeq  F |_{U_1}$ and $g\defeq G|_{V_1}$.

  \ref{enu:intersection_of_lower_level_sets_empty_or_unbounded}) 
  Since in case of ${\rm lev}_\alpha ( F ) \cap{\rm lev}_\beta(G) = \emptyset$ there is nothing to show, 
  we assume that there is a $z_1 \in {\rm lev}_\alpha ( F ) \cap{\rm lev}_\beta(G)$. Choose any 
  $z_2 \in U_2 \cap V_2$ with $z_2 \not = \zerovec$. Due to
  $z_1 + \lambda z_2 \in {\rm lev}_\alpha ( F ) \cap{\rm lev}_\beta(G)$ for all $\lambda \in \R$,
  a whole affine line is contained in ${\rm lev}_\alpha ( F ) \cap{\rm lev}_\beta(G)$.
  So the latter set is unbounded.

  \ref{enu:intersection_of_lower_level_sets_bounded})
  Let $U_2 \cap V_2 = \{\zerovec\}$ and let ${\rm lev}_\alpha ( f )$, ${\rm lev}_\beta(g)$ be bounded.
  If the set ${\rm lev}_\alpha ( F ) \cap{\rm lev}_\beta(G)$ was unbounded, 
  it would contain an unbounded sequence of points 
  $z^{(k)}$, $k\in \N$.
  Due to ${\rm lev}_\alpha ( F ) = {\rm lev}_\alpha ( f ) \oplus U_2$ 
  and    ${\rm lev}_\beta(G)      = {\rm lev}_\beta  (g)    \oplus V_2$
  the $z^{(k)}$ could be written in the form
  $z^{(k)} = u_1^{(k)} + u_2^{(k)} = v_1^{(k)} + v_2^{(k)}$ 
  with first components
  $u_1^{(k)} \in {\rm lev}_\alpha ( f )$, 
  $v_1^{(k)} \in {\rm lev}_\beta  (g)   $,
  forming bounded sequences,
  and second components
  $u_2^{(k)} \in U_2$,
  $v_2^{(k)} \in V_2$,
  forming unbounded sequences.
  Lemma \ref{lem:same_blow_up_factor_gets_all_attachted_affines_strangers_out_of_ball}
  ensures that there is a constant $C > 0$ such that
  $\|u_2^{(k)}- v_2^{(k)}\| \geq C^{-1} \|u_2^{(k)}\|$
  for all $k \in \N$. The unboundedness of the sequence 
  $(\|u_2^{(k)}\|)_{k\in \N}$ 
  along with the boundedness of the sequences 
  $(\|u_1^{(k)} \|_2)_{k\in \N}$ and  $(\|v_1^{(k)} \|_2)_{k\in \N}$
  would therefore result in 
  \begin{alignat*}{1}
    0 
    &= 
    \|z^{(k)} - z^{(k)}\|_2
    = 
    \|u_1^{(k)} - v_1^{(k)} + u_2^{(k)} - v_2^{(k)}\|_2
    \geq 
    \|u_2^{(k)} - v_2^{(k)} \|_2 - \|u_1^{(k)} -v_1^{(k)} \|_2
    \\
    &\geq
    C^{-1}\|u_2^{(k)}\|_2  - (\|u_1^{(k)} \|_2 + \|v_1^{(k)} \|_2)
    \rightarrow + \infty
  \end{alignat*}
  -- a contradiction.
  \\
  \ref{enu:sum_takes_it_minimum})
  Since the level sets ${\rm lev}_{\alpha}( f )$ and
  ${\rm lev}_{\beta}(g)$ of 
  the proper, convex and lower semicontinuous functions $f,g$  
  are nonempty and bounded we know that
  all level sets of $f$ and $g$ are bounded, cf. 
  \cite[Corollary 8.7.1]{Rockafellar1970}.
  Since $\dom F \cap \dom G \not = \emptyset$ there are levels 
  $\tilde\alpha, \tilde\beta \in \R$ with 
  ${\rm lev}_{\tilde\alpha}( F ) \cap {\rm lev}_{\tilde\beta}(G) \not = \emptyset$.
  The bounded sets ${\rm lev}_{\tilde\alpha} ( f )$ and ${\rm lev}_{\tilde\beta}(g)$ are nonempty,
  due to ${\rm lev}_{\tilde\alpha}( f ) \oplus U_2 = {\rm lev}_{\tilde\alpha}( F ) \not = \emptyset$
  and ${\rm lev}_{\tilde\beta}(g) \oplus V_2 = {\rm lev}_{\tilde\beta}(G) \not= \emptyset $.
  Consequently $ f $ and $g$ are bounded from below, see
  \prettyref{det:functions_f_and_g_bounded_from_below}.
  Without loss of generality we may therefore assume $ f  \geq 0$ and $g \geq 0$
  (otherwise we can set 
  $m_ f  \defeq \inf_{u_1\in U_1} f (u_1)$, 
  $m_g    \defeq \inf_{v_1 \in V_1}  g(v_1)$ 
  and replace $ f $, $F$, $\tilde\alpha$ and $g$, $G$, $\tilde\beta$ by 
  $ f    - m_ f $,
  $ F    - m_f $,
  $\tilde\alpha - m_ f $
  and 
  $g      - m_g$,
  $ G     - m_g$,
  $\tilde\beta  - m_g$, 
  respectively),
  i.e. $ F  \geq 0$ and $G \geq 0$.
  Next we show that ${\rm lev}_{\tilde\alpha + \tilde\beta}( F +G)$ is a nonempty compact set.
  We have 
  $
    {\rm lev}_{\tilde\alpha + \tilde\beta}( F  + G) 
    \supseteq 
    {\rm lev}_{\tilde\alpha}( F ) \cap {\rm lev}_{\tilde\beta}(G) \not = \emptyset
  $. 
  Furthermore ${\rm lev}_{\tilde\alpha + \tilde\beta}( F  + G)$ is closed 
  due to being a level set of a lower semicontinuous function.
  Lastly 
  $
    {\rm lev}_{\tilde\alpha + \tilde\beta}( F  + G)
    \subseteq
    {\rm lev}_{\tilde\alpha + \tilde\beta}( F ) \cap {\rm lev}_{\tilde\alpha + \tilde\beta}(G)
  $
  is bounded by \eqref{enu:intersection_of_lower_level_sets_bounded}, since
  the needed boundedness of 
  ${\rm lev}_{\tilde\alpha + \tilde\beta}( f )$ and
  ${\rm lev}_{\tilde\alpha + \tilde\beta}(g)$ 
  is only a special case of the already mentioned level boundedness of
  $f$ and $g$ and therewith ensured.
  %
  Hence ${\rm lev}_{\tilde\alpha + \tilde\beta}( F  + G)$ is non empty and compact.
  Therefore the (proper) lower semicontinuous function 
  $
    ( F  + G)|_{{\rm lev}_{\tilde\alpha + \tilde\beta}( F  + G)}
    =
     F  + G + \iota_{{\rm lev}_{\tilde\alpha + \tilde\beta}( F  + G)}
  $
  must be minimized by an 
  $\check u \in {\rm lev}_{\tilde\alpha + \tilde\beta}( F  + G)$, see
  \cite[1.10 Corollary]{RW04} or 
  \prettyref{thm:lsc_function_takes_minimum}, 
  which clearly also minimizes $ F  + G$.
  Finally we set
  $
    \gamma 
    \defeq 
    F(\check u) + G(\check u) 
    \in 
    (-\infty, \tilde \alpha + \tilde \beta ]
  $
  and note that $\argmin(F+G) = {\rm lev}_\gamma (F+G)$
  is a closed subset of the compact set 
  ${\rm lev}_{\tilde \alpha + \tilde \beta}(F+G)$
  and hence itself compact.
\end{proof}

Next we are interested in the direction of $\argmin (F + G)$.
We will see that -- under certain assumptions -- we have
$(\argmin (F + G) - \argmin (F + G) ) \subseteq P[F]$, which 
is the core ingredient to see that $F$ and $G$ are constant 
on $\argmin (F + G)$.

\begin{lemma} \label{lem_unique}
Let the Euclidean space $\R^n$ be decomposed into the direct sum $\R^n = U_1 \oplus U_2$ 
of two subspaces $U_1,U_2$ and let 
$ F :\R^n\rightarrow \R \cup \{ +\infty \}$ be a convex function which inheres
the translation invariance 
$ F (x) =  F (x+u_2)$ for all $x\in\R^n$ and 
$u_2\in U_2$.  
Furthermore, let $G:\R^n \rightarrow \R \cup \{ +\infty \}$ 
be any convex function. Then the following holds true:
\begin{enumerate}
  \item \label{enu:direction_of_argmin__general_version}
    If $\dom  F \cap \dom G \not = \emptyset$ and 
    $ F $ is strictly convex on $U_1$ then
    all $\hat{x},\tilde{x}\in \argmin_{x \in \R^n} \{  F (x) +G(x) \}$ 
    fulfill 
    $\hat{x}-\tilde{x}\in U_2$ 
    and
    $ F (\hat{x})= F (\tilde{x})$, $G(\hat{x})=G(\tilde{x})$.
  \item \label{enu:direction_of_argmin__essentially_smooth_and_ess_convex_version}
    If ${\rm ri} (\dom F) \cap {\rm ri}( \dom G) \not = \emptyset$ and 
    $ F $ is essentially smooth on $U_1$ and strictly convex on
    ${\rm ri}(\dom F \cap U_1)$ then 
    $\argmin_{x\in \R^n}(F(x) + G(x)) \subseteq {\rm ri} (\dom F)$
    and 
    all $\hat{x},\tilde{x}\in \argmin_{x \in \R^n} \{  F (x) +G(x) \}$ 
    fulfill 
    $\hat{x}-\tilde{x}\in U_2$ 
    and
    $ F (\hat{x})= F (\tilde{x})$, $G(\hat{x})=G(\tilde{x})$.
\end{enumerate}

\end{lemma}

Before proving this lemma we illustrate that in general we really need 
to require $ F $ to be essentially smooth,
in order to guarantee the assertions of part 
\upref{enu:direction_of_argmin__essentially_smooth_and_ess_convex_version}

\begin{figure}[htpb]
\begin{center}
\includegraphics[width=0.35\textwidth]{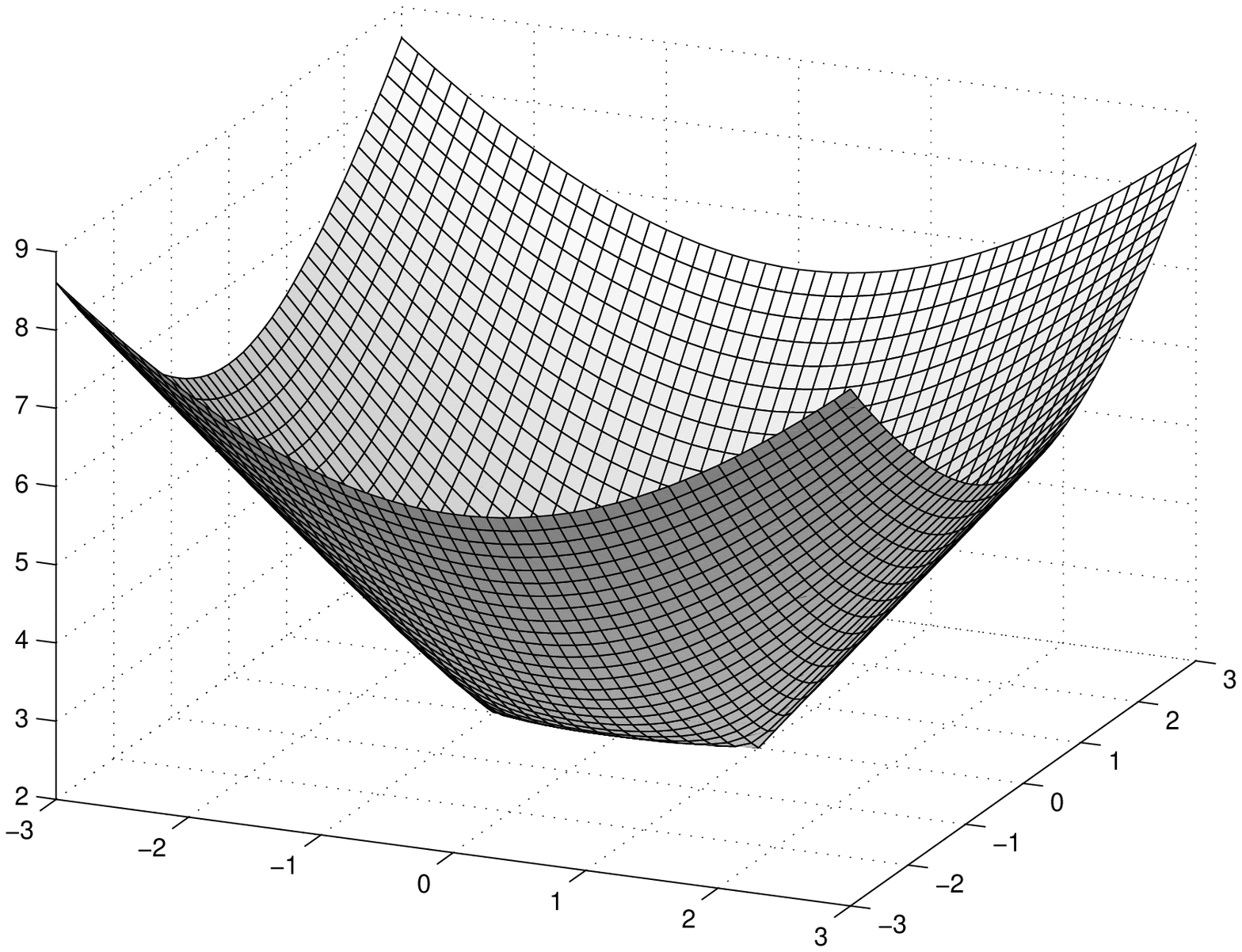}
\includegraphics[width=0.35\textwidth]{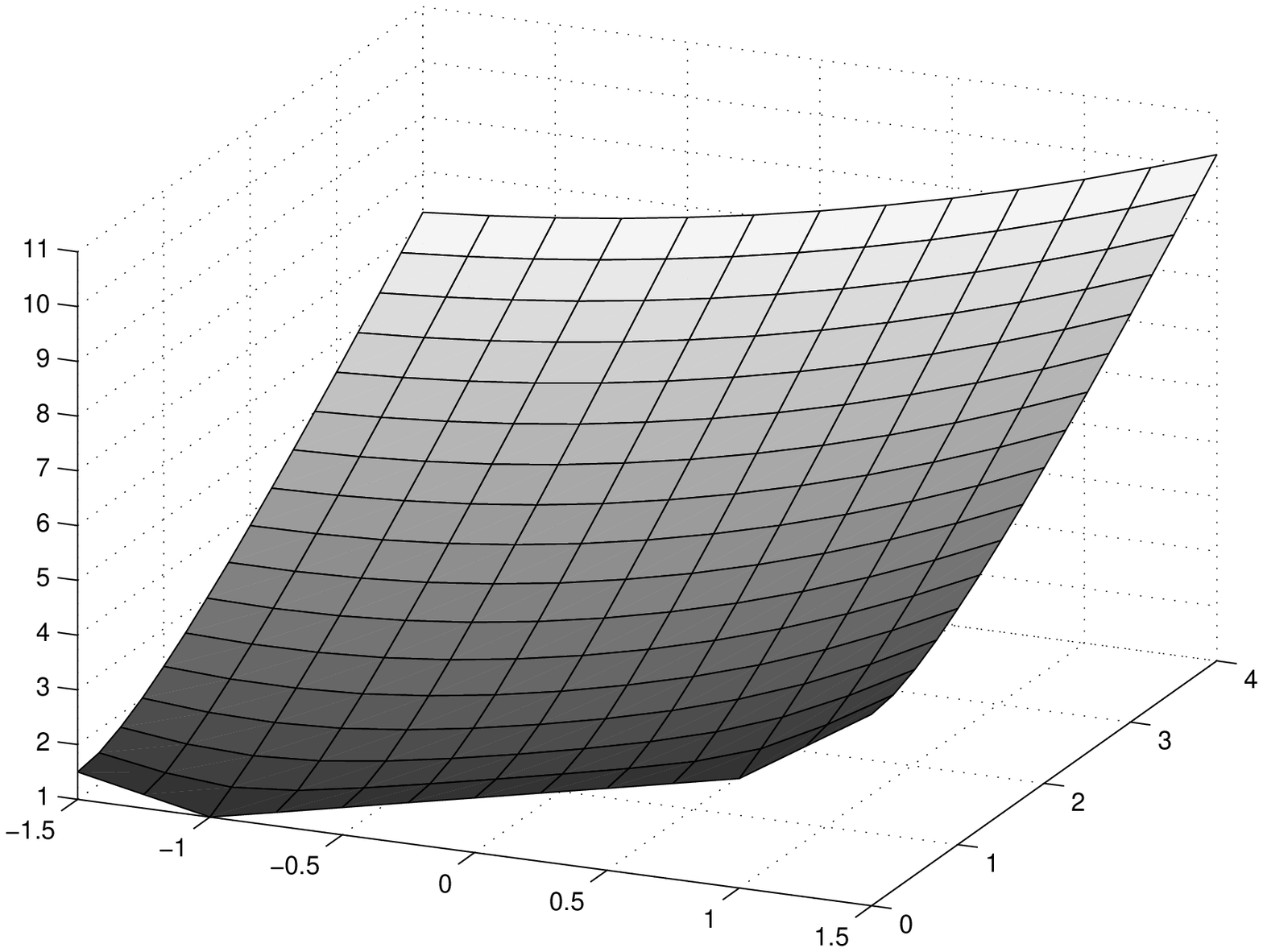}
\includegraphics[width=0.35\textwidth]{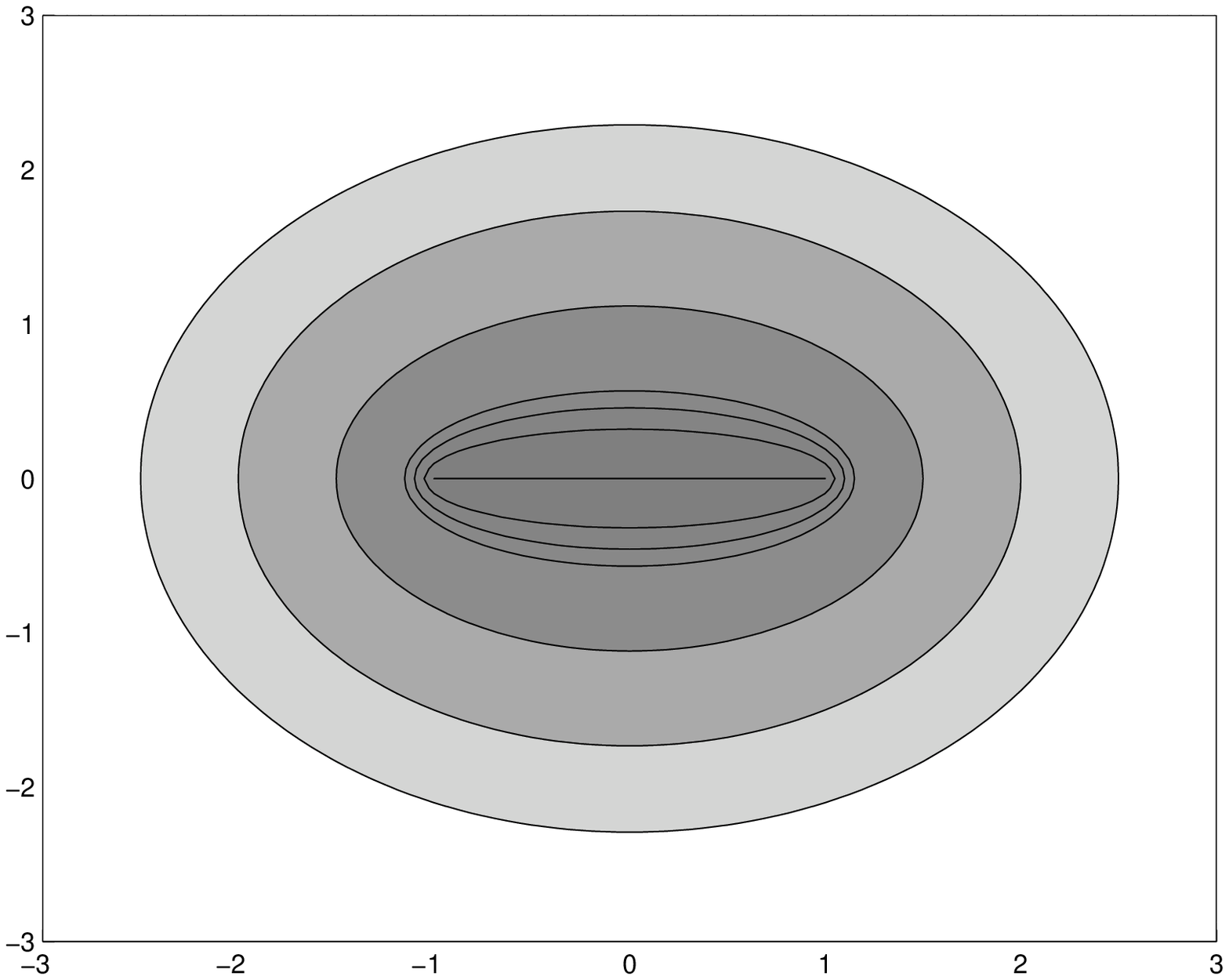}
\caption{Up:  Graph of $h$ and $F$, \; 
	 Down: $\argmin(h) = [-1,1]\times \{0\}$ and some other level sets of $h$ 
	 \label{fig:essentially_smoothness_is_important}}
\end{center}
\end{figure}

\begin{example}
  The shifted Euclidean norm $h_b: \R^2 \rightarrow \R$ given by
  $h_b(x) \defeq \|x - b\|_2$, where $b \in \R^2$ is strictly convex on 
  every straight line which does not meet $b$, by Lemma
  \ref{lem:euclidean_norm_strictly_convex_on_every_straight_line_through_origin}. 
  Set $b = (1,0)^T$ and $b' = -b = (-1,0)^T$ and consider the function
  $h: \R^2 \rightarrow \R$ given by
  \begin{gather*}
    h(x) = h_b(x) + h_{b'}(x) = \|x-b\|_2 + \|x-b'\|_2.
  \end{gather*}
  The only straight line which meets both $b$ and $b'$ is 
  ${\rm aff}(\{b,b'\}) = \R \times \{0\}$.
  Therefore $h$ is strictly convex on all other straight lines in $\R^2$,
  c.f. also Figure \ref{fig:essentially_smoothness_is_important}.
  In particular $h$ is strictly convex in the open upper half plane 
  $\mathbb{H} \defeq \{x\in \R^2: x_2 > 0\}$.
  Set $U_1 \defeq \R^2$, $U_2 \defeq \{\zerovec\}$ and consider the functions 
  $F, G: \R^2 \rightarrow \R \cup \{+\infty\}$ given by
  \begin{align*}
    F(x) \defeq
    \begin{cases}
      h(x) + x_1 & \text{ for } x \in \overline{\mathbb{H}}
      \\
      +\infty    & \text{ for } x \in \R^2 \setminus \overline{\mathbb{H}}
    \end{cases},
    &&
    G(x) \defeq -x_1
  \end{align*}
  Then all general assumptions of Lemma
  \ref{lem_unique}
  are fulfilled just as the assumptions of part
  \upref{enu:direction_of_argmin__essentially_smooth_and_ess_convex_version}
  -- except that $F$ is not essentially smooth on $U_1 = \R^2$; note here
  that $h$ is continuously differentiable in $\R^2 \setminus \{b,b'\}$, 
  so that choosing any boundary point 
  $
    x 
    \in 
    \partial \dom F 
    = 
    \partial \overline{\mathbb{H}}
    = \R \times \{0\}
  $,
  which is different from $b$ and $b'$, we have 
  $
    \lim_{n \rightarrow \infty} \|\nabla F(x_n)\|_2 
    =
    \|\nabla h(x) + (1,0)^T\|_2
    \not =
    +\infty
  $
  for any sequence $(x_n)_{n\in \N}$ in 
  ${\rm int}(\dom F) = \mathbb{H}$, which converges to $x$.

  We have $\argmin \{F+G\} = \argmin h = [-1,1]\times \{0\}$ here, so that 
  $\argmin \{F+G\} \cap {\rm ri}(\dom F) = \emptyset$. Moreover 
  the minimizers $\hat x = (1,0)^T$ and $\tilde x = (-1,0)^T$ neither fulfill
  $\hat x - \tilde x \in U_2$ nor $F(\hat x) = F(\tilde x)$, 
  $G(\hat x) = G(\tilde x)$.
\end{example}

%
{\bf Proof of Lemma \ref{lem_unique}.}
\upref{enu:direction_of_argmin__general_version} 
First we prove that for any $x,y \in \dom  F $  and
the line segment $l(x,y) \defeq \{x + t(y-x):t \in [0,1]\}$ 
the following statements are equivalent:
\begin{itemize}
 \item[a)] $ F \bigl\vert_{l(x,y)}$ is constant,
 \item[b)]  $ F \bigl\vert_{l(x,y)}$ is affine,
 \item[c)]  $y-x\in U_2$.
\end{itemize}
\iftime{Checken und ggf. was dazu einarbeiten: Beziehung dieser drei Aussagen zum Periodenraum und
zu Rockafellar Theorem 8.8 ?!}
We use the unique decompositions $x=x_1 + x_2$, $y=y_1 + y_2$ with $x_1,y_1\in U_1$ and
$x_2,y_2\in U_2$.

a) $\Rightarrow$ b): 
This is clear since a constant function is in particular an affine one.
\\
b) $\Rightarrow$ c): 
If $ F \bigl\vert_{l(x,y)}$ is affine, 
i.e.,
\begin{equation*}
 F (x +t(y-x))  =   F (x) + t( F (y)- F (x))\; \text{ for every } t \in [0,1],
\end{equation*}  
the translation invariance of $ F $ yields
\begin{align*}
 F (x + t(y-x) - x_2 -t(y_2-x_2) )
    &=
 F (x-x_2) + t( F (y-y_2)- F (x-x_2)),\\
 F (x_1 +t(y_1-x_1))
    &=
 F (x_1) + t( F (y_1)- F (x_1))\; \text{ for every }  t \in [0,1],
\end{align*}
so that $ F \bigl\vert_{l(x_1,y_1)}$ is affine as well. 
On the other hand $ F $ is also strictly convex on 
$l(x_1,y_1)$. Both can be simultaneously only true, 
if $x_1=y_1$, which just means that 
$y-x=y_2-x_2 \in U_2$.
\\
c) $\Rightarrow$ a):
  Let $y-x \in U_2$, i.e. $y_1=x_1$, so that $y-x=y_2-x_2$. 
  Therefore and due to the translation invariance of $ F $ we get
  \begin{equation*}
     F (x+t(y-x))
   = F (x+t(y_2-x_2))
   = F (x)
  \end{equation*}
  even for all $t\in\R$. In particular $ F $ is constant on $l(x,y)$.
\\
Now the assertions of part 
\upref{enu:direction_of_argmin__general_version}
can be seen as follows: Due to the convexity of $ F +G$ the whole segment 
$l(\hat{x},\tilde{x})$ 
belongs to
$\argmin\{ F +G\}$ 
so that $ F +G$ is constant on 
$l(\hat{x},\tilde{x})$.
Thus, the convex summands $ F $ and $G$ must be affine on 
$l(\hat{x},\tilde{x}) \subseteq \dom ( F  + G)$.
Now the equivalence b) $\Leftrightarrow $ c) tells us that  
$\hat{x}-\tilde{x} = -(\tilde{x}-\hat{x})\in U_2$ and hence
$ F (\hat{x}) =  F (\tilde{x})$. The remaining $G(\hat{x})=G(\tilde{x})$ 
follows from the last equation and from 
$ F (\hat{x})+G(\hat{x}) =  F (\tilde{x}) + G(\tilde{x})$ since only finite
values occur. 
\\
\upref{enu:direction_of_argmin__essentially_smooth_and_ess_convex_version} 
The function $f\defeq F|_{U_1} : U_1 \rightarrow \R \cup \{+\infty\}$ is
essentially smooth, so that 
${\rm int}_{U_1}(\dom f)$ is in particular a nonempty subset of $U_1$.
Therefore and by Theorem 
\ref{thm:operations_that_interchange_with_direct_sum}
we get
$
  {\rm aff}(\dom F)
  =
  {\rm aff}(\dom F|_{U_1} \oplus U_2)
  =
  {\rm aff}(\dom F|_{U_1}) \oplus U_2
  =
  U_1 \oplus U_2
$.
Lemma 
\ref{lem:essentially_smoothness_does_not_depend_on_periodspace}
now says that $F$ is essentially smooth on ${\rm aff}(\dom F)$.
The therewith applicable
part \upref{enu:argmin_of_sum_contained_in_ri_of_dom_of_ess_smooth_summand} 
of Lemma 
\ref{lem:argmin_of_sum_contained_in_ri_of_dom_of_ess_smooth_summand}
gives $\argmin(F+G) \subseteq {\rm ri}(\dom F)$. Hence the minimizers of 
$F + G$ keep unchanged, if we enlarge the values $F(x)$ outside of
${\rm ri}(\dom F)$ by setting
\begin{gather*}
  \tilde F(x) 
  \defeq 
  \begin{cases}
    F(x),     & \text { for } x \in {\rm ri}(\dom F)
    \\
    +\infty,  & \text { for } x \not \in {\rm ri}(\dom F)
  \end{cases}.
\end{gather*}
Hence we get the remaining assertions for 
$\hat x, \tilde x \in \argmin (F+G) = \argmin (\tilde F + G)$
by applying part \upref{enu:direction_of_argmin__general_version}
to $\tilde F$ and $G$;
note herein that
$\dom \tilde F \cap \dom G = {\rm ri}(\dom F) \cap \dom G \not = \emptyset$,
that 
$\tilde F$ is still convex,
see Theorem \ref{thm:nonemtpy_convex_set_has_nonempty_ri},
and strictly convex
on $U_1$, since $F$ is by assumption strictly convex on 
${\rm ri}(\dom F \cap U_1) = \dom (\tilde F|_{U_1})$, and that 
finally $U_2$ still belongs to the periods space $P[\tilde F]$,
since 
$
  {\rm ri}(\dom F) 
  = 
  {\rm ri}(\dom F|_{U_1} \oplus U_2)
  = 
  {\rm int}_{U_1}(\dom f) \oplus U_2
$, by Theorem \ref{thm:operations_that_interchange_with_direct_sum}.
\hfill $\Box$
\vspace{0.2cm}

\rem{(wegen Formulierung in eingereichter Version) Zweiter Punkt 
$\bullet$ duerfte erfuellt sein,
wenn $F$ essentially strictly convex auf $U_1$ ist}

\begin{theorem} \label{thm:direction_of_argmin__essentially_smooth__strict_convex_sharpend}
  Let $F,G: \R^n \rightarrow \R \cup \{+\infty\}$ be convex functions with 
  ${\rm ri}(\dom F) \cap {\rm ri}(\dom G) \not = \emptyset$. If there is 
  a decomposition 
  \begin{gather*}
    {\rm aff}(\dom F) = \check A \oplus  \check P
  \end{gather*}
  of 
  ${\rm aff}(\dom F)$ into a subspace $\check P$ of $P[F]$ and an 
  affine subspace $\check A \subseteq \R^n$ such that 
  $F$ is essentially smooth on ${\rm aff}(\dom F)$ (or on $\check A$) as 
  well as strictly convex on ${\rm int}_{\check A}(\dom F|_{\check A})$
  then 
  \begin{alignat*}{4}
    & & \argmin_{x\in \R^n}(F(x)+G(x)) \subseteq {\rm ri}(\dom F) 
  \shortintertext{and}
   & \hat x - \tilde x \in \check P, & &
   \begin{split} 
     F(\hat x) &= F(\tilde x),\\ G(\hat x) &= G(\tilde x)
   \end{split}
  \end{alignat*}
  for all $\hat x, \tilde x \in \argmin_{x\in \R^n}(F(x)+G(x))$.
\end{theorem}

\begin{proof}
  Let $\mathring a \in {\rm ri}(\dom F) \cap {\rm ri}(\dom G)$.
  Replacing $F$ and $G$ by $F_1 (\cdot )= F(\cdot - \mathring a)$
  and $G_1(\cdot) = G(\cdot - \mathring a)$, respectively
  would neither change the truth value of the assumptions nor of 
  the assertions; therefore we may without loss of generality assume 
  $\mathring a = \zerovec$, so that ${\rm aff}(\dom F)$ is a vector subspace 
  of $\R^n$. 
  Write $\zerovec = a_0 + p_0$ with some $a_0 \in  \check A$ and $p_0 \in \check P$.
  Due to $F|_{\check A} = F|_{\check A + p_0}$ we see that replacing
  $\check A$ by the vector subspace $ \check A_2 = \check A + p_0$ 
  would neither change the truth value of the assumptions nor of 
  the assertions; therefore we may without loss of generality
  furthermore also
  assume that $\check A$ is a vector subspace of ${\rm aff}(\dom F)$.
  Set now $U_1 \defeq \check A$ and $U_2 \defeq \check P$. Noting that neither 
  the truth value of the assumptions nor of the assertions changes
  when considering $F,G$ and $F+G$ only on the vector space 
  $U_1 \oplus U_2 = {\rm aff}(\dom F)$ and identifying it with some 
  $\R^{n'}$ we obtain all claimed assertions by part 
  \upref{enu:direction_of_argmin__essentially_smooth_and_ess_convex_version}
  of Lemma \ref{lem_unique}; note here that 
  $
    {\rm ri}(\dom F \cap U_1) 
    =
    {\rm ri}(\dom F) \cap U_1
    =
    {\rm int}_{U_1}(\dom F|_{U_1})
  $,  
  $
    {\rm ri}(\dom G|_{{\rm aff}(\dom F)})
    =
    {\rm ri}(\dom G \cap {\rm aff}(\dom F))
    =
    {\rm ri}(\dom G) \cap {\rm aff}(\dom F)
  $
  by Theorem \ref{thm:restricting_set_operations_to_an_affine_subset},
  and note finally that $F$ is in any case essentially smooth on 
  ${\rm aff}(\dom F)$ by Lemma 
  \ref{lem:essentially_smoothness_does_not_depend_on_periodspace}. 
\end{proof}

\begin{remark}
  ~
  \begin{enumerate} 
    \item \label{enu:only_full_periodspace_possible_in_decomposition}
      The assumptions of the just proven theorem can be only valid 
      if in fact $\check P = P[F]$.
    \item 
    \label{enu:smoothness_and_strictly_convex_assumption_independet_of_decomposition}
      The essentially smoothness as well as the strictly convexity 
      assumptions on $F$ 
      keep valid if $\check A$ is replaced by any other affine subset 
      $\tilde A \subseteq \R^n$ with 
      $\tilde A \oplus \check P = {\rm aff}(\dom F)=\check A \oplus \check P$.
  \end{enumerate}
\end{remark}

\begin{proof}
  \upref{enu:only_full_periodspace_possible_in_decomposition} 
  Since $\check P$ is a subspace of $P[F]$ we have $\check P \subseteq P[F]$.
  For the proof of $\check P \supseteq P[F]$ we may assume 
  without loss of generality that the affine space ${\rm aff}(\dom F)$
  is even a vector subspace of $\R^n$ with origin $\zerovec \in \dom F$.
  Then every arbitrarily chosen $p \in P[F]$ belongs to 
  ${\rm aff}(\dom F)$ and can therefore be written in the form 
  $p = \check a + \check p$ with some $\check a \in \check A$,
  $\check p \in \check P$. Hence 
  $\check a = p - \check p \in P[F]$, i.e.
  $
    F(x + \lambda \check a) = F(x)
  $
  for all $x \in \R^n$ and all $\lambda \in \R$.
  Choosing any element $\mathring a$ from the nonempty set
  $
    {\rm ri}(\dom F) \cap \check A
    =
    {\rm ri}(\dom F|_{\check A})
  $,
  see Theorem \ref{thm:nonemtpy_convex_set_has_nonempty_ri} 
  and Theorem \ref{thm:restricting_set_operations_to_an_affine_subset}
  we have in particular
  $ F(\mathring a + \lambda \check a) = F(\mathring a) $
  for all $\lambda \in \R$.
  This is only possible for $\check a = \zerovec$, since 
  $F$ is strictly convex on 
  $
    \{\mathring a + \lambda \check a : \lambda \in \R\}
    \subseteq
    {\rm ri}(\dom F) \cap \check A
    =
    {\rm ri}(\dom F|_{\check A})
    =
    {\rm int}_{\check A}(\dom F|_{\check A})
  $, where we have used Theorem 
  \ref{thm:restricting_set_operations_to_an_affine_subset}.
  Consequently $p = \check a + \check p = \check p \in \check P$.
  \\
  \upref{enu:smoothness_and_strictly_convex_assumption_independet_of_decomposition}
  Writing $\zerovec = \check a_0 + \check p_0$, $\zerovec = \tilde a_0 + \tilde p_0$
  and noting $F|_{\check A} = F|_{\check A + \check p_0}$,
  $F|_{\tilde A} = F|_{\tilde A + \tilde p_0}$
  we may without loss of generality  assume that 
  $\check A$ and $\tilde A$ are vector subspaces.
  Consider the projection 
  $\pi: A \rightarrow \tilde A$, $x=\tilde a + \check p \mapsto \tilde a$  
  of the vector space 
  $A = \check A \oplus \check P = \tilde A \oplus \check P$ 
  onto its subspace $\tilde A$.
  We have $\mathcal{N}(\pi) = \check P$, so that 
  $\alpha \defeq \pi|_{\check A}: \check A \rightarrow \tilde A$ is a 
  vector space isomorphism, which links $F|_{\check A}$ and 
  $F|_{\tilde A}$ both via $F|_{\check A} = F|_{\tilde A} \circ \alpha$
  and its consequence 
  $
    {\rm int}_{\tilde A}(\dom F|_{\tilde A})
    =
    \alpha [{\rm int}_{\check A}(\dom F|_{\check A}) ]
  $.
  Therefore $F|_{\check A}$ is essentially smooth if and only if 
  $F|_{\tilde A}$ is essentially smooth.
  Writing $\check O \defeq {\rm int}_{\check A}(\dom F|_{\check A})$
  and $\tilde O \defeq {\rm int}_{\tilde A}(\dom F|_{\tilde A})$
  we likewise have that $F|_{\check O}$
  is strictly convex if and only if $F|_{\tilde O}$ is strictly convex.
\end{proof}

\section{Homogeneous penalizers and constraints} \label{sec:homegeneous_penalizers_and_constraints}
%

\iftime{Nach Unterteilung des Kapitels in einige Unterkapitel,
musz der nachf. einl. Text noch entsp. umgeschrieben werden}

\addtointroscoll{

This section is divided into two subsections.
In the first subsection we restrict the broad setting of the  
\prettyref{sec:penalizers_and_constraints} to a less general setting by making a particular choice 
for $\Psi$ and by putting some assumptions on $\Phi$.
In Lemma \ref{lem:subgradient_and_conjugate_function_in_our_setting}
we show some implications of the assumptions on $\Phi$ for $\Phi$
itself and its conjugate function $\Phi^*$.
In Remark  \ref{rem:assumptions_for_fenchel_dualty_theorem_fulfilled} 
we will see that the Fenchel Duality Theorem
\ref{fenchel-dual-hom} can be applied within our setting.
The second subsection deals with properties of the minimizing sets.
\iftime{Existenz auslagern in FenchelDualHom?}
In Theorem \ref{thm:existence_of_solvers} we show that 
the problems 
$(P_{1,\tau})$,
$(P_{2,\lambda})$,
$(D_{1,\tau})$,
$(D_{2,\lambda})$
have a solution for $\tau > 0$ and $\lambda > 0$,
if certain conditions are fulfilled.
In \prettyref{thm:localisation_of_the_solvers}
we prove that under the same conditions and an extra condition 
there are intervals $(0,c)$ and $(0,d)$ such that 
${\rm SOL}(P_{1,\tau})$,
${\rm SOL}(P_{2,\lambda})$,
${\rm SOL}(D_{1,\tau})$,
${\rm SOL}(D_{2,\lambda})$
show similar localization behavior for 
$\tau = 0$, $\lambda \in [d, +\infty)$;
$\tau \in (0,c)$, $\lambda \in (0,d)$;
and $\tau \in [c, + \infty)$, $\lambda = 0$.
In \prettyref{thm:theo_lambda_tau} the localization behavior is refined
for $\tau \in (0,c)$ and $\lambda \in (0,d)$. The results there
say that, while $\tau$ runs from $0$ to $c$ and $\lambda$ runs from
$d$ to $0$, all solver sets have to move. Moreover the mappings,
given by
$\tau \mapsto {\rm SOL}(P_{1,\tau})$ and
$\lambda \mapsto {\rm SOL}(P_{2,\lambda})$ are 
the same -- besides a (``direction reversing'') parametrization change $g:(0,c) \rightarrow (0,d)$.
Similar the mappings, given by
$\tau \mapsto {\rm SOL}(D_{1,\tau})$ and
$\lambda \mapsto {\rm SOL}(D_{2,\lambda})$ are 
the same -- besides the same parametrization change $g:(0,c) \rightarrow (0,d)$.
In the remaining parts of that subsection we deal with $g$.

}

\subsection{Setting} \label{subsec:setting}
%
In the rest of this 
thesis, we deal with the functions
\begin{equation*}
\Psi_{1} \defeq \iota_{lev_1 \| \cdot \|}  \quad {\rm and} \quad \Psi_2 \defeq \| \cdot \|,
\end{equation*}
where $\| \cdot \|$ denotes an arbitrary norm in $\mathbb R^{m}$ 
with dual norm 
$\| \cdot \|_*\defeq \max_{\|x\| \le 1} \langle \cdot, x \rangle$.
Constraints and penalizers of this kind appear in many image processing tasks.
Note that
$\Psi_1(\tau^{-1} x) = \iota_{lev_\tau \| \cdot \|} (x) = \tau \iota_{lev_\tau \| \cdot \|} (x)$
for $\tau \in (0, \pinfty)$.
The conjugate functions of $\Psi_1$ and $\Psi_2$ are 
$$
\Psi_1^* = \| \cdot\|_* \quad {\rm and} \quad \Psi_2^* = \iota_{lev_1 \| \cdot \|_*}.
$$
\iftime{Bessere Stelle fuer die Subdifferentiale finde}
and their subdifferentials are known to be
\begin{equation} 
\partial \Psi_1 (x) =
\left\{
\begin{array}{ll}
\{ \zerovec \} & {\rm if} \; \;\|x\| < 1,\\
\{p \in \R^{m}: \langle p,x \rangle = \|p\|_*\} & {\rm if} \; \; \|x\| = 1,\\
\emptyset &  {\rm otherwise} 
\end{array}
\right.
\end{equation}
and
\begin{equation} \label{subgrad_2}
\partial \Psi_2 (x) =
\left\{
\begin{array}{ll}
\{p \in \R^{m}:\|p\|_* \le 1\}  & {\rm if} \; \; \|x\| = 0,\\
\{p \in \R^{m}: \langle p,x \rangle = \|x\|, \|p\|_* = 1 \} & {\rm otherwise}.\\
\end{array}
\right.
\end{equation}
Then the primal problems $(P)$ in \eqref{primal} with $\mu \defeq \tau^{-1} > 0$ in the case $\Psi = \Psi_1$ 
and $\mu \defeq \lambda > 0$ in the case $\Psi = \Psi_2$ become
\begin{eqnarray*}
(P_{1,\tau}) && \argmin_{x \in \R^n}  \left\{ \Phi(x) \st \| Lx\| \le \tau \right\}, 
\\ 
(P_{2,\lambda}) && \argmin_{x \in \R^n} \left\{ \Phi(x) \; + \; \lambda \| Lx\| \right\} 
\end{eqnarray*}
and the dual problems $(D)$ in \eqref{dual} read
\begin{eqnarray*}
(D_{1,\tau}) && \argmin_{p \in \R^m} \left\{\Phi^*(-L^* p) \; + \; \tau \|p\|_* \right\},  
\\
(D_{2,\lambda}) && \argmin_{p \in \R^m} \left\{ \Phi^*(-L^* p)  \st
\| p \|_* \le \lambda \right\} 
\end{eqnarray*}
We will also consider the cases  $\tau=0$ and $\lambda = 0$.
In what follows we will assume that 
$F_P \defeq \Phi: \R^n \rightarrow \R \cup \{+\infty\}$ and
$F_D \defeq \Phi^*(-L^* \cdot): \R^m \rightarrow \R \cup \{ + \infty\}$
are invariant under translation in direction of subspaces $U_{P,2}$
and $U_{D,2}$, respectively.
Speaking now in terms of a general function 
$F: \R^n \rightarrow \R \cup \{+\infty \}$ we could of course 
always make the uninteresting choice $U_2 \defeq \{\zerovec\}$;
so more 
precisely we are interested in those decompositions $\R^n = U_1 \oplus U_2$ with
$F(u + u_2) = F(u)$ for all $u \in \R^n$, $u_2 \in U_2$, in which
$U_2$ is chosen as large as possible, so that the essential properties of
$F$ can be revealed by considering $F|_{U_1}$.
In case of ${\rm aff}(\dom F|_{U_1}) = U_1$ we do not need to refine the decomposition 
$\R^n = U_1 \oplus U_2$ and can think of $F$ to be essentially given by 
$f = F|_{U_1}$.
In case of ${\rm aff}(\dom F|_{U_1}) \subset U_1$, however, it can be convenient to 
refine the decomposition $\R^n = U_1 \oplus U_2$ by writing 
${\rm aff}(\dom F|_{U_1}) = a + X_1$ with some $a \in {\rm aff}(\dom F|_{U_1})$ and a
vector subspace $X_1 \subseteq \R^n$; after choosing some vector subspace 
$X_3$ with $U_1 = a + X_1 \oplus X_3$ and setting $X_2 \defeq U_2$ we have
$\R^n = a + X_1 \oplus X_2 \oplus X_3$ and can think of $F$ to be given 
essentially by $F|_{a + X_1}$, since the inclusion $\dom F \subseteq a + X_1 \oplus X_2$
just means that $F(x) = F(a + x_1 + x_2 + x_3)$ equals $+\infty$ for $x_3 \not = \zerovec$
and $F(a + x_1 + x_2) = F(a + x_1)$ for $x_3 = \zerovec$.
\\[1ex]
In those cases where $\zerovec \in {\rm aff}(\dom F|_{U_1})$ or where $F$ is replaceable
by $F(\cdot - a)$ we can even assume $a = \zerovec$ so that we have
$\R^n = X_1 \oplus X_2 \oplus X_3$ and can think of $F$ to be given in its 
essence by $F|_{X_1}$ on $X_1$, then extended to a larger subspace 
$X_1 \oplus X_2$ by demanding translation invariance in direction $X_2$, 
and finally set to $+ \infty$ on $\R^n \setminus (X_1 \oplus X_2)$.
This is the core structure, which $\Phi$ will now be demanded to have.
In addition $X_1$, $X_2$ and $X_3$ shall be pairwise orthogonal:
\\[1ex]
Let $\Phi$'s domain $\R^n$ have a decomposition 
$\R^n = X_1 \oplus X_2 \oplus X_3$ into pairwise orthogonal
subspaces such that
\begin{equation} \label{new_setting_Phi}
  \Phi(x) = \Phi(x_1 + x_2 + x_3) =
  \begin{cases}
    \phi(x_1)  & \text{ if } x_3 = \zerovec
    \\
    + \infty   & \text{ if } x_3 \not = \zerovec
  \end{cases},
\end{equation}
where $\phi = \Phi|_{X_1}: X_1 \rightarrow \R \cup \{+\infty\}$ is a function meeting the following demands:
\iftime{Ist die Offenheit von dom $\phi$ noch noetig?}

\begin{enumerate}
  \item \label{enu:dom_phi_shall_touch_origin}
      $\dom \phi$ is an open subset of $X_1$ with $\zerovec \in \overline{ \dom \phi }$,
  \item \label{enu:phi_shall_be_strictly_convex_and_ess_smooth}
      $\phi$ belongs to $\Gamma_0(X_1)$ and is strictly convex and essentially smooth (compare \cite[p. 251]{Rockafellar1970}),
  \item \label{enu:phi_should have_a_minimizer}
      $\phi$ has a  minimizer.
\end{enumerate}

The following lemma shows that the subdifferentials of $\phi$ and $\Phi$ are closely related and that $\Phi^*$ is
of the same basic structure as $\Phi$, whereas the roles of $X_2$ and $X_3$ are interchanged. 
Note that for the proof of the first two parts we use only the direct 
decomposition of $\Phi$'s domain $\R^n$ into the pairwise orthogonal 
subspaces 
$X_1, X_2, X_3$; none of the additional properties of $\phi$ is needed.
\begin{lemma} \label{lem:subgradient_and_conjugate_function_in_our_setting}
  For a function $\Phi$ fulfilling the setting in \eqref{new_setting_Phi} and any points $x, x^* \in \R^n$ 
  the following holds true:
  \begin{enumerate}
    \item \label{enu:subgradient_in_our_setting}
      $
        \partial \Phi(x) 
	=
	\partial \Phi(x_1 + x_2 + x_3) 
	=
	\begin{cases}
	  \emptyset                                                         & \text { if } x_3 \not = \zerovec
	  \\
	  \partial \phi(x_1)  \oplus   \{\zerovec\} \oplus  X_3    & \text{ if } x_3 = \zerovec.
	\end{cases}
      $
    \item \label{enu:conjugate_function_in_our_setting}
      $
	\Phi^*(x^*) 
	=
	\Phi^*(x^*_1 + x^*_2 + x^*_3)
	=
	\begin{cases}
	  \phi^*(x_1^*)   & \text{ if } x_2^* = \zerovec
	  \\
	  + \infty	  & \text{ if } x_2^* \not = \zerovec
	\end{cases}
      $, where
    \item \label{enu:dual_properties_in_our_setting}
      \begin{itemize}
        \item 
	  $\phi^*$ belongs to $\Gamma_0(X_1)$ 
	  and is essentially smooth and essentially strictly convex 
	  (compare \cite[p. 253]{Rockafellar1970})
	\item 
	  $\zerovec \in {\rm int}(\dom \phi^*)$ and $\zerovec \in {\rm ri}(\dom \Phi^*)$
      \end{itemize}
  \end{enumerate}

\end{lemma}


\begin{proof}
  \upref{enu:subgradient_in_our_setting} and 
  \upref{enu:conjugate_function_in_our_setting} 
  We rewrite $\Phi$ in the form 
  $\Phi = \Phi_1 \sdirsum \Phi_2 \sdirsum \Phi_3$, where
  \begin{align*}
    &\Phi_1 = \phi       : X_1 \rightarrow \R \cup \{+\infty\},
    &\Phi_2 = 0_{X_2}    : X_2 \rightarrow \R, 
    &&\Phi_3 = \iota_{\{\zerovec\}} : X_3 \rightarrow \R \cup \{+\infty\}.
  \end{align*}
  Since $\R^n = X_1 \oplus X_2 \oplus X_3$ is a direct decomposition 
  into pairwise orthogonal subspaces we can apply Theorem 
  \ref{thm:conjugate_and_subdiff_of_direct_orthogonal_sum}
  and obtain
  \begin{gather*}
    \partial \Phi (x)
    =
    \partial \Phi_1(x_1) 
    \oplus 
    \partial \Phi_2(x_2)
    \oplus 
    \partial \Phi_3(x_3)
    =
    \partial \phi(x_1)
    \oplus 
    \{\zerovec\}
    \oplus 
    S_3(x_3),
  \end{gather*}
  where $S_3(x_3) = \emptyset$ for $x_3 \not = \zerovec$ and 
  $S_3(x_3) = X_3$ for $x_3 = \zerovec$, as well as
  \begin{gather*}
    \Phi^*(x^*)
    =
    \Phi_1^*(x^*_1)    
    +
    \Phi_2^*(x^*_2)
    +
    \Phi_3^*(x^*_3)
    =
    \phi^*(x^*_1)
    +
    \iota_{\zerovec}(x^*_2)
    + 
    0
  \end{gather*}

  \ref{enu:dual_properties_in_our_setting})
  $\phi \in \Gamma_0(X_1)$ implies $\phi^* \in \Gamma_0(X_1)$.
  Changing the coordinate system via an orthogonal transformation
  $\tilde x \mapsto x = Q\tilde x$ changes $\phi$ and $\phi^*$ 
  in the same way:
  If $\tilde \phi (\tilde x) = \phi (Q \tilde x)$ then also
  $\tilde \phi^*(\tilde x) = \phi^*(Q \tilde x)$.
  Hence \cite[Theorem 26.3]{Rockafellar1970} can be extended for functions like
  $\phi, \phi^*:X_1 \rightarrow \R \cup \{+ \infty\}$,
  which are only defined on a subspace $X_1$ of $\R^n$.
  So the strict convexity of $\phi$ implies that $\phi^*$ is essentially smooth
  and the essentially smoothness of $\phi$ implies that $\phi^*$ is
  essentially strictly convex.
  In order to prove $\zerovec \in {\rm int}(\dom \phi^*)$ we note that
  $\argmin \phi$, consisting of just one element, 
  is a nonempty and bounded level set of $\phi$. Consequently all level sets
  ${\rm lev}_{\alpha}(\phi)$, $\alpha \in \R$, are bounded,
  compare \cite[Corollary 8.7.1]{Rockafellar1970}. This implies
  $\zerovec \in {\rm int}(\dom \phi^*)$ (of course regarded relative to $X_1$),
  compare \cite[Corollary 14.2.2]{Rockafellar1970}. 
  Therefrom we finally obtain
  $\zerovec \in {\rm int}( \dom \phi^*) \oplus X_3 = {\rm ri} (\dom \Phi^*)$,
  because
  $\dom \phi^* \oplus X_3 = \dom \Phi^*$ by part
  \ref{enu:conjugate_function_in_our_setting}).
\end{proof}

\begin{remark} \label{rem:assumptions_for_fenchel_dualty_theorem_fulfilled}
  By our setting -- in first line by the condition 
  \upref{enu:dom_phi_shall_touch_origin} on $\Phi$ -- 
  we have $\zerovec \in \overline{\dom \Phi}$ and also 
  $\zerovec \in {\rm ri} (\dom \Phi^*)$
  by Lemma
  \ref{lem:subgradient_and_conjugate_function_in_our_setting}.
  Therefore our setting ensures that the assumptions i) - iv) of Lemma
  \ref{fenchel-dual-hom} are fulfilled:
  Regarding the first three assumptions
  we note $\mathcal{R}(L) = {\rm ri}(\mathcal{R}(L))$ so that 
  every of these assumptions is of the form 
  \begin{gather*}
    {\rm ri}(A) \cap {\rm ri}(B) \not = \emptyset
  \end{gather*}
  with convex subsets $A,B$ of some Euclidean space.
  Both for $\Psi = \Psi_1$ and $\Psi = \Psi_2$ 
  we have $\zerovec \in \overline{A}$ and $\zerovec \in {\rm int}(B)$ for sets 
  $A,B$ corresponding to condition i), ii) or iii)  of Lemma
  \ref{fenchel-dual-hom}, respectively.
  Since $A$ is in any case convex and nonempty there is some
  $a_k \in {\rm ri}(A)$ with
  $a_k \rightarrow \zerovec$,
  cf. Theorem 
  \ref{thm:raypoint_from_ri__leaving_closure__never_comes_back}.
  Hence we have also $a_K \in {\rm int}(B)$ for a large enough $K$.
  In particular 
  $
    {\rm ri}(A) \cap {\rm ri}(B)
    = 
    {\rm ri}(A) \cap {\rm int}(B)
    \not = 
    \emptyset
  $.
  Also the fourth assumption of Lemma \ref{fenchel-dual-hom} is 
  clearly fulfilled in our setting, since 
  $\zerovec \in \mathcal{R}(-L^*) \cap {\rm ri}(\dom \Phi^*)$.
\end{remark}


\subsection{Properties of the solver sets and the relation between their parameters} 

The next theorem shows that all our problems 
$(P_{1,\tau})$,
$(D_{1,\tau})$,
$(P_{2,\lambda})$,
$(D_{2,\lambda})$
have a solution for $\tau > 0$ and $\lambda > 0$ if certain conditions on 
$\argmin \Phi$ and $\mathcal{N}(L) = \argmin \|L \cdot\|$
are fulfilled.

\begin{theorem} \label{thm:existence_of_solvers}
  Let $\Phi\in \Gamma_0(\R^n)$ be a function
  fulfilling the setting \eqref{new_setting_Phi} and let $L \in \R^{m,n}$ so that
  $X_2 \cap  {\cal N} (L) = \{\zerovec\}$ and 
  $\argmin \Phi \cap  {\cal N} (L) = \emptyset$.
  Then all solver sets 
  ${\rm SOL}(P_{1,\tau})$,
  ${\rm SOL}(D_{1,\tau})$,
  ${\rm SOL}(P_{2,\lambda})$,
  ${\rm SOL}(D_{2,\lambda})$
  are nonempty for 
  $\tau \in (0, +\infty)$, $\lambda \in (0, + \infty)$ 
  and the corresponding minima are finite.
\end{theorem}

\begin{proof}
  Note in the following that the requirements i) - iv) of 
  Lemma \ref{fenchel-dual-hom} are fulfilled. 
  Let $\lambda > 0$. Since 
  $\Phi(-L^* \cdot)$
  is lower semicontinuous on the compact Ball
  $
    B 
    \defeq 
    \overline{\mathbb{B}}_{\lambda}(\zerovec)[\|\cdot\|_*] 
    \defeq 
    \{p\in \R^m: \|p\|_* \leq \lambda\}
  $
  we have ${\rm SOL}(D_{2,\lambda}) \not = \emptyset$.
  The attained minimum is finite, because
  $\zerovec \in B \cap \dom(\Phi^*(-L^* \cdot))$
  holds true by part \ref{enu:dual_properties_in_our_setting}) of
  Lemma \ref{lem:subgradient_and_conjugate_function_in_our_setting}.
  Lemma \ref{fenchel-dual-hom} ensures that also 
  ${\rm SOL}(P_{2,\lambda}) \not = \emptyset$,
  where the attained minimum is finite, since
  $\dom ( \Phi + \lambda \|L \cdot\|) = \dom \Phi \not = \emptyset$.
  Let now $\tau > 0$. 
  We get
  ${\rm SOL}(P_{1,\tau}) \not = \emptyset$,
  by part \ref{enu:sum_takes_it_minimum}) of
  Lemma \ref{lem:levelsets_and_existence_of_minimizer_of_sum},
  applied to 
  $F \defeq \Phi$, 
  $U_1 \defeq X_1 \oplus  X_3$, $U_2 \defeq X_2$
  and
  $G \defeq \iota_{{\rm lev}_\tau \|L \cdot\|}$,
  $V_1 \defeq \mathcal{R}(L^*)$, $V_2 \defeq \mathcal{N}(L)$;
  the assumption of this Lemma are checked in 
  Detail \ref{det:pr_of_sol_existence:check_of_assumptions_of_lemma}.
  %
\iftime{EVT. schon in Fenchel Dual Hom reinpacken: angenommene Minima endlich}
  Due to the therein appearing relation
  $\dom \Phi \cap {\rm lev}_\tau \|L \cdot\| \not = \emptyset$
  the attained minimum is finite.
  Lemma \ref{fenchel-dual-hom} gives now 
  ${\rm SOL}(D_{1,\tau}) \not = \emptyset$,
  where the attained minimum is also finite since
  $
    \dom (\Phi^*(-L^* \cdot) +\tau \|\cdot\|_*)
    =
    \dom \Phi^*(-L^* \cdot)
    \not =
    \emptyset
  $.
\end{proof}

Recall in the next theorem that $\inf \emptyset = + \infty$ since any $m \in [-\infty, +\infty]$
is a lower bound of $\emptyset \subseteq [-\infty, +\infty]$.
The theorem states that there are three main areas where our solver sets
${\rm SOL}(P_{1,\tau})$
and
${\rm SOL}(P_{2,\lambda})$
must be located:
either they are completely contained in $\argmin \|L \cdot\| = \mathcal{N}(L)$ or $\argmin \Phi$, 
or they are located ``between'' them, in the sense of
${\rm SOL}(P_\bullet)    \cap \mathcal{N}(L) = \emptyset$
and
${\rm SOL}(P_\bullet)    \cap \argmin \Phi   = \emptyset$.
Similar relations hold true for
${\rm SOL}(D_{1,\tau})$
and
${\rm SOL}(D_{2,\lambda})$.
Note that 
${\rm SOL}(D_{1,\tau}) = \emptyset$ can happen in the border case 
$\tau = 0$ as we show in 
Example \ref{exa:Sol_can_be_empty_for_tau_eq_0}.
Also notice in the following theorem that 
$OP(\Phi, \|L\cdot\|)$ can either be $(0, +\infty)$ or
$[0, +\infty)$ for a function $\Phi$ which fulfills our setting 
\eqref{new_setting_Phi}. In case 
$\tau \not \in OP(\Phi, \|L\cdot\|)$ we have to be carefull
when regarding the problem 
\begin{gather*}
  \argmin_{x \in \R^n}  
  \left\{ \Phi(x) \st  \| Lx\| \le \tau 
  \right\},
\end{gather*}
since rewriting it to  
\begin{gather*}
  \argmin_{x \in \R^n}  
  \left\{ \Phi(x) + \iota_{{\rm lev}_{\tau}\|L \cdot \|}  (x) 
  \right\}
\end{gather*}
is not possible in this case, cf. the table on page 
\pageref{tab:unconstrained_vs_constrained_perspecive}.

%
\begin{theorem}  \label{thm:localisation_of_the_solvers}
  Let $\Phi\in \Gamma_0(\R^n)$ be a function
  fulfilling the setting \eqref{new_setting_Phi} and let $L \in \R^{m,n}$ so that
  $X_2 \cap  {\cal N} (L) = \{\zerovec\}$, 
  $X_3 \cap  \mathcal{R}(L^*) = \{\zerovec\}$
  and 
  $\argmin \Phi \cap  {\cal N} (L) = \emptyset$.
  Then the values
  \begin{align} 
      c &\defeq \inf_{x \in \argmin \Phi } \|Lx\| = \min_{x \in \argmin \Phi } \|Lx\|, 
      \label{eq:the_tau_for_which_SOL(P_1,tau)_arrives_at_argmin_Phi_for_the_first_time}
      \\
      d &\defeq \inf_{p \in  \argmin \Phi^*(-L^* \cdot)} \|p\|_*
	  =
	  \begin{cases}
	    \min\limits_{p \in  \argmin \Phi^*(-L^* \cdot)} \|p\|_*, 
		& \text{ if } \argmin \Phi^*(-L^* \cdot) \not = \emptyset
	    \\
	    + \infty,
		& \text{ if } \argmin \Phi^*(-L^* \cdot)      = \emptyset
	  \end{cases}
      \label{eq:the_lambda_for_which_SOL(P_2,lambda)_arrives_at_argmin_G_for_the_first_time}
  \end{align}
  are positive. Their geometrical meaning for the primal and dual problems is 
  expressed by the equations
  \begin{align*}
    c &= \min \{ \tau \in [0, +\infty): 
	{\rm SOL}(P_{1,\tau}) \cap \argmin \Phi \not = \emptyset \}
      \\
      &= \min \{ \tau \in [0, +\infty): 
	{\rm SOL}(D_{1,\tau}) \cap \{ \zerovec \} \not = \emptyset \}
  \shortintertext{and}
    d &= \inf \{\lambda \in [0, +\infty): 
	{\rm SOL}(P_{2,\lambda}) \cap \mathcal{N}(L) \not = \emptyset\}
      \\
      &= \inf \{\lambda \in [0, +\infty): 
	{\rm SOL}(D_{2,\lambda}) \cap \argmin \Phi^*(-L^*\cdot) \not = \emptyset\},
  \end{align*}
  where the infima are actually minima of the latter two sets, 
  if one of them is not empty. Furthermore the value of $\tau$ allows to locate 
  ${\rm SOL}(P_{1,\tau})$ and ${\rm SOL}(D_{1,\tau})$, according to
  \begin{align*}
    {\rm SOL}(P_{1,\tau}) &\subseteq \mathcal{N}(L),
      &{\rm SOL}(D_{1,\tau}) &\subseteq \argmin\Phi^*(-L^* \cdot),
      &&\text{ if } \tau = 0
    \\
    \Bigg\{
    \begin{split}
      {\rm SOL}(P_{1,\tau}) &\cap \mathcal{N}(L) = \emptyset\\
      {\rm SOL}(P_{1,\tau}) &\cap \argmin\Phi    = \emptyset
    \end{split}
    \Bigg\},
    &
    \Bigg\{
    \begin{split}
       {\rm SOL}(D_{1,\tau}) &\cap \argmin\Phi^*(-L^* \cdot)  = \emptyset\\
       {\rm SOL}(D_{1,\tau}) &\cap \{\zerovec\}                      = \emptyset
    \end{split}
    \Bigg\},
    &&\text{ if } \tau \in (0,c)
    \\
    {\rm SOL}(P_{1,\tau}) &\subseteq \argmin \Phi,
      &{\rm SOL}(D_{1,\tau}) &\subseteq \{\zerovec\},
      &&\text{ if } \tau \in [c, +\infty).
  \end{align*}
  The value of $\lambda$ similar allows to locate
  ${\rm SOL}(P_{2,\lambda})$ and ${\rm SOL}(D_{2,\lambda})$, according to
  \begin{align*}
    {\rm SOL}(P_{2,\lambda}) &\subseteq \mathcal{N}(L),
      &{\rm SOL}(D_{2,\lambda}) &\subseteq \argmin\Phi^*(-L^* \cdot),
      &&\text{ if } \lambda \in [d, + \infty)
    \\
    \Bigg\{
    \begin{split}
      {\rm SOL}(P_{2,\lambda}) &\cap \mathcal{N}(L) = \emptyset\\
      {\rm SOL}(P_{2,\lambda}) &\cap \argmin\Phi    = \emptyset
    \end{split}
    \Bigg\},
    &
    \Bigg\{
    \begin{split}
       {\rm SOL}(D_{2,\lambda}) &\cap \argmin\Phi^*(-L^* \cdot)  = \emptyset\\
       {\rm SOL}(D_{2,\lambda}) &\cap \{\zerovec\}                      = \emptyset
    \end{split}
    \Bigg\},
    &&\text{ if } \lambda \in (0,d)
    \\
    {\rm SOL}(P_{2,\lambda}) &\subseteq \argmin \Phi,
      &{\rm SOL}(D_{2,\lambda}) &\subseteq \{\zerovec\},
      &&\text{ if } \lambda = 0.
  \end{align*}
\end{theorem}

\noindent
\begin{proof}
  In the proof we use the abbreviations
  $\overline{\mathbb{B}}_{r}(a) \defeq \overline{\mathbb{B}}_{r}(a)[\|\cdot\|]$
  and
  $\overline{\mathbb{B}}^*_{r}(a) \defeq \overline{\mathbb{B}}_{r}(a)[\|\cdot\|_*]$.
  \\
  1. $c$ is really a minimum:
  We need only to show that the function 
  $\iota_{\argmin(\Phi)} + \|L \cdot \|$
  attains somewhere in $\R^n$ its minimum.
  In order to apply part \ref{enu:sum_takes_it_minimum}) of Lemma \ref{lem:levelsets_and_existence_of_minimizer_of_sum} we decompose
  $\R^n$ into the orthogonal subspaces $U_1 \defeq X_1 \oplus X_3$, 
  $U_2 \defeq X_2$ and
  $V_1 \defeq \mathcal{R}(L^*)$,  $V_2 \defeq \mathcal{N}(L)$, respectively
  and set
  $F \defeq \iota_{\argmin(\Phi)}$ and
  $G \defeq \|L \cdot\|$, respectively; then all assumptions are
  fulfilled for certain $\alpha, \beta$, see Detail 
  \ref{det:pr_of_locali_theorem:c_is_really_a_min}, so that
  $\iota_{\argmin(\Phi)} + \|L \cdot \|$
  attains indeed its minimum.
  \\
  2. $d$ is really a minimum if $\argmin \Phi^*(-L^* \cdot) \not = \emptyset$ :
  Let $p_0 \in \argmin \Phi^*(-L^* \cdot)$ and set $r \defeq \|p_0\|_*$.
  The set 
  $\argmin \Phi^*(-L^* \cdot)$ 
  is closed, due being 
  a level set of the lower semicontinuous function $\Phi^*(-L^* \cdot)$.
  Hence $C\defeq \argmin \Phi^*(-L^* \cdot) \cap \overline{\mathbb{B}}^*_r$
  is a nonempty compact set, which must provide a minimizer 
  $\check p \in \argmin \Phi^*(-L^* \cdot)$ for the continuous function
  $\|\cdot\|_*|_C$.
  Clearly we also have 
  $\|\check p\|_* = \inf_{p\in \argmin \Phi^*(-L^* \cdot)} \|p\|_*$,
  since $\|p\|_* \geq r = \|p_0\|_* \geq \|\check p\|_*$ for all
  $
    p \in \argmin \Phi^*(-L^* \cdot) 
	\setminus \overline{\mathbb{B}}^*_{r}$.
  \\
  3. Next $c > 0$ and $d > 0$ are proven, where we consider only the 
  interesting case $\argmin \Phi^*(-L^* \cdot) \not = \emptyset$.
  We have 
     \begin{align*}
      c = 0 &\Leftrightarrow
	    \min_{x \in \argmin \Phi} \|Lx\| = 0
	    \\
	    &\Leftrightarrow
	    \exists \check x \in \argmin \Phi: \|L \check x\| = 0
	    \\
	    &\Leftrightarrow
	    \argmin \Phi \cap \mathcal{N}(L) \not = \emptyset.
     \end{align*}
     Since $c \geq 0$ this just means 
     $c>0 \Leftrightarrow \argmin \Phi \cap \mathcal{N}(L) = \emptyset$,
     so that we really obtain $c > 0$.
     Using some calculus from Convex Analysis
     we obtain 
     \begin{align*}
       d=0  & \Leftrightarrow
	      \argmin \Phi \cap \mathcal{N}(L) \not = \emptyset,
     \end{align*}
     see Detail \ref{det:pr_of_locali_theorem:d_greater_than_zero}.
     Due to $d \geq 0$ this just means
     $d>0 \Leftrightarrow \argmin \Phi \cap \mathcal{N}(L) = \emptyset$,
     so that also $d > 0$.
  \\
  4.
  In order to verify that the different views on $c$ and $d$ are really equivalent,
  we set 
  \begin{align*}
    T	      &\defeq \{ \tau \in [0, +\infty) : 
	\exists x_0 \in \argmin(\Phi): \; \tau = \|Lx_0\| \},
  \\
    T_P       &\defeq \{ \tau \in [0, +\infty) : 
	{\rm SOL}(P_{1,\tau}) \cap \argmin \Phi \not = \emptyset \},
  \\
    T_D       &\defeq \{ \tau \in [0, +\infty) : 
	{\rm SOL}(D_{1,\tau}) \cap \{ \zerovec \} \not = \emptyset \}
  \shortintertext{and}
    \Lambda   &\defeq \{\lambda \in [0, +\infty):
	\exists p_0 \in \argmin \Phi^*(-L^* \cdot): \lambda = \|p_0\|_*\},
  \\
    \Lambda_P &\defeq \{\lambda \in [0, +\infty): 
	{\rm SOL}(P_{2,\lambda}) \cap \mathcal{N}(L) \not = \emptyset\},
  \\
    \Lambda_D &\defeq \{\lambda \in [0, +\infty): 
	{\rm SOL}(D_{2,\lambda}) \cap \argmin \Phi^*(-L^*\cdot) \not = \emptyset\},
  \end{align*}
  respectively, and show that
  \begin{align*}
    T          &= 
	\bigcup_{x_0 \in I_T} \{ \|Lx_0\| \}
    ,
    &T_P       &= T_D =
	\bigcup_{x_o \in I_T} [\|Lx_0\|, +\infty)
  \shortintertext{and}
    \Lambda    &= 
	\bigcup_{(x_0, p_0) \in I_{\Lambda}}    \{\|p_0\|_*\}
    ,    
    &\Lambda_D &= \Lambda_P  =
	\bigcup_{(x_0, p_0) \in I_{\Lambda}}    [\|p_0\|_*, +\infty)
    ,
  \end{align*}
  respectively, where 
  $
    I_T
    \defeq
    \{ x_0 \in \R^n : \zerovec \in \partial \Phi(x_0) \}
  $
  and
  $
    I_{\Lambda} 
    \defeq 
    \{ (x_0, p_0)  \in \R^n \times \R^m:  
	Lx_0 = \zerovec, x_0 \in \partial \Phi^*(-L^*p_0)\}
  $
  are some index sets. The above way of representing $T$, $T_P$, $T_D$ and
  $\Lambda$, $\Lambda_P$, $\Lambda_D$, respectively, then
  elucidates $c = \min T = \min T_P = \min T_D $ and
  $d = \inf \Lambda = \inf \Lambda_P = \inf \Lambda_D$, respectively;
  here the headline of part 2 of the proof ensures that the last 
  three infima are actually minima of the respective sets,
  if one -- an thus all -- of them is nonempty.
  For all $\tau \in [0, +\infty)$ we indeed have by Fermat's rule
  \begin{align*}
  \begin{split}
    {}& \tau \in T 
  \\
    \Leftrightarrow{}&
    \exists x_0 \in \R^n: \zerovec \in \partial \Phi (x_0) \wedge \tau = \|Lx_0\|
  \\
    \Leftrightarrow{}&
    \exists x_0 \in I_T: \tau \in \{\|Lx_0\|\}
  \\
    \Leftrightarrow{}&
    \tau \in \bigcup_{x_0 \in I_T} \{ \|Lx_0\| \}        
  \end{split}  
  &&& &&
  \begin{split}
    {}& \tau \in T_P
  \\
    \Leftrightarrow{}&
    \exists x_0 \in \R^n: \|Lx_0\| \leq \tau \wedge \zerovec \in \partial \Phi(x_0)
  \\
    \Leftrightarrow{}&
    \exists x_0 \in I_T:  \|Lx_0\| \leq \tau
  \\
    \Leftrightarrow{}&
    \tau \in \bigcup_{x_0\in I_T} [\|Lx_0\|, +\infty)
  \\
  \end{split}
  \end{align*}
  and -- by using again Fermat's Rule as well as the calculus for 
  subdifferentials, see \cite[p. 222-225]{Rockafellar1970}, 
  $x \in \partial \Phi^*(x^*) \Leftrightarrow x^* \in \partial \Phi(x)$ and
  \eqref{subgrad_2} -- also
  \begin{align*}
    {}& \tau \in T_D
  \\
    \Leftrightarrow{}&
    \exists p_0 \in \R^m: 
    \zerovec \in \partial (\Phi^*(-L^* \cdot) + \tau \|\cdot\|_*)|_{p_0}
    \wedge
    p_0 = \zerovec
  \\
    \Leftrightarrow{}&
    \zerovec \in \partial (\Phi^*(-L^* \cdot))|_{\zerovec} + \tau \partial \|\cdot\|_*|_{\zerovec}
  \\
    \Leftrightarrow{}&
    \zerovec \in -L \partial \Phi^*(-L^* \zerovec) + 
      \tau \overline{\mathbb{B}}_{1}[\|\cdot\|_{**}]
  \\
    \Leftrightarrow{}&
    \exists x_0 \in \R^n : 
      x_0 \in \partial \Phi^*(\zerovec)
      \wedge
      \zerovec \in -Lx_0 + \overline{\mathbb{B}}_{\tau}[\|\cdot\|]
  \\
    \Leftrightarrow{}&
    \exists x_0 \in \R^n : 
      \zerovec \in \partial \Phi(x_0)
      \wedge
      \|Lx_0\| \leq \tau
  \\
    \Leftrightarrow{}&
    \tau \in \bigcup_{x_0\in I_T} [\|Lx_0\|, +\infty)
  \end{align*}
  Similar we obtain for $\lambda \in [0, + \infty)$ the equivalences
  \begin{align*}
    {}& \lambda \in \Lambda
  \\\Leftrightarrow{}&
    \exists p_0 \in \R^m : \zerovec \in \partial (\Phi^*(-L^* \cdot))|_{p_0} \wedge \lambda = \|p_0\|_*
  \\\Leftrightarrow{}&
    \exists p_0 \in \R^m : \zerovec \in L \partial \Phi^*(-L^*p_0) \wedge \lambda = \|p_0\|_*
  \\\Leftrightarrow{}&
    \begin{smallmatrix}
      \exists x_0 \in \R^n, \\
      \exists p_0 \in \R^m
    \end{smallmatrix}:
    x_0 \in \partial \Phi^*(-L^*p_0) \wedge Lx_0 = \zerovec \wedge \lambda = \|p_0\|_*
  \\\Leftrightarrow{}&
    \exists (x_0, p_0) \in I_\Lambda : \lambda = \|p_0\|_*
  \\\Leftrightarrow{}&
    \lambda \in \bigcup_{(x_0,p_0) \in I_\Lambda} \{\|p_0\|_*\}
  \shortintertext{besides}
    {}& \lambda \in \Lambda_D
  \\\Leftrightarrow{}&
    \exists p_0 \in \R^m : \zerovec \in \partial (\Phi^*(-L^* \cdot))|_{p_0} \wedge \lambda \geq \|p_0\|_*
  \\\Leftrightarrow{}&
    \exists p_0 \in \R^m : \zerovec \in L \partial \Phi^*(-L^*p_0) \wedge \lambda \geq \|p_0\|_*
  \\\Leftrightarrow{}&
    \begin{smallmatrix}
      \exists x_0 \in \R^n, \\
      \exists p_0 \in \R^m
    \end{smallmatrix}:
    x_0 \in \partial \Phi^*(-L^*p_0) \wedge Lx_0 = \zerovec \wedge \lambda \geq \|p_0\|_*
  \\\Leftrightarrow{}&
    \exists (x_0, p_0) \in I_\Lambda : \lambda \geq \|p_0\|_*
  \\\Leftrightarrow{}&
    \lambda \in \bigcup_{(x_0,p_0) \in I_\Lambda} [\|p_0\|_*, + \infty)
  \shortintertext{and}
    {}& \lambda \in \Lambda_P
  \\\Leftrightarrow{}&
    \exists x_0 \in \R^n :
      \zerovec \in \partial(\Phi(\cdot) + \lambda\|L\cdot\| )|x_0 \wedge Lx_0 = \zerovec
  \\\Leftrightarrow{}&
    \exists x_0 \in \R^n :
      \zerovec \in \partial\Phi(x_0) + \lambda L^* \partial \|\cdot\| |_{Lx_0} \wedge Lx_0 = \zerovec
  \\\Leftrightarrow{}&
    \exists x_0 \in \R^n :
      \zerovec \in \partial\Phi(x_0) + \lambda L^* \partial \|\cdot\| |_{\zerovec} \wedge Lx_0 = \zerovec
  \\\Leftrightarrow{}&
    \exists x_0 \in \R^n :
      \zerovec \in \partial\Phi(x_0) + L^* \lambda \overline{\mathbb{B}}^*_{1} \wedge Lx_0 = \zerovec
  \\\Leftrightarrow{}&
    \begin{smallmatrix}
      \exists x_0 \in \R^n, \\
      \exists p_0 \in \R^m
    \end{smallmatrix}:
      p_0 \in \lambda \overline{\mathbb{B}}^*_{1} 
      \wedge
      \zerovec \in \partial\Phi(x_0) + L^*p_0
      \wedge 
      Lx_0 = \zerovec
  \\\Leftrightarrow{}&
    \begin{smallmatrix}
      \exists x_0 \in \R^n, \\
      \exists p_0 \in \R^m
    \end{smallmatrix}:
      \|p_0\|_* \leq \lambda
      \wedge
      -L^*p_0 \in \partial\Phi(x_0)
      \wedge 
      Lx_0 = \zerovec
  \\\Leftrightarrow{}&
    \begin{smallmatrix}
      \exists x_0 \in \R^n, \\
      \exists p_0 \in \R^m
    \end{smallmatrix}:
      x_0 \in \partial\Phi^*(-L^*p_0)
      \wedge 
      Lx_0 = \zerovec
      \wedge
      \|p_0\|_* \leq \lambda
  \\\Leftrightarrow{}&
    \exists (x_0, p_0) \in I_\Lambda : \lambda \geq \|p_0\|_*
  \\\Leftrightarrow{}&
    \lambda \in \bigcup_{(x_0,p_0) \in I_\Lambda} [\|p_0\|_*, + \infty).
  \end{align*}

  ~
  \\
  5.
  Finally we prove the 16 claimed relations of the theorem.
  The subset-relations for $\tau = 0$ and $\lambda = 0$ are trivially true.
  In oder to prove the primal relations for $\tau \in (0,c)$ and $\tau \in [c, +\infty)$
  we make use of
  $
    c 
    =
    \min \{\tau \in [0, +\infty) : 
    {\rm SOL}(P_{1,\tau}) \cap \argmin \Phi \not = \emptyset\}.
  $
  For $\tau \in (0,c)$ we directly get
  \[
    {\rm SOL}(P_{1,\tau}) \cap \argmin\Phi = \emptyset;
  \]
  this also implies that any $\hat x \in {\rm SOL}(P_{1,\tau})$, 
  $\tau \in (0,c)$ must fulfill $\|L\hat x\| \geq \tau > 0$. 
  ($\|L \hat x\| < \tau$ would mean that $\hat x$ is a local minimizer of
  $\Phi$, i.e. a global minimizer of the convex function $\Phi$; so we would 
  end up in the contradictory
  $\hat x \in {\rm SOL}(P_{1,\tau}) \cap \argmin\Phi = \emptyset$).
  So we have 
  \[
    {\rm SOL}(P_{1,\tau}) \cap \mathcal{N}(L) = \emptyset
  \]
  for $\tau \in (0,c)$. Furthermore the above reformulation of $c$ ensures
  that there is an $\hat x \in {\rm SOL}(P_{1,\tau}) \cap \argmin\Phi$ 
  for $\tau = c$. Clearly also 
  $\hat x \in {\rm SOL}(P_{1,\tau}) \cap \argmin\Phi$ for $\tau > c$, so that
  ${\rm SOL}(P_{1,\tau}) \cap \argmin \Phi \not = \emptyset$ for
  $\tau \in [c, + \infty)$. 
  Since no solvers of 
  $(P_{1,\tau})$ can be outside of $\argmin \Phi$, as soon as one solver of
  $(P_{1,\tau})$ belongs to this level set of $\Phi$, we even get
  \[
    {\rm SOL}(P_{1,\tau}) \subseteq \argmin\Phi
  \]
  for $\tau \in [c, + \infty)$.
\\[0.5ex]
  In order to prove the dual relations for $\tau \in (0,c)$ and 
  $\tau \in [c, +\infty)$ we use 
  $
    c = \min 
      \{\tau \in [0, +\infty):
	  {\rm SOL}(D_{1,\tau}) \cap \{\zerovec\} \not = \emptyset
      \}
  $.
  For $\tau \in (0,c)$ this immediately implies 
  \[
    {\rm SOL}(D_{1,\tau}) \cap \{\zerovec\} = \emptyset
  \]
  and
  ${\rm SOL}(D_{1,c}) \cap \{\zerovec\} \not = \emptyset$.
  The latter means
  $\Phi^*(-L^* \zerovec) + c \|\zerovec\|_* \leq \Phi(-L^*p) + c \|p\|_*$ 
  for all $p\in \R^m$. For $\tau \in [c, +\infty)$ addition of
  the inequality $(\tau - c)\|\zerovec\|_* \leq (\tau -c) \|p\|_*$ yields
  $\Phi^*(-L^*\zerovec) + \tau \|\zerovec\|_* \leq \Phi^*(-L^*p) + \tau \|p\|_*$
  for all $p\in \R^m$. 
  This just means
  $\zerovec \in {\rm SOL}(D_{1,\tau})$ for $\tau \in [c, +\infty)$.
  We even have
  \[
    {\rm SOL}(D_{1,\tau}) = \{\zerovec\}
  \]
  for $\tau \in [c, +\infty)$:
  Let an additional $\check p \in {\rm SOL}(D_{1,\tau})$ be given.
  In order to prove $\check p = \zerovec$ it suffices to check that  Theorem 
  \ref{thm:direction_of_argmin__essentially_smooth__strict_convex_sharpend}
  can be applied to $F (\cdot) = \Phi^*(-L^* \cdot)$ and 
  $G(\cdot) = \tau \|\cdot\|_*$, since this theorem would then give
  $\tau \|\check p\|_* = G (\check p) = G(\zerovec) = 0$ and hence the wanted
  $\check p = \zerovec$.
  Indeed all assumptions of this theorem are fulfilled:
  Clearly $F$ and $G$ are convex functions with 
  $
    {\rm ri}(\dom F) \cap {\rm ri}(\dom G)
    = 
    {\rm ri}(\dom F) 
    \not = 
    \emptyset
  $.
  Next the needed decomposition 
  ${\rm aff}(\dom F) = \check A_F \oplus \check P_F$ is obtained,
  by using 
  \prettyref{thm:concatenation_stricly_convex_on_ri_with_linear_mapping},
  see Detail \ref{det:pr_of_locali_theorem:SOL_D_1_equals_0}.
  Finally Theorem \ref{thm:concatenation_ess_smooth_with_linear_mapping} 
  ensures that $F = E \circ M$ is essentially smooth on ${\rm aff}(\dom F)$.
  So all assumptions of Theorem
  \ref{thm:direction_of_argmin__essentially_smooth__strict_convex_sharpend}
  are really fulfilled.
  Finally we show 
  \[
    {\rm SOL}(D_{1,\tau}) \cap \argmin \Phi^*(-L^* \cdot) = \emptyset
  \]
  for $\tau \in (0,c)$: 
  Assume that there is a 
  $\hat p \in {\rm SOL}(D_{1,\tau}) \cap \argmin \Phi^*(-L^* \cdot)$ for a 
  $\tau \in (0,c)$.
  The functions $F(\cdot) = \Phi(-L^* \cdot)$ and 
  $G(\cdot) = \tau \|\cdot\|_*$
  fulfill the assumptions of Theorem
  \ref{thm:direction_of_argmin__essentially_smooth__strict_convex_sharpend},
  see Detail \ref{det:pr_of_locali_theorem:SOL_D_1_away_from_argmin},
  so that $\hat p \in \argmin(F+G) \subseteq {\rm int}_A(\dom F)$.
  Consider $F$ and $G$ now only on the vector subspace 
  $A \defeq {\rm aff}(\dom F)$ by setting 
  $f\defeq F|_A \in \Gamma_0(A)$ and 
  $\|\cdot \|' \defeq \|\cdot \|_*|_A \in \Gamma_0(A)$. Since
  $F(x) = + \infty $ for $x \not \in A$ we still had
  $\hat p \in \argmin_{p \in A}(f(p)+ \tau\|p\|')$ and 
  $\hat p \in \argmin_{p\in A} f(p)$.
  The function $f:A \rightarrow \R \cup \{+\infty\}$,
  beeing essentially smooth by Lemma
  \ref{lem:subgradient_and_conjugate_function_in_our_setting}
  and Theorem \ref{thm:concatenation_ess_smooth_with_linear_mapping}, 
  would be differentiable in  
  $\hat p \in {\rm int}_A(\dom F) = {\rm int}_A(\dom f)$.
  By Theorem 
  \ref{thm:subgradient_of_essentially_smooth_and_Gamma_0_function}
  and by Fermat's rule we had
  $\partial f(\hat p) = \{\zerovec\}$.
  Using Fermat's rule and the calculus for subdifferentials we hence obtained
  $
    \zerovec 
    \in 
    \partial (f + \tau \|\cdot \|')|_{\hat p}
    =
    \partial f(\hat p) + \tau \partial \|\cdot\|'|_{\hat p}
    =
    \tau \partial \|\cdot\|'|_{\hat p}
  $.
  The already proven 
  ${\rm SOL}(D_{1,\tau}) \cap \{\zerovec\} = \emptyset$ says $\hat p \not = \zerovec$, so
  that equation \eqref{subgrad_2} implied the contradictory
  $
    \zerovec 
    \in 
    \partial \|\cdot\|'|_{\hat p} 
    \subseteq
    \mathbb{S}_1[(\|\cdot\|')_*]
  $.
  \\[2ex]
  In order to prove the primal relations for $\lambda \in (0,d)$ we make
  use of 
  $
    d 
    = 
    \inf \{\lambda \geq 0: 
	{\rm SOL}(P_{2,\lambda}) \cap \mathcal{N}(L) \not = \emptyset\}
  $,
  while for proving the primal relations for $\lambda \in [d, +\infty)$
  we may assume $d < +\infty$, i.e.
  $
    d
    =
    \min \{\lambda \geq 0: 
	{\rm SOL}(P_{2,\lambda}) \cap \mathcal{N}(L) \not = \emptyset\}
  $,
  since in the vacuous case $d = + \infty$, 
  meaning $[d, +\infty) = \emptyset$, there is nothing to show.
  For $\lambda \in (0,d)$ we then get immediately
  \[
    {\rm SOL}(P_{2,\lambda}) \cap \mathcal{N}(L) = \emptyset
  \]
  and for $\lambda = d$ we get 
  ${\rm SOL}(P_{2,d}) \cap \mathcal{N}(L) \not = \emptyset$.
  The latter means
  $\Phi(\hat x) + d \|L\hat x\| \leq \Phi(x) + d\|Lx\|$
  for all $\hat x\in {\rm SOL}(P_{2,d}) \cap \mathcal{N}(L)$ and 
  $x \in \R^n$. For $\lambda \geq d$ adding 
  $(\lambda -d)\|L\hat x\| \leq (\lambda -d) \|Lx\|$
  hence gives 
  $\Phi(\hat x) + \lambda \|L\hat x\| \leq \Phi(x) + \lambda \|Lx\|$
  for all $\hat x\in {\rm SOL}(P_{2,d}) \cap \mathcal{N}(L)$ and 
  $x \in \R^n$, so that we have
  ${\rm SOL}(P_{2,\lambda}) \cap \mathcal{N}(L) \not = \emptyset$ for
  all $\lambda \in [d, +\infty)$. We even have 
  \[
    {\rm SOL}(P_{2,\lambda}) \subseteq \mathcal{N}(L)
  \]
  for $\lambda \in [d, +\infty)$:
  Choose any $\hat x \in {\rm SOL}(P_{2,\lambda}) \cap \mathcal{N}(L)$ 	
  and consider an arbitrarily chosen $\tilde x \in {\rm SOL}(P_{2,\lambda})$.
  In order to prove $L \tilde x = \zerovec$ it suffices to check that Theorem 
  \ref{thm:direction_of_argmin__essentially_smooth__strict_convex_sharpend}
  can be applied to $F = \Phi$ and $G(\cdot) = \lambda \|L \cdot\|$,
  since this theorem would then give 
  $\lambda \|L\tilde x\| = G(\tilde x) = G(\hat x) = 0$ and hence the 
  needed $L \tilde x = \zerovec$. Indeed all assumptions of this theorems are 
  fulfilled, see Detail 
  \ref{det:pr_of_locali_theorem:SOL_P_2_belongs_to_nullspace}.
  Finally we show 
  \[
    {\rm SOL}(P_{2,\lambda}) \cap \argmin \Phi = \emptyset
  \]
  for $\lambda \in (0,d)$:
\iftime{Dieser neue Beweis der Lagerelation vom 7. oder 8. Maerz kommt ohne die 
Bedingung $\mathcal{R}(L^*) \cap X_3 = \{\zerovec\}$ aus.
Auch bewirkt dieser Weg vereinfachung --> Auch an anderen Stellen moeglich?!}
  It clearly suffices to show that any minimizer of $\Phi$ can never
  belong to
  ${\rm SOL}(P_{2,\lambda})$ for any real $\lambda > 0$.
  To this end fix $\lambda \in(0,+\infty)$ and let an arbitrary
  $\hat x \in \argmin \Phi$ be given.
  Regard $\Phi$ and $\Phi(\cdot) + \lambda \|L \cdot \|$ only on 
  ${\rm span}(\hat x)$ by considering the functions 
  $f,h:\R \rightarrow \R \cup \{+ \infty\}$ given by 
  $f(t) \defeq \Phi(t \hat x)$ and
  $
    h(t) \defeq \Phi(t \hat x) + \lambda \|L (t \hat x)\| 
	 = \Phi(t \hat x) + m|t|
  $, where $m \defeq \lambda \|L \hat x\| > 0$ due to the assumption 
  $\argmin \Phi \cap \mathcal{N}(L) = \emptyset$.
  $\Phi$ is proper, convex, lower semicontinuous and essentially smooth 
  on the affine hull of its effective domain of definition.
  These properties carry over to $f$, 
  see Detail
  \ref{det:pr_of_locali_theorem:SOL_P_2_does_not_intersect_argmin_Phi}.
  Since $1 \in \R$ is clearly a minimizer of $f$ we obtain,
  using part 
  \upref{enu:argmin_contained_in_ri_of_dom}
  of Lemma 
  \ref{lem:argmin_of_sum_contained_in_ri_of_dom_of_ess_smooth_summand},
  that $f$ is differentiable in $1 \in \R$ with derivative
  $f'(1) = 0$. Hence also $h$ is 
  differentiable in $1$ with derivative 
  $h'(1) = f'(1) + m = m > 0$.
  Consequently there is an $\varepsilon > 0$ such that 
  $h(1-\varepsilon) < h(1)$.
  Its rewritten form 
  $
    \Phi\left((1-\varepsilon)\hat x\right) 
      + \lambda \|L (1-\varepsilon)\hat x\|
    <
    \Phi(\hat x) + \lambda \|L \hat x\|
  $
  shows that $\hat x$ is not a minimizer of ${\rm SOL}(P_{2,\lambda})$.
  \\[0.5ex]
  In order to prove the dual relations for $\lambda \in (0,d)$ we make
  use of 
  $
    d 
    = 
    \inf \{\lambda \geq 0: 
	{\rm SOL}(D_{2,\lambda}) \cap \argmin\Phi^*(-L^* \cdot) \not = \emptyset\}
  $,
  while for proving the dual relations for $\lambda \in [d, +\infty)$
  we may assume $d < +\infty$, i.e.
  $
    d
    =
    \min \{\lambda \geq 0: 
	{\rm SOL}(D_{2,\lambda}) \cap \argmin\Phi^*(-L^* \cdot) \not = \emptyset\}
  $,
  since in the vacuous case $d = + \infty$ there is again nothing to show.
  For $\lambda \in (0,d)$ we then get immediately 
  \[
    {\rm SOL}(D_{2,\lambda}) \cap \argmin \Phi^*(-L^* \cdot) = \emptyset;
  \]
  this also implies that any $\hat p \in {\rm SOL}(D_{2,\lambda})$ with
  $\lambda \in (0,d)$ must fulfill $\|\hat p\|_{*} \geq \lambda > 0$.
  ($\|\hat p\|_{*} < \lambda$ would mean that $\hat p$ is a local minimizer of 
  $\Phi^*(-L^* \cdot)$ and hence a global minimizer of this convex function;
  so we would end up in the contradictory
  $\hat p \in {\rm SOL}(D_{2,\lambda}) \cap \argmin \Phi^*(-L^* \cdot) = \emptyset$).
  So we have 
  \[
    {\rm SOL}(D_{2,\lambda}) \cap \{\zerovec\} = \emptyset
  \]
  for $\tau \in (0,d)$. Furthermore the above reformulation of $d$ ensures
  that there is an $\hat p \in {\rm SOL}(D_{2,\lambda}) \cap \argmin \Phi^*(-L^* \cdot)$
  for $\lambda = d$. Clearly also
  $\hat p \in {\rm SOL}(D_{2,\lambda}) \cap \Phi^*(-L^* \cdot)$ for $\lambda > d$,
  so that 
  ${\rm SOL}(D_{2,\lambda}) \cap \argmin \Phi^*(-L^* \cdot) \not = \emptyset$ for
  $\lambda \in [d, +\infty)$.
  Since no solvers of 
  $(D_{2,\lambda})$ can be outside of $\argmin \Phi^*(-L^* \cdot)$, as soon as one solver
  of $(D_{2,\lambda})$ belongs to this level set of $\Phi^*(-L^* \cdot)$, we even get
  \[
    {\rm SOL}(D_{2,\lambda}) \subseteq \argmin \Phi^*(-L^* \cdot)
  \]
  for $\lambda \in [d, +\infty)$.  
\end{proof}

Now we give the announced example, showing that 
${\rm SOL}(D_{1,\tau}) = \emptyset$ can happen in the border case 
$\tau = 0$.
\begin{example} \label{exa:Sol_can_be_empty_for_tau_eq_0}
  The particular choice 
  \begin{gather*}
    \Phi(x) \defeq \phi(x) \defeq 
    \begin{cases}
      x-1 + \log \tfrac{1}{x}  & \text { for } x > 0 
    \\
      \pinfty 			& \text { for } x \leq 0
    \end{cases}
  \end{gather*}
  gives a functions $\Phi: \R \rarr \R \cup \{\pinfty\}$ 
  that fulfills the requirements of our setting along with 
  the identity matrix $L \defeq (1)$ 
  and $\|\cdot\| = |\cdot|$.
  The conjugate function $\Phi^*: \R \rarr \R \cup \{\pinfty\}$
  can explicitely be expressed as 
  \begin{gather*}
    \Phi^*(p) = 
    \begin{cases}
      - \log (1-p)  & \text{ for } p < 1
    \\
      \pinfty 	    & \text{ for } p \geq 1
    \end{cases},
  \end{gather*}
  cf. \cite{CiShSt2012} or \cite[p. 50f]{BorweinLewis2000}. 
  Here clearly the proper function $\Phi^*$ is not bounded below 
  so that ${\rm SOL}(D_{1, 0}) = - \argmin \Phi^* = \emptyset$.
\end{example}

The following theorem specifies the relations between ($P_{1,\tau}$), ($P_{2,\lambda}$),
($D_{1,\tau}$) and ($D_{2,\lambda}$) for the special setting in this section.
We will see that for every $\tau \in (0,c)$,
there exists a {\it uniquely determined} $\lambda$ such that the {\it solution sets} of ($P_{1,\tau}$)
and ($P_{2,\lambda}$) coincide. 
Note that by the Remarks \ref{remark_2} and \ref{remark_2_1}
this is not the case for general functions $\Phi,\Psi \in \Gamma_0(\R^n)$.
Moreover, we want to determine for given $\tau$, the value $\lambda$ such that ($P_{2,\lambda}$)
has the same solutions as ($P_{1,\tau}$).
Note that part i) of \prettyref{thm:constraint_vs_nonconstraint} was not constructive.
%
\begin{theorem} \label{thm:theo_lambda_tau}
Let $\Phi \in \Gamma_0(\R^n)$ be of the form \eqref{new_setting_Phi} and 
let $L \in \R^{m,n}$ such that
$X_2  \cap {\cal N} (L) = \{\zerovec\}$, 
$X_3 \cap  \mathcal{R}(L^*) = \{\zerovec\}$
and 
$\argmin \Phi \cap  {\cal N} (L) = \emptyset$.
Define $c$ 
by  \eqref{eq:the_tau_for_which_SOL(P_1,tau)_arrives_at_argmin_Phi_for_the_first_time}
and $d$ by
\eqref{eq:the_lambda_for_which_SOL(P_2,lambda)_arrives_at_argmin_G_for_the_first_time}.
Then, for $\tau \in (0,c)$ and $\lambda \in (0, d)$, the problems
$(P_{1,\tau})$,
$(P_{2,\lambda})$,
$(D_{1,\tau})$,
$(D_{2,\lambda})$
have solutions with finite minima.
Further there exists a bijective mapping $g: (0,c) \rightarrow (0,d)$
such that for $\tau \in (0,c)$ and $\lambda \in (0,d)$
we have
\begin{alignat*}{2}
  \Bigg\{
  \begin{split}
      {\rm SOL}(P_{1,\tau}) &= {\rm SOL}(P_{2,\lambda}) 
      \\
      {\rm SOL}(D_{1,\tau}) &= {\rm SOL}(D_{2,\lambda}) 
  \end{split}
  \Bigg\} 
  &\quad \text{ if } \quad
  (\tau,\lambda) \in {\rm gr} \, g
\end{alignat*}
and for
$\tau \in (0,c)$, $\lambda \in [0, +\infty)$
or
$\lambda \in (0,d)$, $\tau \in [0, +\infty)$,
\begin{alignat*}{2}
  \Bigg\{
  \begin{split}
      {\rm SOL}(P_{1,\tau}) &\cap {\rm SOL}(P_{2,\lambda}) = \emptyset 
      \\
      {\rm SOL}(D_{1,\tau}) &\cap {\rm SOL}(D_{2,\lambda}) = \emptyset
  \end{split}
  \Bigg\} \quad
  &\text{ if } \quad (\tau,\lambda) \not \in {\rm gr} \, g.
\end{alignat*}
For $(\tau, \lambda) \in {\rm gr} \, g$  any solutions $\hat x$ and $\hat p$ 
of the primal and dual problems, resp., fulfill
\begin{equation*}
\tau = \|L \hat x\| \quad {\rm and} \quad \lambda = \| \hat p \|_*.
\end{equation*}
\end{theorem}
%
\vspace{0.2cm}

\vspace{0.2cm}

\noindent
{\bf Proof.} 
Note in the following that the requirements i) - iv) of 
Lemma \ref{fenchel-dual-hom} are fulfilled for $\tau \in (0, +\infty)$
and $\lambda \in (0, +\infty)$. 

Theorem \ref{thm:existence_of_solvers} ensures that all solver sets
${\rm SOL}(P_{1,\tau})$,
${\rm SOL}(P_{2,\lambda})$,
${\rm SOL}(D_{1,\tau})$,
${\rm SOL}(D_{2,\lambda})$
are not empty for $\tau \in (0,c)$ and $\lambda \in (0, d)$
and that only finite minima are taken.

The core of the proof consists of two main steps:
In the first step we use
Theorem \ref{fenchel-dual-hom}, 
Theorem \ref{thm:direction_of_argmin__essentially_smooth__strict_convex_sharpend}
and 
\prettyref{thm:constraint_vs_nonconstraint} {\rm ii)}
to construct mappings 
$g: (0, c) \rightarrow (0,d)$,
$f: (0, d) \rightarrow (0,c)$
with the following properties:
\begin{align}
  \forall \tau \in (0,c):
  \Bigg\{
  \begin{split}
      {\rm SOL}(P_{1,\tau}) &\subseteq {\rm SOL}(P_{2,g(\tau)}) 
      \\
      {\rm SOL}(D_{1,\tau}) &\subseteq {\rm SOL}(D_{2,g(\tau)}) 
  \end{split}
  \Bigg\},
  \label{eq:sol_1_is_contained_in_sol_2}
  \\
  \forall \lambda \in (0,d):
  \Bigg\{
  \begin{split}
      {\rm SOL}(P_{2,\lambda}) &\subseteq {\rm SOL}(P_{1,f(\lambda)}) 
      \\
      {\rm SOL}(D_{2,\lambda}) &\subseteq {\rm SOL}(D_{1,f(\lambda)}) 
  \end{split}
  \Bigg\}.
  \label{eq:sol_2_is_contained_in_sol_1}
\end{align}
In the second step we verify that $f\circ g = {\rm id}_{(0,c)}$ and $g \circ f= {\rm id}_{(0,d)}$
so that $g$ is bijective and
\eqref{eq:sol_1_is_contained_in_sol_2} and \eqref{eq:sol_2_is_contained_in_sol_1}
actually hold true with equality.
Finally, we deal in a third part with $(\tau,\lambda) \not \in {\rm gr} g$. 
\\[0.7ex]
1. First we show that for all 
$\hat x \in \R^n \setminus \mathcal{N}(L)$,
$\hat p \in \R^m \setminus \{\zerovec\}$
and for all
$\lambda,\tau > 0$
the following equivalence holds true:
\begin{equation}
  \left\{
  \begin{aligned}
      \hat x \in {\rm SOL}(P_{1,\tau}),  \\
      \hat p \in {\rm SOL}(D_{1,\tau}),	 \\
      \lambda = \|\hat p\|_*
  \end{aligned}
  \right\}
  \Leftrightarrow
   \left\{
  \begin{aligned}
      \hat x \in {\rm SOL}(P_{2,\lambda}),  \\
      \hat p \in {\rm SOL}(D_{2,\lambda}),  \\
      \tau    = \|L \hat x\| 
  \end{aligned}
  \right\}.
  \label{eq:solvers_of_1_vs_solvers_of_2}
\end{equation}
We have on the one hand for
$\hat x \in \R^n \setminus \mathcal{N}(L)$,
$\hat p \in \R^m \setminus \{\zerovec\}$,
$\tau > 0$ and $\lambda >0$
the equivalences
\begin{alignat*}{2}
  &~~\phantom{\Leftrightarrow}
    \hat x \in {\rm SOL}(P_{1,\tau}),~ \hat p \in {\rm SOL}(D_{1,\tau})
  &&
\\
  &\Leftrightarrow
    \tau \hat p \in \partial \Psi_1(\tau^{-1}L \hat x),~
  && -L^* \hat p \in \partial \Phi(\hat x)
\\
  &\Leftrightarrow
    \Psi_1(\tau^{-1} L \hat x) + \Psi_1^*(\tau \hat p)
    =
    \langle \tau^{-1}L \hat x, \tau \hat p \rangle,~
  && -L^* \hat p \in \partial \Phi(\hat x)
\\
  &\Leftrightarrow
    \|L \hat x\| \leq \tau,~
    \tau \|\hat p\|_* = \langle L\hat x, \hat p \rangle,~
  && -L^* \hat p \in \partial \Phi(\hat x)
\\
  &\Leftrightarrow
    \|L \hat x\| = \tau,~
    \tau \|\hat p\|_* = \langle L\hat x, \hat p \rangle,~
  && -L^* \hat p \in \partial \Phi(\hat x)
\\
  &\Leftrightarrow
    \|L \hat x\| = \tau,~
    \|L \hat x\| \|\hat p\|_* = \langle L\hat x, \hat p \rangle,~
  && -L^* \hat p \in \partial \Phi(\hat x),    
\end{alignat*}
where we used Lemma \ref{fenchel-dual-hom} in step 1, 
the Fenchel equality \cite[Theorem 23.5]{Rockafellar1970} in step 2
and applied in step 4 the inequality 
$\langle p ,  p' \rangle \leq \|p\| \|p'\|_*$
for $p=L \hat x$, $p' = \hat p$.
On the other hand we obtain similar for
$\hat x \in \R^n \setminus \mathcal{N}(L)$,
$\hat p \in \R^m \setminus \{\zerovec\}$,
$\tau > 0$ and $\lambda >0$
the equivalences
\begin{alignat*}{2}
    &~~\phantom{\Leftrightarrow}
    \hat x \in {\rm SOL}(P_{2,\lambda}),~ \hat p \in {\rm SOL}(D_{2,\lambda})
  &&
\\
  &\Leftrightarrow
    \lambda^{-1} \hat p \in \partial \psi_2(\lambda L \hat x),~
  && -L^* \hat p \in \partial \Phi(\hat x)
\\
  &\Leftrightarrow
    \Psi_2(\lambda L \hat x) + \Psi_2^*(\lambda^{-1} \hat p)
    =
    \langle \lambda L \hat x, \lambda^{-1} \hat p \rangle,~
  && -L^* \hat p \in \partial \Phi(\hat x)
\\
  &\Leftrightarrow
    \lambda \|L \hat x\| = \langle L\hat x, \hat p \rangle,~
    \|\hat p\|_* \leq \lambda,~
  && -L^* \hat p \in \partial \Phi(\hat x)
\\
  &\Leftrightarrow
    \lambda \|L \hat x\| = \langle L\hat x, \hat p \rangle,~
    \|\hat p\|_* = \lambda,~
  && -L^* \hat p \in \partial \Phi(\hat x)
\\
  &\Leftrightarrow
    \|\hat p\|_* = \lambda,~
     \|L \hat x\| \|\hat p\|_* = \langle L\hat x, \hat p \rangle,~
  && -L^* \hat p \in \partial \Phi(\hat x).
\end{alignat*}

Adding the conditions $\lambda = \|\hat p\|_*$ and
$\tau = \|L \hat x\|$, respectively, we see directly
that \eqref{eq:solvers_of_1_vs_solvers_of_2} holds true.
Now we can construct the function $g$ on $(0,c)$ as follows:
Let $\tau \in (0,c)$ and set 
\begin{gather*} 
  g(\tau) \defeq \|\hat p\|_*
\end{gather*}
with any $\hat p \in {\rm SOL}(D_{1,\tau})$; 
this is well defined by Detail 
\ref{det:pr_of_main_theorem:function_g_is_well_defined}.
Theorem \ref{thm:localisation_of_the_solvers} assures
${\rm SOL}(P_{1,\tau}) \cap \mathcal{N}(L) = \emptyset$,
${\rm SOL}(D_{1,\tau}) \cap \{\zerovec\}          = \emptyset$
and 
${\rm SOL}(D_{1,\tau}) \cap \argmin \Phi^*(-L^* \cdot) = \emptyset$,
so that 
\begin{align}
  \|L \hat x\| &> 0, 
  &\|\hat p\|_* &> 0, 
  &\|\hat p\|_* &< d
  \label{eq:x_good_and_p_very_good}
\end{align}
for all $\hat x \in {\rm SOL}(P_{1,\tau})$, and 
for all $\hat p \in {\rm SOL}(D_{1,\tau})$;
see Detail \ref{det:pr_of_main_theorem:norm_of_dual_sol_lower_than_d}
for the last inequality.
By the second and third inequality in \eqref{eq:x_good_and_p_very_good}
we see that
$g(\tau) \in (0,d)$, so that $g:(0,c) \rightarrow (0,d)$.
The wanted inclusions in \eqref{eq:sol_1_is_contained_in_sol_2} follow now
from \eqref{eq:solvers_of_1_vs_solvers_of_2},
which is allowed to apply, by the first and second inequality in 
\eqref{eq:x_good_and_p_very_good}.
\\[0.7ex]
The function $f$ on $(0,d)$ is constructed as follows:
Let $\lambda \in (0,d)$ and set
\begin{gather*}
  f(\lambda) \defeq \|L \hat x\|
\end{gather*}
with any $\hat x \in {\rm SOL}(P_{2,\lambda})$; this is well defined, by Detail 
\ref{det:pr_of_main_theorem:function_f_is_well_defined}.
Theorem \ref{thm:localisation_of_the_solvers} assures
${\rm SOL}(D_{2,\lambda}) \cap \{\zerovec\} = \emptyset$,
${\rm SOL}(P_{2,\lambda}) \cap \mathcal{N}(L) = \emptyset$
and 
${\rm SOL}(P_{2,\lambda}) \cap \argmin \Phi = \emptyset$,
so that 
\begin{align}
  \|\hat p\|_* &> 0, 
  &\|L \hat x\| &> 0, 
  &\|L \hat x\| &< c	
  \label{eq:p_good_and_x_very_good}
\end{align}
for all $\hat p \in {\rm SOL}(D_{2,\lambda})$, and 
for all $\hat x \in {\rm SOL}(P_{2,\lambda})$;
see Detail \ref{det:pr_of_main_theorem:G_of_primal_solution_lower_than_c} 
for the last inequality.
\\[0.5ex]
By the second and third inequality in \eqref{eq:p_good_and_x_very_good} we see that
$f(\lambda) \in (0,c)$, so that $f:(0,d) \rightarrow (0,c)$.
The inclusions in \eqref{eq:sol_2_is_contained_in_sol_1} follow now
from \eqref{eq:solvers_of_1_vs_solvers_of_2},
which is allowed to apply, by the first and second inequality in 
\eqref{eq:p_good_and_x_very_good}.
\\[0.7ex]
2. First we note that
\begin{alignat}{2}
  {\rm SOL}(P_{1,\tau}) 
  &\cap 
  {\rm SOL}(P_{1,\tau'}) 
  &&= 
  \emptyset,
  \label{eq:sol_1_moves}
\\
  {\rm SOL}(D_{2,\lambda})
  &\cap
  {\rm SOL}(D_{2,\lambda'})
  &&=
  \emptyset
  \label{eq:sol_2_moves}
\end{alignat}
for all distinct $\tau, \tau' \in (0,c)$ 
and all distinct $\lambda, \lambda' \in (0,d)$, respectively,
cf. detail 
\ref{det:pr_of_main_theorem:different_taus_from_0_to_c_give_distict_solver_sets}.
\\
Next we prove the bijectivity of $g:(0,c) \rightarrow (0,d)$ 
by showing $f \circ g = {\rm id}_{(0,c)}$
and        $g \circ f = {\rm id}_{(0,d)}$.
In doing so we will also see that
\eqref{eq:sol_1_is_contained_in_sol_2}
actually holds true with equality.
Let $\tau \in (0,c)$ be arbitrarily chosen and
set $\tau' = f(g(\tau))$. 
Using \eqref{eq:sol_1_is_contained_in_sol_2}
and   \eqref{eq:sol_2_is_contained_in_sol_1}
with $\lambda = g(\tau)$ yields
\begin{alignat*}{3}
  {\rm SOL}(P_{1,\tau}) 
  &\subseteq
  {\rm SOL}(P_{2,g(\tau)})
  &&\subseteq
  {\rm SOL}(P_{1,\tau'}),
\\
  {\rm SOL}(D_{1,\tau}) 
  &\subseteq
  {\rm SOL}(D_{2,g(\tau)})
  &&\subseteq
  {\rm SOL}(D_{1,\tau'}).
\end{alignat*}
Since ${\rm SOL}(P_{1,\tau}) \not = \emptyset$ we must have
$\tau = \tau'$ in order to avoid a contradiction to 
\eqref{eq:sol_1_moves}. 
Similarly we can prove for an arbitrarily chosen
$\lambda \in (0,d)$ and
$\lambda' \defeq g(f(\lambda))$ that $\lambda = \lambda'$,
see detail \ref{det:pr_of_main_theorem:g_after_f_is_identity}.
\\[0.7ex]
3. It remains to show 
$ {\rm SOL}(P_{1,\tau}) \cap {\rm SOL}(P_{2,\lambda}) = \emptyset$
and 
$ {\rm SOL}(D_{1,\tau}) \cap {\rm SOL}(D_{2,\lambda}) = \emptyset$
for these 
$ 
  (\tau, \lambda) \in 
  [(0,c)\times [0, +\infty)] \cup [[0, +\infty ) \times (0,d)]
$
with $(\tau, \lambda) \not \in {\rm gr} \, g$.
Having Theorem \ref{thm:localisation_of_the_solvers} in mind, we may
restrict us to those $(\tau, \lambda)\in(0,c)\times(0,d)$
which are not in ${\rm gr} \, g$. For such $\tau$, $\lambda$ we have
$\tau \not = g^{-1}(\lambda)$ and $\lambda \not = g(\tau)$. By 
\eqref{eq:sol_1_moves} and \eqref{eq:sol_2_moves}
we therefore have
$
  {\rm SOL}(P_{1,\tau}) 
  \cap 
  {\rm SOL}(P_{1,g^{-1}(\lambda)}) 
  = 
  \emptyset
$
and 
$ {\rm SOL}(D_{2,\lambda})
  \cap
  {\rm SOL}(D_{2,g(\tau)})
  =
  \emptyset
$.
Substituting 
${\rm SOL}(P_{1,g^{-1}(\lambda)})$ by ${\rm SOL}(P_{2,\lambda})$ and 
${\rm SOL}(D_{2,g(\tau)})$         by ${\rm SOL}(D_{1,\tau})$
we are done.
\hfill $\Box$
\\

Here are some more properties of the function $g$.
%
\begin{corollary}\label{lem:monotone}
  Let the assumptions of \prettyref{thm:theo_lambda_tau} be fulfilled.
  Then the bijection $g: (0,c) \rightarrow (0,d)$ is strictly monotonic decreasing
  and continuous.
\end{corollary}
%
{\bf Proof.}
Since decreasing bijections between open intervals are strict decreasing and continuous
we need only to show 
that $f = g^{-1}: (0,d) \rightarrow (0,c)$ is decreasing.
Let $0<\lambda_1 < \lambda_2 < d$
and
$\hat x_i \in \argmin_{x \in \R^n} \{ \Phi(x) + \lambda_i \Psi(x) \}$, $i = 1,2$, 
where $\Psi(x) \defeq \|Lx\|$.

Then we know that $\tau_i =\Psi(\hat x_i)$, $i=1,2$.
Assume that $\Psi(\hat x_1) < \Psi(\hat x_2)$. Then we obtain with $\lambda_2=\lambda_1+\varepsilon$ and
  $\varepsilon >0$ the contradiction
  \begin{align*}
    \Phi(\hat x_2)+\lambda_2 \Psi(\hat x_2)
    &= \Phi(\hat x_2)+\lambda_1 \Psi(\hat x_2)+\varepsilon \Psi(\hat x_2)\\
    &\geq \Phi(\hat x_1)+\lambda_1 \Psi(\hat x_1)+\varepsilon \Psi(\hat x_2)\\
    &> \Phi(\hat x_1)+\lambda_1 \Psi(\hat x_1)+\varepsilon \Psi(\hat x_1)\\
    &= \Phi(\hat x_1)+\lambda_2 \Psi(\hat x_1).
  \end{align*}
\hfill $\Box$

\begin{remark} \label{rem:g_neither_differentiable_nor_convex}

The function $g$ is in general neither differentiable nor convex 
as the following example shows: The strictly convex function 
$\Phi$, given by
\[
   \Phi(x)
   \defeq
   \begin{cases}
     (x-4)^2      &  \text{ for } x \leq 2 \\
     2(x-3)^2 + 2 &  \text{ for } x > 2
   \end{cases}
\]
has exactly one minimizer, namely $x_0 = 3$.
Clearly $\Phi$, $\|\cdot \| \defeq | \cdot |$ and $L=(1)$ fulfill 
all assumptions of \prettyref{thm:theo_lambda_tau} if we set $X_2 \defeq \{0\}$.
For $\lambda \geq 0$ and $\tau \in (0, c) = (0,x_0)$ we have
%
\begin{equation*}
  \argmin_{x\in\R} \{ \Phi(x) \st  |x|\leq\tau  \}
  =
  \{ \tau \}
  \eqdef
  \{ \hat x \}
  .
\end{equation*}
By \prettyref{thm:theo_lambda_tau} we have 
$\argmin (\Phi(\cdot) + \lambda |\cdot |) = \{\tau\}$
exactly for $\lambda = g(\tau)$. An explicit formula for 
$g(\tau)$ is obtained by applying Fermat's rule: 
$
  0 
  \in 
  \partial(\Phi(\cdot ) + g(\tau)|\cdot|)|_\tau
  =
  (\{\Phi'(\cdot)\} + g(\tau) \partial |\cdot |)|_\tau
  =
  \{\Phi'(\tau) + g(\tau)\}
$; 
by rearranging we get
\[
  g(\tau)
  = 
  -\Phi'(\tau)
  =
  \left.
  \begin{cases}
    2(4 - \tau)  & \text{ for } 0 < \tau < 2 \\
    4		 & \text{ for } \tau = 2 \\
    4(3 - \tau)  & \text{ for } 2 < \tau < x_0
  \end{cases} 
  \right\rbrace
\]
Obviously $g$ is neither differentiable nor convex.
\end{remark}

  \appendix  
\clearpage   
\numberwithin{theorem}{chapter}

\chapter{Supplementary Linear Algebra and Analysis} \label{chap:LinA}


\begin{lemma} \label{lem:criterion_for_linearity}
  Let $V$ and $W$ be vector spaces over $\R$.
  A mapping $\varphi: V \rightarrow W$ is linear if
  \begin{enumerate}
    \item 
      $\varphi(v+v') = \varphi(v) + \varphi(v')$ for all $v, v' \in V$,
    \item \label{enu:shorted_homogenity}
      $\varphi(t v) = t \varphi(v)$ for all $v \in V$ and all $t \in [0,1]$.
  \end{enumerate}
\end{lemma}

Note that only $t \in [0,1]$ is required.

\begin{proof}[Proof of \prettyref{lem:criterion_for_linearity}]
  By assumption $\varphi$ is additive. Moreover 
  $\varphi$ is also homogeneous:
  Let $v \in V$ be arbitrarily chosen.
  In case $t \in [0,1]$ we have $\varphi(t v) = t \varphi (v)$ by 
  assumption \upref{enu:shorted_homogenity}.
  In case 
  $t \in (1, +\infty)$
  application of the same assumption 
  to $t' \defeq \tfrac{1}{t} \in [0,1]$ and $v' \defeq tv \in V$ yields
  $
    \varphi(tv)
    =
    t t' \varphi(v')
    =
    t \varphi(t'v')
    =
    t\varphi(v)
  $.
  Using $\varphi(\tilde t \tilde v) = \tilde  t \varphi( \tilde v)$ 
  for $\tilde v \in V$, $\tilde t \in (0, + \infty)$ and 
  $
    \varphi(-v) + \varphi(v) 
    = 
    \varphi(-v + v) 
    =
    \varphi(\zerovec)
    = 
    \varphi (0\cdot\zerovec)
    = 
    0\varphi(\zerovec) 
    = 
    \zerovec
  $,
  i.e. $\varphi(-v) = - \varphi(v)$
  we finally obtain also in case $t \in (-\infty, 0)$ the equation 
  $\varphi(tv) = \varphi(-t (-v))= -t \varphi(-v) = t \varphi(v)$.
\end{proof}

The following Lemma provides a useful inequality, which reflects 
the fact that a direct decomposition $X = X_1 \oplus X_2$ 
of an Euclidean vector space $X$ of finite dimension can only consist
of subspaces $X_1$ and $X_2$ which form a strict positive angle 
$\alpha \in (0, \frac12 \pi]$,
analytically described by 
\begin{gather*}
  -1 
  < 
  \cos(\pi-\alpha) 
  = 
  \inf_{h_1 \in X_1 \setminus \{\zerovec\}, h_2 \in X_2 \setminus \{\zerovec\}}
  \frac{\langle h_1, h_2 \rangle}{\|h_1\|_2 \|h_2\|_2}.
\end{gather*}
The equivalent inequality
$
  \inf_{h_1 \in X_1\cap \mathbb{S}_1, h_2 \in X_2\cap \mathbb{S}_1}
  \langle h_1, h_2 \rangle > -1
$
follows indeed easily from the inequality of the next theorem for 
$\|\cdot \| = \| \cdot \|_2$, see Detail
\ref{det:appendix_pr_of_relation_between_inequalities}.
Note however that the above 
inequality
and the inequality in Lemma 
\ref{lem:same_blow_up_factor_gets_all_attachted_affines_strangers_out_of_ball}
are in general only true in finite-dimensional spaces.
These inequalities do
not directly transfer to infinite dimensional inner 
product spaces as the example
$
  X 
  = 
  {\rm span} \{e_1\} \oplus 
  {\rm span}\{e_1 +  \frac12 e_2, e_1 + \frac13 e_3, e_1 + \frac14 e_4, \dots\}
  \subseteq  
  l_2(\R) 
$
shows; recall here that the notation $X = X_1 \oplus X_2$ 
still shall mean only an inner 
decomposition in the sense of pure vector spaces without demanding
additional properties like (topological) 
closeness on $X_1$ and $X_2$.

\begin{lemma} 
\label{lem:same_blow_up_factor_gets_all_attachted_affines_strangers_out_of_ball}
  Let $X_1, X_2$ be subspaces of $\R^n$ with $X_1 \cap X_2 = \{\zerovec\}$ and 
  let $\|\cdot\|$ be any norm on $\R^n$. Then there is a constant $C\geq 1$
  such that 
  \begin{gather*}
    \|h_1\| \leq C \|h_1 + h_2\|
  \end{gather*}
  for all $h_1 \in X_1$ and $h_2 \in X_2$.
\end{lemma}

\begin{proof}
  It suffices to find a constant $C > 0$ for which the claimed  
  inequality holds true, since enlarging the constant then clearly
  keeps the inequality true.
  In case $h_1 = \zerovec$ the inequality is fulfilled for any $C > 0$.
  Therefore we may assume without loss of generality that 
  $h_1 \in X_1 \cap \mathbb{S}_1$;
  note therefore that the following statements are equivalent:
  \begin{align*}
    \exists C > 0 \; 
    \forall h_1 \in X_1 \setminus \{\zerovec\} \; 
    \forall h_2 \in X_2:
      \|h_1\| \leq C \|h_1 + h_2\|,
    \\
    \exists C > 0 \; 
    \forall x_1 \in X_1 \cap \mathbb{S}_1 \; 
    \forall x_2 \in X_2:
      \|x_1\| \leq C \|x_1 + x_2\|.
  \end{align*}
  So we need only to find a constant $C > 0$ such that 
  $\frac{1}{C} \leq \|h_1 + h_2\|$ 
  for all $h_1 \in X_1 \cap \mathbb{S}$ and all $h_2 \in X_2$.
  We have 
  \begin{gather*}
    \|h_1 + h_2\| \geq |\|h_2\| - \|h_1\|| = \|h_2\| - 1 \geq 2 
  \end{gather*}
  for $\|h_2\| \geq 3$, on the one hand.
  The mapping 
  $
    \varphi : 
    (X_1 \cap \mathbb{S}_1) \times (X_2 \cap \overline{\mathbb{B}}_{3})
    \rightarrow
    \R
  $,
  given by $\varphi(h_1, h_2) \defeq \|h_1 + h_2\|$, is continuous on its 
  compact domain of definition. Therefore $\varphi$ attains its minimum 
  $\check c = \varphi(\check h_1, \check h_2)$ for some 
  $\check h_1 \in \mathbb{S} \cap X_1$,
  $\check h_2 \in X_2 \cap \overline{\mathbb{B}}_{3}$.
  Combining $X_1 \cap X_2 = \{\zerovec\}$ and $\|\check h_1\| \not = 0$ ensures
  $\check h_2 \not = -\check h_1$, so that
  $\check c = \|\check h_1 + \check h_2\| > 0$ and hence 
  $\|h_1 + h_2\| \geq \check c > 0$ for all
  $h_1 \in X_1 \cap \mathbb{S} $ and
  $h_2 \in X_2 \cap \overline{\mathbb{B}}_{3}$ on the other hand.
  In total we have 
  $\|h_1 + h_2\| \geq \min \{2, \check c\} > 0 $ for 
  $h_1 \in X_1 \cap \mathbb{S}$ and $h_2 \in X_2$.
  Setting $C \defeq \frac{1}{\min \{2, \check c\}} > 0$
  we are done.
\end{proof}

Next we introduce the notion of an affine mapping via four equivalent 
conditions;
note therein that condition 
\upref{enu:def_affine_mapping_via_convex_combination}
can also be demanded for a function $f$ 
which is defined only on a nonempty convex set.
For condition 
\upref{enu:def_affine_mapping_via_affine_combination} and
\upref{enu:def_affine_mapping_via_linear_mapping_without_fixing_origin}
c.f. also \cite[p. 7]{Rockafellar1970}\DissVersionForMeOrHareBrainedOfficialVersion
  { and \cite{wiki:affinetransform_2012_07_25__08_43}.}{.}

\begin{definition}
  Let $A$,$A'$ be nonempty affine subspaces of $\R^n$ and 
  $U$, $U' \subseteq \R^n$ the corresponding vector subspaces 
  that are parallel to $A$ and $A'$, respectively.
  A mapping $f: A \rightarrow A'$ is called {\bf affine}, iff one 
  of the following equivalent conditions is fulfilled:
  \begin{enumerate}
    \item \label{enu:def_affine_mapping_via_convex_combination}
      $f(a_1 + t(a_2 - a_1)) = f(a_1) + t(f(a_2) - f(a_1))$
      \quad for all $a_1, a_2 \in A$ and all $t \in [0,1]$,
    \item \label{enu:def_affine_mapping_via_affine_combination} 
      $f(a_1 + t(a_2 - a_1)) = f(a_1) + t(f(a_2) - f(a_1))$
      \quad for all $a_1, a_2 \in A$ and all $t \in \R$,
    \item \label{enu:def_affine_mapping_via_linear_mapping_without_fixing_origin}
      There is a linear mapping $\varphi:U \rightarrow U'$ such that 
      \\
      $f(a_2) - f(a_1) = \varphi(a_2 -a_1)$ 
      \quad for all $a_1, a_2 \in A$,
    \item \label{enu:def_affine_mapping_via_linear_mapping_and_fixed_origin}
      There is a linear mapping $\varphi:U \rightarrow U'$ and a point 
      $a_0 \in A$ such that 
      \\
      $f(a) = f(a_0) + \varphi(a - a_0)$ 
      \quad for all $a \in A$.
  \end{enumerate}
\end{definition}

\begin{remark}
  The four conditions are really equivalent:
  \\
  ``\upref{enu:def_affine_mapping_via_linear_mapping_and_fixed_origin}
    $\Rightarrow$
    \upref{enu:def_affine_mapping_via_linear_mapping_without_fixing_origin}'':
  Let $\varphi: U \rightarrow U'$ be linear and $a_0 \in A$ such that
  $f(a) = f(a_0) + \varphi(a-a_0)$ for all $a \in A$. Then we get 
  \begin{align*}
    f(a_2) - f(a_1) 
    &= 
    f(a_2) - f(a_0) - [f(a_1) - f(a_0)]
  \\&= 
    \varphi(a_2 - a_0) - \varphi (a_1 -a_0)
  \\&= 
    \varphi(a_2 - a_0 - [a_1 -a_0] )
  \\&= 
    \varphi(a_2 - a_1)
  \end{align*}
  for all $a_1, a_2 \in A$.
  %
  \\
  ``\upref{enu:def_affine_mapping_via_linear_mapping_without_fixing_origin}
    $\Rightarrow$
    \upref{enu:def_affine_mapping_via_affine_combination}'':
  Using \upref{enu:def_affine_mapping_via_linear_mapping_without_fixing_origin}
  for $a_1' = a_1 \in A$ and $a_2' = a_1 + t(a_2 -a_1) \in A$
  we get
  \begin{align*}
    f(a_1 + t(a_2 -a_1)) - f(a_1) 
    &=
    \varphi(a_2' - a_1')
    =
    \varphi(t(a_2 -a_1))
  \\&=
    t\varphi(a_2 - a_1)
    =
    t(f(a_2) - f(a_1))
  \end{align*}
  for all $a_1,a_2 \in A$ and all $t \in \R$, so that 
  \upref{enu:def_affine_mapping_via_affine_combination}
  holds true.
  \\
  ``\upref{enu:def_affine_mapping_via_affine_combination}
    $\Rightarrow$
    \upref{enu:def_affine_mapping_via_convex_combination}''
    is obviously true.
  \\
  ``\upref{enu:def_affine_mapping_via_convex_combination}
    $\Rightarrow$
    \upref{enu:def_affine_mapping_via_linear_mapping_and_fixed_origin}'':
  Choose any $a_0 \in A$ and set 
  $\varphi(u) \defeq f(a_0 + u) - f(a_0)$ for $u \in U$.
  Then clearly $\varphi: U \rightarrow U'$ and 
  $f(a) = f(a_0) + \varphi(a-a_0)$ for all $a \in A = a_0 \oplus  U$.
  It remains to show that $\varphi$ is linear. By Lemma 
  \ref{lem:criterion_for_linearity} it suffices to show that $\varphi$ 
  is additive and fulfills $\varphi(t u) = t \varphi(u)$ for all $u \in U$
  and all $t \in [0,1]$. In order to prove the latter let $u \in U$
  be arbitrarily chosen. 
  Using \upref{enu:def_affine_mapping_via_convex_combination} 
  with  $a_1 = a_0 \in A$ and $a_2 = a_0 + u \in a_0 + U = A$
  we obtain indeed 
  \begin{align*}
    \varphi(tu)
    &=
    f(a_0 + tu) - f(a_0)
  \\&=
    f(a_0 + t(a_2 -a_0)) - f(a_0)
  \\& =
    f(a_0) + t[f(a_2) -f(a_0)] - f(a_0)
  \\&=
    t[f(a_0 + u) - f(a_0)]
  \\&=
    t \varphi(u)
  \end{align*}
  for all $t \in [0,1]$.
  In order to prove the additivity of we note that choosing $t = \frac{1}{2}$
  in \upref{enu:def_affine_mapping_via_convex_combination} gives the equation
  $f(\frac{1}{2}(a_1+a_2)) = \frac{1}{2}[f(a_1) + f(a_2)]$
  for all $a_1, a_2 \in A$.
  For arbitrarily chosen $u,u' \in U$ we obtain therefrom and by 
  $\frac{1}{2} \in [0,1]$ the identity
  \begin{align*}
    \varphi(u + u') 
    &= 
    f(a_0 + u + u') - f(a_0)
    =
    f\left(\tfrac{1}{2}([a_0 + 2u] + [a_0 + 2u'])\right)  - f(a_0)
  \\&=
    \tfrac{1}{2}f(a_0 + 2u) + \tfrac{1}{2}f(a_0 + 2 u') - f(a_0)
  \\&=
    \tfrac{1}{2}\varphi(2u) + \tfrac{1}{2}\varphi(2u')
     =
    \varphi(\tfrac{1}{2} 2u) + \varphi(\tfrac{1}{2} 2u')
     =
    \varphi(u) + \varphi(u').
  \end{align*}
  So $\varphi$ is additive as well.
\end{remark}

\chapter{Supplementary Convex Analysis} \label{chap:convexAna}

\begin{lemma} \label{lem:convexity_condition_inside_and_outside}
  Let $F:\R^n \rightarrow \R \cup \{+\infty\}$ be a convex function.
  \begin{enumerate}
    \item \label{enu:convexity_condition_in_between_and_outside}
      For any two points $x,y \in \dom F$ and $\lambda \in \R$ we have
      \begin{align}
      \label{ineq:convexity_condition_between_two_points_of_dom}
        F((1-\lambda) x + \lambda y) &\leq (1-\lambda) F(x) + \lambda F(y)
	\quad \text{ if } \lambda \in [0,1],
      \\
      \label{ineq:convexity_condition_outside_of_two_points_of_dom}
	 F((1-\lambda) x + \lambda y) &\geq (1-\lambda) F(x) + \lambda F(y)
	\quad \text{ if } \lambda \in \R \setminus (0,1).
      \end{align}
    \item \label{enu:convex_function_constant_on_line_segment_spaned_by_3_points}
      If there are three different collinear points $a,b,c \in \dom F$ 
      which yield the same value $F(a) = F(b) = F(c)$ then 
      $F$ is constant on the line segment ${\rm co}(\{a,b,c\})$
      spanned by these three points.
  \end{enumerate}
\end{lemma}

\begin{proof}
  \upref{enu:convexity_condition_in_between_and_outside}
  The inequality \eqref{ineq:convexity_condition_between_two_points_of_dom}
  is just the inequality from the definition of convexity.
  In order to prove 
  \eqref{ineq:convexity_condition_outside_of_two_points_of_dom}
  we set 
  \begin{equation} \label{eq:def_of_point_that_runs_on_line_spanned_by_two_points}
    z_\lambda \defeq x + \lambda(y-x) = (1-\lambda)x + \lambda y
  \end{equation}
  for $\lambda \in \R \setminus (0,1)$.
  If $F(z_\lambda) = +\infty$ we clearly have
  $F(z_\lambda) = +\infty \geq (1-\lambda) F(x) + \lambda F(y)$. 
  Assume now 
  $F(z_\lambda) < + \infty$, i.e. $z_\lambda \in \dom F$.
  In case $\lambda \geq 1$ rewriting equation 
  \eqref{eq:def_of_point_that_runs_on_line_spanned_by_two_points} 
  yields the convex combination
  $
    y 
    = 
    -\frac{1-\lambda}{\lambda}x + \frac{1}{\lambda}z_\lambda
    =
    (1 - \frac{1}{\lambda})x + \frac{1}{\lambda}z_\lambda
  $
  and hence by the convexity of $F$ the inequality
  $
    F(y) 
    \leq 
    (1-\frac{1}{\lambda})F(x) +  \frac{1}{\lambda} F(z_\lambda)
  $.
  Since only finite values occur this can be rewritten as
  $
    \frac{1}{\lambda} F(z_\lambda) \geq (\frac{1}{\lambda}-1)F(x) + F(y)
  $
  which is equivalent to the claimed inequality in
  \eqref{ineq:convexity_condition_outside_of_two_points_of_dom}, since
  $\lambda \geq 1 > 0$.
  In case $\lambda \leq 0$ we can similar write $x$ as convex combination
  $
    x 
    = 
    - \frac{\lambda}{1-\lambda}y + \frac{1}{1- \lambda}  z_\lambda
    =
    (1- \frac{1}{1-\lambda})y + \frac{1}{1- \lambda}  z_\lambda
  $
  so that the convexity of $F$ yields the inequality
  $F(x) \leq (1-\frac{1}{1-\lambda})F(y) + \frac{1}{1-\lambda} F(z_\lambda)$.
  Since only finite values occur this can be rewritten as
  $\frac{1}{1-\lambda}F(z_\lambda) \geq F(x)+ \frac{\lambda}{1-\lambda}F(y)$
  which is equivalent to the claimed inequality in
  \eqref{ineq:convexity_condition_outside_of_two_points_of_dom}, since
  $1-\lambda \geq 1 > 0$.
  \\
  \upref{enu:convex_function_constant_on_line_segment_spaned_by_3_points}  
  Without loss of generality we may assume that $b$ is the 
  point ``between'' the endpoints $a$ and $c$, so that 
  ${\rm co}\{a,b,c\} \eqdef l(a,b)$ is the line segment between $a$ and $c$.
  Set $v \defeq F(a) = F(b) = F(c) \in \R$.
  We have to show that any $z \in l(a,c)$ also fulfills
  $F(z) = v$. In case $z \in l(a,b)$,  we can write $z$ as convex combination
  $z = (1-\lambda)a + \lambda b$ with some $\lambda \in [0,1]$
  and as affine combination
  $z = (1-\lambda')b + \lambda' c$ with some 
  $\lambda' \in \R\setminus(0,1)$, respectively. So 
  inequalities \eqref{ineq:convexity_condition_between_two_points_of_dom}
  and \eqref{ineq:convexity_condition_outside_of_two_points_of_dom}
  give $F(z) \leq (1-\lambda)F(a) + \lambda F(b) = v$ and 
  $F(z) \geq (1-\lambda')F(b) + \lambda' F(c) = v$, respectively.
  All in all we thus have $F(z) = v$.
  In case $z \in l(b,c) = l(c,b)$ we get the assertion analogously
  by interchanging the roles of $a$ and $c$.
\end{proof}

Of course norms are not strictly convex. 
However we have the following lemma.

\begin{lemma} \label{lem:euclidean_norm_strictly_convex_on_every_straight_line_through_origin}
 The Euclidean norm $\|\cdot\|_{2}: \R^n \rightarrow \R$ is strictly convex on every straight line,
 which does not contain the origin $\zerovec$.
\end{lemma}

\begin{proof}
  Let $l$ be a straight line in $\R^n$ with $\zerovec \not \in l$ and let $x, y \in l$ be 
  two distinct points. The strict Cauchy-Schwarz Inequality
  $\langle x,y \rangle < \|x\|_{2} \|y\|_{2}$ holds true for $x$ and $y$, since these vectors are 
  linearly independent. For all $\lambda \in (0,1)$ we hence get
  \begin{align*}
    \|\lambda x + (1- \lambda)y \|_{2}^2
    &=
    \|\lambda x\|_{2}^2 + \|(1-\lambda)y\|_{2}^2 
      +
      2\lambda(1-\lambda) \langle x,y \rangle
    \\
    &<
    \|\lambda x\|_{2}^2 + \|(1-\lambda)y\|_{2}^2 
      +
      2\lambda(1-\lambda) \|x\|_{2} \|y\|_{2}
    \\
    &=
    (\|\lambda x\|_{2} + \|(1-\lambda)y\|_{2})^2
  \end{align*}
  and therewith the needed
  $\|\lambda x + (1-\lambda)y\|_{2} < \lambda \|x\|_{2} + (1-\lambda) \|y\|_{2}$.
\end{proof}

The following Theorem is obtained from 
\cite[p. 52]{Rockafellar1970} and \cite[Theorem 7.4]{Rockafellar1970}.
\begin{theorem} \label{thm:closure_of_proper_convex_function}
  Let $f: \R^n \rightarrow \R \cup \{+\infty\}$ be a 
  proper, convex function. Its closure ${\rm cl}f$ fulfills
  \begin{enumerate}
    \item \label{enu:closure_of_function_as_liminf}
      ${\rm cl}f(x_0) = \liminf_{x \rightarrow x_0}f(x)$
      for every $x_0 \in \R^n$.
    \item \label{enu:closure_of_proper_convex_function_is_almost_original}
      ${\rm cl}f$ is a proper convex and lower 
      semicontinuous function which agrees with $f$
      except perhaps at relative boundary points of 
      $\dom f$.
  \end{enumerate}
\end{theorem}

For the proof of the following theorem see 
\cite[Corollary 7.5.1]{Rockafellar1970}
\begin{theorem} 
\label{thm:limit_behaviour_of_proper_convex_lsc_functions}
  For a function $F \in \Gamma_0(\R^n)$ one has 
  \begin{gather*}
    F(x^*)
    = 
    \lim_{\lambda \uparrow 1}  F( (1- \lambda)a + \lambda x^* )
  \end{gather*}
  for every $a \in \dom F$ and every $x^* \in \R^n$.
\end{theorem}

\iftime{Noch scharf fassen, was hier genau
mit differentiable gemeint sein soll; 
Funktion soll halt linearisierbar sein}
For the proof of the following theorem cf.
\cite[Theorem 26.1]{Rockafellar1970} after identifying 
${\rm aff}(\dom F)$ with some $\R^m$.
\begin{theorem} \label{thm:subgradient_of_essentially_smooth_and_Gamma_0_function}
  Let $F \in \Gamma_0(\R^n)$ be essentially smooth on 
  ${\rm aff}(\dom F) \eqdef A$. Then $\partial (F|_A)(x)$ 
  contains at most one subgradient for every $x \in \R^n$.
  In case $x \not \in {\rm ri}(\dom F)$ we have
  $\partial (F|_A)(x) = \emptyset$ 
  while in case $x \in {\rm ri}(\dom F)$ there is
  exactly one subgradient in $\partial (F|_A)(x)$. 
  In particular the function
  $F|_A$ is subdifferentiable in every 
  $x \in {\rm ri}(\dom F)$.  
\end{theorem}

\begin{lemma} 
\label{lem:argmin_of_sum_contained_in_ri_of_dom_of_ess_smooth_summand}
  Let $F: \R^n \rightarrow \R \cup \{+\infty\}$ be a proper and convex 
  function, which is essentially smooth on 
  $A \defeq {\rm aff}(\dom F)$. Then 
  \begin{enumerate}
    \item 
    \label{enu:argmin_of_sum_contained_in_ri_of_dom_of_ess_smooth_summand}
      $\argmin_{x\in \R^n}(F(x)+G(x)) \subseteq {\rm ri}(\dom F)$ for 
      every convex function $G: \R^n \rightarrow \R \cup \{+\infty\}$
      with ${\rm ri}(\dom F) \cap {\rm ri}(\dom G) \not = \emptyset$.
    \item \label{enu:argmin_contained_in_ri_of_dom}
      $\argmin_{x\in \R^n}F(x) \subseteq {\rm ri}(\dom F)$ and 
      $F|_A$ is differentiable in every $\hat x \in \argmin F$.
  \end{enumerate}

\end{lemma}

\begin{proof}
  \upref{enu:argmin_of_sum_contained_in_ri_of_dom_of_ess_smooth_summand}
    Let all assumptions be fulfilled. By Theorem 
    \ref{thm:closure_of_proper_convex_function} we may further assume 
    without loss of generality that $F$ is closed, 
    i.e. lower semicontinuous, since replacing $F$ by ${\rm cl}F$
    would neither affect the assumptions nor the assertions of the 
    theorem. Let 
    $\hat x \in \argmin(F+G)$.
    Restricting $F$ and $G$ to $A = {\rm aff}(\dom F)$ by setting 
    $f \defeq F|_A$ and $g \defeq G|_A$ we still have 
    $\hat x \in \argmin (f+g)$.
    Using Theorem 
    \ref{thm:restricting_set_operations_to_an_affine_subset}
    we see that still 
    \begin{gather*}
      {\rm ri}(\dom f) \cap {\rm ri}(\dom g)
      =
      {\rm ri}(\dom F) \cap {\rm ri}(\dom G \cap A)
      =
      {\rm ri}(\dom F) \cap {\rm ri}(\dom G) \cap A
      \\
      =
      ({\rm ri}(\dom F) \cap A )\cap {\rm ri}(\dom G) 
      =
      {\rm ri}(\dom F) \cap {\rm ri}(\dom G)
      \not = 
      \emptyset.
    \end{gather*}
    Using the therewith applicable Sum rule and Fermat's rule we 
    obtain
    \begin{gather*}
      \zerovec 
      \in 
      \partial (f+g)(\hat x)
      =
      \partial f(\hat x) + \partial g(\hat x).
    \end{gather*}
    In particular $\partial f(\hat x) \not = \emptyset$ so that 
    the essentially smoothness of $f$ gives 
    $\hat x \in {\rm int}_A(\dom f) = {\rm ri}(\dom F)$ 
    by Theorem
    \ref{thm:subgradient_of_essentially_smooth_and_Gamma_0_function}.
\\
  \upref{enu:argmin_contained_in_ri_of_dom}
    The inclusion
    follows from the just proven by choosing $G \equiv 0$ since then 
    $
      {\rm ri}(\dom F) \cap {\rm ri}(\dom G) 
      = 
      {\rm ri}(\dom F)
      \not = 
      \emptyset
    $
    by Theorem 
    \ref{thm:nonemtpy_convex_set_has_nonempty_ri}.
    From the inclusion we now also get the differentiability assertion 
    by applying Theorem
    \ref{thm:subgradient_of_essentially_smooth_and_Gamma_0_function}.
\end{proof}

The proofs of the following two theorems 
can be found in \cite[p. 45]{Rockafellar1970}.
\begin{theorem} \label{thm:raypoint_from_ri__leaving_closure__never_comes_back}
  Let $C$ be a convex set in $\R^n$. Let $\mathring x \in {\rm ri}(C)$ and 
  $x\in \overline{C}$. Then $(1- \lambda)\mathring x + \lambda x$ belongs to ${\rm ri}(C)$
  (and hence in particular to C) for $0 \leq \lambda < 1$.
\end{theorem}

\begin{theorem} \label{thm:nonemtpy_convex_set_has_nonempty_ri}
  Let $C$ be any convex set in $\R^n$. Then $\overline{C}$ and
  ${\rm ri}(C)$ are convex sets in $\R^n$, having the same affine hull,
  and hence the same dimension, as $C$.
  In particular ${\rm ri}(C) \not = \emptyset$ if $C \not = \emptyset$.
\end{theorem}

The following theorem is obtained from
\cite[Theorem 7.6]{Rockafellar1970} and \cite[Theorem 6.2]{Rockafellar1970}.

\begin{theorem} \label{thm:ri_of_levelset}
  For a proper, convex function 
  $F: \R^n \rightarrow \R \cup \{+ \infty\}$
  and $\tau \in ( \inf F, + \infty)$ we have
  \begin{align*}
    {\rm ri} ({\rm lev}_{\tau} F)
    =
    {\rm ri} ({\rm lev}_{<\tau} F)
    =
    {\rm lev}_{<\tau} F \cap {\rm ri} (\dom F).
  \end{align*}
  Furthermore all these sets have the same dimension as $\dom F$. 
\end{theorem}

\begin{theorem} \label{thm:restricting_set_operations_to_an_affine_subset}
  Let $C$ be a convex set in $\R^n$, and let $A$ be an affine set
  in $\R^n$ which contains
  a point of ${\rm ri} (C)$. Then 
  \begin{align}
    \label{eq:ri_restricted_to_affine_subset}
    {\rm ri}(A \cap C) &= A \cap {\rm ri} (C),
    \\
    \label{eq:closure_restricted_to_affine_subset}
    \overline{A \cap C} &= A \cap \overline{C},
    \\
    \label{eq:rb_restricted_to_affine_subset}
    {\rm rb}(A \cap C) &= A \cap {\rm rb} (C),
    \\
    \label{eq:aff_restricted_to_affine_subset}
    {\rm aff}(A \cap C) &= A \cap {\rm aff}(C).   
  \end{align}
\end{theorem}

\begin{proof}
  For the proof of the first and the second equality 
  see \cite[Corollary 6.5.1]{Rockafellar1970}.
  With these statements we now also get 
  \begin{align*}
    {\rm rb}(A \cap C) 
    &= 
    \overline{A \cap C} \setminus {\rm ri}(A \cap C)
    =
    (A \cap \overline{C}) \setminus (A \cap {\rm ri}(C) )
    = 
    A \cap (\overline{C} \setminus {\rm ri}(C))
    =
    A \cap {\rm rb}(C).
  \end{align*}
  For the proof of the remaining forth statement let
  $a \in A \cap {\rm ri }(C)$.
  Since the truth value of the assertion stays unchanged when translating 
  the coordinate system we may assume $a = \zerovec$, so that 
  ${\rm aff}(A) = {\rm span}(A)$,
  ${\rm aff}(C) = {\rm span}(C)$ and
  ${\rm aff}(A \cap C) = {\rm span}(A \cap C)$.
  Due to 
  $
    {\rm span}(A \cap C) 
    = 
    {\rm span}( A \cap ({\rm span}(C) \cap C) )
    =
    {\rm span}( (A \cap {\rm span}(C) ) \cap C) 
  $
  and
  $ A \cap {\rm span}(C) = (A \cap  {\rm span}(C) ) \cap {\rm span} (C) $
  we may restrict us to subspaces $A \subseteq {\rm span}(C)$, so that we can 
  identify ${\rm span}(C)$ with $\R^m$ where $m = \dim({\rm span}(C))$.
  Choose $\varepsilon > 0$ so small that 
  $\mathbb{B}_{\varepsilon} \subseteq C$.
  Then 
  $
    {\rm span}(A) 
    = 
    {\rm span}(A \cap \mathbb{B}_{\varepsilon})
    \subseteq 
    {\rm span}(A \cap C)
    \subseteq
    {\rm span}(A) \cap {\rm span}(C)
    =
    {\rm span}(A) \cap \R^m
    =
    {\rm span}(A)
  $
  so that we have in particular
  $
    {\rm span}(A \cap C)
    =
    {\rm span}(A) \cap {\rm span}(C)
    =
    A \cap {\rm span}(C)
  $.
\end{proof}

\begin{theorem} 
\label{thm:directness_of_sum_of_convex_sets_eq_carries_over_to_affine_hulls}
  For convex subsets $C_1$ and $C_2$ of $\R^n$ the following are equivalent:
  \begin{enumerate}
    \item \label{enu:unique_decomposition_of_sum_of_convex_sets}
      $C_1 +  C_2 =  C_1 \oplus C_2$,
\DissVersionForMeOrHareBrainedOfficialVersion      
      {i.e. each $c \in C \defeq C_1 + C_2$ has a unique decomposition $c = c_1 + c_2$ into components  $c_1 \in C_1$, $c_2 \in C_2$.}
{}
    \item \label{enu:unique_decomposition_of_sum_of_affine_hull_of_convex_sets}
      ${\rm aff} (C_1) + {\rm aff} (C_2) = {\rm aff} (C_1) \oplus {\rm aff} (C_2)$.
  \end{enumerate}
\end{theorem}

\begin{proof}
  Assume without loss of generality that $C_1$ and $C_2$ are not empty.
  Translating $C_1$ or $C_2$ does neither change the truth value of the 
  statement $C_1 + C_2 = C_1 \oplus C_2$ nor the truth value of the statement
  ${\rm aff}(C_1) + {\rm aff}(C_2) = {\rm aff}(C_1) \oplus  {\rm aff}(C_2)$.
  Without loss of generality we may therefore assume 
  $\zerovec \in {\rm ri}(C_1)$ and $\zerovec \in {\rm ri}(C_2)$.
  \\[0.5ex]
  Clearly \upref{enu:unique_decomposition_of_sum_of_affine_hull_of_convex_sets} 
  implies 
  \upref{enu:unique_decomposition_of_sum_of_convex_sets},
  since $C_1 \subseteq {\rm aff }(C_1)$ and $C_2 \subseteq {\rm aff}(C_2)$.
  We show the remaining direction 
  $
    \upref{enu:unique_decomposition_of_sum_of_convex_sets}
    \Rightarrow
    \upref{enu:unique_decomposition_of_sum_of_affine_hull_of_convex_sets}
  $
  by proving its contrapositive; assume that the sum
  ${\rm aff} (C_1) + {\rm aff} (C_2)$ is not direct,
  so that there are distinct $a_1, a_1' \in {\rm aff}(C_1)$ and 
  distinct $a_2, a_2' \in {\rm aff}(C_2)$ such that $a_1 + a_2 = a_1' + a_2'$.
  Let $a$ be any of the four points and let $C$ be the corresponding
  set $C_1$ or $C_2$. We can find a $\lambda_a >0$
  such that $\lambda_a a\in C$; indeed, by $\zerovec \in {\rm ri}(C)$
  there is an $\varepsilon > 0$ such that 
  $\mathbb{B}_\varepsilon \cap {\rm aff}(C) \subseteq C$.
  Hence and since ${\rm aff}(C)$ is an affine set we get
  $
    \lambda_a a
    = 
    \lambda_a a + (1-\lambda)\zerovec 
    \in 
    {\rm aff}(C) \cap \mathbb{B}_\varepsilon
    \subseteq C
  $ 
  for $\lambda_a$ chosen sufficiently small.
  For sufficiently small chosen $\lambda >0$ the four points
  $c_1 \defeq \lambda a_1$, $c_1' \defeq \lambda a_1'$ and 
  $c_2 \defeq \lambda a_2$, $c_2' \defeq \lambda a_2'$ belong 
  hence to $C_1$ and $C_2$, respectively, and 
  fulfill still $c_1 + c_2 = c_1' + c_2'$ under preservation of
  the distinctions $c_1 \not = c_1'$ and $c_2 \not = c_2'$.
  In particular the sum $C_1 + C_2$ is also not direct.
\end{proof}

\begin{remark}
  The condition that $C_1$ and $C_2$ are convex is essential to guarantee
  the implication 
  $
   C_1 +  C_2 =  C_1 \oplus C_2
    \Rightarrow
   {\rm aff} (C_1) + {\rm aff} (C_2) = {\rm aff} (C_1) \oplus {\rm aff} (C_2)
  $
  as the following example shows:
  Consider the sum of the upper circle line 
  $C_1 \defeq \{(\cos(t), \sin(t)): t \in [0, \pi]\}$
  with the line 
  $C_2\defeq \{(0,\lambda) \in \R^2: \lambda \in \R\}$.
  We have 
  $C_1 + C_2 = [-1,1] \times \R = C_1 \oplus C_2  $.
  However ${\rm aff}(C_1) = \R^2$ and ${\rm aff}(C_2) = C_2$, so that 
  the sum ${\rm aff}(C_1) + {\rm aff}(C_2) $
  is clearly not direct.

\end{remark}

\begin{lemma} \label{lem:addition_restricted_is_homeo}
  Assume that two nonempty convex sets $C_1,C_2 \subseteq \R^n$
  give a direct sum $C_1 \oplus C_2$.
  Restricting the vector addition $+: \R^n \times \R^n \rightarrow \R^n$
  to $C_1 \times C_2$ gives then a homeomorphism between the 
  product space $C_1 \times C_2$ and the 
  (topological) subspace $C_1 \oplus C_2$ of $\R^n$.
\end{lemma}

\begin{proof}
  By theorem 
  \ref{thm:directness_of_sum_of_convex_sets_eq_carries_over_to_affine_hulls}
  we know that the sum of 
  ${\rm aff}(\dom C_1) \eqdef A_1$ and ${\rm aff}(\dom C_2) \eqdef A_2$ is also 
  a direct one. Therefore it suffices to show that 
  $+|_{A_1 \times A_2}$ is a homeomorphism between $A_1 \times A_2$ 
  and $A_1 \oplus A_2$.
  Choose any $a^* = (a_1^*, a_2^*) \in A_1 \times A_2$ and set 
  $X_1 \defeq A_1 - a_1^*$ and $X_2 \defeq A_2 - a_2^*$. Noting that
  $+|_{A_1 \times A_2}$ is a homeomorphism between $A_1 \times A_2$ and
  $A_1 \oplus A_2$ if and only if $f\defeq+|_{X_1 \times X_2}$ is a 
  homeomorphism between $X_1 \times X_2$ and $X_1 \oplus X_2$
  it suffices to prove the latter.
  To this end note that $f$ is clearly continuous and surjective.
  Since $X_1 + X_2 = X_1 \oplus X_2$ we see that $f$ is also injective 
  and hence bijective. Finally 
  $f^{-1}: X_1 \oplus X_2 \rightarrow X_1 \times X_2$
  is continuous:
  Let $x = x_1 + x_2 \in X_1 \oplus X_2$ and let 
  $x^{(k)} = x_1^{(k)} + x_2^{(k)} \in X_1 \oplus X_2$ converge to $x$.
  We have to show that 
  $f^{-1}(x^{(k)}) = (x_1^{(k)}, x_2^{(k)})$ converges to 
  $f^{-1}(x) = (x_1, x_2)$.
  By Lemma 
  \ref{lem:same_blow_up_factor_gets_all_attachted_affines_strangers_out_of_ball}
  we know that there exists a constant $C > 0$ such that
  $\|h_1\| \leq C \|h_1 + h_2\|$ for all $h_1 \in X_1, h_2 \in X_2$.
  In particular we obtain
  \begin{gather*}
    \|x_1^{(k)} - x_1\| 
    \leq 
    C \|(x_1^{(k)} - x_1) + (x_2^{(k)} - x_2)\|
    =
    C \|x^{(k)} - x\|
    \rightarrow
    0
  \end{gather*}
  as $k \rightarrow + \infty$,
  so that $x_1^{(k)} \rightarrow x_1$ as $k \rightarrow +\infty$.
  By role reversal we obtain also $x_2^{(k)} \rightarrow x_2$ as 
  $k \rightarrow + \infty$, so that really 
  $(x_1^{(k)}, x_2^{(k)}) \rightarrow (x_1, x_2) $ as 
  $k \rightarrow + \infty$.
\end{proof}

The key in the previous proof was that the directness of the sum 
of two convex sets $C_1, C_2$ keep maintained when enlarging these 
sets to their affine hull.
This is, however, in general not true for a direct sum $C_1 \oplus C_2$,
where one of the summands $C_1, C_2$ is not convex.
In such cases it can happen that 
$+|_{C_1 \times C_2}: C_1 \times C_2 \rightarrow C_1 \oplus C_2$ 
is no longer a homeomorphism, as the following 
example illustrates:

\begin{example}
  Consider the non-convex set $C_1 \defeq \{0,1\}$ and the convex set 
  $C_2 \defeq [0,1)$.
  Although their sum $C_1 + C_2 = [0,2) = C_1 \oplus C_2$ is a direct one,
  the sum ${\rm aff}(C_1) + {\rm aff}(C_2) = \R + \R$ is not direct and 
  $+|_{C_1 \times C_2}$ is not a homeomorphism between $C_1 \times C_2$ 
  and $C_1 \oplus C_2$, since these topological spaces are not at all 
  homeomorphic: $C_1 \oplus C_2 = [0,2)$ is a connected space while 
  $C_1 \times C_2 = [\{0\} \times [0,1) ] \cup [\{1\} \times [0,1) ]$
  is not a connected space.
\end{example}

\begin{theorem} \label{thm:operations_that_interchange_with_direct_sum}
  Let $C$ and $A$ be a convex and an affine subset of $\R^n$, respectively,
  whose sum $C+A$ is direct. Then the following holds true:
  \begin{align}
    \label{eq:ri_of_direct_sum_with_affine_subset}
    {\rm ri}(A \oplus  C) &= A \oplus  {\rm ri} (C),
    \\
    \label{eq:closure_of__direct_sum_with_affine_subset}
    \overline{A \oplus  C} &= A \oplus  \overline{C},
    \\
    \label{eq:rb_of_direct_sum_with_affine_subset}
    {\rm rb}(A \oplus  C) &= A \oplus  {\rm rb} (C),
    \\
    \label{eq:aff_of_direct_sum_with_affine_subset}
    {\rm aff}(A \oplus  C) &= A \oplus {\rm aff}(C).    
  \end{align}
\end{theorem}

\begin{proof}
  Assume without loss of generality that $A$ and $C$ are not empty.
  Note first that the ``largest'' sum of the four right hand side sums,
  i.e. the sum $A + {\rm aff}(C)$ is a direct one by Theorem 
  \ref{thm:directness_of_sum_of_convex_sets_eq_carries_over_to_affine_hulls}.
  Hence the other three sums 
  $A + {\rm ri} (C)$,
  $A + \overline{C}$ and
  $A + {\rm rb} (C)$
  are direct all the more.
  Noting that the truth value of the statement 
  ${\rm aff}(A + C) = A + {\rm aff}(C)$
  does not change when translating $A$ or $C$ we may assume 
  $\zerovec \in A$ and $\zerovec \in C$ without loss of generality, so that
  in particular
  $A \subseteq A + C$ and $C \subseteq A + C$.
  We then get
  \begin{multline*}
    A + {\rm aff}(C) 
    = 
    {\rm aff}(A) + {\rm aff}(C)
    \subseteq
    {\rm aff}(A + C) + {\rm aff}(A + C)
    =
    {\rm span}(A + C) + {\rm span}(A + C)
    \\
    =
    {\rm span}(A + C)
    \subseteq
    {\rm span}( {\rm span}(A) + {\rm span}(C))
    =
    {\rm span}(A) + {\rm span}(C)
    =
    A + {\rm aff}(C)
  \end{multline*}
  and therewith 
  $A \oplus {\rm aff}(C) = {\rm span}(A \oplus C) = {\rm aff}(A \oplus C) $.

  Consider now the topological spaces $C_1 \defeq {\rm aff}(A) = A$ and
  $C_2 \defeq {\rm aff}(C)$ and their product space $C_1 \times C_2$,
  equipped with the product topology.
  We have 
  \begin{align*}
    {\rm int}_{C_1 \times C_2}(A \times C)
    &=
    {\rm int}_{C_1}(A) \times {\rm int}_{C_2}(C)
    =
    A \times {\rm int}_{C_2}(C),
  \\
    \overline{A \times C}^{C_1 \times C_2}
    &=
    \overline{A}^{C_1} \times \overline{C}^{C_2}
    =
    A \times \overline{C}^{C_2}
  \\
  \shortintertext{and}
    {\rm \partial}_{C_1 \times C_2} (A \times C)
    &=
    \overline{A \times C}^{C_1 \times C_2}
      \setminus
      {\rm int}_{C_1 \times C_2}(A \times C)
  \\
    &=
    \left(A \times \overline{C}^{C_2} \right)
      \setminus 
      \left(A \times {\rm int}_{C_2}(C) \right)
  \\
    &=
    A \times \left(\overline{C}^{C_2} \setminus  {\rm int}_{C_2}(C) \right)
  \\
    &=
    A \times {\rm \partial}_{C_2}(C).
  \end{align*}
  By means of the homeomorphism
  $+|_{C_1 \times C_2}: C_1 \times C_2 \rightarrow C_1 \oplus C_2$
  from lemma  \ref{lem:addition_restricted_is_homeo}
  these three equations can be translated to 
  \begin{align*}
    {\rm int}_{C_1 + C_2} (A + C)
    &=
    A + {\rm int}_{C_2}(C),
  \\
    \overline{A + C}^{C_1 + C_2}
    &=
    A + \overline{C}^{C_2}
  \\
  \shortintertext{and}
    {\rm \partial}_{C_1 + C_2} (A + C)
    &=
    A + {\rm \partial}_{C_2}(C),
  \end{align*}
  which gives the equations 
  \eqref{eq:ri_of_direct_sum_with_affine_subset},
  \eqref{eq:closure_of__direct_sum_with_affine_subset},
  \eqref{eq:rb_of_direct_sum_with_affine_subset}.
\end{proof}

The following theorem is a special case of 
\cite[Corollary 2.4.5]{ZA02} and an equation used in its proof. 
Cf. also \cite[Theorem 10.5]{RW04}.
Note that we need an \emph{orthogonal} decomposition
$\R^n = X_1 \oplus X_2 \dots \oplus  X_n$ in order to guarantee
$
  \langle x, x^* \rangle 
  = 
  \sum_{i=1}^n \langle x_i, x^*_i \rangle
$.

\begin{theorem} \label{thm:conjugate_and_subdiff_of_direct_orthogonal_sum}
  Let $\R^n = X_1 \oplus \dots \oplus X_2$ be a decomposition of $\R^n$
  into pairwise orthogonal vector subspaces $X_1, \dots , X_n$. For
  any  proper functions $f_i : X_i \rightarrow \R \cup \{+ \infty\}$
  and their 
  semidirect sum 
  $
    f = f_1 \sdirsum f_2 \sdirsum \dots \sdirsum f_n : 
    \R^n \rightarrow \R \cup \{+ \infty\}, 
    f(x) = f(x_1 + \dots + x_n) 
    \defeq 
    \sum\limits_{i=1}^n f_i(x_i)
  $ 
  we have
  \begin{enumerate}
    \item \label{enu:conjugate_of_direct_and_orthogonal_sum}
      $
	[f_1 \sdirsum f_2 \sdirsum \dots \sdirsum f_n]^*
	=
	f_1^* \sdirsum f_2^* \sdirsum \dots \sdirsum f_n^*
      $,
      i.e.\\
      $
	f^*(x^*) 
	= 
	f^*(x_1^* + \dots + x_n^*)
	=
	\sum\limits_{i=1}^n f_i^*(x_i^*)
      $
      for every $x^* \in \R^n$.
    \item \label{enu:subdifferential_of_direct_and_orthogonal_sum}
      $
	\partial f(x) 
	= 
	\partial f(x_1 + \dots + x_n)
	=
	\bigoplus\limits_{i=1}^n \partial f_i(x_i)
      $
      for every $x \in \R^n$.
  \end{enumerate}

\end{theorem}

\begin{proof}
  \upref{enu:conjugate_of_direct_and_orthogonal_sum}
  For any $x^* = (x_1^*, \dots, x_n^*)$ we have
  \begin{gather*}
    f(x^*)
    \defeq 
    \sup_{x\in \R^n} [\langle x, x^* \rangle - f(x)]
    = 
    \sup_{x_1 \in X_1} \dots \sup_{x_n \in X_n} 
	\sum_{i=1}^n [\langle x_i, x_i^* \rangle - f_i(x_i) ]
    = \sum_{i=1}^n f_i^*(x_i^*).
  \end{gather*}

  \upref{enu:subdifferential_of_direct_and_orthogonal_sum}
  Let $x = x_1 + \dots + x_n$ be arbitrarily chosen. In case 
  $x_i \not \in \dom f_i$ for some $i$ the equation and the directness 
  of its right-hand side sum holds vacuously true. 
  In case $x_i \in \dom f_i$ for all $i \in \{1, \dots, n\}$
  the claimed equation also holds true since for any
  $x^* = x_1^* + \dots x_n^*$ we have
  the equivalences
  \begin{align*}
    {}& 
    x^* \in \partial f(x)
  \\
    \Leftrightarrow{}&
    \forall z = z_1 + z_2 + \dots + z_n \in \R^n:
    f(z) \geq f(x) + \langle z-x, x^* \rangle 
  \\
    \Leftrightarrow{}&
    \forall z = z_1 + z_2 + \dots + z_n \in \R^n:
    f(z) - f(x) - \langle z-x, x^* \rangle \geq 0
  \\
    \Leftrightarrow{}&
    \forall z = z_1 + z_2 + \dots + z_n \in \R^n:
    \sum_{i=1}^n [f_i(z_i)-f_i(x_i) - \langle z_i - x_i, x_i^* \rangle] \geq 0
  \\
    \Leftrightarrow{}&
    \forall i \in \{1, \dots , n\} \; \forall z_i \in X_i:
    f_i(z_i)-f_i(x_i) - \langle z_i - x_i, x_i^* \rangle \geq 0
  \\
    \Leftrightarrow{}&
    \forall i \in \{1, \dots , n\}:
    x_i^* \in \partial f_i (x_i).
  \end{align*}
  Finally note that the directness of the sum 
  $\partial f_1(x_1) \oplus \dots \oplus \partial f_n(x_n)$
  is inherited from the direct sum $X_1 \oplus \dots \oplus X_n$.  
\end{proof}

As corollary of the previous Theorem 
\ref{thm:conjugate_and_subdiff_of_direct_orthogonal_sum}
we get the following theorem.

\begin{theorem} 
\label{thm:subdifferential_for_functions_with_not_full_dim_dom}
  Let $f: \R^n \rightarrow \R \cup \{+\infty \}$ be a  proper function
  and let $\dom f$ be contained in some affine subset 
  $A$ of $\R^n$ with difference space $U$. Then 
  \begin{gather*}
    \partial f(x) = 
    \begin{cases}
      \emptyset 			& \text{ if } x \not \in A \\
      \partial f|_A (x) \oplus U^\perp  & \text{ if } x \in A
    \end{cases}    
  \end{gather*}
  for every $x \in \R^n$.
\end{theorem}

\begin{proof}
  There is an $x_0 \in \dom f$. Translating the origin of the 
  coordinate system to $x_0$ through replacing $f$ by 
  $f(\cdot - x_0)$ would not affect the truth value of the claimed
  equation. Therefore we may assume $x_0 = \zerovec$ without loss of 
  generality, so that $A = U$ is even a vector subspace of $\R^n$.
  Setting $X_1 \defeq A = U$, $X_2 \defeq U^\perp$ and defining proper 
  functions
  $f_1: X_1 \rightarrow \R \cup \{+\infty\}$,
  $f_2: X_2 \rightarrow \R \cup \{+\infty\}$
  by 	
  \begin{align*}
    f_1(x_1) &\defeq f|_{X_1}(x_1) && \text{  and } 
    &&
    f_2(x_2) \defeq 
    \begin{cases}
      0	      & \text{ if } x_2 = \zerovec \\
      +\infty & \text{ if } x_2 \not = \zerovec
    \end{cases}
  \end{align*}
  allows us to write $f$ in the form 
  $f(x) = f(x_1 + x_2) = f_1(x_1)+f_2(x_2)$
  for all $x \in \R^n$.
  Applying Theorem 
  \ref{thm:conjugate_and_subdiff_of_direct_orthogonal_sum}
  yields
  \begin{align*}
    \partial f(x) 
    &=
    \partial f(x_1 + x_2)
    =
    \partial f_1(x_1) \oplus \partial f_2(x_2)
  \\
    &= 
    \begin{cases}
      \emptyset 			& \text{ if } x_2 \not = \zerovec \\
      \partial f_1(x_1) \oplus U^\perp  & \text{ if } x_2 = \zerovec
    \end{cases}
  \\
    &=
    \begin{cases}
      \emptyset 			& \text{ if } x \not \in U\\
      \partial f_1(x) \oplus U^\perp    & \text{ if } x \in U
    \end{cases}        
  \\
    &=
    \begin{cases}
      \emptyset 			& \text{ if } x \not \in A\\
      \partial f|_A(x) \oplus U^\perp   & \text{ if } x \in A
    \end{cases}
  \end{align*}
  for every $x \in \R^n$.
\end{proof}

\chapter{Elaborated details} \label{chap:details}

\begin{detail}
  \label{det:intersection_of_compact_sets_not_compact}
  The intersection of compact subsets of a non-Hausdorff space
  does not need to be compact again:
  We will construct an example for this phenomenon
  in three steps.
  First we will obtain a non-Hausdorff space 
  $\decorspace{X}{'}$ by gluing
  two copies of the interval 
  $([0,1], [0,1] \Cap \mathcal{O}^{\otimes 1})$ to an 
  ``interval'' which has two different right-hand side endpoints 
  $\overline{1}, \underline{1}$.
  Next we will show that homeomorphic copies of the original 
  spaces are contained in $\decorspace{X}{'}$ as 
  certain subspaces $(X_1', X_1' \Cap \mathcal{O}')$ and 
  $(X_2', X_2' \Cap \mathcal{O}')$.
  Finally we will show that the intersection 
  $(X_1' \cap X_2', (X_1' \cap X_2') \Cap \mathcal{O}')$
  of these compact subspaces is homeomorphic to the half-open 
  interval $([0,1), [0,1) \Cap \mathcal{O}^{\otimes 1})$
  and hence not compact.
  Consider the space 
  \begin{gather*}
  \big(  
    \underbrace{(\{-1\} \times [0,1]) \cup (\{1\} \times [0,1])}_{\eqdef X},  
    \underbrace{X \Cap \mathcal{O}_\R^{\otimes 2}}_{\eqdef \mathcal{O}}
  \big),
  \end{gather*}
  consisting of two copies
  $\{-1\} \times [0,1] \eqdef X_{-1}$ and 
  $\{1\} \times [0,1] \eqdef X_{1}$ of the interval $[0,1]$,
  equipped with the usual topology.
  In order to glue the space $\decorspace{X}{}$ to an ``interval'' with 
  two right-hand side endpoints
  we set 
  \begin{gather*}
    X' \defeq [0,1) \cup \{\underline{1}, \overline{1}\},
  \end{gather*}
  where $\underline{1}$ and $\overline{1}$ are two different elements
  which are not contained in $[0,1)$; moreover we equip $X'$ with the 
  identification topology $\mathcal{O}'$ which is induced by 
  $\mathcal{O}$ and the mapping $f: X \rarr X'$ given by
  \begin{gather*}
    f(t,a) 
    \defeq 
    \begin{cases}
      t 	    & \text{ for } t \in [0,1) \\
      \overline{1}  & \text{ for } t = 1  \text{ and } a = 1 \\
      \underline{1} & \text{ for } t = 1 \text{ and } a = -1.
    \end{cases}
  \end{gather*}
  The space $\decorspace{X}{'}$ is not a Hausdorff space since every
  $\topology{X}'$--neighborhoods $\overline{U}$ of $\overline{1}$
  has nonempty intersection with every $\topology{X}'$--neighborhood 
  $\underline{U}$ of $\underline{1}$ because both $\overline{U}$ and 
  $\underline{U}$ contain infinitely many of the points 
  $1-\tfrac1n$, $n\in \N$.
  However $\overline{1}$ and $\underline{1}$ are the 
  only distinct points in $\decorspace{X}{'}$ which can not be separated 
  from each other by distinct neighborhoods; i.e. all subspaces of 
  $\decorspace{X}{'}$, which contain at most one of these endpoints,
  are Hausdorff spaces.
  In particular the sets 
  \begin{gather*}
    X'_a \defeq f[X_a],
  \end{gather*}
  $a \in \{-1, 1\}$, are Hausdorff spaces. By 
  \prettyref{lem:continuous_mapping_from_compact_space_to_hausdorff_space_is_closed}
  the mapping $f|_{X_a}$, $a \in \{-1, 1\}$ acts as a homeomorphism 
  between $(X_a, X_a \Cap \mathcal{O})$ and 
  $(X'_a, X'_a \Cap \mathcal{O}')$ for $a \in \{-1 , 1\}$.
  In particular both $X'_{-1}$ and $X'_{1}$ are compact subsets of 
  $\decorspace{X}{'}$.
  However their intersection 
  \begin{gather*}    
    X'_{-1} \cap X'_{1} 
    =
    f|_{X_1}\big[\{1\} \times [0,1) \big]
  \end{gather*}
  is homeomorphic to 
  $\big(\{1\} \times [0,1), (\{1\} \times [0,1))\Cap \mathcal{O}\big)$,
  i.e. to $\big([0,1), [0,1) \Cap \mathcal{O}^{\otimes{1} } \big)$ 
  and hence not compact.
  We note here that our construction could also be done
  in a more elegant way if we had used the mapping 
  $\id_{[0,1)}: [0,1) \rarr [0,1)$ as ``Anheftungsabbildung'' 
  in order to stick two copies of the interval 
  $([0,1], [0,1] \Cap \mathcal{O}^{\otimes 1})$ together,
  cf. \cite[p. 54]{Jnch2001}; however this would bring the need 
  to introduce further topological notions.
  Moreover the constructed space $\decorspace{X}{'}$ should be 
  homeomorphic to the space presented by Steen and Seebach in section 
  ``Telophase Topology'' of their book 
  ``Counterexamples in Topology'', see \cite[p. 92]{Steen1978}.
\end{detail}

\begin{detail}
  \label{det:intersection_of_closed_and_compact_sets_is_again_closed_and_compact}
  The intersection of two both compact and closed subsets $K_1, K_2$ of
  a topological space $(X, \mathcal{O})$ is again closed and compact:
  Clearly $K_1 \cap K_2$ is again a closed subset of 
  $(X, \mathcal{O})$.
  Due to
  \begin{gather*}
    K_1 \cap K_2 
    = 
    K_1 \cap (K_1 \cap K_2) 
    \in 
    K_1 \Cap \mathcal{A}(X,\mathcal{O})
    =
    \mathcal{A}(K_1, K_1 \Cap \mathcal{O})
  \end{gather*}
  the intersection $K_1 \cap K_2$ is also a closed subset of the
  compact space $(K_1, K_1 \Cap \mathcal{O})$ and hence a compact subset 
  of this space by part \upref{enu:closeness_implies_compactness} of 
  Theorem \ref{thm:relation_closed_compact}. 
  From the compactness of the subspace 
  $
    \big(K_1 \cap K_2, (K_1 \cap K_2) \Cap (K_1 \Cap \mathcal{O}) \big)
  $
  of $(K_1, K_1 \Cap \mathcal{O})$
  we conclude that 
  \begin{gather*}
    \big (K_1 \cap K_2, (K_1 \cap K_2) \Cap (K_1 \Cap \mathcal{O}) \big )
    =
    \big (K_1 \cap K_2, (K_1 \cap K_2) \Cap \mathcal{O} \big)  
  \end{gather*}
  is also a compact subspace of 
  the original space $(X, \mathcal{O})$, since
  beeing compact is an intrinsic property of a topological (sub)space,
  cf. Definition \ref{def:compactness_of_a_space};
  i.e. $K_1 \cap K_2$ is a compact subset of $(X, \mathcal{O})$.
\end{detail}


\begin{detail} \label{det:definition_in_steen_not_totally_correct}
  The Definition in \cite[p. 74]{Steen1978} is not totally correct:
  In that book the right order topology for a linearly ordered space
  $(X, \leq)$
  is said to be the topology which is generated by basis sets
  of the form $S_a = \{x | x > a\}$. 
  However the whole space $X$ needs in general to be added to that set 
  system in order to really obtain a basis for a topology:
  Consider for instance the linearly ordered set 
  $(X, \leq) \defeq ([\minfty, \pinfty], \leq)$.
  The union of all sets $S_a$ is only the set 
  $(\minfty, \pinfty] \not = X$.
  Instead of adding the set $X$ to the set system formed by the $S_a$
  the problem could also be repaired by replacing the word ``basis'' by 
  ``subbasis''.
  For the left order order topology there is the very same problem.
  It can be repaired analogously.
\end{detail}


\begin{detail} \label{det:closed_left_halfray_is_compact_in_opposite_topology}
  $K' \defeq \{ z \in Z : z \leq z' \}$ is a compact subset of 
  $(Z, \mathcal{T}_\geq)$: 
  Let $(O'_i)_{i_\in I}$ be some open covering of $K'$.
  At least one of these open sets, lets call it $O'$, must cover 
  $z'$ and  hence also every $z \leq z'$, 
  i.e. every $z \in K'$, by the interval-like structure of the set 
  $O' \in \mathcal{T}_\geq$.
  Taking $O'$ already yields the needed finite subcover.
\end{detail}

\begin{detail} \label{det:explanation_equivalences_star_and_diamond_in_thm_characterize_lsc_and_coercive}
  The equivalences in $(\ast)$ and $(\diamond)$ in the proof of 
  \prettyref{thm:characterize_lsc_and_coercive} hold true:
  Note that the harder direction ``$\Larr$'' of the equivalence in $(\ast)$ is true,
  since every compact set $K \in \mathcal{K}(\R^n) = \mathcal{KA}(\R^n)$
  is contained in the closed ball $\closedball[R][\zerovec]$, 
  if the radius $R$ is chosen large enough. The other direction 
  ``$\Rarr$'' is true since we can simply choose 
  $K = \closedball[R][\zerovec]$.
  Next we proof the equivalence in $(\diamond)$. The totally ordered set
  $
    (Z, \leq) \defeq ([\minfty, \pinfty], \leq_{[\minfty, \pinfty]})
  $
  with the natural order on $[\minfty, \pinfty]$ has both a minimum and 
  a maximum. Hence part
  \upref{enu:compact_closed_sets_which_do_not_hit_max_el_are_compl_of_its_open_neighborhoods}
  of \prettyref{lem:compact_and_closed_subsets_in_right_order_topology}
  can be applied and we obtain
  $
    \mathcal{KA}_{\{\pinfty\}}([\minfty, \pinfty], \mathcal{T})
    =
    \{Z \setminus U' : U' \in \mathcal{U}'(\pinfty) \cap \mathcal{T}\}
  $.
  After taking complements this reads 
  \begin{gather*}
    \mathcal{U}'(\pinfty) \cap \mathcal{T} 
    = 
    \{Z \setminus K' :
      K' \in \mathcal{KA}_{\{\pinfty\}}([\minfty, \pinfty], \mathcal{T})
    \}
  \end{gather*}
  which directly shows that the equivalence in  $(\diamond)$ is true.  
\end{detail}

\begin{detail} 
\label{det:continuity_notions_at_pinfty_are_equivalent}
  $
    \widehat f: 
	    (\R^n_\infty, \mathcal{O}^{\otimes n}_\infty) 
	    \rarr 
	    ([\minfty, \pinfty], \mathcal{T})
  $
  is continuous at the point $\infty$ if and only if
  $
      \widehat f: 
	    (\R^n_\infty, \mathcal{O}^{\otimes n}_\infty) 
	    \rarr 
	    ([\minfty, \pinfty], \mathcal{O}_\leq)
  $
  is continuous at the point $\infty$:
  Let 
  $
    \widehat f: 
	    (\R^n_\infty, \mathcal{O}^{\otimes n}_\infty) 
	    \rarr 
	    ([\minfty, \pinfty], \mathcal{T})
  $
  be continuous at the point $\infty$, i.e.
  for any $\mathcal{T}$--neighborhood $T$ of 
  $\pinfty = \widehat f (\infty)$ there is a neighborhood 
  $U \in \mathcal{O}^{\otimes n}_\infty$ of $\infty$ with
  $\widehat f[U] \subseteq T$.
  In order to show that 
  $
      \widehat f: 
	    (\R^n_\infty, \mathcal{O}^{\otimes n}_\infty) 
	    \rarr 
	    ([\minfty, \pinfty], \mathcal{O}_\leq)
  $
  is continuous at the point $\infty$ let any 
  $\mathcal{O}_\leq$--neighborhood $O$ of 
  $\pinfty = \widehat f(\infty)$ be given.
  Since $O$ contains a set of the form
  $(\alpha, \pinfty] \eqdef T \in \mathcal{T}$
  we obtain with some corresponding 
  neighborhood $U \in \mathcal{O}^{\otimes n}_\infty$ of $\infty$ 
  the inclusion $f[U] \subseteq T \subseteq O$
  and have therewith shown one implication.
  Let now, to the contrary,
  $
      \widehat f: 
	    (\R^n_\infty, \mathcal{O}^{\otimes n}_\infty) 
	    \rarr 
	    ([\minfty, \pinfty], \mathcal{O}_\leq)
  $
  be continuous at the point $\infty$, i.e.
  for any $\mathcal{O}_\leq$--neighborhood $O$ of 
  $\pinfty = \widehat f (\infty)$ there is a neighborhood
  $U \in \mathcal{O}^{\otimes n}_\infty$ of $\infty$ with
  $\widehat f[U] \subseteq O$.
  In particular the mapping 
  $
    \widehat f: 
	    (\R^n_\infty, \mathcal{O}^{\otimes n}_\infty) 
	    \rarr 
	    ([\minfty, \pinfty], \mathcal{T})
  $
  is also continuous in $\pinfty$,
  since every $\mathcal{T}$--neighborhood of $\pinfty$ 
  is also a $\mathcal{O}_\leq$--neighborhood 
  of $\pinfty$.
\end{detail}

\begin{detail} \label{det:product_of_locally_compact_Hausdorff_is_of_same_type}
  The product space 
  $
    \decorspace{Y}{'} \varotimes \decorspace{Y}{''}
    \eqdef \decorspace{Y}{}
  $ of
  two locally compact Hausdorff spaces is again a locally compact Hausdorff
  space:
  Let $(x', x''), (y', y'')$  be two different points in 
  $\decorspace{Y}{}$ with, say, $x' \neq y'$.
  Since $\decorspace{Y}{'}$ is a Hausdorff space there exist
  disjoint neighborhoods $U'$ and $V'$ of $x'$ and $y'$, 
  respectively.
  Then clearly $U \defeq U' \times Y''$ and 
  $V \defeq V' \times Y''$ are disjoint neighborhoods of 
  $(x',x'')$ and $(y',y'')$, respectively, in
  $\decorspace{Y}{}$. So the latter topological space is again a
  Hausdorff space.
  Moreover $\decorspace{Y}{}$ is also locally compact:
  Let $y = (y', y'') \in Y$.
  Since $\decorspace{Y}{'}$ and $\decorspace{Y}{''}$ are locally compact 
  there exist compact neighborhoods 
  $U' \in \mathcal{U}'(y')$ and $U'' \in \mathcal{U}''(y'')$.
  The neighborhood $U \defeq U' \times U''$ of $y$ is then compact 
  in virtue of Tichonov's Theorem \ref{thm:little_tichonov}.   
\end{detail}



\begin{detail} \label{det:coercivity_assertion_in_example_is_special_case_of_thm_sum_coercive_on_certain_subspaces}
  The coercivity assertion of 
  \prettyref{lem:application_of_compact_continuations_to_get_coercivity_and_lsc_of_sum}
  is contained in \prettyref{thm:sum_coercive_on_certain_subspaces} 
  as special case:
  $F_1$ and $G_1$ are coercive; for instance 
  $F_1 = \phi \circ H|_{\mathcal{R}(H^*)}$ is a concatenation of 
  the coercive mapping $\phi$ and the injective and 
  hence normcoercive linear mapping 
  $H|_{\mathcal{R}(H^*)}: \mathcal{R}(H^*) \rarr \mathcal{R}(H)$;
  cf. also the proof of 
  \prettyref{thm:linear_mapping_coercive_iff_injective}.
  Moreover the mappings 
  $F_1: X_1 \rarr (\minfty, \pinfty]$ and 
  $G_1: Y_1 \rarr (\minfty, \pinfty]$ are lower semicontinuous and hence
  in particular locally bounded. 
  Finally $F_2 = 0_{X_2}$ and $G_2 = 0_{Y_2}$ are clearly bounded below.
\end{detail}

\begin{detail} \label{det:both_subsequences_bounded_above}
  Both 
  $(\|\widecheck x_{n_k}\|)_{k \in \N}$ and 
  $(\|\widecheck y_{n_k}\|)_{k \in \N}$
  would be bounded above by some $B > 0$:
  If one of this sequences, say $(\widecheck x_{n_k})$ without loss of
  generality, would be unbounded there would be a subsequence 
  $(\widecheck x_{n_{k_j}})_{j \in \N}$ with
  $\|\widecheck x_{n_{k_j}}\|_X \rarr \pinfty$ as $j \rarr \pinfty$.
  Since $\widecheck F$ is normcoercive we would get 
  $\|\widecheck F(\widecheck x_{n_{k_j}})\|_Z \rarr \pinfty$
  as $j \rarr \pinfty$. This contradicts
  \eqref{eq:functions_bounded_on_subsequence}.
\end{detail}

\begin{detail} \label{det:compact_subset_contained_in_closed_half_ray}
  There is an element $b \in Z \setminus \MAX_\leq(Z)$ with 
  $K' \subseteq b]$:
  If $Z \setminus \MAX_\leq(Z)$ contains  a maximum $\widehat b$
  then clearly $K' \subseteq Z \setminus \MAX_\leq(Z) = \widehat b]$.
  If $Z \setminus \MAX_\leq(Z)$ contains no maximum then we can write
  \begin{gather*}
    Z \setminus \MAX_\leq(Z)
    =
    \bigcup_{b \in Z \setminus \MAX_\leq(Z)} b)
  \end{gather*}
  so that the sets $b)$, where $b \in Z \setminus \MAX_\leq(Z)$, form
  in particular an open cover of $K'$. Due to the compactness of 
  $K'$ there are finitely many 
  $b_1, \dots , b_n \in Z \setminus \MAX_\leq(Z)$ with 
  \begin{gather*}
     K' \subseteq \bigcup_{i=1}^n b_i).
  \end{gather*}
  Denoting the largest of 
  the $b_i$ with $b$ we hence have 
  $K' \subseteq \bigcup_{i=1}^n b_i) \subseteq b]$.
\end{detail}

\begin{detail} \label{det:intersection_between_sum_of_spaces_and_intersection_of_their_orthogonal_complements_is_trivial}
  The subspaces $X_1 + W_1$ and $(X_1^\perp \cap W_1^\perp)$
  have trivial intersection:
  Writing an arbitrarily chosen 
  $x \in (X_1 + W_1) \cap X_1^\perp \cap W_1^\perp$ in the form 
  $x = x_1 + w_1$ with some $x_1 \in X_1$ and $w_1 \in W_1$
  we get $\langle x, x_1 \rangle = 0$ and 
  $\langle x, w_1 \rangle = 0$. Addition gives 
  $\langle x, x\rangle = 0$ and hence $x = \zerovec$.
\end{detail}

\begin{detail} \label{det:sum_unique_only_up_to_constant}
  For real-valued functions 
  $F_1, \widetilde{F_1}: X_1 \rarr \R \cup \{\pinfty\}$,
  $F_2, \widetilde{F_2}: X_2 \rarr \R \cup \{\pinfty\}$
  with $F_1 \sdirsum F_2 = \widetilde{F_1} \sdirsum \widetilde{F_2}$ 
  there is a constant $C \in \R$ such that $F_1 = \widetilde {F_1} + C$ 
  and $F_2 = \widetilde {F_2} - C$:  
  For all $x_1 \in X_1$ and $x_2 \in X_2$ we have
  $F_1(x_1) + F_2(x_2) = \widetilde{F_1}(x_1) + \widetilde{F_2}(x_2)$.
  Since only finite values occur we can rearrange the latter and obtain
  $F_1(x_1) - \widetilde{F_1}(x_1) = \widetilde{F_2}(x_2) - F_2(x_2)$
  for all $x_1 \in X_1$ and $x_2 \in X_2$.
  In particular the functions 
  $F_1 - \widetilde{F_1}: X_1 \rarr \R$ and 
  $\widetilde{F_2} - F_2: X_2 \rarr \R$
  are constant on $X_1$ and $X_2$, respectively; by the previous 
  equality, they take the same constant value. 
  Denoting this value by ``$C$'' we are done.  
\end{detail}

\begin{detail} \label{det:sum_not_even_unique_up_to_constant}
  If one of the functions $F_1, F_2, \widetilde{F_1}, \widetilde{F_2}$
  takes the value $\pinfty$ there is no guarantee 
  that, e.g. $F_2$ and $\widetilde{F_2}$ differ merely 
  by a real constant; consider for instance the functions
  $F_1 = \widetilde{F_1} \equiv + \infty$ on $X_1$.
  Then $F_1 \sdirsum F_2 = \widetilde{F_1} \sdirsum \widetilde{F_2}$ for 
  any functions $F_2, \widetilde{F_2}: X_2 \rarr \R \cup \{\pinfty\}$.
\end{detail}

\MayChangesPartiallyPerformedVersionOrHareBrainedOfficialVersion{%
\begin{detail} \label{det:orthogonality_can_be_assumed_wlog}
  Without loss of generality, we may assume
  $X_1 = X_2^\perp$, $Y_1 = Y_2^\perp$ and $Z_1 = Z_2^\perp$;
  otherwise we can replace
  \begin{alignat*}{3}
    F_1 {}& \text{ by  } \;
      \widetilde {F_1} &&= F_1 \circ \pi_{X_1,X_2}|_{X_2^\perp},
  \\
    G_1 {}& \text{ by  } \;
      \widetilde {G_1} &&= G_1 \circ \pi_{Y_1,Y_2}|_{Y_2^\perp},
  \\
    H_1 {}& \text{ by  } \;
      \widetilde {H_1} &&= H_1 \circ \pi_{Z_1,Z_2}|_{Z_2^\perp},
  \end{alignat*}
  and continue the proof with theses new functions instead of the 
  original functions due to the following three reasons: 
  \begin{enumerate}
    \item 
      The assumptions on $F_1, G_1$
      carry over to $\widetilde {F_1}, \widetilde{G_1}$:
      Using part \upref{enu:nullspace_of_projection} of 
      \prettyref{lem:nullspace_of_projections} we see that 
      the new functions differ from the original functions merely 
      by bijective linear transformations of their image domains.
      Since the involved spaces are of finite dimension these linear 
      bijections are even homeomorphisms. In particular the locally
      boundedness assumption on the original functions carries over 
      to the new functions.
      Also the coercivity assumption on the original functions 
      carries over to the new functions by part
      \upref{enu:semidirect_sum_with_zero_independet_from_choice_of_complement}
      of \prettyref{lem:semidirect_sum_with_zerofunction}.
    \item 
      $H$ stays unchanged when replacing the old function by the 
      new ones: part 
      \upref{enu:semidirect_sum_with_zero_independet_from_choice_of_complement}
      of \prettyref{lem:semidirect_sum_with_zerofunction} gives 
      $F_1 \sdirsum 0_{X_2} = \widetilde{F_1} \sdirsum 0_{X_2}$ and 
      $G_1 \sdirsum 0_{X_2} = \widetilde{G_1} \sdirsum 0_{X_2}$ so that
      \begin{align*}
	H & = (F_1 \sdirsum 0_{X_2}) + (G_1 \sdirsum 0_{Y_2})
      \\
	  & = (\widetilde {F_1} \sdirsum 0_{X_2}) 
	      + (\widetilde{G_1} \sdirsum 0_{Y_2})
      \end{align*}
    \item 
      After proving the coercivity of $\widetilde {H_1}$ also 
      the coercivity of $H_1$ would follow:
      Using parts
      \upref{enu:sum_of_semidir_sum_with_zero_is_of_the_same_structure}
      and 
      \upref{enu:semidirect_sum_with_zero_independet_from_choice_of_complement}
      of \prettyref{lem:semidirect_sum_with_zerofunction}
      we can rewrite $H$ in the form 
      \begin{gather*}
        H = H_1 \sdirsum 0_{Z_2} 
	  = \widetilde{H_1} \sdirsum 0_{Z_2}
      \end{gather*}
      so that part 
      \upref{enu:semidirect_sum_with_zero_independet_from_choice_of_complement}
      of \prettyref{lem:semidirect_sum_with_zerofunction}
      ensures that $\widetilde {H_1}$ is coercive iff $H_1$ is coercive.
  \end{enumerate} 
\end{detail} 

\may{Details 14 und 15 vertauscht (damit sie chronologisch zitiert werden)}
\begin{detail} \label{det:coercive_and_locally_bounded_below_imply_bounded_below}
  Both $F_1$ and $G_1$ are bounded below:
  Let $\|\cdot \|_2$ be the Euclidean norm in $\R^n$.
  After setting 
  \begin{align*}
    (X, \|\cdot\|) \defeq (X_1, \|\cdot \|_2|_{X_1})
    &&
    (Z, \leq) \defeq ((\minfty, \pinfty], \leq_{(\minfty, \pinfty]})
  \end{align*}
  with the natural ordering $\leq_{(\minfty, \pinfty]}$ on 
  $(\minfty, \pinfty]$ we can apply 
  \prettyref{thm:coercive_mapping_in_finite_dim_space_bounded_below_already_when_locally_bounded_below}
  to $F_1: (X, \|\cdot\|) \rarr (Z, \leq)$
  and obtain that $F_1$ is bounded from below.
  Likewise we see that also $G_1$ is bounded below.
\end{detail}
}%
{\begin{detail} \label{det:coercive_and_locally_bounded_below_imply_bounded_below}
  Both $F_1$ and $G_1$ are bounded below:
  Let $\|\cdot \|_2$ be the Euclidean norm in $\R^n$.
  After setting 
  \begin{align*}
    (X, \|\cdot\|) \defeq (X_1, \|\cdot \|_2|_{X_1})
    &&
    (Z, \leq) \defeq ((\minfty, \pinfty], \leq_{(\minfty, \pinfty]})
  \end{align*}
  with the natural ordering $\leq_{(\minfty, \pinfty]}$ on 
  $(\minfty, \pinfty]$ we can apply 
  \prettyref{thm:coercive_mapping_in_finite_dim_space_bounded_below_already_when_locally_bounded_below}
  to $F_1: (X, \|\cdot\|) \rarr (Z, \leq)$
  and obtain that $F_1$ is bounded from below.
  Likewise we see that also $G_1$ is bounded below.
\end{detail}

\begin{detail} \label{det:orthogonality_can_be_assumed_wlog}
  Without loss of generality, we may assume
  $X_1 = X_2^\perp$, $Y_1 = Y_2^\perp$ and $Z_1 = Z_2^\perp$;
  otherwise we can replace
  \begin{alignat*}{3}
    F_1 {}& \text{ by  } \;
      \widetilde {F_1} &&= F_1 \circ \pi_{X_1,X_2}|_{X_2^\perp},
  \\
    G_1 {}& \text{ by  } \;
      \widetilde {G_1} &&= G_1 \circ \pi_{Y_1,Y_2}|_{Y_2^\perp},
  \\
    H_1 {}& \text{ by  } \;
      \widetilde {H_1} &&= H_1 \circ \pi_{Z_1,Z_2}|_{Z_2^\perp},
  \end{alignat*}
  and continue the proof with theses new functions instead of the 
  original functions due to the following three reasons: 
  \begin{enumerate}
    \item 
      The assumptions on $F_1, G_1$
      carry over to $\widetilde {F_1}, \widetilde{G_1}$:
      Using part \upref{enu:nullspace_of_projection} of 
      \prettyref{lem:nullspace_of_projections} we see that 
      the new functions differ from the original functions merely 
      by bijective linear transformations of their image domains.
      Since the involved spaces are of finite dimension these linear 
      bijections are even homeomorphisms. In particular the locally
      boundedness assumption on the original functions carries over 
      to the new functions.
      Also the coercivity assumption on the original functions 
      carries over to the new functions by part
      \upref{enu:semidirect_sum_with_zero_independet_from_choice_of_complement}
      of \prettyref{lem:semidirect_sum_with_zerofunction}.
    \item 
      $H$ stays unchanged when replacing the old function by the 
      new ones: part 
      \upref{enu:semidirect_sum_with_zero_independet_from_choice_of_complement}
      of \prettyref{lem:semidirect_sum_with_zerofunction} gives 
      $F_1 \sdirsum 0_{X_2} = \widetilde{F_1} \sdirsum 0_{X_2}$ and 
      $G_1 \sdirsum 0_{X_2} = \widetilde{G_1} \sdirsum 0_{X_2}$ so that
      \begin{align*}
	H & = (F_1 \sdirsum 0_{X_2}) + (G_1 \sdirsum 0_{Y_2})
      \\
	  & = (\widetilde {F_1} \sdirsum 0_{X_2}) 
	      + (\widetilde{G_1} \sdirsum 0_{Y_2})
      \end{align*}
    \item 
      After proving the coercivity of $\widetilde {H_1}$ also 
      the coercivity of $H_1$ would follow:
      Using parts
      \upref{enu:sum_of_semidir_sum_with_zero_is_of_the_same_structure}
      and 
      \upref{enu:semidirect_sum_with_zero_independet_from_choice_of_complement}
      of \prettyref{lem:semidirect_sum_with_zerofunction}
      we can rewrite $H$ in the form 
      \begin{gather*}
        H = H_1 \sdirsum 0_{Z_2} 
	  = \widetilde{H_1} \sdirsum 0_{Z_2}
      \end{gather*}
      so that part 
      \upref{enu:semidirect_sum_with_zero_independet_from_choice_of_complement}
      of \prettyref{lem:semidirect_sum_with_zerofunction}
      ensures that $\widetilde {H_1}$ is coercive iff $H_1$ is coercive.
  \end{enumerate} 
\end{detail} }%


\begin{detail} \label{det:hyperplane_intersection_with_subspace_yields_here_again_a_hyperplane}
  $H^{=}$ is a hyperplane in $U = {\rm aff}(\dom \Psi)$:
  The subspace $H^{=} \defeq H_{p,\alpha}^{=} \cap U$ 
  is of dimension 
  $
    \dim H^{=} 
    = 
    \dim H_{p,\alpha}^{=} + \dim U - \dim (U+H_{p,\alpha}^{=})
    \in
    n - 1 + \dim U - \{n, n-1\}
    =
    \{\dim U, \dim U - 1\}
  $.
  The set $H_{p,\alpha}^{=}$ does not completely contain $S$; consequently
  $H^{=} \subseteq H_{p,\alpha}^{=}$ can not completely contain ${\rm aff}(\dom \Psi) \supseteq S$ all the more, so that only
  $\dim H^{=} = \dim ({\rm aff}(\dom \Psi)) - 1$ can be true.
  Therefore $H^{=}$ is a hyperplane in ${\rm aff}(\dom \Psi)$.
\end{detail}



%

\begin{detail} \label{det:gradient_gets_large_when_approaching_boundary}
  For $\alpha  \in (0, \frac{1}{2})$ we have
  $\|\nabla g_\alpha (z^{(k)})\|_{2} \rightarrow + \infty$  as 
  $k \rightarrow + \infty$ for any sequence 
  $(z^{(k)})_{k \in \N}$ in $Q$, converging to some boundary point 
  $z^{(\infty)}$ of $Q$:
  Since all norms in $\R^2$ are equivalent it suffices to show 
  $\|\nabla g_\alpha (z^{(k)})\|_{\infty} \rightarrow + \infty$.
  We have
  $\nabla g_\alpha (z) = - \alpha z_1^{\alpha - 1}z_2^{\alpha -1} (z_2, z_1)^T$
  for all $z \in Q$, so that 
  $
    \|\nabla g_\alpha(z)\|_{\infty} 
    = 
    \alpha z_1^{\alpha - 1}z_2^{\alpha -1} \max\{z_2, z_1\}
  $
  for these $z$. In case $z^{(\infty)} = (0,0)^T$ we thus have for 
  $\alpha \in (0, \frac{1}{2})$ the estimate
  \begin{align*}
    \|\nabla g_\alpha (z^{(k)})\|_{\infty} 
    &\geq 
    \alpha 
      [\max \{z^{(k)}_1, z^{(k)}_2\}]^{\alpha -1} 
      [\max \{z^{(k)}_1, z^{(k)}_2\}]^{\alpha -1} 
      \max \{z^{(k)}_1, z^{(k)}_2\}
    \\
    &=
     \alpha [\max \{z^{(k)}_1, z^{(k)}_2\}]^{2 \alpha -1}
    \rightarrow
    +\infty
  \end{align*}
  as $k \rightarrow + \infty$.
  In case $z^{(\infty)} \not = (0,0)$ we may assume, due to symmetry reasons,
  $z^{(\infty)}_1 = 0$ and $z^{(\infty)}_2> 0$ without loss of generality.
  We then obtain
  \begin{gather*}
    \|\nabla g_\alpha (z^{(k)})\|_{\infty}
    =
    [z^{(k)}_1]^{\alpha -1}
    (\alpha [z^{(k)}_2]^{\alpha -1} \max\{z^{(k)}_2, z^{(k)}_1\})
    \rightarrow 
    +\infty
  \end{gather*}
  as $k \rightarrow +\infty$, even for $\alpha \in (0, 1 )$.    
\end{detail}

%

\begin{detail} \label{det:functions_f_and_g_bounded_from_below}
  The functions $ f $ and $g$ are bounded from below:
  If, say $ f $, was not bounded from below there would be a sequence $(u_k)_{k\in \N}$
  in the compact level set ${\rm lev}_{\tilde\alpha}( f )$ with $ f (u_k) \rightarrow - \infty$
  for $k \rightarrow + \infty$.
  However, after choosing a subsequence which converges to some
  $u \in {\rm lev}_{\tilde\alpha}( f )$ we had 
  $ f (u) = - \infty$, by the lower semicontinuity of $ f $.
  But this would mean that $ f $ is not proper -- a contradicition.
\end{detail}

%

\begin{detail} \label{det:pr_of_sol_existence:check_of_assumptions_of_lemma}
  All assumptions of 
  part \ref{enu:sum_takes_it_minimum}) of
  Lemma \ref{lem:levelsets_and_existence_of_minimizer_of_sum} are 
  fulfilled for 
  $F \defeq \Phi$, 
  $U_1 \defeq X_1 \oplus  X_3$, $U_2 \defeq X_2$
  and
  $G \defeq \iota_{{\rm lev}_\tau \|L \cdot\|}$,
  $V_1 \defeq \mathcal{R}(L^*)$, $V_2 \defeq \mathcal{N}(L)$, for appropriately
  chosen $\alpha$ and $\beta$:
  \begin{itemize}
    \item 
      $U_2 \cap V_2 = \{\zerovec\}$ holds true, beeing an assumption of the
      current theorem.
    \item 
      $
	\dom F \cap \dom G
	= 
	\dom \Phi \cap {\rm lev}_\tau{\|L \cdot\|} 
	\not =
	\emptyset
      $:
      Each neighborhood of $\zerovec \in \overline{\dom F}$ intersects $\dom F$.  
      Since $\tau > 0$ ensures $\zerovec \in {\rm int}({\rm lev}_\tau{\|L\cdot\|})$
      we thus have in particular for this neighborhood
      $
	\emptyset 
	\not = 
	\dom F \cap {\rm int}({\rm lev}_\tau{\|L\cdot\|})
	\subseteq
	\dom F \cap {\rm lev}_\tau{\|L\cdot\|}
      $.
    \item 
      ${\rm lev}_{\alpha}(F|_{U_1})$ is nonempty and bounded for
      an $\alpha \in \R$:
      Denoting the unique minimizer of the strictly convex function 
      $\phi = \Phi|_{X_1}$ by $\check x$ and setting
      $\alpha \defeq \phi(\check x)$ we see that 
      $
	{\rm lev}_{\alpha}(F|_{U_1})
	=
	{\rm lev}_{\alpha}(\phi) \oplus \{\zerovec\}
	=
	\{\check x\}
      $
      is nonempty and bounded.
    \item 
      Finally ${\rm lev}_{\beta}(G|_{V_1})$
      is nonempty and bounded for any $\beta \geq 0$,
      since  $G|_{V_1}$ is a norm 
      -- namely the norm on $V_1$, which makes $(V_1, G|_{V_1})$ isometrically isomorph
      to $(\mathcal{R}(L),$ $\| \cdot \||_{\mathcal{R}(L)})$, by virtue of the bijection
      $L|_{\mathcal{R}(L^*)} :\mathcal{R}(L^*) \rightarrow \mathcal{R}(L)$.
  \end{itemize}
\end{detail}

%

\begin{detail} \label{det:pr_of_locali_theorem:c_is_really_a_min}
  All assumptions of part \ref{enu:sum_takes_it_minimum}) of 
  Lemma \ref{lem:levelsets_and_existence_of_minimizer_of_sum} 
  are fulfilled for $U_1 \defeq X_1 \oplus X_3$, 
  $U_2 \defeq X_2$,
  $V_1 \defeq \mathcal{R}(L^*)$,  $V_2 \defeq \mathcal{N}(L)$ and 
  $F \defeq \iota_{\argmin(\Phi)}$,
  $G \defeq \|L \cdot\|$, for appropriate choice of $\alpha$ and $\beta$:
  \begin{itemize}
    \item 
      $F$, $G$ are in $\Gamma_0(\R^n)$ 
      and have the needed translation invariance.
    \item 
      $U_2 \cap V_2 = \{\zerovec\}$ holds true, beeing an assumption of the current theorem.
    \item 
      ${\rm lev}_{\alpha} (F|_{U_1})$ is nonempty and bounded for an $\alpha$:
      Denoting the unique minimizer of $\phi$ with $\check x_1$ we have
      $
	\argmin \Phi
	= 
	\{\check x_1\} \oplus X_2 \subseteq X_1 \oplus X_2
      $.
      For $\alpha \defeq F(\check x_1) = 0$ the set
      ${\rm lev}_{\alpha} (F|_{U_1}) = \{\check x_1\}$ is then 
      obviously nonempty and bounded.
    \item 
      Finally ${\rm lev}_{\beta}(G|_{V_1})$ is nonempty and bounded for any 
      $\beta \geq 0$, since 
      $G|_{V_1}$ is a norm -- namely the norm on $V_1$, 
      which makes $(V_1, G|_{V_1})$ isometrically isomorph
      to $(\mathcal{R}(L),$ $\| \cdot \||_{\mathcal{R}(L)})$, by virtue of the bijection
      $L|_{\mathcal{R}(L^*)} :\mathcal{R}(L^*) \rightarrow \mathcal{R}(L)$.
  \end{itemize}
\end{detail}

\begin{detail} \label{det:pr_of_locali_theorem:d_greater_than_zero}
  $d=0  \Leftrightarrow \argmin \Phi \cap \mathcal{N}(L) \not = \emptyset:$
  Using Fermat's Rule, see \cite[p. 264, l. 8]{Rockafellar1970};
  $\zerovec \in {\rm ri}(\dom \Phi^*)$,
  see part \ref{enu:dual_properties_in_our_setting}) in Lemma \ref{lem:subgradient_and_conjugate_function_in_our_setting},
  in order to apply the chain rule, see \cite[Theorem 23.9]{Rockafellar1970} and
  $x\in\partial \Phi^*(x^*) \Leftrightarrow x^* \in \partial \Phi(x)$,
  see \cite[Corollary 23.5.1]{Rockafellar1970} 
  we obtain 
  \begin{align*}
    d=0  & \Leftrightarrow
	  \zerovec \in \argmin \Phi^*(-L^* \cdot)
	\\
	& \Leftrightarrow
	  \zerovec \in \partial [\Phi^*(-L^* \cdot)]|_{\zerovec}
	\\
	& \Leftrightarrow
	  \zerovec \in -L \partial \Phi^*(-L^*\zerovec)
	\\
	& \Leftrightarrow
	  \exists x \in \R^n: x \in \partial \Phi^*(\zerovec) \wedge \zerovec = -Lx
	\\
	& \Leftrightarrow
	  \exists x \in \R^n: \zerovec \in \partial \Phi(x) \wedge x \in \mathcal{N}(L)
	\\
	& \Leftrightarrow
	  \argmin \Phi \cap \mathcal{N}(L) \not = \emptyset.
  \end{align*}
\end{detail}

\begin{detail} \label{det:pr_of_locali_theorem:SOL_D_1_equals_0}
  There is a decomposition 
  ${\rm aff}(\dom F) = \check A_F \oplus \check P_F$ such that 
  $\check P_F$ is a subspace of $P[F]$ and such that $F$ is strictly convex 
  on ${\rm int}_{\check A}(\dom F|_{\check A})$:
  We set $E = \Phi^*: \R^n \rightarrow \R \cup \{+\infty \}$,
  $M(\cdot) = -L^* \cdot$. 
  Note now that $\zerovec \in {\rm ri}(\dom E) \cap \mathcal{R}(M)$ 
  and that ${\rm aff}(\dom E) = X_1 \oplus X_3$, where
  $X_3$ is a subspace of $P[E]$, by Lemma 
  \ref{lem:subgradient_and_conjugate_function_in_our_setting},
  and where $E = \Phi^*$ is strictly convex 
  on ${\rm int}_{X_1}(\dom \Phi^*|_{X_1}) = {\rm ri}(\dom \Phi^*|_{X_1})$,
  since it is even essentially strictly convex on $X_1$
  by Lemma \ref{lem:subgradient_and_conjugate_function_in_our_setting}.
  Thus we can use 
  \prettyref{thm:concatenation_stricly_convex_on_ri_with_linear_mapping}
  and obtain that ${\rm aff}(\dom F)$ can be decomposed in the 
  claimed way.
\end{detail}

\begin{detail} \label{det:pr_of_locali_theorem:SOL_D_1_away_from_argmin}
  The functions $F(\cdot) = \Phi(-L^* \cdot)$ and 
  $G(\cdot) = \tau\|\cdot\|_*$
  fulfill the assumptions of Theorem
  \ref{thm:direction_of_argmin__essentially_smooth__strict_convex_sharpend}:
  Due to $\zerovec = -L^*\zerovec \in {\rm ri}(\dom \Phi^*)$ and 
  $\dom \Phi^* = X_1 \oplus X_3$ we see that
  \prettyref{thm:concatenation_stricly_convex_on_ri_with_linear_mapping}
  can be applied to $E = \Phi^*$ and $M(\cdot)= -L^* \cdot$. Thereby we 
  get a decomposition ${\rm aff}(\dom F) = \check A \oplus \check P$
  of ${\rm aff}(\dom F) \eqdef A$ into a vector subspace $\check P$ of the 
  periods space $P[F]$ and an affine subspace $\check A \subseteq \R^n$
  such that $F$ is strictly convex on
  ${\rm int}_{\check A}(\dom F|_{\check A})$.
  Furthermore $F$ is essentially smooth on $A$ by Theorem 
  \ref{thm:concatenation_ess_smooth_with_linear_mapping}.
\end{detail}

\begin{detail} \label{det:pr_of_locali_theorem:SOL_P_2_belongs_to_nullspace}
  The assumptions of Theorem 
  \ref{thm:direction_of_argmin__essentially_smooth__strict_convex_sharpend}
  are fulfilled for $F = \Phi$ and $G(\cdot) = $ $\lambda \|L \cdot\|$:  
  Clearly $F$ and $G$ are convex functions with
  ${\rm ri}(\dom F) \cap {\rm ri}(\dom G) \not = \emptyset$.
  Moreover the decomposition ${\rm aff}(\dom F) = X_1 \oplus X_2$, or rather 
  their components, have the needed properties 
  by our setting's assumptions:
  $X_2$ is a subspace of $P[F]$, $F$ is strictly convex on 
  ${\rm int}_{X_1}(\dom F|_{X_1})= {\rm ri}(\dom F|_{X_1}) \not = \emptyset$,
  and lastly $F$ is essentially smooth on $X_1$.
\end{detail}

\begin{detail} \label{det:pr_of_locali_theorem:SOL_P_2_does_not_intersect_argmin_Phi}
  $f$ is again proper, convex, lower semicontinuous and essentially 
  smooth:
  $f$ is proper since $\Phi \eqdef F$ is proper and because 
  $f(1)=F(\hat x) < +\infty$.
  Moreover $f$ also inherits convexity and lower semicontinuity from $F$.
  Finally $f$ is essentially smooth:
  Part \upref{enu:argmin_contained_in_ri_of_dom} of Lemma 
  \ref{lem:argmin_of_sum_contained_in_ri_of_dom_of_ess_smooth_summand}
  gives $\hat x \in \argmin F \subseteq {\rm ri}(\dom F)$, so that 
  Theorem \ref{thm:concatenation_ess_smooth_with_linear_mapping}
  can be applied to $E = F$ and $M(\cdot) = \cdot \hat x$, giving the 
  essentially smoothness of $f = F \circ M$ on ${\rm aff}(\dom f) = \R$;
  note for the last equality -- in the nontrivial case 
  $\hat x \not = \zerovec$ -- the above $\hat x \in {\rm ri}(\dom F)$ 
  and our setting assumption $\zerovec \in \overline{\dom F}$.
\end{detail}

%

\begin{detail} \label{det:pr_of_main_theorem:function_g_is_well_defined}
  The function $g$, given by $g(\tau) \defeq \|\hat p\|_*$ with any 
  $\hat p \in {\rm SOL}(D_{1,\tau})$, $\tau \in (0,c)$ is well defined,
  since ${\rm SOL}(D_{1,\tau}) \not = \emptyset$ and since
  Theorem 
  \ref{thm:direction_of_argmin__essentially_smooth__strict_convex_sharpend}
  ensures $\|\hat p\|_* = \|\hat q\|_*$ for any other 
  $\hat q \in {\rm SOL}(D_{1,\tau})$:
  Consider $F(\cdot) = \Phi(-L^* \cdot)$ and $G(\cdot) = \tau \|\cdot\|_*$.
  Due to $\zerovec = -L^*\zerovec \in {\rm ri}(\dom \Phi^*)$ and 
  $\dom \Phi^* = X_1 \oplus X_3$ we see that
  Theorem \ref{thm:concatenation_stricly_convex_on_ri_with_linear_mapping}
  can be applied to $E = \Phi^*$ and $M(\cdot)= -L^* \cdot$. Thereby we 
  get a decomposition ${\rm aff}(\dom F) = \check A \oplus \check P$
  of ${\rm aff}(\dom F) \eqdef A$ into a vector subspace $\check P$ of the 
  periods space $P[F]$ and an affine subspace $\check A \subseteq \R^n$
  such that $F$ is strictly convex on
  ${\rm int}_{\check A}(\dom F|_{\check A})$.
  We may assume without loss of generality that $\check A$
  is a vector subspace as well, since $\zerovec \in A$.
  Furthermore $F$ is essentially smooth on $A$ by Theorem 
  \ref{thm:concatenation_ess_smooth_with_linear_mapping}
  and even on $\check A$ by Lemma
  \ref{lem:essentially_smoothness_does_not_depend_on_periodspace}.
  Theorem
  \ref{thm:direction_of_argmin__essentially_smooth__strict_convex_sharpend}
  can thus be applied, giving
  $\tau \|\hat p \|_* = G(\hat p) = G(\hat q) = \tau \|\hat q\|_*$.
  Since $\tau \not = 0$ we get the claimed $\|\hat p\|_* = \|\hat q\|_*$.
\end{detail}

\begin{detail} \label{det:pr_of_main_theorem:norm_of_dual_sol_lower_than_d}
  $\|\hat p\|_* < d:$
  \prettyref{thm:constraint_vs_nonconstraint} {\rm ii)} ensures 
  $\hat p \in {\rm SOL}(D_{2,\|\hat p\|_*})$; 
  hence we must have $\|\hat p\|_* < d$ since the
  assumption $\|\hat p\|_* \geq d$ would imply, 
  by \prettyref{thm:localisation_of_the_solvers}, that
  $\hat p \in {\rm SOL}(D_{2,\|\hat p\|}) \subseteq \argmin \Phi^*(-L^* \cdot)$,
  resulting in 
  $\hat p \in {\rm SOL}(D_{1,\tau}) \cap \argmin \Phi^*(-L^* \cdot)$.
  This contradicts the relation 
  ${\rm SOL}(D_{1,\tau}) \cap \argmin \Phi^*(-L^* \cdot) = \emptyset$
  from \prettyref{thm:localisation_of_the_solvers} 
  which holds since $\tau \in (0,c)$.
\end{detail}

\begin{detail} \label{det:pr_of_main_theorem:function_f_is_well_defined}
  The function $f$, given by
  $f(\lambda) \defeq \|L \hat x\|$ with any $\hat x \in {\rm SOL}(P_{2,\lambda})$,
  $\lambda \in (0,d)$,
  is well defined,
  since ${\rm SOL}(P_{2,\lambda}) \not = \emptyset$ and since
  Theorem 
  \ref{thm:direction_of_argmin__essentially_smooth__strict_convex_sharpend} 
  ensures $\|L \hat x\| = \|L \tilde x\|$ for any other 
  $\tilde x \in {\rm SOL}(P_{2,\lambda})$: 
  For $F = \Phi$ and 
  $G(\cdot) = \lambda \|L\cdot\|$ all assumptions of Theorem
  \ref{thm:direction_of_argmin__essentially_smooth__strict_convex_sharpend}
  are fulfilled;
  note herein that $F$ and $G$ are convex functions with
  ${\rm ri}(\dom F) \cap {\rm ri}(\dom G) \not = \emptyset$ and that
  the decomposition ${\rm aff}(\dom F) = X_1 \oplus X_2$ fits to the 
  assumptions of Theorem
  \ref{thm:direction_of_argmin__essentially_smooth__strict_convex_sharpend}:
  $X_2$ is a subspace of $P[F]$ and $F$ is strictly convex on 
  ${\rm int}_{X_1}(\dom F|_{X_1})= {\rm ri}(\dom F|_{X_1}) \not = \emptyset$.
  Lastly $F$ is essentially smooth on $X_1$.
  Applying Theorem 
  \ref{thm:direction_of_argmin__essentially_smooth__strict_convex_sharpend}
  gives now 
  $
    \lambda \| L \hat x \| 
    = G(\hat x)
    = G(\tilde x) 
    = \lambda \| L \tilde x \|
  $
  and hence the claimed $\| L \hat x \| = \| L \tilde x \|$.
\end{detail}

\begin{detail} \label{det:pr_of_main_theorem:G_of_primal_solution_lower_than_c}
  $\|L \hat x\| < c:$
  \prettyref{thm:constraint_vs_nonconstraint} {\rm ii)} ensures 
  $\hat x \in {\rm SOL}(P_{1,\|L \hat x\|})$;
  so we must have $\|L \hat x\| < c$,
  since the assumption $\|L \hat x\| \geq c$ would imply, 
  by \prettyref{thm:localisation_of_the_solvers}, that
  $\hat x \in {\rm SOL}(P_{1,\|L \hat x\|}) \subseteq \argmin \Phi$,
  resulting in 
  $\hat x \in {\rm SOL}(P_{2,\lambda}) \cap \argmin \Phi$.
  This contradicts the relation 
  ${\rm SOL}(P_{2,\lambda}) \cap \argmin \Phi = \emptyset$
  from \prettyref{thm:localisation_of_the_solvers} which holds 
  since $\lambda \in (0,d)$.
\end{detail}

\begin{detail} \label{det:pr_of_main_theorem:different_taus_from_0_to_c_give_distict_solver_sets}
  The equations
  \begin{alignat*}{2}
    {\rm SOL}(P_{1,\tau}) 
    &\cap 
    {\rm SOL}(P_{1,\tau'}) 
    &&= 
    \emptyset,
  \\
    {\rm SOL}(D_{2,\lambda})
    &\cap
    {\rm SOL}(D_{2,\lambda'})
    &&=
    \emptyset
  \end{alignat*}
  hold true for all distinct $\tau, \tau' \in (0,c)$ 
  and all distinct $\lambda, \lambda' \in (0,d)$, respectively:
  If there were e.g. distinct $\lambda, \lambda' \in (0,d)$ with,
  say $\lambda < \lambda'$, such that there would be a
  $
    \hat p 
    \in 
    {\rm SOL}(D_{2,\lambda}) 
    \cap 
    {\rm SOL}(D_{2,\lambda'})
  $
  we had $\|\hat p\|_* \leq \lambda < \lambda'$ and
  $\hat p \in \argmin \Phi^*(-L^* \cdot)$ subject to 
  $\| \cdot \|_* \leq \lambda'$,  
  so that $\hat p$ would be a local minimizer of $\Phi^*(-L^* \cdot)$.
  Hence, $\hat p \in \argmin \Phi^*(-L^* \cdot)$, 
  by the convexity of $\Phi^*(-L^* \cdot)$.
  This, however, contradicts
  $
    \argmin \Phi^*(-L^* \cdot)
    \cap 
    {\rm SOL}(D_{2,\lambda'})
    =
    \emptyset
  $,
  which holds by Theorem \ref{thm:localisation_of_the_solvers}
  since $\lambda' \in (0,d)$.
  The proof of the other equation is done just analogously.
\end{detail}

\begin{detail} \label{det:pr_of_main_theorem:g_after_f_is_identity}
  For an arbitrarily chosen
  $\lambda \in (0,d)$ and
  $\lambda' \defeq g(f(\lambda))$ we have $\lambda = \lambda'$:
  Using \eqref{eq:sol_2_is_contained_in_sol_1}
  and   \eqref{eq:sol_1_is_contained_in_sol_2}
  with $\tau = f(\lambda)$ yields
  \begin{alignat*}{3}
    {\rm SOL}(P_{2,\lambda}) 
    &\subseteq
    {\rm SOL}(P_{1,f(\lambda)})
    &&\subseteq
    {\rm SOL}(P_{2,\lambda'}),
  \\
    {\rm SOL}(D_{2,\lambda}) 
    &\subseteq
    {\rm SOL}(D_{1,f(\lambda)})
    &&\subseteq
    {\rm SOL}(D_{2,\lambda'})
  \end{alignat*}
  Since ${\rm SOL}(D_{2,\lambda}) \not = \emptyset$ we must have
  $\lambda = \lambda'$, in order to avoid a contradiction to 
  \eqref{eq:sol_2_moves}.
\end{detail}

%

\begin{detail} \label{det:appendix_pr_of_relation_between_inequalities}
  Lemma 
  \ref{lem:same_blow_up_factor_gets_all_attachted_affines_strangers_out_of_ball}
  implies 
  $
    \inf_{h_1 \in X_1\cap \mathbb{S}_1, h_2 \in X_2\cap \mathbb{S}_1}
    \langle h_1, h_2 \rangle > -1
  $ by the following reason:
  By this lemma there is a constant $C \geq 1 > 0$ such that 
  $
    \frac{1}{C^2} \|h_1\|_2^2 
    \leq 
    \|h_1 + h_2\|_2^2
    =
    \|h_1\|_2^2 + \|h_2\|_2^2 + 2 \langle h_1,h_2 \rangle
  $
  for all $h_1 \in X_1$ and $h_2 \in X_2$.
  For $h_1 \in X_1 \cap \mathbb{S}_1$ 
  and $h_2 \in X_2 \cap \mathbb{S}_1$
  we obtain in particular
  $
    \langle h_1, h_2 \rangle
    \geq
    \frac12 [\frac{1}{C^2} -1 -1]
    =
    -1 + \frac{1}{2 C^2} 
    \eqdef 
    \gamma
  $,
  so that 
  $
    \inf_{h_1 \in X_1\cap \mathbb{S}_1, h_2 \in X_2\cap \mathbb{S}_1}
    \langle h_1, h_2 \rangle 
    \geq 
    \gamma 
    > 
  -1
  $
  holds indeed true.
\end{detail}

\DissVersionForMeOrHareBrainedOfficialVersion
   {\include{./suggestions_for_further_work}
    \include{./leveled_spaces_und_level_coercivity}}
{} 

    \bibliographystyle{abbrv}
    \bibliography{./dr_arbeit}
\section*{Own publication}

\begin{tabular}{p{3.3cm}p{\textwidth-2\tabcolsep-3.3cm}}
  [CiShSt2012]
    & R.~Ciak, B.~Shafei, and G.~Steidl
      Homogeneous penalizers and constraints in convex image restoration.
      {\em Journal of Mathematical Imaging and Vision},
      47(3):210--230, 2013, published online October 2012.
\end{tabular}

\pagestyle{plain}
\clearpage   

\section*{Wissenschaftlicher Werdegang}

\begin{tabular}{p{3.3cm}p{\textwidth-2\tabcolsep-3.3cm}}
  06/2001
    & {\it Hochschulreife}
\\
    & Matthias-Grünewald-Gymnasium, Tauberbischofsheim
\\[0.7ex]
  10/2001 - 12/2009
    & {\it Studium der Physik und Mathematik}
\\
    & Julius-Maximilians-Universität Würzburg
\\[0.7ex]
  12/2009
    & {\it Diplom in Mathematik}
\\
    & Diplomarbeit: Eine Variationsmethode für die Koebefunktion
\\[0.7ex]
  01/2010 - 03/2010
    & {\it Teilnahme an mehreren Kursen}
\\  
    & Rechenzentrum, Universität Würzburg
\\[0.7ex]
  04/2010 - 07/2010
    & {\it Nebenberuflicher wissenschaftlicher Mitarbeiter}
\\
    & Fakultät für Mathematik und Informatik, Universität Würzburg
\\[0.7ex]
  08/2010 - 03/2011
    & {\it Praktikum}
\\
    & Fraunhofer ITWM, Kaiserslautern
\\[0.7ex]
  ab 04/2011
    & {\it Doktorand} (Dissputation am 9. Oktober 2014)
\\
    & Fachbereich Mathematik, TU Kaiserslautern
\end{tabular}

\section*{Scientific Career}

\begin{tabular}{p{3.3cm}p{\textwidth-2\tabcolsep-3.3cm}}
  06/2001
    & {\it University entrance qualification}
\\
    & Matthias-Grünewald-Gymnasium, Tauberbischofsheim
\\[0.7ex]
  10/2001 - 12/2009
    & {\it Undergraduate studies in Physics and Mathematics}
\\
    & Julius Maximilians University of Würzburg
\\[0.7ex]
  12/2009
    & {\it Diploma in Mathematics}
\\
    & Diploma thesis: Eine Variationsmethode für die Koebefunktion
\\[0.7ex]
  01/2010 - 03/2010
    & {\it Participation in ceveral courses}
\\
    & Computer center, University of Würzburg
\\[0.7ex]
  04/2010 - 07/2010
    & {\it Teaching Assistantship}
\\
    & Departement of Mathematics and Computer Science, University of Würzburg
\\[0.7ex]
  08/2010 - 03/2011
    & {\it Internship}
\\
    & Fraunhofer ITWM, Kaiserslautern
\\[0.7ex]
  from 04/2011
    & {\it Ph. D. student}  (PhD thesis defense on October 9, 2014)
\\
    & Fachbereich Mathematik, TU Kaiserslautern

\end{tabular}

\clearpage   
\section*{Danksagung}
\thispagestyle{empty}

Ganz vielen Dank möchte ich zuallererst meiner Yoga--Lehrerin
Susanne sagen. Wer weiß, ob oder wie ich die Zeit 
auf der anderen Straßenseite gegenüber, überstanden hätte,
wenn ich nicht das Glück gehabt hätte, daß gerade sie 
im Unisport Yoga Vidya unterrichtet. \linebreak[3]
Vielen Dank Susanne für die Gelegenheiten und Hilfestellungen
auf Nahes und doch manchmal so Weites aufmerksam und achtsam(er)
zu werden.

Herzlichen Dank hier auch nochmal an Hemmi-Maria Schaar,
dafür daß sie mich auf das wertvolle Buch 
``Haben oder Sein'' aufmerksam gemacht hat,
an Jessica Borsche für ihre, in ``Amtsstuben''
nicht selbstverständlich anzutreffende, freundliche und hilfsbereite 
Art, an meine Eltern für mannigfache Unterstützungen,
besonders bei Umzügen und als ich im Krankenhaus war und für die 
Wochen danach. Hier auch vielen lieben Dank an mein Schwesterherz,
insbesondere fürs Beantworten so vieler Fragen.
%

%

Danke auch an alle, die die Zeit meines Doktorandendaseins 
bereichert haben. Besonders an Andreas, Micha, Elmi, Sarah, Maria,
Jin Yu, Sophie, Lena und Jochen. Sarah und Lena vielen Dank 
für die vielen schönen und liebevollen Karten.
Insbesondere Maria, Jochen und Andreas, sowie meinen Eltern, 
vielen Dank auch dafür, daß sie durch
ihr Sein und Sosein erkennbar machten, 
daß gewisse idealistische Grundeinstellungen
der Erosion zu trotzen vermögen auch heutzutage noch.

Danke an Martin für Aufmunterungen und gute berufliche und private 
Gespräche und für sehr viele gute Vorschläge,
an Friederike für ihre Hilfe und Vorschläge zur Verbesserung 
der Einleitung sowie für gute Gespräche beruflicher wie privater
Natur.
Beiden und meinem Schwesterherz vielen Dank, 
daß sie ihre guten Englischkenntnisse
mit mir teilten und halfen an vielen und wichtigen Stellen, 
den Text besser werden zu lassen. Für Verbesserungsvorschläge 
hier auch nochmal herzlichen Dank an meinen Freund Elmi.

Danke an Gabi, für die Momente in denen wir beide ganz Mensch waren,
und ebenfalls für die Stellen, welche ich sah und welche,
die ich nicht wahrnahm oder wahrnehme, 
an denen sie sich für mich einsetzte. Danke auch für 
die vielen Korrektur- und Verbesserungsvorschläge für die Diss.

Ein Dankeschön für die Bereitschaft meine Dissertation zu begutachten
geht jeweils an Gabi und an ihre Kollegin Frau Professorin Gerlind Plonka-Hoch.

Für ihre Hilfsbereitschaft danke ich Kirsten, Tobi, Nico 
und Jin Yu -- auch für, obwohl oder vielleicht vielmehr weil ich 
zu vielen Zeiten nicht in der Lage war ihn (immer) zu sehen, 
den Korb mit den wunderbaren Sachen.

Für Hilfe bei Latex--Fragen möchte ich vielen Danke sagen --
neben zahlreichen Bloggern, die ihr Wissen mit anderen teilten,
besonders Behrang, Tanja, Sören und Ronny.
Ihnen, Stanislav und Jan und den verbleibenden heutigen oder 
ehemaligen Gruppenmitgliedern, auch der anderer AG's
vielen Dank für gute Momente und Zeiten beruflicher wie privater Natur.
Tanja hier nochmal ein herzliches Dankeschön für ihren Hinweis auf 
die Klamm im  Karlstal bei Trippstadt.


%
%
%
%
%
%
%
%
%
%

%
%

\clearpage   
\cleardoublepage
%
%
%
\section*{Some remarks to the thesis}
\thispagestyle{plain}
Between the preceding thesis and the ``vorgelegte Dissertation'' there are 
some minor differences. When handing in the ``vorgelegte Dissertation'' 
the ``Summary'' and the ``Zusammenfassung'' were printed on separate pages outside
of the thesis, whereas here they were included inside the thesis itself.
Moreover Typos, obvious small local errors and certain inconsequencies in notation were corrected.
In particular the zerovector of the Euclidean space $\R^n$ should now everywhere be denoted
by $\zerovec$ (with exception for $n=1$ where the notation $0$ might be used).

We finally note that an electronic version of this work is available via ArXive, see\\
\url{http://arxiv.org/a/ciak_r_1}  

The reader may want to check this webpage also for Erata / Update 
(maybe additionally containing a new space concept, 
which was not yet developed enough to be included in the 
``vorgelegte Dissertation'')

\DissVersionForMeOrHareBrainedOfficialVersion{\newpage \listoftodos}{}

\end{document}